%% file: rational_quivers.tex
\documentclass[a4paper,12pt,reqno]{amsart}
\usepackage{amsmath}
\usepackage{amssymb}
\usepackage{amsfonts}
\usepackage{amscd}
\usepackage{amsthm}
\usepackage{mathrsfs}
\usepackage{tikz}
\usetikzlibrary{arrows, cd, positioning, shapes}
\tikzset{
    every node/.style={inner sep=2pt},
    arr/.style={thin, -{Stealth[scale=1.0]}},
    revarr/.style={thin, {Stealth[scale=1.0]}-},
    loopstyle/.style={thin, -{Stealth[scale=1.0]},
        loop left, min distance=1.5cm, out=135, in=225, distance=1.5cm
    }
}
\tikzcdset{arrow style=tikz}

\usepackage[all]{xy}
\usepackage[bb=boondox,bbscaled=.95,cal=boondoxo]{mathalfa}
\usepackage[numbers]{natbib}
\usepackage[a4paper, margin=2cm]{geometry}
\usepackage{hyperref}
\usepackage{longtable}

\newif\ifrqleanmarks
\rqleanmarkstrue
\makeatletter
\newcommand{\rqleanglyph}{{\normalfont\sffamily\scriptsize[Lean]}}
\newcommand{\rqleanentry}[5]{\@namedef{rqlean@#1}{{#2}{#3}{#4}{#5}}}
\newcommand{\rqleanurl}[3]{\rqLeanRepo/blob/\rqLeanCommit/\rqLeanPathPrefix#1\#L#2-L#3}
\newcommand{\rqlean@mark}[4]{\href{\rqleanurl{#2}{#3}{#4}}{\rqleanglyph}}
\newcommand{\rqlean@put}[1]{\marginpar[\hfill\rqlean@this{#1}]{\rqlean@this{#1}}}
\newcommand{\rqlean@this}[1]{\expandafter\expandafter\expandafter\rqlean@mark\csname rqlean@#1\endcsname}
\let\rqlean@pending\@empty
\newcommand{\rqlean@fire}{\ifx\rqlean@pending\@empty\else\rqlean@put{\rqlean@pending}\global\let\rqlean@pending\@empty\fi}
\newcommand{\rqleanmark}[1]{\ifrqleanmarks\@ifundefined{rqlean@#1}{\PackageWarning{rqlean}{No Lean entry for #1}}{%
  \ifvmode\gdef\rqlean@pending{#1}\everypar\expandafter{\the\everypar\rqlean@fire}\else\rqlean@put{#1}\fi}\fi
  \ignorespaces}
\newcommand{\rqleandecl}[2]{\ifx\rqLeanDocs\@empty\texttt{#1}\else\href{\rqLeanDocs find/?pattern=#2\#doc}{\texttt{#1}}\fi}
\newcommand{\rqleansource}[4]{\href{\rqleanurl{#1}{#3}{#4}}{\texttt{#2}} & \href{\rqleanurl{#1}{#3}{#4}}{\mbox{#3--#4}}}
\makeatother
\input{lean-refs}

\newcommand{\CC}{{\rm\bf C}}
\newcommand{\RR}{{\rm\bf R}}
\newcommand{\QQ}{{\rm\bf Q}}
\newcommand{\ZZ}{{\rm\bf Z}}

\newcommand{\FF}{{\rm\bf F}}
\newcommand{\EE}{{\rm\bf E}}

\DeclareMathOperator{\Spec}{\mathrm{Spec}}
\DeclareMathOperator{\Specm}{\mathrm{Specm}}

\DeclareMathOperator{\SL}{\mathrm{SL}}

\DeclareMathOperator{\SO}{\mathrm{SO}}

\DeclareMathOperator{\Aut}{\mathrm{Aut}}
\DeclareMathOperator{\End}{\mathrm{End}}
\DeclareMathOperator{\Hom}{\mathrm{Hom}}

\DeclareMathOperator{\Gal}{\mathrm {Gal}}
\DeclareMathOperator{\Sym}{\mathrm {Sym}}

\DeclareMathOperator{\res}{\mathrm{Res}}

\DeclareMathOperator{\id}{\mathrm{id}}
\DeclareMathOperator{\Lie}{\mathrm{Lie}}
\DeclareMathOperator{\HC}{\mathrm{HC}}
\DeclareMathOperator{\Rep}{\mathrm{Rep}}

\newcommand{\liesl}{{\mathfrak {sl}}}

\theoremstyle{plain}
\newtheorem{theorem}{Theorem}[section]
\newtheorem{lemma}[theorem]{Lemma}
\newtheorem{corollary}[theorem]{Corollary}
\newtheorem{proposition}[theorem]{Proposition}

\theoremstyle{remark}
\newtheorem{remark}[theorem]{Remark}
\newtheorem{definition}[theorem]{Definition}
\newtheorem{example}[theorem]{Example}

\bibpunct{[}{]}{,}{n}{}{,}

\begin{document}

    \title{Rational structures on quivers\\and a generalization of Gelfand's equivalence}

    \author{Fabian Januszewski}
    \address{Institut f\"ur Mathematik, Fakult\"at EIM, Paderborn University, Warburger Str.\ 100, 33098 Paderborn, Germany}
    \email{fabian.januszewski@math.uni-paderborn.de}

    \begin{abstract}
      We introduce the notion of rational structure on a quiver and associated representations to establish a coherent framework for studying quiver representations in separable field extensions. This notion is linked to a refinement of the notion of $K$-species, which we term \'etale $K$-species. We establish a categorical anti-equivalence between $K$-rational quivers and \'etale $K$-species, which extends to an equivalence of their respective representation categories.

      For $K$-rational quivers there is a canonical notion of base change, which suggests a corresponding notion of base change for (\'etale) $K$-species which we elaborate.

      As a primary application, we generalize Gelfand's celebrated equivalence between certain blocks of Harish-Chandra modules for $\mathrm{SL}_2(\RR)$ and representations of the Gelfand quiver to a rational setting. To this end, we define a $\QQ$-rational structure on the Gelfand quiver and its representations. A key technical tool, which we call unipotent stabilization, is developed to construct the functor from rational Harish-Chandra modules to nilpotent rational quiver representations. We prove that this functor is an equivalence. A similar result is established for the cyclic quiver. Both results hold over every field $E$ of characteristic $0$, for suitable $E$-rational structures on the two quivers; for $E=\CC$ they recover Gelfand's equivalence. In this rational framework the defining relation of the Gelfand quiver acquires a descent-theoretic meaning: its Galois orbit is stable, it descends to a single relation on the associated $\QQ$-species, and a $\QQ$-rational representation satisfies it if and only if the composite of the two arrows through the vertex $-$ is a $\QQ$-rational endomorphism of the space at the central vertex. This allows us to recast our results in the language of $\QQ$-species.
    \end{abstract}

    \maketitle

    {\tiny
      \tableofcontents
    }

    \section*{Introduction}

    Since the pioneering work of Gabriel \cite{Gabriel1972}, the representation theory of quivers is well understood over algebraically closed fields. For general non-algebraically closed fields, by the foundational work of Dlab and Ringel \cite{DlabRingel1975,DlabRingel1976}, the notion $K$-species is a suitable generalization of quivers. Within this framework the passage from a non-algebraically closed field to an algebraically closed field is very useful and at times essential in certain arguments.

    In the first half of our work, we build a coherent framework which allows for the study of quiver representations over non-algebraically closed fields on the one hand, and which on the other hand admits a suitable notion of base change for passage to (potentially algebraically closed) field extensions.

    The underlying concept is that of a rational structure on a quiver and associated representations. A central theme is to relate this notion to a suitable refinement of the well-known concept of $K$-species.

    In this context, we introduce the notion of {\em \'etale $K$-species} and establish a categorical anti-equivalence between the category of $K$-rational quivers and the category of \'etale $K$-species (Theorem \ref{thm:KrationalquiveretaleKspeciesantiequivalence}). This structural result is complemented by a corresponding equivalence between the respective categories of representations (Theorem \ref{thm:species_equivalence}).

    The natural notion of base change for $K$-rational quivers has a natural analogue for \'etale $K$-species, which in turn suggests a general analogue for the notion of base change for general $K$-species (cf.\ Remark \ref{rmk:simplealgebraspecies}). As is expected, restriction and base change for $K$-rational quivers and \'etale $K$-species form adjoint pairs of functors (Propositions \ref{prop:quiverrestrictionbasechangeadjunction} and \ref{prop:speciesrestrictionbasechangeadjunction}).

    \smallskip
    With a view towards applications in representation theory, we apply this framework to the study of rational Harish-Chandra modules for $\SL_2(\RR)$. Motivated by (expected) rationality and integrality properties of automorphic $L$-functions of (hypothetical) motivic origin, a rational theory of Harish-Chandra modules was developed by G\"unter Harder, Michael Harris and the author (cf.\ \cite{HarderRaghuram-Book}, \cite{Harris2012,Harris2020Erratum}, \cite{Januszewski18}).

    In I.~M.~Gelfand's influential 1970 ICM address \cite{Gelfand1970}, Gelfand observed that the category of nilpotent complex representations of the quiver with relation
    \begin{center}
    \begin{tikzpicture}[baseline=(current bounding box.center)]
    \node (A) {$-$};
    \node (B) [right=6em of A] {$\star$};
    \node (C) [right=6em of B] {$+$};
    \node (D) [right=2em of C] {$a_{-} b_{-} = a_{+} b_{+}$};
    \draw [arr, bend left=20] (A) to node [midway, above] {$a_-$} (B);
    \draw [arr, bend left=20] (B) to node [midway, below] {$b_-$} (A);
    \draw [arr, bend right=20] (B) to node [midway, below] {$b_+$} (C);
    \draw [arr, bend right=20] (C) to node [midway, above] {$a_+$} (B);
    \end{tikzpicture}
    \end{center}
    which now bears his name, is equivalent to the blocks $\HC_\lambda(\liesl_2,\SO(2))$ of Harish-Chandra modules for $\SL_2(\RR)$ which contain the $\sqrt{\lambda}$-dimensional irreducible representation of $\SL_2$.

    We generalize Gelfand's equivalence to arbitrary fields of characteristic $0$. To achieve this, we first define a $\QQ$-rational structure on Gelfand's quiver (and the cyclic quiver) and obtain a $\QQ$-rational quiver $\Gamma_\QQ$ (cf.\ Examples \ref{ex:gelfand} and \ref{ex:cyclic}). The associated \'etale $\QQ$-species is the following $\QQ$-species $S_{\rm Gelfand}$:
    \begin{center}
    \begin{tikzpicture}[baseline=(current bounding box.center)]
    \node (Mstar) {$\QQ=L_{\{\star\}}$};
    \node (M+)    [right=4em of Mstar] {$\quad\quad L_{\{\pm\}}=\QQ[\sqrt{-1}]$};
    \draw [arr, bend right=20] (Mstar) to node [below] {$\QQ[\sqrt{-1}]$} (M+);
    \draw [arr, bend right=20] (M+) to node [above] {$\QQ[\sqrt{-1}]$} (Mstar);
    \end{tikzpicture}
    \end{center}
    Here the $\QQ,\QQ[\sqrt{-1}]$- and $\QQ[\sqrt{-1}],\QQ$-bimodules
    \[
      {}_{\{\star\}}M_{\{\pm\}}=\QQ[\sqrt{-1}],\quad
      {}_{\{\pm\}}M_{\{\star\}}=\QQ[\sqrt{-1}],
    \]
    carry the natural bimodule structures. In the case of the cyclic quiver the bimodule structure turns out to be slightly more interesting (cf.\ Example \ref{ex:cyclic}).

    The construction of a $\QQ$-rational structure on a quiver representation associated with a Harish-Chandra module is non-trivial and generally requires a procedure which we term {\em unipotent stabilization} (see Subsection \ref{sec:unipotentstabilization}). The properties of unipotent stabilization play a central role not only in the construction itself but also in the proof of our main result, Theorem \ref{thm:EEQQequivalence}, for $\ell\geq1$. In the case $\ell=0$ of the cyclic quiver (Theorem \ref{thm:cyclicEEQQequivalence}) the stabilization is trivial: the semi-linear maps are given by complex conjugation itself.

    We establish the desired equivalences of abelian categories over $\QQ$ (Theorems \ref{thm:EEQQequivalence} and \ref{thm:cyclicEEQQequivalence}, Corollaries \ref{cor:gelfandspecies} and \ref{cor:cyclicspecies}):

    \begingroup
    \renewcommand{\thetheorem}{A}\renewcommand{\theHtheorem}{A}
    \begin{theorem}\label{thm:introA}\rqleanmark{thm:introA}
      For $\lambda=\ell^2\geq 1$, the functor $\EE_\QQ$ defines equivalences of abelian categories
      \begin{equation}
        \EE_\QQ\colon\HC_\lambda(\liesl_2,\SO(2))_\QQ \simeq \Rep^{\mathrm{nil}}(\Gamma_\QQ)
        \label{eq:HCquiverequivalence}
      \end{equation}
      and
      \begin{equation}
        \EE_\QQ\colon\HC_\lambda(\liesl_2,\SO(2))_\QQ \simeq \Rep^{\mathrm{nil}}(S_{\rm Gelfand},R_{\rm Gelfand}).
        \label{eq:HCspeciesequivalence}
      \end{equation}
      For $\lambda=0$, the analogous functor defines equivalences
      \[
        \EE_\QQ\colon\HC_0(\liesl_2,\SO(2))_\QQ \simeq \Rep^{\mathrm{nil}}(\Gamma_\QQ) \simeq \Rep^{\mathrm{nil}}(S(\Gamma_\QQ)),
      \]
      where $\Gamma_\QQ$ denotes the $\QQ$-rational cyclic quiver.
    \end{theorem}
    \addtocounter{theorem}{-1}
    \endgroup

    Here, $\HC_\lambda(\liesl_2,\SO(2))_\QQ$ denotes the block of $\QQ$-rational Harish-Chandra modules with generalized infinitesimal character $\lambda$ containing the irreducible $\ell$-dimensional representation of $\SL_2$. On the right hand side of \eqref{eq:HCquiverequivalence}, $\Rep^{\mathrm{nil}}(\Gamma_\QQ)$ is the category of nilpotent, finite-dimensional, $\QQ$-rational representations of the Gelfand quiver, endowed with a suitable $\QQ$-rational structure (cf. Example \ref{ex:gelfand}).

    The species form \eqref{eq:HCspeciesequivalence} follows by combining \eqref{eq:HCquiverequivalence} with Theorem \ref{thm:species_equivalence} on the equivalence of $\QQ$-rational quiver representations and \'etale $\QQ$-species representations, which preserves nilpotence and carries the defining relation of the Gelfand quiver to the relation $R_{\rm Gelfand}$ (Corollary \ref{cor:nilpotentrelations} and Example \ref{ex:gelfandspeciesrelation}).

    Nothing in these constructions is specific to $\QQ$. For a field $E$ of characteristic $0$ we put $L_E:=E[\sqrt{-1}]$ and $G_E:=\Gal(L_E/E)$ (Definition \ref{def:sqrtminusone}).
    Our second main result is the following (Theorems \ref{thm:EEequivalence} and \ref{thm:cyclicEEequivalence}).

    \begingroup
    \renewcommand{\thetheorem}{B}\renewcommand{\theHtheorem}{B}
    \begin{theorem}\label{thm:introB}\rqleanmark{thm:introB}
      Let $E$ be a field of characteristic $0$, let $\ell\geq0$ and $\lambda=\ell^2$. The functor $\EE_E$ defines an equivalence of abelian categories
      \[
        \EE_E\colon\HC_\lambda(\liesl_2,\SO(2))_E \simeq \Rep^{\mathrm{nil}}(\Gamma_E),
      \]
      where $\Gamma_E$ is the $E$-rational Gelfand quiver if $\ell\geq1$ and the $E$-rational cyclic quiver if $\ell=0$ (Definition \ref{def:GammaE}).
    \end{theorem}
    \addtocounter{theorem}{-1}
    \endgroup

    The group $G_E$ has order $1$ or $2$. If $\sqrt{-1}\in E$, then $G_E$ is trivial and $\Gamma_E$ is split over $E$. Then $\Rep^{\mathrm{nil}}(\Gamma_E)$ is the category of nilpotent finite-dimensional representations over $E$ of the Gelfand quiver satisfying its relation, respectively of the cyclic quiver. We prove Theorem \ref{thm:introB} directly over $E$, and we obtain Theorem \ref{thm:introA} as its case $E=\QQ$. This is possible since the infinitesimal character $\lambda=\ell^2$ is rational and all scalars occurring in the construction are non-zero rational numbers. The rational structures are semi-linear with respect to $G_E$, which for $E=\QQ$ is $\Gal(\QQ[\sqrt{-1}]/\QQ)$. For $E=\CC$, Theorem \ref{thm:introB} recovers Gelfand's equivalence, and for $E=\RR$ it gives its real form (Remark \ref{rmk:classicalcases}). The equivalences $\EE_E$ are compatible with base change along every extension $E''/E$, and base change induces isomorphisms $E''\otimes_E\Hom\cong\Hom$ on both sides (Proposition \ref{prop:basechange}). Base change is, however, in general not essentially surjective (Remark \ref{rmk:basechangeES}). The same method gives equivalences for twisted forms: for every quadratic extension $D'/E$, the category $\HC_\lambda(\mathfrak{g},K_{D'/E})_E$ of Definition \ref{def:HClambdaDE} is equivalent to $\Rep^{\mathrm{nil}}(\Gamma_{D'/E})$ (Theorem \ref{thm:twistedEEequivalence}). For $E\subseteq\RR$ these include the categories for all $E$-forms of $\SO(2)$ (Remark \ref{rmk:formsofSO2}).

    The species form \eqref{eq:HCspeciesequivalence} generalizes to the fields $E$ of characteristic $0$ with $\sqrt{-1}\not\in E$ (Theorem \ref{thm:speciesE}): then $\QQ[\sqrt{-1}]\otimes_\QQ E=L_E$ is a field, the species $S_{\rm Gelfand}\otimes_\QQ E$ has the shape of $S_{\rm Gelfand}$ with $(\QQ,\QQ[\sqrt{-1}])$ replaced by $(E,L_E)$, and for $\lambda=\ell^2\geq1$ we have
    \[
      \EE_E\colon\HC_\lambda(\liesl_2,\SO(2))_E \simeq \Rep^{\mathrm{nil}}(S_{\rm Gelfand}\otimes_\QQ E,R_E),
      \qquad
      R_E:=R_{\rm Gelfand}\otimes_\QQ E=E\cdot\sqrt{-1}.
    \]
    For $\ell=0$, the species form holds over every $E$. If $\sqrt{-1}\in E$ and $\ell\geq1$, the species $S_{\rm Gelfand}\otimes_\QQ E$ becomes the split species of the split Gelfand quiver, base change being a non-trivial operation for (\'etale) $K$-species (cf.\ Definition \ref{def:basechangeforspecies}). Then no relation in the sense of Definition \ref{def:speciesrelations} expresses $a_{-} b_{-} = a_{+} b_{+}$ (Remark \ref{rmk:splitspecies}), and the quiver formulation $\Rep^{\mathrm{nil}}(\Gamma_E)$ is the appropriate one.

    The rational structure gives the defining relation of the Gelfand quiver a descent-theoretic meaning. The non-trivial element $c$ of $\Gal(\QQ[\sqrt{-1}]/\QQ)$ carries $a_{-} b_{-} = a_{+} b_{+}$ to $a_{+} b_{+} = a_{-} b_{-}$, so the relation set is Galois-stable, and on every $\QQ$-rational representation the two composites are conjugate under the semi-linear automorphism $\varphi_{\star,c}$ at the central vertex. Consequently, a $\QQ$-rational representation satisfies $a_{-} b_{-} = a_{+} b_{+}$ if and only if the endomorphism $\phi_{a_-}\circ\phi_{b_-}$ of the space at $\star$ commutes with $\varphi_{\star,c}$, i.e., is $\QQ$-rational (cf.\ Example \ref{ex:gelfand}). On the associated \'etale $\QQ$-species this condition is the single relation $R_{\rm Gelfand}$ (cf.\ Example \ref{ex:gelfandspeciesrelation}). The relation is not a consequence of the rational structure.
    
    Finally, in Section \ref{sec:examples}, we illustrate our main equivalence with a handful of fundamental examples, including finite-dimensional representations, discrete series, and principal series representations, and describe their corresponding quiver and species representations.

\subsection*{Formalization}

The results of this paper listed in Appendix~\ref{sec:lean}, \rqleannumformalized{} labelled statements including Theorems~\ref{thm:introA} and~\ref{thm:introB}, are formalized in the proof assistant Lean~4~\cite{RQLean}, on top of the library mathlib.\ifrqleanmarks{} Each of them carries the mark~\rqleanglyph, which links to the corresponding Lean declaration.\fi{} The formalization uses Lean~\rqLeanToolchain{} and mathlib at commit \texttt{\rqMathlibCommitShort}. Its proofs contain no unproved steps and use no axioms other than the three standard axioms of Lean's type theory: propositional extensionality (\texttt{propext}), the axiom of choice (\texttt{Classical.choice}) and the soundness of quotients (\texttt{Quot.sound}). For every declaration of Appendix~\ref{sec:lean} the command \texttt{\#print axioms} lists no other axiom. The repository, the version cited here and its archived copy are given in~\cite{RQLean}; the package is built by \texttt{lake exe cache get} followed by \texttt{lake build} in its directory \texttt{lean}. The formalization was produced with the assistance of Claude Code (Anthropic); see the declaration on generative AI at the end of the paper.

The Lean statements follow the statements of this paper, with the following modelling choices. The category $\HC_\lambda(\mathfrak{g},K)_E$ of Definition~\ref{def:HClambdaE} is formalized through the weight decomposition of $L_E\otimes_E M$: an object is an $E$-vector space $M$ together with a decomposition $L_E\otimes_E M=\bigoplus_{n\in\ZZ}M_n$ and operators $H$, $X$, $Y$ which satisfy the relations of $\mathfrak{sl}_2$ and are compatible with the semi-linear action of $G_E$, where $H$ acts on $M_n$ by $n$ and $X$, $Y$ shift the weights by $\pm2$, subject to the conditions of Definition~\ref{def:HClambdaE} on finite generation, admissibility, the generalized infinitesimal character and the action of $-1_2$. The comparison with $(\mathfrak{g}_E,K_E)$-modules is formalized for $E=\CC$ (Remark~\ref{rmk:HClambdaC}), not for general $E$. The abelian structure of $\Rep^{\mathrm{nil}}(\Gamma_E)$ is the one transported along $\EE_E$ (Remark~\ref{rmk:EEabelian}). The main equivalences are stated for the functors constructed in this paper, not merely as the existence of an equivalence. Where a Lean statement differs in form from the printed one, or is weaker, the documentation of the Lean declaration records the difference; for instance, Lemma~\ref{lem:quiverbasechangecompat} is formalized for finite $L/K$ only. The following labelled statements are not formalized: \rqleanunformalized.

\medskip    
\noindent{\bf Acknowledgements.}
The author thanks Igor Burban and William Crawley-Boevey for fruitful discussions. The author acknowledges support by the Deutsche Forschungsgemeinschaft (DFG, German Research Foundation) -- SFB-TRR 358/1 2023 -- 491392403.


\subsection*{Notation}

All homomorphisms of commutative rings are assumed to be unital. Let $K$ denote a field and let $V$ and $W$ denote right and left $K$-vector spaces respectively. Then $V\otimes_K W$ denotes the respective tensor product. If $L$ is another field and if $V$ carries an $L,K$-bimodule structure, then $V\otimes_KW$ is canonically a left $L$-vector space; likewise, if $W$ carries a $K,L$-bimodule structure, then $V\otimes_KW$ is canonically a right $L$-vector space. Moreover, if $\sigma\colon K\to L$ is a ring homomorphism, then for any right vector space $V$ and any $L$-vector space $W$ we consider the twisted tensor product $V\otimes_{K,\sigma}W$, where for all $a\in K$ and any $v\in V$, $w\in W$:
\[
va\otimes w=v\otimes\sigma(a)w.
\]
Similarly, if $V$ is a right $L$-vector space, and $W$ a left $K$-vector space, we understand $V\otimes_{\sigma,K}W$ via the analogous relation
\[
v\otimes aw=v\sigma(a)\otimes w.
\]
For a finite or infinite Galois extension $L/K$ of fields we write $\Gal(L/K)$ for the Galois group with the profinite topology. I.\,e.\ if $L/K$ is finite, then $\Gal(L/K)$ is discrete. All $\Gal(L/K)$-modules are considered continuous with the target being topologized discretely. The same convention applies to all other actions of $\Gal(L/K)$ occurring in this paper, in particular to semi-linear actions on vector spaces and to actions on quivers: they are always assumed continuous with the target being topologized discretely. For finite $L/K$ this is no condition.

\section{Rational structures on quivers}\label{sec:rationalstructures}

Since Galois descent plays a central role, we recall some well-known concepts and statements.

\subsection{Galois descent}

Let $L/K$ be a finite or infinite Galois extension with Galois group $\Gal(L/K)$. For $\sigma\in\Gal(L/K)$ and $\alpha\in L$, we sometimes write $\alpha^\sigma=\sigma(\alpha)$, even though it is strictly speaking a left action.

\label{def:homgaloisaction}\rqleanmark{def:homgaloisaction}Let $V, W$ be $L$-vector spaces, each endowed with a semi-linear action of $\Gal(L/K)$ in the sense introduced below. The Galois group $\Gal(L/K)$ then acts on the set of $L$-linear maps $\Hom_L(V,W)$ via
\[
\forall\sigma\in\Gal(L/K),\,\phi\in\Hom_L(V,W)\colon\quad \phi^\sigma:=\sigma\circ\phi\circ\sigma^{-1}.
\]
If the actions on $V$ and $W$ are continuous and $V$ is finite-dimensional over $L$, then this action is continuous with respect to the discrete topology on $\Hom_L(V,W)$.

\label{def:semilinear}\rqleanmark{def:semilinear}We call a map $\phi\colon V\to W$ {\em semi-linear} with respect to a given $\sigma\in\Gal(L/K)$, or simply $\sigma$-linear, if for all $v\in V$ and all $\alpha\in L$:
\[
\phi(\alpha\cdot v)=\sigma(\alpha)\phi(v).
\]

\label{def:kstructure}\rqleanmark{def:kstructure}Every $L$-vector space $V$ can be considered as a $K$-vector space via the embedding $K\to L$. A {\em $K$-structure} on an $L$-vector space $V$ is a group homomorphism
\[
\varphi\colon\Gal(L/K)\to\Aut_K(V),\;\sigma\mapsto\varphi_\sigma
\]
with the property that for each $\sigma\in\Gal(L/K)$, the map
\[
\varphi_\sigma\colon V\to V
\]
is $\sigma$-linear, and that, in accordance with our general convention, the resulting action
\[
\Gal(L/K)\times V\to V,\quad(\sigma,v)\mapsto\varphi_\sigma(v),
\]
is continuous, $V$ being topologized discretely. Explicitly, continuity means that every $v\in V$ is fixed by $\varphi_\sigma$ for all $\sigma$ in an open subgroup of $\Gal(L/K)$; for finite $L/K$ the continuity condition is vacuous. In the following, we identify $\varphi$ with the collection $\varphi_\bullet=(\varphi_\sigma)_{\sigma\in\Gal(L/K)}$.

\label{def:krationalmap}\rqleanmark{def:krationalmap}We call a map $\phi\in\Hom_L(V,W)$ between two $L$-vector spaces with $K$-structure {\em $K$-rational} if $\phi$ commutes with the two given actions of $\Gal(L/K)$ on $V$ and $W$.

\label{prop:canonicalkstructure}\rqleanmark{prop:canonicalkstructure}If $V$ is a $K$-vector space, then $L\otimes_K V$ carries a canonical $K$-structure, which is explicitly given by
\[
\varphi_\sigma(\alpha\otimes v)=\sigma(\alpha)\otimes v,\quad \sigma\in\Gal(L/K),\,\alpha\in L,\,v\in V.
\]
Every $K$-linear map $\phi\colon V\to W$ then induces a $K$-rational map ${\bf1}\otimes\phi\colon L\otimes V\to L\otimes W$.

The following result is well known:

\begin{theorem}[Galois descent]\label{thm:galoisdescent}\rqleanmark{thm:galoisdescent}
  Let $L/K$ be a finite or infinite Galois extension. The category $\rm{Vec}_K$ of $K$-vector spaces is equivalent, via the functor
  \[
  V\;\mapsto\; \left(L\otimes_K V,(\varphi_\sigma)_{\sigma}\right)
  \]
  to the category of $L$-vector spaces with a $K$-structure --- the defining semi-linear action of $\Gal(L/K)$ being continuous, as required above --- and $K$-rational maps as morphisms. A quasi-inverse functor is given by
  \[
  \left(V,(\varphi_\sigma)_\sigma\right)\;\mapsto\;V^{\Gal(L/K)}:=\{v\in V\mid\forall\sigma\in\Gal(L/K)\colon \varphi_\sigma(v)=v\}.
  \]
\end{theorem}

\begin{remark}\label{rmk:continuitynecessity}
  For infinite $L/K$ the continuity of the semi-linear action, which is part of our definition of a $K$-structure, is an essential hypothesis in Theorem \ref{thm:galoisdescent} and not a mere convention: a non-continuous semi-linear action of $\Gal(L/K)$ on an $L$-vector space $V$ need not arise via base change from a $K$-vector space, and the canonical map $L\otimes_K V^{\Gal(L/K)}\to V$ may fail to be an isomorphism. For finite $L/K$ every semi-linear action is continuous.
\end{remark}

\subsection{Rational structures on quivers}

\begin{definition}[Quiver (Gabriel)]\label{def:quiver}\rqleanmark{def:quiver}
  A (classical) quiver is a quadruple $\Gamma=(V,E,s,t)$ consisting of a finite set of vertices $V$, a finite set of edges $E$, and two maps $s,t\colon E\to V$, which assign to each edge $e\in E$ a source $s(e)\in V$ and a target $t(e)\in V$.
\end{definition}

\begin{definition}[$K$-rational (\'etale) quivers]\label{def:rationalquivers}\rqleanmark{def:rationalquivers}
  An {\em (\'etale) $K$-rational structure} on a quiver $\Gamma=(V,E,s,t)$ is a group homomorphism $\rho\colon\Gal(L/K)\to\Aut(\Gamma)$ into the automorphism group of $\Gamma$ for a Galois extension $L/K$ with Galois group $\Gal(L/K)$, such that the actions of $\Gal(L/K)$ on $V$ and $E$ induced by $\rho$ are continuous when $V$ and $E$ are considered as discrete sets. We write $\Gamma_K:=(\Gamma,\rho)$ for the quiver $\Gamma$ endowed with the rational structure $\rho$. We say that $\Gamma_K$ is an {\em (\'etale) quiver over $K$}.
\end{definition}

\begin{definition}[Split rational quivers]\label{def:splitrationalquivers}\rqleanmark{def:splitrationalquivers}
  We call a quiver $\Gamma_K=(\Gamma,\rho)$ over $K$ {\em constant} or {\em split over $K$} if $\rho$ is the trivial group homomorphism, i.\,e.\ if the $\Gal(L/K)$ actions on $V$ and $E$ are both trivial.
\end{definition}

\begin{remark}\label{rmk:galoissetdiagram}\rqleanmark{rmk:galoissetdiagram}
A $K$-rational quiver is a diagram in the category of finite $\Gal(L/K)$-sets:
   \begin{center}
   \begin{tikzpicture}[>=Stealth]
    \node (E) at (0,1) {$E$};
    \node (V_left) at (-1.3,0) {$V$};
    \node (V_right) at (1.3,0) {$V$};
    \draw[->] (E) -- node[above] {$s$} (V_left);
    \draw[->] (E) -- node[above] {$t$} (V_right);
    \end{tikzpicture}
   \end{center}
\end{remark}

\begin{remark}\label{rmk:rationalstructureexplicit}
Specifically, a $K$-rational structure $\rho$ on a quiver $\Gamma=(V,E,s,t)$ is given by:
\begin{itemize}
\item[(i)] A (continuous left) $\Gal(L/K)$-action
  \[\rho_V\colon\Gal(L/K)\to\Sym(V)\] on the set of vertices $V$,
  \item[(ii)] a (continuous left) $\Gal(L/K)$-action \[\rho_E\colon\Gal(L/K)\to\Sym(E)\] on the set of edges,
\end{itemize}
such that $s$ and $t$ are homomorphisms of $\Gal(L/K)$-sets, i.\,e., the following compatibilities must be satisfied: For all $\sigma\in\Gal(L/K)$ and all edges $e\in E$ it holds that
\begin{equation}
  s(\sigma e)=\sigma s(e)\quad\text{and}\quad t(\sigma e)=\sigma t(e).
  \label{eq:stgaloisequivariance}
\end{equation}
In other words: We reformulated Gabriel's original definition in the category of (continuous discrete) $\Gal(L/K)$-sets.
\end{remark}

\begin{remark}[Relations]\label{def:galoisrelations}\rqleanmark{def:galoisrelations}
  If $\Gamma$ is a quiver with relations, we require that the Galois action respects them, i.e., that each element of the Galois group maps a given relation to a given relation or, equivalently, that the image of $\rho$ lies in the automorphism group of the quiver $\Gamma$ with relations.
\end{remark}

\begin{remark}\label{rmk:finitequotient}\rqleanmark{rmk:finitequotient}
  Let $K^{\rm sep}$ denote a separable algebraic closure of $K$. Then any of the Galois groups $\Gal(L/K)$ occurring in Definition \ref{def:rationalquivers} may be (non-canonically) considered as a quotient of the absolute group $\Gal(K^{\rm sep}/K)$. Therefore all actions in Definition \ref{def:rationalquivers} arise from continuous actions of $\Gal(K^{\rm sep}/K)$ on $E$ and $V$. $E$ and $V$ being finite, any $K$-rational structure on the quiver $\Gamma$ factors over a finite quotient $\Gal(K^{\rm sep}/K)\to\Gal(L_0/K)$ with $L_0/K$ a finite Galois extension. Therefore we may often assume that $L/K$ is a finite Galois extension without loss of generality.

  That being said, in certain situations, such as when considering categorical aspects of $K$-rational quivers, the discussion simplifies if we consider $\Gal(K^{\rm sep}/K)$-actions, as is the case in the following definition.
\end{remark}

\begin{definition}[Rational homomorphisms of quivers]\label{def:rationalquiverhom}\rqleanmark{def:rationalquiverhom}
  A homomorphism $\phi\colon\Gamma_K\to \Gamma_K'$ of two quivers $\Gamma_K$ and $\Gamma_K'$ over $K$ is a homomorphism of quivers $\phi\colon\Gamma\to\Gamma'$ which commutes with the $\Gal(K^{\rm sep}/K)$-Galois actions on $\Gamma_K$ and $\Gamma_K'$.
\end{definition}

\begin{definition}[Base change]\label{def:quiverbasechange}\rqleanmark{def:quiverbasechange}
  For any quiver $\Gamma_K=(\Gamma,\rho)$ with rational structure given by an action of $\Gal(L/K)$, we have for any intermediate field $L/E/K$, $E/K$ finite, the $E$-rational quiver
  \[
  \Gamma_E=E\otimes_K\Gamma_K:=(\Gamma,\rho|_{\Gal(L/E)}),
  \]
  which we call the {\em base change} of $\Gamma_K$ to $E$. The restricted action $\rho|_{\Gal(L/E)}$ is again continuous, $\Gal(L/E)$ being endowed with the subspace topology, so that $\Gamma_E$ is indeed an $E$-rational quiver.
\end{definition}

\begin{remark}\label{rmk:honeclassification}
  If $\Gamma$ is a quiver and $L/K$ is a Galois extension, then all isomorphism classes of $K$-rational structures on $\Gamma$ can be classified using Galois cohomology
  \[
  H^1(\Gal(K^{\rm sep}/K);\Aut(\Gamma)).
  \]
  Here, we can assume without loss of generality that $\Aut(\Gamma)$ is endowed with the trivial Galois action, and cohomology is understood as continuous cohomology of the profinite group $\Gal(K^{\rm sep}/K)$ with discrete coefficients, in accordance with the continuity required in Definition \ref{def:rationalquivers}. If $\Aut(\Gamma)$ is abelian, it follows in particular that
  \[
  H^1(\Gal(K^{\rm sep}/K);\Aut(\Gamma))=\Hom_{\rm cont}(\Gal(K^{\rm sep}/K),\Aut(\Gamma)),
  \]
  the group of continuous homomorphisms.
\end{remark}

\begin{definition}[Restriction]\label{def:quiverrestriction}\rqleanmark{def:quiverrestriction}
  For any intermediate field $L/E/K$, $E/K$ finite, and any $E$-rational quiver $\Gamma_E=(\Gamma,\rho)$, with $\Gamma=(V,E_\Gamma,s,t)$ (we write $E_\Gamma$ for the edge set here, $E$ denoting the intermediate field) we consider the quiver
  \begin{align*}
    \res_{E/K}\Gamma_E=(&\Gal(L/K)\times_{\Gal(L/E)}V,\\
    &\Gal(L/K)\times_{\Gal(L/E)}E_\Gamma,\\
    &\Gal(L/K)\times_{\Gal(L/E)}s,\\
    &\Gal(L/K)\times_{\Gal(L/E)}t\,)
  \end{align*}
  which we call the {\em restriction} of $\Gamma_E$ to $K$. Since $E/K$ is finite, $\Gal(L/E)$ is an open subgroup of $\Gal(L/K)$, and the induced actions of $\Gal(L/K)$ on the sets $\Gal(L/K)\times_{\Gal(L/E)}V$ and $\Gal(L/K)\times_{\Gal(L/E)}E_\Gamma$ are again continuous, so that $\res_{E/K}\Gamma_E$ is indeed a $K$-rational quiver.
\end{definition}

\begin{proposition}\label{prop:quiverrestrictionbasechangeadjunction}\rqleanmark{prop:quiverrestrictionbasechangeadjunction}
  For any intermediate field $L/E/K$, $E/K$ finite, and any $E$-rational quiver $\Gamma_E=(\Gamma,\rho)$, and any $K$-rational quiver $\Gamma_K'$, we have a natural isomorphism
  \begin{equation}
    \Hom_{E}(\Gamma_E,E\otimes_K\Gamma_K')\cong
    \Hom_{K}(\res_{E/K}\Gamma_E,\Gamma_K').
  \end{equation}
\end{proposition}

\begin{proof}
  This is an elementary application of the universal property of the functor
  \[
  \Gal(L/K)\times_{\Gal(L/E)}(-)\colon \{\text{finite $\Gal(L/E)$-sets}\}\to\{\text{finite $\Gal(L/K)$-sets}\},
  \]
  which is left adjoint to the forgetful functor
  \[
  \{\text{finite $\Gal(L/K)$-sets}\}\to\{\text{finite $\Gal(L/E)$-sets}\}.
  \]
\end{proof}

\begin{remark}\label{rmk:finitefields}
  We remark that quivers over finite fields $\FF_q$ together with an automorphism (of order $2$) have been previously studied in \cite{DengDuParshallWang2008} in the context of quantum groups.
\end{remark}

\subsection{Rational quiver representations}\label{sec:rationalquiverrepresentations}

  \begin{definition}[Representations of rational quivers]\label{def:rationalrepresentations}\rqleanmark{def:rationalrepresentations}
    Let $\Gamma_K$ be a quiver over $K$ with respect to a Galois extension $L/K$. A ($K$-rational) {\em representation} $M_K$ of $\Gamma_K$ consists of:
    \begin{itemize}
    \item[(i)]
      A family $M=(M(v))_{v\in V}$ of $L$-vector spaces.
    \item[(ii)] Semi-linear isomorphisms $\varphi_{v,\sigma}\colon M(v)\to M(\sigma v)$ for each $v\in V$ and $\sigma\in\Gal(L/K)$ such that the cocycle condition
      \begin{equation}
        \varphi_{\tau v,\sigma}\circ\varphi_{v,\tau}=\varphi_{v,\sigma\tau}
        \label{eq:semi-linearactiononrep}
      \end{equation}
      is satisfied, and such that, in accordance with our general convention, the resulting action of $\Gal(L/K)$ on the disjoint union $\coprod_{v\in V}M(v)$, considered as a discrete set, is continuous. As always, the continuity condition is automatic if $L/K$ is finite.
    \item[(iii)] A classical quiver representation of $\Gamma$ on $M$, i.e., for each edge $e\in E$, an $L$-linear map $\phi_e\colon M(s(e))\to M(t(e))$ is given, which is compatible with the Galois action in the following sense:
      For each edge $e\in E$ and each $\sigma\in\Gal(L/K)$, we have
      \begin{equation}
        \phi_{\sigma e}\circ\varphi_{s(e),\sigma}=\varphi_{t(e),\sigma}\circ\phi_e.
        \label{eq:koecheractiononrep}
      \end{equation}
    \end{itemize}
  \end{definition}

  \begin{definition}[Homomorphisms of rational quiver representations]\label{def:rationalrepresentationhom}\rqleanmark{def:rationalrepresentationhom}
    A {\em homomorphism} of $K$-rational representations $M=(M(v))_v$, $N=(N(v))_v$ is a classical homomorphism $\left(\psi_v\colon M(v)\to N(v)\right)_v$ over $L$, which is compatible with the rational structures $\varphi^M$ on $M$ and $\varphi^N$ on $N$: For all $v\in V$ and all $\sigma\in\Gal(L/K)$, we have
    \begin{equation}
      \psi_{\sigma v}\circ\varphi^M_{v,\sigma}=\varphi^N_{v,\sigma}\circ\psi_v.
      \label{eq:koecheractiononrephom}
    \end{equation}
  \end{definition}
  
  \begin{remark}\label{rmk:semilinearity}
    Semi-linearity in point (ii) means specifically
    \begin{equation}
      \varphi_{v,\sigma}(a\cdot m)=\sigma(a)\cdot\varphi_{v,\sigma}(m)
    \end{equation}
    for all $v\in V$, $m\in M(v)$, $a\in L$ and $\sigma\in\Gal(L/K)$.
  \end{remark}

  \begin{definition}\label{def:repcategories}\rqleanmark{def:repcategories}
    For a $K$-rational quiver $\Gamma_K$, we define $\Rep(\Gamma_K)$ as the category of finite-dimensional $K$-rational representations of $\Gamma_K$. Moreover we define the full subcategory $\Rep^{\rm nil}(\Gamma_K)$ in $\Rep(\Gamma_K)$ consisting of finite-dimensional {\em nilpotent} $K$-rational representations of $\Gamma_K$. Here a representation $M_K$ is called {\em nilpotent} if there is an integer $n\geq0$ such that for every path $p=e_m\cdots e_1$ in $\Gamma$ of length $m\geq n$ the composite $\phi_p:=\phi_{e_m}\circ\dots\circ\phi_{e_1}$ vanishes; this is a property of the underlying classical representation over $L$.
  \end{definition}

  \begin{proposition}[Base change for representations]\label{prop:repbasechange}\rqleanmark{prop:repbasechange}
    Let $M_K=((M(v))_v,(\varphi_{v,\sigma})_{v,\sigma},(\phi_e)_e)$ be a (nilpotent) representation of a $K$-rational quiver $\Gamma_K$ and let $L/E/K$ be an intermediate field, $E/K$ finite. Then
    \[
    M_E:=E\otimes_KM_K:=\left((M(v))_v,(\varphi_{v,\sigma})_{v\in V,\sigma\in\Gal(L/E)},(\phi_e)_e\right),
    \]
    obtained by restricting the semi-linear action to $\Gal(L/E)$, is a (nilpotent) representation of the base change $E\otimes_K\Gamma_K$. Base change of (nilpotent) representations is functorial. If $M_K$ is finite-dimensional, the notation is justified by Galois descent: the canonical map $\alpha\otimes m\mapsto\alpha m$ is an isomorphism of $E$-vector spaces
    \[
    E\otimes_K\Bigl(\bigoplus_{v\in V}M(v)\Bigr)^{\Gal(L/K)}\;\cong\;\Bigl(\bigoplus_{v\in V}M(v)\Bigr)^{\Gal(L/E)}
    \]
    between the base change to $E$ of the $K$-vector space underlying $M_K$ and the $E$-vector space underlying $M_E$.
  \end{proposition}

  \begin{proof}
    The conditions of Definition \ref{def:rationalrepresentations} for $\Gal(L/E)$ are restrictions of those for $\Gal(L/K)$, and nilpotence does not involve the Galois action. Morphisms of $K$-rational representations are in particular compatible with the restricted actions. In particular, the continuity condition of Definition \ref{def:rationalrepresentations} holds for $M_E$, the relevant action of $\Gal(L/E)$ being the restriction of the continuous action attached to $M_K$ to a subgroup carrying the subspace topology. The last assertion follows from Theorem \ref{thm:galoisdescent}, applied to $L/K$ and to $L/E$, to the finite-dimensional $L$-vector space $\bigoplus_vM(v)$ with its semi-linear action.
  \end{proof}

  \begin{remark}\label{rmk:repstack}
    Considering everything over the absolute Galois group $\Gal(K^{\rm sep}/K)$, we see that $(\Rep(\Gamma_E))_{K\subseteq E\subseteq K^{\rm sep}}$ and $(\Rep^{\rm nil}(\Gamma_E))_{K\subseteq E\subseteq K^{\rm sep}}$ define stacks in the \'etale topology on $K$.
  \end{remark}

\subsection{Base change for Hom}

  \begin{proposition}\label{prop:homdescent}\rqleanmark{prop:homdescent}
    Let $\Gamma_K$ be a quiver with a $K$-rational structure. For $K$-rational representations $M_K$ and $N_K$ of $\Gamma_K$, we have for any intermediate field $L/E/K$, $E/K$ finite, a canonical isomorphism
    \begin{equation}
      E\otimes_K\Hom_{\Gamma_K}(M_K,N_K) \cong \Hom_{\Gamma_E}(M_E,N_E)
    \end{equation}
    of finite-dimensional $E$-vector spaces.
  \end{proposition}

  \begin{proof}
    We first consider the case $E=L$. We aim to show the existence of a canonical isomorphism of $L$-vector spaces
    \[
    L\otimes_K\Hom_{\Gamma_K}(M_K,N_K) \cong \Hom_{\Gamma}(M, N),
    \]
    where $\Hom_{\Gamma}(M, N)$ denotes the space of homomorphisms of the underlying classical quiver representations over $L$.
    
    Let $G = \Gal(L/K)$. We consider the $L$-vector space $W := \Hom_{\Gamma}(M, N)$. An element $\psi \in W$ is a collection of $L$-linear maps $\psi = (\psi_v)_{v\in V}$ such that for every edge $e \in E$, we have $\psi_{t(e)}\circ \phi_{e,M} = \phi_{e,N}\circ\psi_{s(e)}$.
    
    We define a semilinear action of $G$ on $W$. For $\sigma \in G$ and $\psi \in W$, let $\psi^\sigma$ be the homomorphism with components
    \[
    (\psi^\sigma)_v := \varphi_{N,v,\sigma} \circ \psi_{\sigma^{-1}v} \circ \varphi_{M,\sigma^{-1}v,\sigma}^{-1}.
    \]
    This defines a map $(\psi^\sigma)_v\colon M(v) \to N(v)$. It is $L$-linear, because for any $\alpha \in L$, the $\sigma$-linearity of $\varphi_{N,v,\sigma}$ and $\sigma^{-1}$-linearity of $\varphi_{M,\sigma^{-1}v,\sigma}^{-1}$ ensure that for all $m\in M(v)$:
    \begin{align*}
      (\psi^\sigma)_v(\alpha\cdot m) &= \varphi_{N,v,\sigma} \circ \psi_{\sigma^{-1}v} \circ \varphi_{M,\sigma^{-1}v,\sigma}^{-1}(\alpha\cdot m) \\
      &= \varphi_{N,v,\sigma} \left(\psi_{\sigma^{-1}v} \left(\sigma^{-1}(\alpha)\cdot\varphi_{M,\sigma^{-1}v,\sigma}^{-1}(m)\right)\right) \\
      &= \varphi_{N,v,\sigma} \left(\sigma^{-1}(\alpha)\cdot(\psi_{\sigma^{-1}v} \circ \varphi_{M,\sigma^{-1}v,\sigma}^{-1})(m)\right) \\
      &= \sigma(\sigma^{-1}(\alpha)) \cdot (\psi^\sigma)_v(m)\\
      &= \alpha \cdot (\psi^\sigma)_v(m).
    \end{align*}
    One can verify that $\psi^\sigma$ respects the edge relations and thus $\psi^\sigma \in W$. The map $\psi \mapsto \psi^\sigma$ defines a semilinear action of $G$ on the $L$-vector space $W$. This action is continuous for the discrete topology on $W$: since $M_K$ and $N_K$ are finite-dimensional and the semi-linear actions on $M$ and $N$ are continuous, any given $\psi\in W$ is determined by the images of finitely many basis vectors, and all vertices, vectors and coefficients involved are fixed by a common open subgroup of $G$; hence the stabilizer of $\psi$ is open.
    
    The space of $K$-rational homomorphisms $\Hom_{\Gamma_K}(M_K, N_K)$ consists of those $\psi \in W$ that are invariant under this action. An element $\psi$ is invariant if $\psi^\sigma = \psi$ for all $\sigma \in G$. The condition $\psi_v = (\psi^\sigma)_v$ for all $v$ is equivalent to
    \[
    \psi_v = \varphi_{N,v,\sigma} \circ \psi_{\sigma^{-1}v} \circ \varphi_{M,\sigma^{-1}v,\sigma}^{-1}.
    \]
    Replacing $v$ with $\sigma v$, we get $\psi_{\sigma v} = \varphi_{N,\sigma v,\sigma} \circ \psi_v \circ \varphi_{M,v,\sigma}^{-1}$, which is precisely the rationality condition from Definition \ref{def:rationalrepresentationhom} after rearranging. Thus, we have identified $\Hom_{\Gamma_K}(M_K, N_K)$ as the space of $G$-invariants $W^G$.
    
    By the theorem on Galois descent for vector spaces (cf.\ Theorem \ref{thm:galoisdescent}), the canonical map $L\otimes_K W^G \to W$ given by $\alpha\otimes w \mapsto \alpha w$ is an isomorphism. Substituting $W = \Hom_{\Gamma}(M, N)$ and $W^G = \Hom_{\Gamma_K}(M_K, N_K)$ concludes the proof in the case $E=L$.

    The general case for a general intermediate field $L/E/K$ follows from the observation that we have in the already established notation the canonical map
    \[
    \Hom_{\Gamma_K}(M_K, N_K)=W^{\Gal(L/K)} \to W^{\Gal(L/E)} =\Hom_{E\otimes_K\Gamma_K}(E\otimes_K M_K, E\otimes_K N_K)
    \]
    extends to an isomorphism over $E$, again by Galois descent.
  \end{proof}

  \section{\'Etale $K$-species}
  
  \subsection{The notion of \'etale $K$-species}

  \begin{definition}[{$K$-species (Gabriel, Dlab, Ringel)}]\label{def:kspecies}\rqleanmark{def:kspecies}
    A {\em $K$-species} $S=(L_i,{}_iM_j)_{i,j\in I}$ is a finite collection $(L_i)_{i\in I}$ of division rings containing $K$ in their center, together with for each pair $i,j\in I$ an $L_i,L_j$-bimodule ${}_iM_j$, on which the left and right $K$-vector space structures agree and that ${}_iM_j$ is finite-dimensional over $K$.
  \end{definition}

  We refine this definition for our context.
  \begin{definition}[\'Etale $K$-species]\label{def:etalekspecies}\rqleanmark{def:etalekspecies}
    An {\em \'etale $K$-species} is a $K$-species $S=(L_i,{}_iM_j)_{i,j\in I}$ where $L_i$ is a finite separable extension of $K$ and ${}_iM_j$ is a commutative $L_i\otimes_KL_j$-algebra which is finite \'etale as $K$-algebra, the $L_i,L_j$-bimodule structure being the one induced by the algebra structure. Equivalently, ${}_iM_j$ is a commutative finite \'etale $K$-algebra together with two $K$-algebra homomorphisms $L_i\to{}_iM_j$ and $L_j\to{}_iM_j$, its left and right {\em structure maps}.
  \end{definition}
  
  \begin{remark}\label{rmk:etalerepconvention}\rqleanmark{rmk:etalerepconvention}
    A representation of an \'etale $K$-species $S$ is understood to be a representation of the underlying $K$-species. We insist on the additional algebra structure on ${}_iM_j$ only to ensure that we can reconstruct the original $K$-rational quiver from the \'etale quiver associated to a $K$-rational quiver.
  \end{remark}

  \begin{definition}[Split \'etale $K$-species]\label{def:splitetalespecies}\rqleanmark{def:splitetalespecies}
    An \'etale $K$-species $S=(L_i,{}_iM_j)_{i,j\in I}$ is {\em split}, if for all $i,j\in I$ we have $K\cong L_i$ and ${}_iM_j\cong\prod\limits_{\varepsilon} K$ as $K$-algebras.
  \end{definition}

  \begin{definition}[Morphisms of \'etale $K$-species]\label{def:specieshom}\rqleanmark{def:specieshom}
    Let $S=(L_i,{}_iM_j)_{i,j\in I}$ and $T=(\widetilde{L}_i,{}_i\widetilde{M}_j)_{i,j\in \widetilde{I}}$ denote two \'etale $K$-species. A {\em morphism $f\colon S\to T$ of \'etale $K$-species} consists of the following data:
    \begin{itemize}
      \item[(i)] A map $\iota\colon \widetilde{I}\to I$ of index sets (note the direction),
      \item[(ii)] For each $\widetilde{\imath}\in \widetilde{I}$ a $K$-algebra homomorphism $\iota_{\widetilde{\imath}}\colon L_{\iota(\widetilde{\imath})}\to\widetilde{L}_{\widetilde{\imath}}$,
      \item[(iii)] For all $\widetilde{s},\widetilde{t}\in \widetilde{I}$ a $K$-algebra homomorphism $\iota_{\widetilde{s}\widetilde{t}}\colon {}_{\iota(\widetilde{s})}M_{\iota(\widetilde{t})}\to{}_{\widetilde{s}}\widetilde{M}_{\widetilde{t}}$ which is compatible with the structure maps transported along $\iota_{\widetilde{s}}$ and $\iota_{\widetilde{t}}$: the composite of the left (resp.\ right) structure map of ${}_{\iota(\widetilde{s})}M_{\iota(\widetilde{t})}$ with $\iota_{\widetilde{s}\widetilde{t}}$ equals the composite of $\iota_{\widetilde{s}}$ (resp.\ $\iota_{\widetilde{t}}$) with the left (resp.\ right) structure map of ${}_{\widetilde{s}}\widetilde{M}_{\widetilde{t}}$.
    \end{itemize}
    Composition is given by composing the index maps (in the opposite order) and the algebra homomorphisms. Equivalently, the data (i) and (ii) amount to a $K$-algebra homomorphism $\prod_{i\in I}L_i\to\prod_{\widetilde{\imath}\in\widetilde{I}}\widetilde{L}_{\widetilde{\imath}}$, and (iii) to compatible edge data.
  \end{definition}

  \begin{remark}\label{rmk:notfullsubcategory}
    Consider $K$-species with morphisms given by the same data as in the preceding definition, except that in (iii) the maps $\iota_{\widetilde{s}\widetilde{t}}$ are only required to be $K$-linear and compatible with the bimodule structures transported along $\iota_{\widetilde{s}}$ and $\iota_{\widetilde{t}}$. Then the resulting category of \'etale $K$-species is {\bf not} a full subcategory: for instance, the zero map ${}_{\iota(\widetilde{s})}M_{\iota(\widetilde{t})}\to{}_{\widetilde{s}}\widetilde{M}_{\widetilde{t}}$ is a bimodule homomorphism, but not a homomorphism of (unital) $K$-algebras unless its target vanishes.
  \end{remark}
  
  For \'etale $K$-species we formulate
  \begin{definition}[Base change for \'etale $K$-species]\label{def:basechangeforspecies}\rqleanmark{def:basechangeforspecies}
    Let $N/K$ denote an extension.
    \begin{itemize}
      \item[(i)] For any \'etale $K$-species $S$ as above we decompose for each $i\in I$
    \begin{equation}
      L_i\otimes_K N=\prod_{j\in J(i)}L_{j}'
      \label{eq:speciesvertexdecomposition}
    \end{equation}
    where each $L_{j}'$ is a finite separable extension of $N$ and without loss of generality the indexing sets $J(i)$ are pairwise disjoint for $i\in I$. We put
    \[
    I_N:=\bigsqcup_{i\in I}J(i)
    \]
    and write $\pi\colon I_N\to I$ for the corresponding fiber map, which maps all $j\in J(i)$ to $i$.
    \item[(ii)] Put $A_i:=L_i\otimes_K N$ for $i\in I$, so that $A_i=\prod_{k\in J(i)}L_k'$ by \eqref{eq:speciesvertexdecomposition}, and write $e_k\in A_{\pi(k)}$ for the idempotent of the factor $L_k'$. For $i,j\in I$ the $N$-algebra ${}_iM_j\otimes_KN$ is a commutative $A_i\otimes_NA_j$-algebra through the structure maps of ${}_iM_j$. For any $k,\ell\in I_N$ we define
    \[
      {}_kM_\ell':=
      L_k'\otimes_{A_{\pi(k)}}\left({}_{\pi(k)}M_{\pi(\ell)}\otimes_KN\right)\otimes_{A_{\pi(\ell)}}L_\ell'
      \;=\;e_k\left({}_{\pi(k)}M_{\pi(\ell)}\otimes_KN\right)e_\ell.
    \]
    Equivalently, ${}_kM_\ell'=L_k'\otimes_{L_{\pi(k)}}({}_{\pi(k)}M_{\pi(\ell)})\otimes_{A_{\pi(\ell)}}L_\ell'$, where ${}_{\pi(k)}M_{\pi(\ell)}$ is a right $A_{\pi(\ell)}$-module through its right structure map and the $N$-structure of $L_k'$; note that the right-hand tensor product is formed over $A_{\pi(\ell)}=L_{\pi(\ell)}\otimes_KN$ and not over $L_{\pi(\ell)}$.
    \end{itemize}
    Then ${}_kM_\ell'$ is canonically a commutative $L_k'\otimes_NL_\ell'$-algebra, finite \'etale over $N$, and we call the \'etale $N$-species $S\otimes_K N:=(L_k',{}_kM_\ell')_{k,\ell\in I_N}$ the {\em base change} of $S$ in $N/K$.
  \end{definition}

  \begin{remark}\label{rmk:speciesmoduledecomposition}\rqleanmark{rmk:speciesmoduledecomposition}
    We remark that we have the canonical decomposition
    \begin{equation}
      N\otimes_K({}_iM_j)\otimes_K N=\bigoplus_{k\in J(i),\ell\in J(j)}L_k'\otimes_{L_i}({}_iM_j)\otimes_{L_j}L_\ell'
      \label{eq:speciesmoduledecomposition}
    \end{equation}
    of the base change of ${}_iM_j$ along $N$ on both sides. The edge algebras of $S\otimes_KN$ arise instead from the single base change: since $1=\sum_{k\in J(i)}e_k$ in $A_i$ and $1=\sum_{\ell\in J(j)}e_\ell$ in $A_j$, we have
    \[
      {}_iM_j\otimes_KN=\prod_{k\in J(i),\ell\in J(j)}{}_kM_\ell',
    \]
    which is the analogue of \eqref{eq:speciesvertexdecomposition} for the edge algebras.
  \end{remark}

  \begin{remark}[Transitivity of base change]\label{rmk:basechangetransitivity}\rqleanmark{rmk:basechangetransitivity}
    For any tower $N'/N/K$ we have a canonical identification $I_{N'}=(I_N)_{N'}$ and a canonical isomorphism
    \[
    S\otimes_K N'=(S\otimes_K N)\otimes_N N'
    \]
    of $N'$-species. The verification is straightforward: it follows from the uniqueness of the decomposition of a finite \'etale algebra into a product of fields, applied to $L_i\otimes_KN'=(L_i\otimes_KN)\otimes_NN'$, together with the corresponding associativity of the tensor products in Definition \ref{def:basechangeforspecies}(ii). For $N'/N/K$ contained in a Galois extension $L/K$ splitting $S$, the statement is also a consequence of the transitivity of base change for $K$-rational quivers (Definition \ref{def:quiverbasechange}) and Theorem \ref{thm:KrationalquiveretaleKspeciesantiequivalence} below.
  \end{remark}

  \begin{remark}\label{rmk:simplealgebraspecies}
    If $D$ is a finite-dimensional division algebra over $K$ containing $K$ in its center, and if the center of $D$ is separable over $K$ (which is automatic in characteristic $0$), then for any extension $L/K$ the algebra $D\otimes_K L$ is semi-simple over $L$ and contains $L$ in its center. (Without the separability assumption this fails: for $K=\FF_p(t)$ and $D=L=K(t^{1/p})$, the algebra $D\otimes_KL$ is not reduced.) Decomposing $D\otimes_K L$ into simple $L$-algebras and invoking Artin--Wedderburn to pass to Morita equivalent division algebras over $L$, we may extend Definition \ref{def:basechangeforspecies} to base change for arbitrary $K$-species.
  \end{remark}

  \begin{definition}[Restriction for \'etale $K$-species]\label{def:restrictionforspecies}\rqleanmark{def:restrictionforspecies}
    Let $N/K$ denote a finite separable extension. For any \'etale $N$-species $S$ we consider the \'etale $K$-species
    \[
    \res_{N/K}S=(\res_{N/K} L_i,\res_{N/K}{}_iM_j)_{i,j\in I},
    \]
    where $\res_{N/K} V$ denotes for any $N$-vector space $V$ the underlying $K$-vector space. Since $N/K$ is finite and separable, each $\res_{N/K}L_i$ is a finite separable extension of $K$ and each $\res_{N/K}{}_iM_j$ is finite \'etale as $K$-algebra, so that $\res_{N/K}S$ is indeed an \'etale $K$-species.
  \end{definition}

  \begin{proposition}\label{prop:speciesrestrictionbasechangeadjunction}\rqleanmark{prop:speciesrestrictionbasechangeadjunction}
    For any finite separable extension $N/K$, any \'etale $K$-species $S$ and any \'etale $N$-species $S'$ we have a natural isomorphism
    \begin{equation}
      \Hom_N(S\otimes_K N,S')\cong \Hom_K(S,\res_{N/K}S').
      \label{eq:restrictionadjunction}
    \end{equation}
  \end{proposition}

  \begin{proof}
    Although the proof is elementary, it will be a consequence of Proposition \ref{prop:quiverrestrictionbasechangeadjunction} combined with Theorem \ref{thm:KrationalquiveretaleKspeciesantiequivalence} below, by observing that the given antiequivalences commute with the corresponding restriction and base change functors.
  \end{proof}

  \subsection{The \'etale $K$-species associated to a $K$-rational quiver}

    Consider an (\'etale) $K$-rational quiver $\Gamma_K$ with underlying vertex set $V$ for the Galois extension $L/K$. For notational simplicity, we write $G:=\Gal(L/K)$. We associate to $\Gamma_K$ an \'etale $K$-species $S(\Gamma_K)$ as follows.

  Let $v\in V$ be arbitrary and let $G_v\subseteq G$ denote the stabilizer of $v$. Since the action of $G$ on the discrete set $V$ is continuous, $G_v$ is an open subgroup of $G$, whence $L^{G_v}$ is a finite extension of $K$ also for infinite $L/K$; this is what renders the species $S(\Gamma_K)$ below \'etale. Then the conjugacy class of $G_v$ is an invariant of the orbit $Gv$ and we have an intermediate field $L(v):=L^{G_v}$ of $L/K$ which is associated to the orbit $Gv$ and unique up to $G$-conjugation. In particular, the isomorphism class of the field $L(v)$ only depends on the Galois orbit $Gv$ of $v$. More precisely, for $w\in Gv$ the map
  \[
  L(v)\to L(w),\quad\alpha\mapsto\sigma\alpha,\qquad\text{where }\sigma\in G\text{ with }\sigma v=w,
  \]
  is a $K$-isomorphism which does not depend on the choice of $\sigma$: if also $\sigma'v=w$, then $\sigma'=\sigma\eta$ with $\eta\in G_v$, and $\eta$ acts trivially on $L(v)=L^{G_v}$. We call it the {\em canonical isomorphism} $L(v)\cong L(w)$; these isomorphisms are compatible with composition. By abuse of notation we put
  \begin{equation}
    L_{Gv}:=L(v).
    \label{eq:abusivedefinitionofLv}
  \end{equation}
  Then $L_{Gv}$ is well defined up to canonical $K$-isomorphism.

  Moreover, we observe that for any edge $e\in E$ of the underlying quiver, we have
  \[
  L(e):=L^{G_{e}},
  \]
  a finite extension of $K$ for the same reason, which provides us with a diagram
  \begin{equation}
    L(s(e))\xrightarrow\subseteq L(e)\xleftarrow\supseteq L(t(e)).
    \label{eq:edgediagram}
  \end{equation}
  Therefore $L(e)$ is an $L(s(e)),L(t(e))$-bimodule (indeed a commutative $L(s(e))\otimes_KL(t(e))$-algebra) in a canonical way. Replacing $e$ with any Galois-conjugate $\sigma e$ provides us with
  \[
  L(s(\sigma e))\xrightarrow\subseteq L(\sigma e)\xleftarrow\supseteq L(t(\sigma e)),
  \]
  which in light of \eqref{eq:stgaloisequivariance} is $K$-isomorphic to \eqref{eq:edgediagram} via the canonical isomorphisms $\alpha\mapsto\sigma\alpha$ on all three fields.

  Once again the isomorphism class of $L(e)$ with its additional structure as $L(s(e)),L(t(e))$-bimodule does not depend on the choice of representative of the Galois orbit of $e$, and by abuse of notation we put
  \begin{equation}
    L_{Ge}:=L(e),
    \label{eq:abusivedefinitionofLe}
  \end{equation}
  which is well defined up to canonical $K$-isomorphism with respect to the choice of representative inside the Galois orbit of $e$. Note, however, that $L(e)$ is a bimodule over $L(s(e))$ and $L(t(e))$, which in general are not the fields attached to the chosen representatives of the vertex orbits of $s(e)$ and $t(e)$; in the following definition the bimodule structure over the latter is obtained through the canonical isomorphisms, and this can produce a genuine twist (cf.\ Example \ref{ex:cyclic}).

  \begin{definition}[Species associated to a rational quiver]\label{def:speciesofquiver}\rqleanmark{def:speciesofquiver}
    Let $\Gamma_K:=(\Gamma,\rho)$ be a quiver over $K$ with underlying quiver $\Gamma=(V,E,s,t)$. We put $I:=G\backslash V$.
    \begin{itemize}
      \item[(i)] For each $i\in I$ we use convention \eqref{eq:abusivedefinitionofLv} to define
      \[
      L_i:=L(v_i),\quad\text{for any fixed } v_i\in i.
      \]
      \item[(ii)] For each $i,j\in I$ we put $E_{ij} = \{ e \in E \mid s(e) \in i, t(e) \in j \}$, fix for each $\varepsilon\in G\backslash E_{ij}$ a representative $e_\varepsilon\in\varepsilon$, and fix $\sigma_\varepsilon,\tau_\varepsilon\in G$ with
      \begin{equation}
        s(e_\varepsilon)=\sigma_\varepsilon v_i\quad\text{and}\quad
        t(e_\varepsilon)=\tau_\varepsilon v_j.
        \label{eq:endsofedgetranslations}
      \end{equation}
      We consider the commutative finite \'etale $K$-algebra
      \[
        {}_iM_j :=\bigoplus_{\varepsilon\in G\backslash E_{ij}} L(e_\varepsilon)
      \]
      (in the sense of \eqref{eq:abusivedefinitionofLe}) with the left and right structure maps whose $\varepsilon$-components are
      \[
        \lambda^s_\varepsilon\colon L_i=L(v_i)\to L(e_\varepsilon),\;\alpha\mapsto\sigma_\varepsilon\alpha,
        \qquad
        \lambda^t_\varepsilon\colon L_j=L(v_j)\to L(e_\varepsilon),\;\alpha\mapsto\tau_\varepsilon\alpha,
      \]
      i.e., the canonical isomorphisms $L(v_i)\cong L(s(e_\varepsilon))$ and $L(v_j)\cong L(t(e_\varepsilon))$ followed by the inclusions \eqref{eq:edgediagram}. These maps are well defined, since $\sigma_\varepsilon\alpha\in L^{G_{\sigma_\varepsilon v_i}}=L(s(e_\varepsilon))\subseteq L(e_\varepsilon)$, and they do not depend on the choice of $\sigma_\varepsilon$ and $\tau_\varepsilon$ in \eqref{eq:endsofedgetranslations}. This makes ${}_iM_j$ an $L_i,L_j$-bimodule.
    \end{itemize}
    We call $S(\Gamma_K):=(L_i,{}_iM_j)_{i,j\in I}$ the {\em (\'etale) $K$-species associated to $\Gamma_K$}.
  \end{definition}

  \subsection{The $K$-rational quiver associated to an \'etale $K$-species}
  
  Let conversely $S=(L_i,{}_iM_j)_{i,j\in I}$ be an \'etale $K$-species. We will associate to $S$ a $K$-rational quiver $\Gamma_K(S)$ as follows.

  Let $L$ denote a splitting field for the collection of the finitely many separable extensions $L_i/K$ for $i\in I$ and the finite \'etale commutative $K$-algebras ${}_iM_j$, $i,j\in I$. We may assume that $L/K$ is finite and Galois with Galois group $G=\Gal(L/K)$. We put
  \[
  V := \bigsqcup_{i\in I}\Hom_K(L_i,L).
  \]
  Then $V$ is a left $G$-set and each subset $\Hom_{K}(L_i,L)$ corresponds to a Galois orbit in $V$. We formulate this observation structurally.

  \begin{proposition}\label{prop:Viorbits}\rqleanmark{prop:Viorbits}
  The fibers of the canonical map
  \[
  \pi\colon V\to I,\;\sigma\mapsto i,\text{ if }\sigma\in\Hom_K(L_i,L),
  \]
  are in canonical bijection with the $G$-orbits in $V$.
  \end{proposition}

  Write $\Specm{}_iM_j$ for the set of maximal ideals in ${}_iM_j$ and for $\varepsilon\in\Specm{}_iM_j$ put
  \[
  L_{\varepsilon}:={}_iM_j/\varepsilon.
  \]
  Then we have a canonical decomposition
  \begin{equation}
    {}_iM_j=\prod_{\varepsilon\in\Specm{}_iM_j}L_\varepsilon.
  \end{equation}
  For each $\varepsilon$ the structure maps of ${}_iM_j$ induce embeddings $\lambda^s_\varepsilon\colon L_i\to L_{\varepsilon}$ and $\lambda^t_\varepsilon\colon L_j\to L_{\varepsilon}$. Fix a $K$-embedding $\kappa\colon L_j\to L$ and regard $L$ as an $L_j$-algebra through $\kappa$. Then over $L$ we have
  \begin{equation}
    {}_iM_j\otimes_{L_j,\kappa} L=\prod_{\varepsilon\in\Specm{}_iM_j}
    \prod_{\tau}L_\varepsilon\otimes_{L_{\varepsilon},\tau}L\;=\;\prod_{\varepsilon\in\Specm{}_iM_j}\prod_{\tau}L,
    \label{eq:etalebialgebradecomposition}
  \end{equation}
  canonically as commutative $L_i\otimes_KL$-algebras, where for each $\varepsilon$ the embedding $\tau$ runs through the set of $K$-embeddings $\tau\colon L_\varepsilon\to L$ with $\tau\circ\lambda^t_\varepsilon=\kappa$, i.e., through $\Hom_{L_j}(L_{\varepsilon},L)$, and where on the $(\varepsilon,\tau)$-component the left $L_i$-structure is given through $\tau\circ\lambda^s_\varepsilon\in\Hom_K(L_i,L)$. We define
  \[
  E:=\bigsqcup_{i,j\in I}\bigsqcup_{\varepsilon\in\Specm{}_iM_j}\Hom_{K}(L_{\varepsilon},L).
  \]
  Note that $E$ carries a canonical left $G$ action and we have
  \begin{proposition}\label{prop:iMjorbits}\rqleanmark{prop:iMjorbits}
  The fibers of the canonical map
  \[
  p\colon E\to \bigsqcup_{i,j\in I}\Specm{}_iM_j=:\mathcal{I},\;
  \sigma\mapsto \varepsilon\in \Specm{}_iM_j,\text{ if }\sigma\in\Hom_{K}(L_{\varepsilon},L),
  \]
  are in canonical bijection with the $G$-orbits in $E$.
  \end{proposition}
  
  The source and target maps $s,t\colon E\to V$ are defined as follows. For $\varepsilon\in\Specm{}_iM_j$ and $\tau\in\Hom_K(L_\varepsilon,L)$ we put
  \begin{equation}
    s(\tau):=\tau\circ\lambda^s_{\varepsilon}\in \Hom_K(L_i,L)\subseteq V,
    \label{eq:tauespilon}
  \end{equation}
  and
  \begin{equation}
  t(\tau):=\tau\circ\lambda^t_{\varepsilon}\in \Hom_K(L_j,L)\subseteq V.
    \label{eq:sigmaespilon}
  \end{equation}
  Then $s,t$ are evidently $G$-equivariant and we may formulate
  \begin{definition}[{$K$-rational quiver associated to an \'etale $K$-species}]\label{def:quiverofspecies}\rqleanmark{def:quiverofspecies}
    Let $S=(L_i,{}_iM_j)_{i,j\in I}$ denote an \'etale $K$-species. Then we call
    \[
    \Gamma_K(S):=(V,E,s,t)
    \]
    with $V,E,s,t$ as above together with its canonical structure of a diagram of $\Gal(L/K)$-sets the {\em $K$-rational quiver associated to $S$}.
  \end{definition}

  \begin{theorem}\label{thm:KrationalquiveretaleKspeciesantiequivalence}\rqleanmark{thm:KrationalquiveretaleKspeciesantiequivalence}
    The maps
    \[
    S(-)\colon \{\text{\rm $K$-rational quivers}\}\to\{S\;\text{\rm \'etale $K$-species}\},
    \]
    \[
    \Gamma_K(-)\colon \{\text{\rm \'etale $K$-species}\}\to\{\text{\rm $K$-rational quivers}\},
    \]
    extend to quasi-inverse contravariant functors of the underlying categories, inducing anti-equivalences of the underlying full subcategories of split objects.
  \end{theorem}

  \begin{proof}
    The proof consists of two main parts. First, we show that for any $K$-rational quiver $\Gamma_K$, the canonical map $\Gamma_K \to \Gamma_K(S(\Gamma_K))$ is an isomorphism. Second, we show that for any \'etale $K$-species $S$, the canonical map $S \to S(\Gamma_K(S))$ is an isomorphism. Throughout the proof, we let $G = \Gal(L/K)$ with $L$ sufficiently big such that all actions are defined as actions of $\Gal(L/K)$ and all occurring finite \'etale algebras split over $L$. This is no loss of generality: any finitely many $K$-rational quivers and \'etale $K$-species (and the morphisms between them) are defined over, resp.\ split by, a common finite Galois extension $L/K$, and by the remark preceding Definition \ref{def:rationalquiverhom} the sets of homomorphisms do not change when $L$ is enlarged (the actions being inflated along $\Gal(L'/K)\to\Gal(L/K)$ for $L'\supseteq L$); likewise $\Gamma_K(S)$ computed with $L$ and with $L'$ agree after inflation, since every $K$-embedding of a field split by $L$ into $L'$ lands in $L$. Hence it suffices to prove the statement for $K$-rational quivers with $G$-action and \'etale $K$-species split by $L$, for a fixed finite Galois extension $L/K$. Alternatively, we may choose $L=K^{\rm sep}$ at the expense of working with an infinite Galois extension throughout.

    \paragraph{Step 1: $\Gamma_K(S(\Gamma_K)) \cong \Gamma_K$.}\ 

    \smallskip
    Let $\Gamma_K = (V,E,s,t)$ be a $K$-rational quiver with its $G$-action. We apply the functor $S(-)$ to obtain the \'etale $K$-species $S(\Gamma_K) = (L_i, {}_iM_j)_{i,j \in I}$. Recall that $I=G\backslash V$ is the set of $G$-orbits of vertices. For each orbit $i\in I$, we fix a representative $v_i \in i$, so $L_i=L(v_i)=L^{G_{v_i}}$. For any two orbits $i,j \in I$, the bimodule ${}_iM_j$ is defined as
    \[
      {}_iM_j = \bigoplus\limits_{\varepsilon \in G\backslash E_{ij}} L(e_\varepsilon),
    \]
    where $E_{ij}=\{e\in E \mid s(e)\in i, t(e)\in j\}$, and for each orbit $\varepsilon=Ge \in G\backslash E_{ij}$, we choose a representative $e_\varepsilon \in \varepsilon$, so $L(e_\varepsilon) = L^{G_{e_\varepsilon}}$.

    Now we apply the functor $\Gamma_K(-)$ to $S(\Gamma_K)$ to obtain a new $K$-rational quiver $\Gamma'_K = (V',E',s',t')$. We must construct a $G$-equivariant isomorphism $(\Phi_V, \Phi_E): \Gamma_K \to \Gamma'_K$.

    \smallskip
    \noindent{\bf Natural bijection on the vertices:} By definition,
    \[
    V' = \bigsqcup\limits_{i\in I} \Hom_K(L_i, L).
    \]
    For a given vertex $v \in V$, let $i=Gv$ be its orbit. We have $v=\tau v_i$ for some $\tau \in G$. We define a map $\Phi_V\colon V \to V'$ by sending $v$ to the $K$-algebra homomorphism $\phi_v \in \Hom_K(L_i,L)$ given by $\phi_v(\alpha) = \tau \alpha$ for all $\alpha\in L_i = L(v_i)$. This map is well-defined, since if $v=\tau'v_i$ for another $\tau'\in G$, then $\tau'=\tau \eta$ for some $\eta\in G_{v_i}$, and for $\alpha\in L(v_i)$, $\tau'(\alpha) = \tau \eta \alpha = \tau \alpha$. The map $\Phi_V$ is a bijection of $G$-sets. The inverse maps an embedding $\phi \in \Hom_K(L(v_i),L)$ to $\tau v_i$, where $\tau\in G$ is any element such that $\phi(\alpha)=\tau \alpha$ for all $\alpha\in L(v_i)$. Such a $\tau\in G$ exists and is unique up to left multiplication by an element of $G_{v_i}$. The $G$-equivariance is readily checked: for $\sigma\in G$, $\Phi_V(\sigma v)=\Phi_V(\sigma \tau v_i)$ is the map $\alpha \mapsto \sigma \tau \alpha$, which is precisely $\sigma \Phi_V(v)$.

    \smallskip
    \noindent{\bf Natural bijection on the edges:} By definition,
    \[
    E' = \bigsqcup\limits_{i,j\in I}\bigsqcup\limits_{\varepsilon'\in\Specm({}_iM_j)}\Hom_{K}(L_{\varepsilon'},L).
    \]
    The set of maximal ideals $\Specm({}_iM_j)$ of
    \[
      {}_iM_j = \bigoplus\limits_{\varepsilon \in G\backslash E_{ij}} L(e_\varepsilon)
    \]
    is in natural bijection with the set of edge orbits $G\backslash E_{ij}$. For an orbit $\varepsilon \in G\backslash E_{ij}$, the corresponding residue field is $L_{\varepsilon'} = L(e_\varepsilon)$. Thus, we can write
    \[
    E' = \bigsqcup\limits_{i,j\in I}\bigsqcup_{\varepsilon \in G\backslash E_{ij}} \Hom_{K}(L(e_\varepsilon),L).
    \]
    This is the disjoint union over all edge orbits in $G\backslash E$.

    For an edge $e\in E$, let $\varepsilon=Ge$ be its orbit. We have $e=\sigma e_\varepsilon$ for some $\sigma\in G$. We define $\Phi_E\colon E \to E'$ by sending $e$ to the $K$-algebra homomorphism $\phi_e \in \Hom_{K}(L(e_\varepsilon),L)$ given by $\phi_e(\alpha)=\sigma \alpha$. As with vertices, this map is well-defined and provides a $G$-equivariant bijection $E \to E'$.

    \smallskip
    \noindent{\bf Compatibility with source and target maps:} We must check that $\Phi_V(s(e)) = s'(\Phi_E(e))$ and $\Phi_V(t(e)) = t'(\Phi_E(e))$. Let $e=\sigma e_\varepsilon \in E$ for $\sigma\in G$, with $s(e_\varepsilon)\in i=Gv_i$ and $t(e_\varepsilon)\in j=Gv_j$. Note that we may not assume $v_i=s(e_\varepsilon)$ and $v_j=t(e_\varepsilon)$, since the representatives $v_i$ are fixed once for each vertex orbit (in Example \ref{ex:cyclic}, $s(e_\varepsilon)$ and $t(e_\varepsilon)$ lie in the same orbit but are distinct). Instead, let $\sigma_\varepsilon,\tau_\varepsilon\in G$ be as in \eqref{eq:endsofedgetranslations}, i.e., $s(e_\varepsilon)=\sigma_\varepsilon v_i$ and $t(e_\varepsilon)=\tau_\varepsilon v_j$.

    The source map $s'$ for $\Gamma'_K$ is defined as $s'(\phi_e) = \phi_e \circ \lambda^s_\varepsilon$, where $\lambda^s_\varepsilon\colon L_i \to L(e_\varepsilon)$, $\alpha\mapsto\sigma_\varepsilon\alpha$, is the left structure map of $S(\Gamma_K)$ (cf.\ \eqref{eq:endsofedgetranslations}). So, for $\alpha\in L_i = L(v_i)$, we have
    \[
    s'(\phi_e)(\alpha) = \phi_e(\sigma_\varepsilon\alpha) = \sigma\sigma_\varepsilon \alpha.
    \]
    On the other hand, $s(e) = s(\sigma e_\varepsilon)=\sigma s(e_\varepsilon)=\sigma\sigma_\varepsilon v_i$. Thus, by definition of $\Phi_V$ (with $\sigma\sigma_\varepsilon$ in place of $\tau$), $\Phi_V(s(e))$ is the map $\alpha \mapsto \sigma\sigma_\varepsilon \alpha$. Hence $\Phi_V(s(e)) = s'(\Phi_E(e))$. Likewise, the target map is $t'(\phi_e)=\phi_e\circ\lambda^t_\varepsilon$ with $\lambda^t_\varepsilon(\alpha)=\tau_\varepsilon\alpha$, so that $t'(\phi_e)(\alpha)=\sigma\tau_\varepsilon\alpha$ for $\alpha\in L_j=L(v_j)$, while $t(e)=\sigma t(e_\varepsilon)=\sigma\tau_\varepsilon v_j$, whence $\Phi_V(t(e))$ is the map $\alpha\mapsto\sigma\tau_\varepsilon\alpha$ as well. Hence $\Phi_V(t(e)) = t'(\Phi_E(e))$. This concludes the proof that $\Gamma_K(S(\Gamma_K))\cong \Gamma_K$.

    \paragraph{Step 2: $S(\Gamma_K(S)) \cong S$.}\ 

    \smallskip
Let $S=(L_i, {}_iM_j)_{i,j\in I}$ be an \'etale $K$-species. We construct the $K$-rational quiver $\Gamma_K(S)=(V,E,s,t)$. Then we apply $S(-)$ to obtain $S' = S(\Gamma_K(S))=(L'_{i'},{}_{i'}M'_{j'})_{i',j'\in I'}$. We need to construct an isomorphism of \'etale $K$-species from $S$ to $S'$.

    \smallskip
    \noindent{\bf Natural bijection on the index sets:} The index set of $S'$ is $I' = G\backslash V$. By construction,
    \[
    V=\bigsqcup\limits_{i\in I}\Hom_K(L_i,L).
    \]
    Each set $\Hom_K(L_i,L)$ is a single $G$-orbit in $V$. Thus, there is a natural bijection $\iota\colon I \to I'$ where $\iota(i)$ is the orbit $\Hom_K(L_i,L)$ (cf.\ Proposition \ref{prop:Viorbits}).

    \smallskip
    \noindent{\bf Natural isomorphisms of the fields:} For $i'=\iota(i)\in I'$, the field $L'_{i'}$ is $L^{G_{v'}}$ for a chosen representative $v' \in i'$. Let us choose $v'\colon L_i \to L$ to be an embedding. The stabilizer $G_{v'}$ consists of all $\sigma\in G$ such that $\sigma\circ v' = v'$. By Galois theory, the fixed field $L^{G_{v'}}$ is exactly the image $v'(L_i)$, which is $K$-isomorphic to $L_i$ via $v'$. Thus, we have $K$-algebra isomorphisms $\iota_i\colon L_i \to L'_{i'}$ for all $i$.

    \smallskip
    \noindent{\bf Natural isomorphisms of the bimodules (edge algebras):} For $i'=\iota(i), j'=\iota(j) \in I'$, the bimodule is ${}_{i'}M'_{j'} = \bigoplus_{\varepsilon'' \in G\backslash E_{i'j'}} L(e_{\varepsilon''})$.
    
    The set of edges $E_{i'j'}$ consists of all $\tau\in E$ with $s(\tau)\in i'$ and $t(\tau)\in j'$. By construction of
    \[
    E = \bigsqcup\limits_{k,l\in I}\bigsqcup\limits_{\varepsilon\in\Specm({}_kM_l)}\Hom_K(L_\varepsilon,L),
    \]
    an edge $\tau\in E$ has source in $\Hom_K(L_k,L)$ and target in $\Hom_K(L_l,L)$. Thus $E_{i'j'}$ consists of edges constructed from the bimodules ${}_iM_j$, i.e.,
    \[
    E_{i'j'} = \bigsqcup\limits_{\varepsilon\in\Specm({}_iM_j)}\Hom_K(L_\varepsilon,L).
    \]
    The $G$-orbits $\varepsilon''$ in $E_{i'j'}$ are precisely the sets $\Hom_K(L_\varepsilon,L)$ for $\varepsilon \in\Specm({}_iM_j)$.
    
    For such an orbit $\varepsilon''$, we pick a representative $e_{\varepsilon''}$, which is an embedding $\tau\colon L_\varepsilon \to L$. The field $L(e_{\varepsilon''}) = L^{G_{\tau}}$ is isomorphic to $L_\varepsilon$.
    
    The $L'_{i'},L'_{j'}$-bimodule structure on $L(e_{\varepsilon''}) \cong L_\varepsilon$ is induced by the maps $s(e_{\varepsilon''})$ and $t(e_{\varepsilon''})$. These are $s(\tau) = \tau\circ\lambda^s_\varepsilon \in i'$ and $t(\tau) = \tau\circ\lambda^t_\varepsilon \in j'$, where $\lambda^s_\varepsilon\colon L_i\to L_\varepsilon$ and $\lambda^t_\varepsilon\colon L_j\to L_\varepsilon$ are the algebra structure maps. More precisely, by the definition of $S(-)$ the left structure map $L'_{i'}=L(v')\to L(e_{\varepsilon''})$ is $\beta\mapsto\sigma''\beta$ for any $\sigma''\in G$ with $\sigma''v'=s(\tau)$. Under the identifications $\iota_i\colon L_i\to L'_{i'}$, $x\mapsto v'(x)$, and $L_\varepsilon\cong L(e_{\varepsilon''})$, $y\mapsto\tau(y)$, it sends $v'(x)$ to $\sigma''(v'(x))=s(\tau)(x)=\tau(\lambda^s_\varepsilon(x))$; hence it corresponds to $\lambda^s_\varepsilon$, and likewise the right structure map corresponds to $\lambda^t_\varepsilon$. So this structure corresponds exactly to the original $L_i,L_j$-bimodule structure on $L_\varepsilon$. Therefore,
    \[
      {}_{i'}M'_{j'} = \bigoplus_{\varepsilon \in \Specm({}_iM_j)} L_\varepsilon.
    \]
    Since ${}_iM_j$ is a finite \'etale $K$-algebra, it decomposes into a product of fields
    \[
      {}_iM_j = \prod\limits_{\varepsilon \in \Specm({}_iM_j)} L_\varepsilon.
    \]
    For a finite index set, the direct sum and direct product of bimodules coincide. Thus we have an isomorphism of $L_i,L_j$-bimodules (and of $K$-algebras, compatible with the structure maps) $\iota_{ij}\colon {}_iM_j \to {}_{i'}M'_{j'}$.
    
    Since $\iota$ is bijective and all $\iota_i$, $\iota_{ij}$ are isomorphisms, the collection $(\iota, (\iota_i)_i, (\iota_{ij})_{i,j})$, read with the direction required in the definition of morphisms of \'etale $K$-species (index map $\iota^{-1}\colon I'\to I$), defines an isomorphism of \'etale $K$-species $S \to S(\Gamma_K(S))$.
    
    \paragraph{Functoriality and split objects:} Let $f\colon\Gamma_K\to\Gamma'_K$ be a morphism of $K$-rational quivers, and write $v'_{i'}$, $e'_{\varepsilon'}$ for the chosen representatives in $\Gamma'_K$. Since $G_v\subseteq G_{f(v)}$ and $G_e\subseteq G_{f(e)}$, we have $L(f(v))\subseteq L(v)$ and $L(f(e))\subseteq L(e)$ for all $v\in V$, $e\in E$. We define $S(f)\colon S(\Gamma'_K)\to S(\Gamma_K)$ in the sense of the definition of morphisms of \'etale $K$-species as follows:
    \begin{itemize}
      \item[(i)] the index map is the orbit map $I=G\backslash V\to I'=G\backslash V'$, $i\mapsto i':=Gf(v_i)$;
      \item[(ii)] for $i\in I$ choose $\rho_i\in G$ with $f(v_i)=\rho_iv'_{i'}$; the vertex component is $L'_{i'}=L(v'_{i'})\to L(v_i)=L_i$, $\alpha\mapsto\rho_i\alpha$, i.e., the canonical isomorphism $L(v'_{i'})\cong L(f(v_i))$ followed by the inclusion $L(f(v_i))\subseteq L(v_i)$;
      \item[(iii)] for $i,j\in I$ and $\varepsilon\in G\backslash E_{ij}$ let $\varepsilon'=Gf(e_\varepsilon)$ and choose $\rho_\varepsilon\in G$ with $f(e_\varepsilon)=\rho_\varepsilon e'_{\varepsilon'}$; the $\varepsilon$-component of the edge map ${}_{i'}M'_{j'}=\bigoplus_{\varepsilon'}L(e'_{\varepsilon'})\to{}_iM_j=\bigoplus_\varepsilon L(e_\varepsilon)$ is the projection to the $\varepsilon'$-component followed by $\alpha\mapsto\rho_\varepsilon\alpha$.
    \end{itemize}
    As in Step 1, none of these maps depends on the choice of $\rho_i$ or $\rho_\varepsilon$. The compatibility with the structure maps of $S(\Gamma'_K)$ and $S(\Gamma_K)$ holds because both composites send $\alpha\in L'_{i'}$ to $g\alpha$, resp.\ $g'\alpha$, for elements $g,g'\in G$ with $gv'_{i'}=f(s(e_\varepsilon))=g'v'_{i'}$, and such elements induce the same map on $L(v'_{i'})$; likewise for targets. Conversely, for a morphism $g\colon S\to T$ of \'etale $K$-species with index map $\iota\colon\widetilde{I}\to I$, the morphism $\Gamma_K(g)\colon\Gamma_K(T)\to\Gamma_K(S)$ is given by precomposition: a vertex $\tau\colon\widetilde{L}_{\widetilde{\imath}}\to L$ is sent to $\tau\circ\iota_{\widetilde{\imath}}\in\Hom_K(L_{\iota(\widetilde{\imath})},L)$, and an edge $\tau\colon\widetilde{L}_{\widetilde{\varepsilon}}\to L$, $\widetilde{\varepsilon}\in\Specm{}_{\widetilde{s}}\widetilde{M}_{\widetilde{t}}$, is sent to the edge of $\Gamma_K(S)$ determined by the $K$-algebra homomorphism ${}_{\iota(\widetilde{s})}M_{\iota(\widetilde{t})}\to{}_{\widetilde{s}}\widetilde{M}_{\widetilde{t}}\to\widetilde{L}_{\widetilde{\varepsilon}}\xrightarrow{\tau}L$, i.e., by its kernel $\varepsilon\in\Specm{}_{\iota(\widetilde{s})}M_{\iota(\widetilde{t})}$ and the induced embedding $L_\varepsilon\to L$. These assignments are $G$-equivariant, compatible with sources and targets by the compatibility condition in (iii) of the definition of morphisms, and compatible with identities and composition. A direct computation with the explicit formulas shows that the isomorphisms $(\Phi_V,\Phi_E)$ of Step 1 and $(\iota,(\iota_i),(\iota_{ij}))$ of Step 2 are natural with respect to these morphisms. Note that the direction of the index map in the definition of morphisms of \'etale $K$-species is forced: $S(f)$ has index map $I\to I'$ but vertex components $L'_{i'}\to L_i$ running the opposite way. This shows that $S(-)$ and $\Gamma_K(-)$ are quasi-inverse contravariant functors.
    
    If a quiver $\Gamma_K$ is split over $K$, its $G$-action is trivial. Then $G_v=G$ for all $v\in V$ and $G_e=G$ for all $e\in E$. Consequently, $L_i=L^G=K$ and $L(e_\varepsilon)=K$ for all $i$ and $\varepsilon$. The resulting species $S(\Gamma_K)$ is a split \'etale $K$-species. Conversely, starting from a split species $S$ (with all $L_i=K$ and ${}_iM_j$ a product of copies of $K$), the resulting quiver $\Gamma_K(S)$ will have trivial $G$-action on its vertices and edges, making it a split $K$-rational quiver. The anti-equivalence thus restricts to an anti-equivalence between the full subcategories of split objects.
  \end{proof}

  With a view to base change along arbitrary field extensions, we use Theorem \ref{thm:KrationalquiveretaleKspeciesantiequivalence} to extend Definition \ref{def:quiverbasechange}.

  \begin{definition}[General base change]\label{def:quiverbasechangegeneral}\rqleanmark{def:quiverbasechangegeneral}
    Let $\Gamma_K$ be a $K$-rational quiver and let $N/K$ be an arbitrary field extension. We call the $N$-rational quiver
    \[
    N\otimes_K\Gamma_K:=\Gamma_N\left(S(\Gamma_K)\otimes_KN\right)
    \]
    associated to the base change $S(\Gamma_K)\otimes_KN$ of the \'etale $K$-species $S(\Gamma_K)$ (Definition \ref{def:basechangeforspecies}) the {\em base change} of $\Gamma_K$ to $N$.
  \end{definition}

  \begin{lemma}\label{lem:quiverbasechangecompat}\rqleanmark{lem:quiverbasechangecompat}%
    Let $\Gamma_K$ be a $K$-rational quiver with respect to a Galois extension $L/K$, and let $L/E/K$ be an intermediate field with $E/K$ finite. Then the base change $E\otimes_K\Gamma_K$ in the sense of Definition \ref{def:quiverbasechangegeneral} is canonically isomorphic to the base change $E\otimes_K\Gamma_K$ in the sense of Definition \ref{def:quiverbasechange}.
  \end{lemma}

  \begin{proof}
    By Remark \ref{rmk:finitequotient} we may assume that $L/K$ is finite: the action factors over a finite Galois subextension of $L/K$, which we may choose to contain $E$. The fields $L_k'$ of $S(\Gamma_K)\otimes_KE$ are quotients of $L(v)\otimes_KE$, and the field factors of the algebras ${}_kM_\ell'$ are quotients of $L(e)\otimes_KE$, for vertices $v$ and edges $e$ of $\Gamma_K$. They therefore embed into $L$, and we may compute $\Gamma_E(S(\Gamma_K)\otimes_KE)$ with respect to $L/E$.

    We first observe that the anti-equivalences of Theorem \ref{thm:KrationalquiveretaleKspeciesantiequivalence} commute with base change (cf.\ the proof of Proposition \ref{prop:speciesrestrictionbasechangeadjunction}). Let $S$ be an \'etale $K$-species split by $L$. For every finite \'etale $K$-algebra $A$, restriction along $A\to A\otimes_KE$ is a bijection
    \[
    \Hom_E(A\otimes_KE,L)\to\Hom_K(A,L),
    \]
    and it is $\Gal(L/E)$-equivariant. Decomposing $A$ and $A\otimes_KE$ into products of fields, both sides are disjoint unions of the sets of embeddings of the factors into $L$. For $A=L_i$ this identifies $\bigsqcup_{k\in J(i)}\Hom_E(L_k',L)$ with $\Hom_K(L_i,L)$ by \eqref{eq:speciesvertexdecomposition}. For $A={}_iM_j$ we have ${}_iM_j\otimes_KE=\prod_{k\in J(i),\ell\in J(j)}{}_kM_\ell'$ by Definition \ref{def:basechangeforspecies}(ii), so that the edges of $\Gamma_E(S\otimes_KE)$ with source in some $J(i)$ and target in some $J(j)$ are identified with the edges of $\Gamma_K(S)$ whose source lies over $i$ and whose target lies over $j$. Since the structure maps of $S\otimes_KE$ are the base changes of those of $S$, these bijections are compatible with source and target maps. In this way, we obtain a canonical isomorphism of $E$-rational quivers
    \[
    \Gamma_E(S\otimes_KE)\cong\left(\Gamma_K(S),\rho|_{\Gal(L/E)}\right)=E\otimes_K\Gamma_K(S).
    \]

    We apply this to $S=S(\Gamma_K)$ and compose with the base change to $E$ of the inverse of the canonical isomorphism $\Gamma_K\to\Gamma_K(S(\Gamma_K))$ of Step 1 in the proof of Theorem \ref{thm:KrationalquiveretaleKspeciesantiequivalence}. Explicitly, a vertex $\alpha\colon L_k'\to L$ of $E\otimes_K\Gamma_K$ in the sense of Definition \ref{def:quiverbasechangegeneral}, $k\in J(i)$, is mapped to the vertex $\tau v_i$ of $\Gamma_K$. Here $\tau\in G$ is any element with $\tau\beta=\alpha(\beta\otimes1)$ for all $\beta\in L_i=L(v_i)$, i.e., $\tau|_{L_i}$ is the composition $L_i\to L_i\otimes_KE\to L_k'\xrightarrow{\alpha}L$; edges are treated in the same way. It depends only on the choices made in the definitions of $S(\Gamma_K)$ and of $S(\Gamma_K)\otimes_KE$.
  \end{proof}

  \subsection{Examples}

    In the following examples we choose $G=\Gal(\QQ[\sqrt{-1}]/\QQ)$, but one may also choose $G=\Gal(\CC/\RR)$ or any other quadratic extension $L/K$ for that matter. We denote by $c\in G$ the non-trivial element.

    \begin{example}\label{ex:twoloops}\rqleanmark{ex:twoloops}
    We consider the restriction of the quiver
    \begin{center}
            \hspace*{2em}\begin{tikzpicture}
              \node (v1) at (1,0) {$\diamond$};
              \draw[loopstyle] (v1) to [loop left] node[midway, left, draw=none] {$a$} (v1);
            \end{tikzpicture}
    \end{center}
    in the extension $\QQ[\sqrt{-1}]/\QQ$ in the sense of Definition \ref{def:quiverrestriction}. This is the disconnected quiver 
    \begin{center}
            \hspace*{2em}\begin{tikzpicture}
                \node (v1) at (-0.25,0) {$-$};
                \node (v2) at (2.25,0) {$+$};
                \draw[loopstyle] (v1) to [loop left] node[midway, left, draw=none] {$a_-$} (v1);
                \draw[loopstyle] (v2) to [loop right] node[midway, right, draw=none] {$a_+$} (v2);
            \end{tikzpicture}
    \end{center}
    endowed the following action of $G$, i.\,e.\ the non-trivial element $c\in G$ acts via:
    \[
    c\pm :=\mp,\quad c a_\pm:=a_\mp.
    \]
    The vertices $+$ and $-$ lie in the same orbit $\{+,-\}$ and $a_+$ and $a_-$ lie in one orbit $\{a_+,a_-\}$. Hence $I=G\backslash V=\{\{+,-\}\}$ has a single element. The stabilizers of vertices and edges are trivial in all cases, so that the associated \'etale $\QQ$-species is given by
    \begin{center}
            \hspace*{2em}\begin{tikzpicture}
              \node (v1) at (1,0) {$\QQ[\sqrt{-1}]=L_{\{\pm\}}\quad$};
                \draw[loopstyle] (v1) to [loop right] node[midway, right, draw=none] {${}_{\{\pm\}}M_{\{\pm\}}=\QQ[\sqrt{-1}]$} (v1);
            \end{tikzpicture}
    \end{center}
    Here the $\QQ[\sqrt{-1}],\QQ[\sqrt{-1}]$-bimodule
    \[
    {}_{\{\pm\}}M_{\{\pm\}}=\QQ[\sqrt{-1}]
    \]
    carries the canonical bimodule structure given by left and right multiplication.
    \end{example}

    \begin{example}\label{ex:cyclic}\rqleanmark{ex:cyclic}
    We consider the cyclic quiver $Q_{\rm cyclic}$:
    \begin{center}
            \begin{tikzpicture}
                \node (v1) at (-0.25,0) {$-$};
                \node (v2) at (2.25,0) {$+$};

                \draw[arr] (v1) to [bend left=20] node[midway, above, draw=none] {$a$} (v2);
                \draw[arr] (v2) to [bend left=20] node[midway, below, draw=none] {$b$} (v1);
            \end{tikzpicture}
    \end{center}
    and endow it with an action of $G$ by defining:
    \[
    c\pm :=\mp,\quad c a:=b,\quad cb:=a.
    \]
    Then $+$ and $-$ lie in the same orbit $\{+,-\}$ and $a$ and $b$ lie in one orbit $\{a,b\}$. Again $I=\{\{+,-\}\}$ has a single element and again the stabilizers of vertices and edges are trivial in all cases. However, in this case the bimodule structure is twisted: choosing $v_{\{\pm\}}=-$ and $e_{\{a,b\}}=a$ as representatives, we have $s(a)=-$ and $t(a)=+=c(-)$, so that $\sigma_{\{a,b\}}={\bf1}$ and $\tau_{\{a,b\}}=c$ in \eqref{eq:endsofedgetranslations}. Hence
    \[
    {}_{\{\pm\}}M_{\{\pm\}}=\QQ[\sqrt{-1}]^{\rm twisted}
    \]
    is the field $\QQ[\sqrt{-1}]$ with left structure map the identity and right structure map $c$: the left $L_{\{\pm\}}$-action is the standard one, while the right action is $m\cdot\alpha:=m\,c(\alpha)$, via complex conjugation. (Choosing other representatives yields an isomorphic species.)

    The associated \'etale $\QQ$-species is given by
    \begin{center}
            \hspace*{2em}\begin{tikzpicture}
                \node (v1) at (1,0) {$\QQ[\sqrt{-1}]=L_{\{\pm\}}\quad$};
                \draw[loopstyle] (v1) to [loop right] node[midway, right, draw=none] {${}_{\{\pm\}}M_{\{\pm\}}=\QQ[\sqrt{-1}]^{\rm twisted}$} (v1);
            \end{tikzpicture}
    \end{center}
  \end{example}

  \begin{example}\label{ex:gelfand}\rqleanmark{ex:gelfand}
    We consider the Gelfand quiver $Q_{\rm Gelfand}$, which is the quiver with relation
    \begin{center}
    \begin{tikzpicture}[baseline=(current bounding box.center)]
    \node (A) {$-$};
    \node (B) [right=6em of A] {$\star$};
    \node (C) [right=6em of B] {$+$};
    \node (D) [right=2em of C] {$a_{-} b_{-} = a_{+} b_{+}$};
    \draw [arr, bend left=20] (A) to node [midway, above] {$a_{-}$} (B);
    \draw [arr, bend left=20] (B) to node [midway, below] {$b_{-}$} (A);
    \draw [revarr, bend left=20] (B) to node [midway, above] {$a_{+}$} (C);
    \draw [revarr, bend left=20] (C) to node [midway, below] {$b_{+}$} (B);
    \end{tikzpicture}
    \end{center}
    We define the $G$-action as:
    \[
    c\pm :=\mp,\quad c\star := \star,\quad c a_{\pm}:=a_{\mp},\quad cb_{\pm}:=b_{\mp}.
    \]
    This action of $c$ is compatible with the given relation $a_{-} b_{-} = a_{+} b_{+}$ and thus defines a $\QQ$-structure (or $\RR$-structure or $K$-structure) on the Gelfand quiver. More precisely, $c$ carries the relation $a_{-} b_{-} = a_{+} b_{+}$ to the relation $a_{+} b_{+} = a_{-} b_{-}$, interchanging its two sides, so that the set of relations $\{a_{-} b_{-} = a_{+} b_{+},\;a_{+} b_{+} = a_{-} b_{-}\}$ is stable under the $G$-action. (As equations, the two members of this set are trivially equivalent; in terms of the element $r=a_{-}b_{-}-a_{+}b_{+}$ of the path algebra, $c(r)=-r$, so the ideal generated by $r$ is $G$-stable.) On the level of representations, $G$-equivariance \eqref{eq:koecheractiononrep} shows for every representation $M$ over $\QQ$ (resp.\ $\RR$ or $K$) that the two composites are conjugate under the semi-linear automorphism $\varphi_{\star,c}$ of $M(\star)$:
    \[
      \phi_{a_{+}}\circ\phi_{b_{+}}
      \;=\;
      \varphi_{\star,c}\circ\left(\phi_{a_{-}}\circ\phi_{b_{-}}\right)\circ\varphi_{\star,c}^{-1}.
    \]
    Consequently, a representation $M$ over $\QQ$ (resp.\ $\RR$ or $K$) satisfies the relation $a_{-} b_{-} = a_{+} b_{+}$ if and only if the $L$-linear endomorphism $\phi_{a_{-}}\circ\phi_{b_{-}}$ of $M(\star)$ commutes with $\varphi_{\star,c}$, i.e., is $\QQ$-rational (resp.\ $\RR$- or $K$-rational). Since $G_\star=G$, by Theorem \ref{thm:galoisdescent} this means that $\phi_{a_{-}}\circ\phi_{b_{-}}$ is the base change of a $\QQ$-linear (resp.\ $\RR$- or $K$-linear) endomorphism of $M(\star)^G$. We emphasize that the relation itself is not a consequence of $G$-equivariance: equivariance only guarantees the $\varphi_{\star,c}$-conjugacy of the two composites, whereas the relation asserts their equality, i.e., the rationality of $\phi_{a_{-}}\circ\phi_{b_{-}}$.

    Because of $G\pm=\{+,-\}$ and $G\star = \{\star\}$, the resulting \'etale $\QQ$-species has associated field extensions
    \begin{align*}
    L(\star)&=\QQ[\sqrt{-1}]^{G_\star}=\QQ[\sqrt{-1}]^{G}=\QQ,\\
    L(\pm)&=\QQ[\sqrt{-1}]^{G_{\pm}}=\QQ[\sqrt{-1}]^{\{\bf1\}}=\QQ[\sqrt{-1}].
    \end{align*}
    Likewise, for the two arrows $a_\pm$ we have
    \begin{align*}
    L(a_\pm)&=\QQ[\sqrt{-1}]^{G_{a_\pm}}=\QQ[\sqrt{-1}]^{\{\bf1\}}=\QQ[\sqrt{-1}],\\
    L(b_\pm)&=\QQ[\sqrt{-1}]^{G_{b_\pm}}=\QQ[\sqrt{-1}]^{\{\bf1\}}=\QQ[\sqrt{-1}].
    \end{align*}
    The index set is $I=\{\{\star\},\{+,-\}\}$. Therefore, the associated \'etale $\QQ$-species is
    \begin{center}
    \begin{tikzpicture}[baseline=(current bounding box.center)]
    \node (Mstar) {$\QQ=L_{\{\star\}}$};
    \node (M+)    [right=4em of Mstar] {\quad\quad$L_{\{\pm\}}=\QQ[\sqrt{-1}]$};
    \draw [arr, bend right=20] (Mstar) to node [below] {$\QQ[\sqrt{-1}]$} (M+);
    \draw [arr, bend right=20] (M+) to node [above] {$\QQ[\sqrt{-1}]$} (Mstar);
    \end{tikzpicture}
    \end{center}
    Here the $\QQ,\QQ[\sqrt{-1}]$- and $\QQ[\sqrt{-1}],\QQ$-bimodules
    \[
      {}_{\{\star\}}M_{\{\pm\}}=\QQ[\sqrt{-1}],\quad
      {}_{\{\pm\}}M_{\{\star\}}=\QQ[\sqrt{-1}],
    \]
    carry the natural bimodule structures. The $c$-twisted structures would provide us with isomorphic $K$-species.

  To illustrate Theorem \ref{thm:KrationalquiveretaleKspeciesantiequivalence} in this case, observe that
  \begin{align*}
  \Hom_\QQ(L_{\{\star\}},\QQ[\sqrt{-1}])&=
  \Hom_\QQ(\QQ,\QQ[\sqrt{-1}])=\{\iota\colon\QQ\to\QQ[\sqrt{-1}]\},\\[0.5em]
  \Hom_\QQ(L_{\{\pm\}},\QQ[\sqrt{-1}])&=
  \Hom_\QQ(\QQ[\sqrt{-1}],\QQ[\sqrt{-1}])=G=\{{\bf1},c\},
  \end{align*}
  Hence, the set of vertices of $\Gamma_K(S(Q_{\rm Gelfand}))$ is
  \[
  V=\{\iota,{\bf1},c\}\;\cong\;\{\star,+,-\},
  \]
  and the edge set is given by
  \begin{align*}
  E&=
  \Hom_{\QQ}(L_{\{\star\}\to\{\pm\}},\QQ[\sqrt{-1}])\sqcup
  \Hom_{\QQ}(L_{\{\pm\}\to\{\star\}},\QQ[\sqrt{-1}])\\
  &=
  \Hom_{\QQ}(\QQ[\sqrt{-1}],\QQ[\sqrt{-1}])\sqcup
  \Hom_{\QQ}(\QQ[\sqrt{-1}],\QQ[\sqrt{-1}])\\
  &=G\sqcup G.
  \end{align*}
  The elements of $G\sqcup G$ give rise to the four directed edges of the Gelfand quiver. We obtain source and targets by precomposition with $\iota$ and $\sigma\in G$. Which composition is the source and which composition is the target depends on the direction of the arrow between $\{\star\}$ and $\{\pm\}$.
  \end{example}

  \subsection{Representations of \'etale $K$-species}

  \begin{definition}[Representation of a $K$-species]\label{def:speciesrepresentation}\rqleanmark{def:speciesrepresentation}
  Let $S=(L_i, {}_iM_j)_{i,j\in I}$ be a $K$-species. A \emph{representation} of $S$ is a pair $(W, f)$ consisting of a collection $W = (W_i)_{i \in I}$ of finite-dimensional right $L_i$-vector spaces, and a collection $f = (f_{ij})_{i,j \in I}$ of $L_j$-linear maps
  \[
    f_{ij} \colon W_i \otimes_{L_i} {}_iM_j \to W_j.
  \]
  A \emph{morphism} $\psi \colon (W, f) \to (W', f')$ between two representations of $S$ is a collection of $L_i$-linear maps $\psi_i \colon W_i \to W'_i$ such that for all $i,j \in I$, the following diagram commutes:
  \begin{center}
  \begin{tikzpicture}[>=Stealth]
  \matrix (m) at (0,0) [matrix of math nodes, row sep=3em, column sep=4em]
  {
    W_i \otimes_{L_i} {}_iM_j & W_j \\
    W'_i \otimes_{L_i} {}_iM_j & W'_j \\
  };
  \path[-Stealth] (m-1-1) edge node[above] {$f_{ij}$} (m-1-2);
  \path[-Stealth] (m-1-1) edge node[left] {$\psi_i \otimes {\bf1}$} (m-2-1);
  \path[-Stealth] (m-1-2) edge node[right] {$\psi_j$} (m-2-2);
  \path[-Stealth] (m-2-1) edge node[below] {$f'_{ij}$} (m-2-2);
  \end{tikzpicture}
  \end{center}
  The category of finite-dimensional representations of $S$ is denoted by $\Rep(S)$.
  \end{definition}

  \begin{remark}\label{rmk:leftrightvectorspaces}
      Since the fields $L_i$ are commutative, we may identify left and right vector spaces, but we maintain the distinction for clarity in tensor products. We consider ${}_iM_j$ as left $L_i$ and as a right $L_j$-vector space.
  \end{remark}

  \begin{definition}[Relations for $K$-species]\label{def:speciesrelations}\rqleanmark{def:speciesrelations}
    Let $S=(L_i,{}_iM_j)_{i,j\in I}$ be a $K$-species. For a path $p=(i_0,i_1,\dots,i_n)$ in $I$ of length $n\geq 1$ we write
    \[
    {}_{i_0}M_p:=
    {}_{i_0}M_{i_1}\otimes_{L_{i_1}}{}_{i_1}M_{i_2}\otimes_{L_{i_2}}\cdots\otimes_{L_{i_{n-1}}}{}_{i_{n-1}}M_{i_n}
    \]
    for the associated $L_{i_0},L_{i_n}$-bimodule, and, for a representation $(W,f)$ of $S$, we write
    \[
    f_p\colon W_{i_0}\otimes_{L_{i_0}}{}_{i_0}M_p\to W_{i_n}
    \]
    for the iterated composition $f_{i_{n-1}i_n}\circ\left(f_{i_{n-2}i_{n-1}}\otimes{\bf1}\right)\circ\dots\circ\left(f_{i_0i_1}\otimes{\bf1}\right)$, which is $L_{i_n}$-linear.

    A \emph{relation} for $S$ is a family $R=(R_p)_p$, indexed by the paths $p$ in $I$ of length $\geq 1$, of $L_{i_0},L_{i_n}$-sub-bimodules $R_p\subseteq{}_{i_0}M_p$. A representation $(W,f)$ of $S$ \emph{satisfies} $R$ if
    \[
    f_p\left(W_{i_0}\otimes_{L_{i_0}}R_p\right)=0
    \]
    for every path $p$. We write $\Rep(S,R)\subseteq\Rep(S)$ for the full subcategory of those finite-dimensional representations of $S$ which satisfy $R$, and $\Rep^{\rm nil}(S,R)$ for the full subcategory of $\Rep(S,R)$ consisting of the nilpotent ones. Here a representation $(W,f)$ of $S$ is called {\em nilpotent} if there is an integer $n\geq1$ such that $f_p=0$ for every path $p$ in $I$ of length $\geq n$. We write $\Rep^{\rm nil}(S)$ for the full subcategory of nilpotent representations in $\Rep(S)$.
  \end{definition}

  \begin{example}\label{ex:gelfandspeciesrelation}\rqleanmark{ex:gelfandspeciesrelation}
    Let $L/K$ be a quadratic extension with non-trivial automorphism $c$, as permitted in the examples of \'etale $K$-species above, and let $S$ be the \'etale $K$-species associated to the $K$-rational Gelfand quiver $\Gamma_K$ of Example \ref{ex:gelfand} for $G=\Gal(L/K)$. We equip $S$ with a relation. Write $\star$ and $\pm$ for its two vertices $\{\star\}$ and $\{+,-\}$, so that $L_\star=K$ and $L_\pm=L$. The only path of length $2$ from $\star$ to $\star$ is $p=(\star,\pm,\star)$, and
    \[
    {}_\star M_p=
    {}_{\{\star\}}M_{\{\pm\}}\otimes_{L_{\{\pm\}}}{}_{\{\pm\}}M_{\{\star\}}
    =L\otimes_{L}L
    =L,
    \]
    a $K,K$-bimodule of dimension $2$ over $K$. We fix $\alpha\in L$ with $\alpha\neq0$ and $c(\alpha)=-\alpha$, and we let $R$ be the relation for $S$ whose only non-zero member is the $K$-line
    \[
    R_p:=K\cdot\alpha\;\subseteq\;{}_\star M_p.
    \]
    The line $K\cdot\alpha$ is the $(-1)$-eigenspace of $c$ on $L$, so it does not depend on the choice of $\alpha$. For $K=\QQ$ and $L=\QQ[\sqrt{-1}]$ we take $\alpha=\sqrt{-1}$, and we write $S_{\rm Gelfand}=S(\Gamma_\QQ)$ and
    \[
    \left(R_{\rm Gelfand}\right)_p:=\QQ\cdot\sqrt{-1}.
    \]

    This is the relation $a_{-}b_{-}=a_{+}b_{+}$ of Example \ref{ex:gelfand}, read through the functor $F$ of Theorem \ref{thm:species_equivalence}. Indeed, let $M$ be a $K$-rational representation of $\Gamma_K$ and $(W,f)=F(M)$, so that $W_\star=M(\star)^{G}$ and $W_\pm=M(-)$. Choosing $-$ as the representative of the orbit $\{+,-\}$ and $a_{-}$, $b_{-}$ as the representatives of the two edge orbits, the definition of $f$ gives $f_{\star\pm}(w\otimes a)=a\cdot\phi_{b_{-}}(w)$ --- the stabilizer of $-$ being trivial --- and, the stabilizer of $\star$ being all of $G$,
    \[
    f_{\pm\star}(u\otimes a')=a'\cdot\phi_{a_{-}}(u)+c(a')\cdot\varphi_{\star,c}\left(\phi_{a_{-}}(u)\right).
    \]
    Writing $z=aa'\in L={}_\star M_p$ and using the $G$-equivariance \eqref{eq:koecheractiononrep} twice, in the form $\varphi_{\star,c}\circ\phi_{a_{-}}=\phi_{a_{+}}\circ\varphi_{-,c}$ and $\varphi_{-,c}\circ\phi_{b_{-}}=\phi_{b_{+}}\circ\varphi_{\star,c}$, together with $\varphi_{\star,c}(w)=w$ for $w\in W_\star$, we obtain
    \begin{equation}
      f_p(w\otimes z)=
      z\cdot\left(\phi_{a_{-}}\circ\phi_{b_{-}}\right)(w)
      +c(z)\cdot\left(\phi_{a_{+}}\circ\phi_{b_{+}}\right)(w).
      \label{eq:gelfandspeciescomposite}
    \end{equation}
    At $z=\alpha$ the right hand side is $\alpha\cdot\left(\phi_{a_{-}}\phi_{b_{-}}-\phi_{a_{+}}\phi_{b_{+}}\right)(w)$. Since $M(\star)=W_\star\otimes_K L$ by \eqref{eq:WitoMvi}, the representation $F(M)$ satisfies $R$ if and only if $M$ satisfies the relation $a_{-}b_{-}=a_{+}b_{+}$. In particular, $F(M)$ satisfies $R_{\rm Gelfand}$ if and only if a $\QQ$-rational representation $M$ of $\Gamma_\QQ$ satisfies $a_{-}b_{-}=a_{+}b_{+}$.

    We stress that $R$ is a \emph{single} relation, the descent of the $G$-stable relation set of Example \ref{ex:gelfand}: writing $X:=\phi_{a_{-}}\circ\phi_{b_{-}}$ and $Y:=\phi_{a_{+}}\circ\phi_{b_{+}}=\varphi_{\star,c}\circ X\circ\varphi_{\star,c}^{-1}$, \eqref{eq:gelfandspeciescomposite} shows that $F(M)$ satisfies $R$ if and only if $X=Y$ on $W_\star$, i.e., if and only if $X$ commutes with $\varphi_{\star,c}$ and is thus $K$-rational. It is, however, a genuine condition on $(W,f)$, and not one that $G$-equivariance already supplies: as Example \ref{ex:gelfand} records, equivariance guarantees only that the two composites appearing in \eqref{eq:gelfandspeciescomposite} are $\varphi_{\star,c}$-conjugate, whereas $R$ asserts their equality.
  \end{example}

  \begin{definition}[Base change for representations of \'etale $K$-species]\label{def:speciesrepbasechange}\rqleanmark{def:speciesrepbasechange}
    Let $S=(L_i, {}_iM_j)_{i,j\in I}$ be an \'etale $K$-species and $W=(W_i,f_{ij})_{i,j\in I}$ a representation of $S$. Write $S\otimes_K N=(L_k',{}_kM_\ell')_{k,\ell\in I_N}$ for the base change of $S$ inside an extension $N/K$ (Definition \ref{def:basechangeforspecies}).
    \begin{itemize}
    \item[(i)] For any $i\in I$ and any $k\in J(i)\subseteq I_N$ put
      \[W_k':=W_i\otimes_{L_i}L_k'.\]
    \item[(ii)] For any $i,j\in I$ and $\ell\in J(j)$ we have a canonical decomposition
      \[
      (W_i\otimes_{L_i}{}_iM_j)\otimes_{L_j} L_\ell'=\bigoplus_{k\in J(i)}W_k'\otimes_{L_k'}{}_kM_\ell',
      \]
      obtained from $W_i\otimes_{L_i}{}_iM_j\otimes_{L_j}L_\ell'=W_i\otimes_{L_i}({}_iM_j\otimes_KN)e_\ell$ and $1=\sum_{k\in J(i)}e_k$ (notation of Definition \ref{def:basechangeforspecies}). For $k\in J(i)$ and $\ell\in J(j)$ let
      \[
      f_{k\ell}'\colon\;
      W_k'\otimes_{L_k'}{}_kM_\ell'\to
      W_j\otimes_{L_j} L_\ell'=W_\ell'
      \]
      be the restriction of $f_{ij}\otimes{\bf1}_{L_\ell'}\colon(W_i\otimes_{L_i}{}_iM_j)\otimes_{L_j} L_\ell'\to W_j\otimes_{L_j} L_\ell'$ to the $k$-th summand. (The individual summand $W_k'\otimes_{L_k'}{}_kM_\ell'$ is in general strictly smaller than $(W_i\otimes_{L_i}{}_iM_j)\otimes_{L_j} L_\ell'$; only the sum over $k\in J(i)$ recovers it.)
    \end{itemize}
    Then $W\otimes_KN:=(W_k',f_{k\ell}')_{k,\ell\in I_N}$ is a representation of $S\otimes_K N$ which we call the {\em base change} of $W$.
  \end{definition}

  \begin{remark}\label{rmk:repbasechangegeneral}
    In light of Remark \ref{rmk:simplealgebraspecies} this notion of base change naturally extends to representations of $K$-species.
  \end{remark}

  \subsection{Categorical equivalence of representations}
  
  We construct a pair of quasi-inverse functors
  \[
  F \colon \Rep(\Gamma_K) \to \Rep(S(\Gamma_K))\quad\text{and}\quad H \colon \Rep(S(\Gamma_K)) \to \Rep(\Gamma_K).
  \]

  \paragraph{The functor $F \colon \Rep(\Gamma_K) \to \Rep(S(\Gamma_K))$.}
  Let
  \[
  M_K = ((M(v))_{v\in V}, (\varphi_{v,\sigma})_{v\in V,\sigma\in G}, (\phi_e)_{e\in E})
  \]
  be a $K$-rational representation of $\Gamma_K$. In order to define its image $F(M_K) = (W_i, f_{ij})_{i,j\in I}$ in $\Rep(S(\Gamma_K))$, we define the $L_i$-right vector spaces $W_i$ first. 

  \smallskip
  For each orbit $i \in I = G\backslash V$, we choose a representative $v_i \in i$ and define the $L_i$-vector space $W_i := M(v_i)^{G_{v_i}}$. The continuous semi-linear action of $G_{v_i}$ on the $L$-vector space $M(v_i)$ makes the space of invariants $M(v_i)^{G_{v_i}}$ into a vector space over the fixed field $L_i = L^{G_{v_i}}$. By Galois descent (Thm.~\ref{thm:galoisdescent}), we have a canonical isomorphism of $L$-vector spaces
    \begin{equation}
      W_i\otimes_{L_i} L\cong M(v_i).
      \label{eq:WitoMvi}
    \end{equation}

    \smallskip
    The definition of the maps $f_{ij}$ is more involved. Here the representatives $v_i$ are those used in the definition of $S(\Gamma_K)$. For each orbit $\varepsilon\in G\backslash E$ we use the representative $e_\varepsilon \in \varepsilon$ and the elements $\sigma_\varepsilon,\tau_\varepsilon\in G$ with
    \[
      s(e_\varepsilon)=\sigma_\varepsilon v_i\quad\text{and}\quad
      t(e_\varepsilon)=\tau_\varepsilon v_j
    \]
    fixed in \eqref{eq:endsofedgetranslations}, where $i,j\in I$ denote the $G$-orbits of $s(e_\varepsilon)$ and $t(e_\varepsilon)$ respectively, and we put $L_\varepsilon=L(e_\varepsilon)$. Then $L(e_\varepsilon)$ is an $L(v_i),L(v_j)$-bimodule via $(\sigma_\varepsilon,\tau_\varepsilon)$, i.e., via the structure maps $\lambda^s_\varepsilon$ and $\lambda^t_\varepsilon$ of $S(\Gamma_K)$.

    \begin{proposition}\label{prop:ledecomposition}\rqleanmark{prop:ledecomposition}
      For any right $L(e_\varepsilon)$-vector space $V$ we have a canonical isomorphism of $L$-vector spaces
    \begin{equation}
      V\otimes_{\tau_\varepsilon,L(v_j)}L
      =
      \prod_{\begin{subarray}c\eta G_{\tau_\varepsilon^{-1}e_\varepsilon}\in\\G_{v_j}/G_{\tau_\varepsilon^{-1}e_\varepsilon}\end{subarray}} V\otimes_{L(e_\varepsilon),\eta\tau_\varepsilon^{-1}}L.
      \label{eq:Leepsilondecomposition}
    \end{equation}
    On pure tensors, this isomorphism is given by
    \(
    v\otimes a\mapsto (v\otimes a)_{\eta}.
    \)
    \end{proposition}

    \begin{proof}
    We have the canonical bimodule decomposition
    \begin{equation}
      L(e_\varepsilon)\otimes_{\tau_\varepsilon,L(v_j)}L
      =
      \prod_{\tau\colon L(\tau_\varepsilon^{-1}e_\varepsilon)\to L}
      L(e_\varepsilon)\otimes_{L(e_\varepsilon),\tau\tau_\varepsilon^{-1}}L,
      \label{eq:Leepsilondecompositionpre}
    \end{equation}
    where $\tau$ runs through all $L(v_j)$-linear embeddings (cf.\ \eqref{eq:endsofedgetranslations}).

    Writing $\iota_{\varepsilon}\colon L(\tau_\varepsilon^{-1}e_\varepsilon)\to L$ for the canonical inclusion corresponding to the inclusion $G_{\tau_\varepsilon^{-1}e_\varepsilon}\to G$, each $L(v_j)$-linear $\tau$ is of the form
    \[
    \tau=\eta\iota_{\varepsilon},\quad\text{for a unique coset }\eta G_{\tau_\varepsilon^{-1}e_\varepsilon} \in G_{v_j}/G_{\tau_\varepsilon^{-1}e_\varepsilon}.
    \]
    Hence \eqref{eq:Leepsilondecompositionpre} implies that for any right $L(e_\varepsilon)$-vector space $V$ we have \eqref{eq:Leepsilondecomposition} and the map has the given description.
    \end{proof}

    {\em Convention.} Here and in the sequel, for each $\varepsilon$ the index $\eta$ runs through a system of representatives of the finite set $G_{v_j}/G_{\tau_\varepsilon^{-1}e_\varepsilon}$, and all objects and maps below are independent of this choice. Indeed, for $g\in G_{\tau_\varepsilon^{-1}e_\varepsilon}$ we have $\tau_\varepsilon g\tau_\varepsilon^{-1}\in G_{e_\varepsilon}\subseteq G_{s(e_\varepsilon)}=G_{\sigma_\varepsilon v_i}$, hence
    \[
    \eta g\tau_\varepsilon^{-1}\sigma_\varepsilon=\eta\tau_\varepsilon^{-1}\sigma_\varepsilon h\quad\text{with}\quad h:=\sigma_\varepsilon^{-1}\tau_\varepsilon g\tau_\varepsilon^{-1}\sigma_\varepsilon\in G_{v_i};
    \]
    in particular $\eta g\tau_\varepsilon^{-1}\sigma_\varepsilon v_i=\eta\tau_\varepsilon^{-1}\sigma_\varepsilon v_i$, and $h$ acts trivially on $L(v_i)$ and on $M(v_i)^{G_{v_i}}$. Since the index sets are finite, we freely identify finite direct products with finite direct sums.

    \begin{proposition}\label{prop:WiiMjLpre}\rqleanmark{prop:WiiMjLpre}
      We have a canonical isomorphism of $L$-vector spaces
    \begin{equation}
    W_i \otimes_{L_i} {}_iM_j\otimes_{L_j}L=
    \bigoplus_{\varepsilon\in G\backslash E_{ij}}
    \prod_{\eta}
    M(v_i)^{G_{v_i}} \otimes_{L(v_i),\eta\tau_\varepsilon^{-1}\sigma_\varepsilon}L,
    \label{eq:WiiMjLpre}
    \end{equation}
    where $\eta$ runs through a system of representatives for $G_{v_j}/G_{\tau_\varepsilon^{-1}e_\varepsilon}$.
    \end{proposition}

    \begin{proof}
    Recalling the definition of ${}_iM_j$ we rewrite
    \[
    W_i \otimes_{L_i} {}_iM_j\otimes_{L_j}L=
    \bigoplus_{\varepsilon\in G\backslash E_{ij}}
    M(v_i)^{G_{v_i}} \otimes_{L(v_i),\sigma_\varepsilon} L(e_\varepsilon)\otimes_{\tau_\varepsilon,L(v_j)}L.
    \]
    In light of the decomposition \eqref{eq:Leepsilondecomposition}, identity \eqref{eq:WiiMjLpre} follows.
    \end{proof}

    \begin{proposition}\label{prop:generalextensionofscalars}\rqleanmark{prop:generalextensionofscalars}
      For any $v\in V$ and any $\sigma\in G$ the given semi-linear map $\varphi_{v,\sigma}\colon M(v)\to M(\sigma v)$ induces a map $M(v)^{G_{v}}\to M(\sigma v)$ which extends to a unique $L$-linear map
    \begin{equation}
      M(v)^{G_{v}}\otimes_{L(v),\sigma}L\to M(\sigma v),\;m\otimes a\mapsto \varphi_{v,\sigma}(m)\cdot a,
      \label{eq:generalextensionisomorphism}
    \end{equation}
    which is an isomorphism of $L$-vector spaces. Moreover, by the same universal property, we have for any $\tau\in G$ a commutative square
    \begin{equation}
      \begin{tikzpicture}[>=Stealth]
        \matrix (m) at (0,0) [matrix of math nodes, row sep=3em, column sep=6.5em]
                {
                  M(v)^{G_{v}}\otimes_{L(v),\sigma}L & M(v)^{G_{v}}\otimes_{L(v),\tau\sigma}L \\
                  M(\sigma v) & M(\tau\sigma v) \\
                };
                \path[-Stealth] (m-1-1) edge node[font=\scriptsize, above] {$m\otimes a\mapsto m\otimes\tau(a)$} (m-1-2);
                \path[-Stealth] (m-1-1) edge (m-2-1);
                \path[-Stealth] (m-1-2) edge (m-2-2);
                \path[-Stealth] (m-2-1) edge node[below] {$\varphi_{\sigma v,\tau}$} (m-2-2);
      \end{tikzpicture}
      \label{eq:semi-linearsquare}
    \end{equation}
    \end{proposition}

    \begin{proof}
      Existence and uniqueness of \eqref{eq:generalextensionisomorphism} follows by the universal property of the left hand side. Galois descent shows that \eqref{eq:generalextensionisomorphism} is an isomorphism of $L$-vector spaces. The commutativity of \eqref{eq:semi-linearsquare} follows by the universal property of $M(v)^{G_{v}}\otimes_{L(v),\sigma}L$.
    \end{proof}

    \begin{corollary}\label{cor:WiiMjL}\rqleanmark{cor:WiiMjL}
    We have canonical $L$-linear isomorphisms
    \begin{equation}
      \psi_{i,j,\varepsilon,\eta}\colon M(v_i)^{G_{v_i}} \otimes_{L(v_i),\eta\tau_\varepsilon^{-1}\sigma_\varepsilon}L\to M(\eta\tau_\varepsilon^{-1}\sigma_\varepsilon v_i).
      \label{eq:Llinearisos}
    \end{equation}
    In particular, we have for all $i,j\in I$ a canonical isomorphism
    \begin{equation}
    W_i \otimes_{L_i} {}_iM_j\otimes_{L_j}L=
    \bigoplus_{\varepsilon\in G\backslash E_{ij}}
    \prod_{\eta}
    M(\eta\tau_\varepsilon^{-1}\sigma_\varepsilon v_i).
    \label{eq:WiiMjL}
    \end{equation}
    \end{corollary}

    \begin{proof}
       We apply Proposition \ref{prop:generalextensionofscalars} to each summand on the right hand side of \eqref{eq:WiiMjLpre}.
    \end{proof}
    
    The identification \eqref{eq:WiiMjL} gives rise to the following $L(v_j)$-rational structure on $W_i \otimes_{L_i} {}_iM_j\otimes_{L_j}L$. For any $\sigma\in \Gal(L/L(v_j))=G_{v_j}$ we consider the map
    \begin{align*}
    \Theta_{ij}(\sigma)\colon\;
    \bigoplus_{\varepsilon\in G\backslash E_{ij}}
    \prod_{\eta}
    M(\eta\tau_\varepsilon^{-1}\sigma_\varepsilon v_i)
    &\to
    \bigoplus_{\varepsilon\in G\backslash E_{ij}}
    \prod_{\eta}
    M(\eta\tau_\varepsilon^{-1}\sigma_\varepsilon v_i)\\
    (m_{\varepsilon,\eta})_{\varepsilon,\eta}&\mapsto
    \left(\varphi_{\sigma^{-1}\eta\tau_\varepsilon^{-1}\sigma_\varepsilon v_i,\sigma}(m_{\varepsilon,\sigma^{-1}\eta})\right)_{\varepsilon,\eta}.
    \end{align*}
    This map is clearly $\sigma$-linear and satisfies the cocycle condition
    \[
    \forall\sigma,\tau\in G_{v_j}:\quad\Theta_{ij}(\sigma)\circ\Theta_{ij}(\tau)=\Theta_{ij}(\sigma\tau).
    \]
    The resulting action of $G_{v_j}$ is moreover continuous: the index set $G_{v_j}/G_{\tau_\varepsilon^{-1}e_\varepsilon}$ is finite, $G_{\tau_\varepsilon^{-1}e_\varepsilon}$ being open in $G_{v_j}$, and for any given element $(m_{\varepsilon,\eta})_{\varepsilon,\eta}$, all $\sigma$ in a sufficiently small open subgroup of $G_{v_j}$ fix each of the finitely many cosets $\eta$ as well as each component $m_{\varepsilon,\eta}$, the latter by the continuity condition in Definition \ref{def:rationalrepresentations}. Hence we obtain an $L(v_j)$-rational structure on $W_i \otimes_{L_i} {}_iM_j\otimes_{L_j}L$ as claimed.

    \begin{proposition}\label{prop:rationalstructures}\rqleanmark{prop:rationalstructures}
      The $L(v_j)$-rational structure $\Theta_{ij}$ on $W_i \otimes_{L_i} {}_iM_j\otimes_{L_j}L$ agrees with the canonical $L(v_j)$-rational structure
      \[
      W_i \otimes_{L_i} {}_iM_j\;\subseteq\; W_i \otimes_{L_i} {}_iM_j\otimes_{L_j}L.
      \]
    \end{proposition}

    \begin{proof}
      Note that the canonical $L_j$-rational structure on $W_i \otimes_{L_i} {}_iM_j\otimes_{L_j}L$ corresponds to the semi-linear automorphisms
      \[w\otimes m\otimes a\mapsto w\otimes m\otimes\sigma a,\]
      for $\sigma\in\Gal(L/L(v_j))$.

      In order to see that both rational structures agree, we observe that for any $\sigma\in\Gal(L/L(v_j))$ we consider the semi-linear automorphism
    \[
    {\bf1}_{V}\otimes \sigma\colon V\otimes_{\tau_\varepsilon,L(v_j)}L\to V\otimes_{\tau_\varepsilon,L(v_j)}L,\;v\otimes a\mapsto v\otimes\sigma a.
    \]
    In light of the explicit description of the right $L$-vector space isomorphism \eqref{eq:WiiMjLpre}, we have the commutative square
    \begin{center}
      \begin{tikzpicture}[>=Stealth]
        \node (N1) at (0,0) {$V\otimes_{\tau_\varepsilon,L(v_j)}L$};
        \node (N2) at (5,0) {$\prod\limits_{\eta} V\otimes_{L(e_\varepsilon),\eta\tau_\varepsilon^{-1}\iota_\varepsilon}L$};
        \node (N3) at (0,-2) {$V\otimes_{\tau_\varepsilon,L(v_j)}L$};
        \node (N4) at (5,-2) {$\prod\limits_{\eta} V\otimes_{L(e_\varepsilon),\eta\tau_\varepsilon^{-1}\iota_\varepsilon}L$};
        \draw[-Stealth] (N1.east) -- (N2.west) node[midway,above] {\eqref{eq:WiiMjLpre}};
        \draw[-Stealth] (N1.south) -- (N3.north) node[midway,left] {$\mathbf{1}_{V}\otimes \sigma$};
        \draw[-Stealth] (N2.south) -- (N4.north);
        \draw[-Stealth] (N3.east) -- (N4.west) node[midway,below] {\eqref{eq:WiiMjLpre}};
      \end{tikzpicture}
    \end{center}
    where the right vertical map is given by
    \begin{equation}
      (m_\eta\otimes a_\eta)_{\eta}\mapsto (m_{\sigma^{-1}\eta}\otimes \sigma a_{\sigma^{-1}\eta} )_{\eta}
      \label{eq:explicitdecompositiontwisting}
    \end{equation}
    This map descends mutatis mutandis to a $\sigma$-linear automorphism of the right hand side of \eqref{eq:WiiMjLpre}. By the commutativity of \eqref{eq:semi-linearsquare}, this shows that the corresponding map on the right hand side of \eqref{eq:WiiMjL} is explicitly given by
    \[
    (m_{\varepsilon,\eta})_{\varepsilon,\eta}\mapsto
    \left(\varphi_{\sigma^{-1}\eta\tau_\varepsilon^{-1}\sigma_\varepsilon v_i,\sigma}(m_{\varepsilon,\sigma^{-1}\eta})\right)_{\varepsilon,\eta},
    \]
    which agrees with the Definition of $\Theta_{ij}(\sigma)$. This shows the claim.
    \end{proof}
    
    To define the maps
    \[f_{ij} \colon W_i \otimes_{L_i} {}_iM_j \to W_j\]
    we first base change to $L$ on the domain and codomain to define
    \[f_{ij}\otimes{\bf1}_L \colon W_i \otimes_{L_i} {}_iM_j\otimes_{L_j}L \to W_j\otimes_{L_j}L\]
    using \eqref{eq:WiiMjL}. Then we will define $f_{ij}$ as a suitable Galois descent of the yet-to-be-defined map $f_{ij}\otimes{\bf1}_L$.
    
    As for the definition of $f_{ij}\otimes{\bf1}_L$, we put
    \begin{align*}
    \Phi_{ij}\colon
    \bigoplus_{\varepsilon\in G\backslash E_{ij}}
    \prod_{\eta}
    M(v_i)^{G_{v_i}} \otimes_{L(v_i),\eta\tau_\varepsilon^{-1}\sigma_\varepsilon}L
    &\to
    M(v_j)=M(v_j)^{G_{v_j}}\otimes_{L(v_j)}L,\\
    (m_{\varepsilon,\eta}\otimes a_{\varepsilon,\eta})_{\varepsilon,\eta}&\mapsto
    \sum_{\varepsilon,\eta}
    \varphi_{t(e_\varepsilon),\eta\tau_\varepsilon^{-1}}\phi_{e_\varepsilon}\varphi_{v_i,\sigma_\varepsilon}(m_{\varepsilon,\eta})\cdot a_{\varepsilon,\eta}.
    \end{align*}
    Here $\varphi_{v_i,\sigma_\varepsilon}\colon M(v_i)\to M(s(e_\varepsilon))$ and $\varphi_{t(e_\varepsilon),\eta\tau_\varepsilon^{-1}}\colon M(t(e_\varepsilon))\to M(\eta\tau_\varepsilon^{-1}t(e_\varepsilon))=M(\eta v_j)=M(v_j)$. By \eqref{eq:semi-linearactiononrep}, \eqref{eq:koecheractiononrep} and the convention above, the $(\varepsilon,\eta)$-term does not depend on the representative $\eta$ of its coset: replacing $\eta$ by $\eta g$ with $g\in G_{\tau_\varepsilon^{-1}e_\varepsilon}$ inserts $\varphi_{t(e_\varepsilon),g'}$ with $g'=\tau_\varepsilon g\tau_\varepsilon^{-1}\in G_{e_\varepsilon}$, and $\varphi_{t(e_\varepsilon),g'}\circ\phi_{e_\varepsilon}\circ\varphi_{v_i,\sigma_\varepsilon}=\phi_{e_\varepsilon}\circ\varphi_{v_i,\sigma_\varepsilon h}$ with $h\in G_{v_i}$ as above, which agrees with $\phi_{e_\varepsilon}\circ\varphi_{v_i,\sigma_\varepsilon}$ on $M(v_i)^{G_{v_i}}$.

    \begin{proposition}\label{prop:phidescends}\rqleanmark{prop:phidescends}
      $\Phi_{ij}$ descends to an $L_j$-linear map
      \[
      f_{ij}\colon W_i \otimes_{L_i} {}_iM_j \to W_j.
      \]
    \end{proposition}

    \begin{proof}
      We show that $\Phi_{ij}$ is $G_{v_j}$-equivariant. For an arbitrary element $(m_{\varepsilon,\eta}\otimes a_{\varepsilon,\eta})_{\varepsilon,\eta}$ we compute:
      \begin{align*}
      \Phi_{ij}\left(\Theta_{ij}(\sigma)((m_{\varepsilon,\eta}\otimes a_{\varepsilon,\eta})_{\varepsilon,\eta})\right)
      &=
      \Phi_{ij}\left((m_{\varepsilon,\sigma^{-1}\eta}\otimes \sigma a_{\varepsilon,\sigma^{-1}\eta} )_{\varepsilon,\eta}\right)\\
      &=
      \sum_{\varepsilon,\eta}
      \varphi_{t(e_\varepsilon),\eta\tau_\varepsilon^{-1}}\phi_{e_\varepsilon}\varphi_{v_i,\sigma_\varepsilon}(m_{\varepsilon,\sigma^{-1}\eta})\cdot \sigma a_{\varepsilon,\sigma^{-1}\eta}\\
      &=
      \sum_{\varepsilon,\eta}
      \varphi_{t(e_\varepsilon),\sigma\eta\tau_\varepsilon^{-1}}\phi_{e_\varepsilon}\varphi_{v_i,\sigma_\varepsilon}(m_{\varepsilon,\eta})\cdot \sigma a_{\varepsilon,\eta}\\
      &=
      \varphi_{v_j,\sigma}
      \left(
      \sum_{\varepsilon,\eta}
      \varphi_{t(e_\varepsilon),\eta\tau_\varepsilon^{-1}}\phi_{e_\varepsilon}\varphi_{v_i,\sigma_\varepsilon}(m_{\varepsilon,\eta})\cdot a_{\varepsilon,\eta}
      \right)\\
      &=
      \varphi_{v_j,\sigma}
      \left(
      \Phi_{ij}\left((m_{\varepsilon,\eta}\otimes a_{\varepsilon,\eta})_{\varepsilon,\eta}\right)
      \right).
      \end{align*}
      where we invoked \eqref{eq:explicitdecompositiontwisting} for the first identity. This proves the Galois-equivariance with respect to the rational structure defined by $\Theta_{ij}$. By Proposition \ref{prop:rationalstructures}, this rational structure agrees with the canonical rational structure. Therefore, $\Phi_{ij}$ descends via Galois descent to an $L_j$-linear map $f_{ij}$ as claimed.
    \end{proof}

    \begin{proposition}\label{prop:fonmorphisms}\rqleanmark{prop:fonmorphisms}
      For any morphism $\psi \colon M_K \to N_K$ in $\Rep(\Gamma_K)$ the collection of $L$-linear maps $\psi_v \colon M(v) \to N(v)$ for $v\in V$ induces for any $i\in I=G\backslash V$ a canonical $L_i$-linear map
      \[
      \psi_{i}\colon M(v_i)^{G_{v_i}}\to N(v_i)^{G_{v_i}}.
      \]
      The collection $(\psi_i)_{i\in I}$ defines a morphism $F(M_K)\to F(N_K)$ of $K$-species representations.
      \end{proposition}

      \begin{proof}
        We observe that for any $i\in I$ condition \eqref{eq:koecheractiononrephom} implies that $\psi_{v_i}$ is $G_{v_i}$-equivariant and hence descends to an $L^{G_{v_i}}$-linear map as claimed.
        
      Assume that
      \[
      M_K=((M(v))_{v\in V},(\varphi_{v,\sigma}),(\phi_e)_{e\in E})\quad\text{and}\quad
      N_K=((N(v))_{v\in V},(\varphi'_{v,\sigma}),(\phi_e')_{e\in E}).
      \]
      Then an elementary calculation shows that the $L$-linear maps $\Phi_{ij}$ underlying the definition of $f_{ij}$ commute with the collection of the maps $\psi_{v_i}$ by \eqref{eq:koecheractiononrep}. By Galois descent, this implies that the maps $\psi_{i}$ for $i\in I$ define a $K$-species homomorphism as claimed.      
      \end{proof}

      \begin{theorem}\label{thm:quiveretalespeciesrepresentations}\rqleanmark{thm:quiveretalespeciesrepresentations}
        The maps
        \begin{equation}
        F\colon M_K=((M(v))_{v\in V},(\varphi_{v,\sigma})_{v\in V,\sigma\in G},(\phi_e)_{e\in E})\;\mapsto\;((M(v_i)^{G_{v_i}})_{i\in I},(f_{ij})_{i,j\in I}),
        \end{equation}
        and
        \begin{equation}
        F\colon \left((\psi_v)_{v\in V}\colon M_K\to N_K\right)\;\mapsto\;(\psi_i)_{i\in I}
        \end{equation}
        defines a functor $\Rep(\Gamma_K)\to\Rep(S(\Gamma_K))$.
      \end{theorem}

      \begin{proof}
        By our previous considerations both maps are well defined. The verification of the functor properties is straightforward.
      \end{proof}

  \paragraph{The functor $H \colon \Rep(S(\Gamma_K)) \to \Rep(\Gamma_K)$.}
  Let
  \[
  (W_{i}, f_{ij})_{i,j\in I}
  \]
  be a representation of $S(\Gamma_K)$. We are going to define a representation of the $K$-rational quiver $\Gamma_K$
  \[
  H((W_i, f_{ij})) = ((M(v))_{v\in V}, (\varphi_{v,\sigma})_{v\in V,\sigma\in G},(\phi_e)_{e\in E})
  \]
  as follows.

  We fix for every $i\in I$ a $K$-rational structure
  \begin{equation}
    W_{i,K}\subseteq W_i,
    \label{eq:KrationalWi}
  \end{equation}
  Then for each vertex $v \in V$, we consider its orbit $i=Gv$, which by our previous choices in the previous section above, agrees with $Gv_{i}$. With these choices we put
  \begin{equation}
    M(v) := W_{i,K}\otimes_{K}L.
    \label{eq:MvforWi}
  \end{equation}
  Then we have for every $\sigma\in G$ a canonical $\sigma$-linear isomorphism
  \begin{equation}
    \varphi_{v,\sigma}\colon M(v)\to M(\sigma v),\quad w\otimes a\mapsto w\otimes \sigma a.
    \label{eq:varphi2}
  \end{equation}
  These isomorphisms satisfy the cocycle condition \eqref{eq:semi-linearactiononrep}, and the resulting action of $G$ on $\coprod_{v\in V}M(v)$ is continuous, since it is obtained by base change from the $K$-rational structures \eqref{eq:KrationalWi} and the $G$-action on $V$ is continuous.

  In order to define for every $e\in E$ an $L$-linear map $\phi_e\colon M(s(e))\to M(t(e))$ satisfying \eqref{eq:koecheractiononrep}, we proceed as follows. We first consider the following analogue of Corollary \ref{cor:WiiMjL}.

  \begin{proposition}\label{prop:hdecomposition}\rqleanmark{prop:hdecomposition}
    For every $i,j\in I$ and every $\varepsilon\in E_{ij}$ recall the chosen representative $e_\varepsilon\in\varepsilon$ and let $\sigma_\varepsilon,\tau_\varepsilon\in G$ be as in \eqref{eq:endsofedgetranslations}. Then in light of definition \eqref{eq:MvforWi}, we have a canonical isomorphism
    \begin{equation}
    W_i \otimes_{L_i} {}_iM_j\otimes_{L_j}L=
    \bigoplus_{\varepsilon\in G\backslash E_{ij}}
    \bigoplus_{\eta}
    M(\eta\tau_\varepsilon^{-1}\sigma_\varepsilon v_i),
    \label{eq:WiiMjL2}
    \end{equation}
    where $\eta$ runs through a system of representatives for $G_{v_j}/G_{\tau_\varepsilon^{-1}e_\varepsilon}$.
  \end{proposition}

  \begin{proof}
    Observe first that we have an isomorphism
    \[
    W_{i,K}\otimes_K L_i\to W_i.
    \]
    Recalling our definition ${}_iM_j=\bigoplus_{\varepsilon\in G\backslash E_{ij}}L(e_\varepsilon)$ gives us
    \begin{align*}
      W_i \otimes_{L_i} {}_iM_j\otimes_{L_j}L
      &=\bigoplus_{\varepsilon\in G\backslash E_{ij}}W_{i}\otimes_{L(v_i),\sigma_\varepsilon} L(e_\varepsilon)\otimes_{\tau_\varepsilon,L(v_j)}L\\
      &=\bigoplus_{\varepsilon\in G\backslash E_{ij}}W_{i}\otimes_{L(v_i),\sigma_\varepsilon} \prod_{\eta}L(e_\varepsilon)\otimes_{L(e_\varepsilon),\eta\tau_\varepsilon^{-1}}L&\text{(by \eqref{eq:Leepsilondecomposition})}\\
      &=\bigoplus_{\varepsilon\in G\backslash E_{ij}}W_{i}\otimes_{L(v_i)} \prod_{\eta}L(v_i)\otimes_{L(v_i),\eta\tau_\varepsilon^{-1}\sigma_\varepsilon}L,
    \end{align*}
    where $\eta$ runs through a system of representatives for $G_{v_j}/G_{\tau_\varepsilon^{-1}e_\varepsilon}$ as before. By our choice of $K$-rational structure, this agrees with
    \[
    W_{i,K}\otimes_K L(v_i)\otimes_{L(v_i)} \prod_{\eta}L(v_i)\otimes_{L(v_i),\eta\tau_\varepsilon^{-1}\sigma_\varepsilon}L=
    \bigoplus_\eta W_{i,K}\otimes_K L(v_i)\otimes_{L(v_i),\eta\tau_\varepsilon^{-1}\sigma_\varepsilon}L,
    \]
    where we identified the finite direct product with a finite direct sum. Now the $\eta$-th direct summand on the right hand side is canonically isomorphic to $M(\eta\tau_\varepsilon^{-1}\sigma_\varepsilon v_i)$ via \eqref{eq:varphi2}. This proves the claim.
  \end{proof}

  For every $i,j\in I$ and every $e\in E_{ij}$ let $\sigma_e\in G$ be such that $\sigma_e e_{\varepsilon(e)}=e$ for $\varepsilon(e)=Ge$. Note that
  \[
  \sigma_e\tau_{\varepsilon(e)} v_j=
  \sigma_e t(e_{\varepsilon(e)})=
  t(\sigma_e e_{\varepsilon(e)})=
  t(e).
  \]
  Then we define $\phi_{e}\colon M(s(e))\to M(t(e))$ via \eqref{eq:WiiMjL2} as the composition
  \begin{align*}
      M(s(e))&= M(\sigma_e s(e_{\varepsilon(e)}))\\
      &=M(\sigma_e \sigma_{\varepsilon(e)} v_i)\\
      &\to M(\tau_{\varepsilon(e)}^{-1}\sigma_{\varepsilon(e)} v_i)&\text{(via $\varphi_{\sigma_e\sigma_{\varepsilon(e)}v_i,\tau_{\varepsilon(e)}^{-1}\sigma_e^{-1}}$)}\\
      &\to  \bigoplus_{\varepsilon\in G\backslash E_{ij}}
    \bigoplus_{\eta}
    M(\eta\tau_\varepsilon^{-1}\sigma_\varepsilon v_i)\\
    &=W_i \otimes_{L_i} {}_iM_j\otimes_{L_j}L\\
    &\to W_j \otimes_{L_j} L&\text{(via $f_{ij}\otimes{\bf1}_L$)}\\
    &\to M(t(e)) & \text{(via $\varphi_{v_j,\sigma_e\tau_{\varepsilon(e)}}$)}
  \end{align*}
  Note that this composition is $L$-linear by definition. It does not depend on the choice of $\sigma_e$: any other choice is of the form $\sigma_eg$ with $g\in G_{e_{\varepsilon(e)}}$. Put $\varepsilon=\varepsilon(e)$ and $\eta':=\tau_{\varepsilon}^{-1}g\tau_{\varepsilon}\in G_{\tau_\varepsilon^{-1}e_\varepsilon}\subseteq G_{v_j}$. By the cocycle condition \eqref{eq:semi-linearactiononrep}, replacing $\sigma_e$ by $\sigma_eg$ replaces the first map by its composite with $\varphi_{\tau_\varepsilon^{-1}\sigma_\varepsilon v_i,\eta'^{-1}}$ and the last map by its composite with $\varphi_{v_j,\eta'}$; since $\eta'$ fixes the coset of the summand $\eta={\bf1}$, the former is the restriction to that summand of the semi-linear automorphism ${\bf1}\otimes{\bf1}\otimes\eta'^{-1}$ of $W_i \otimes_{L_i} {}_iM_j\otimes_{L_j}L$, transported along \eqref{eq:WiiMjL2} (on the summands, ${\bf1}\otimes{\bf1}\otimes\sigma$ is given by the maps \eqref{eq:varphi2}), and since $f_{ij}\otimes{\bf1}_L$ intertwines ${\bf1}\otimes{\bf1}\otimes\eta'^{-1}$ with ${\bf1}\otimes\eta'^{-1}=\varphi_{v_j,\eta'^{-1}}$, the two contributions cancel.
  \begin{proposition}\label{prop:hcocycle}\rqleanmark{prop:hcocycle}
    For all $\sigma\in G$ and all $e\in E$ the $L$-linear map $\phi_e$ satisfies \eqref{eq:koecheractiononrep}.
  \end{proposition}

  \begin{proof}
    Let $i,j\in I$, $e\in E_{ij}$ and $\sigma\in G$. Then
    \begin{align*}
      \tau_{\varepsilon(\sigma e)}^{-1}\sigma_{\sigma e}^{-1}\sigma\sigma_e\tau_{\varepsilon(e)}v_j
      &=
      \tau_{\varepsilon(\sigma e)}^{-1}\sigma_{\sigma e}^{-1}\sigma\sigma_et(e_{\varepsilon(e)})\\
      &=
      \tau_{\varepsilon(\sigma e)}^{-1}t(\sigma_{\sigma e}^{-1}\sigma\sigma_ee_{\varepsilon(e)})\\
      &=
      \tau_{\varepsilon(\sigma e)}^{-1}t(\sigma_{\sigma e}^{-1}\sigma e)\\
      &=
      \tau_{\varepsilon(\sigma e)}^{-1}t(e_{\varepsilon(\sigma e)})\\
      &=
      v_j.
    \end{align*}
    Therefore,
    \[
          \eta_{e,\sigma}:=\tau_{\varepsilon(\sigma e)}^{-1}\sigma_{\sigma e}^{-1}\sigma\sigma_e\tau_{\varepsilon(e)}\in G_{v_j}.
    \]
    Now consider the square
    \begin{center}
    \begin{tikzpicture}[>=Stealth]
    \node (A) at (0,0) {$M(s(e))$};
    \node (B) at (6,0) {$M(s(\sigma e))$};
    \node (C) at (0,-2) {$M(\tau_{\varepsilon(e)}^{-1}\sigma_{\varepsilon(e)} v_i)$};
    \node (D) at (6,-2) {$M(\tau_{\varepsilon(\sigma e)}^{-1}\sigma_{\varepsilon(\sigma e)} v_i)$};

    \draw[->] (A) -- (B) node[midway, above] {$\varphi_{s(e),\sigma}$};
    \draw[->] (A) -- (C) node[midway, left] {$\varphi_{\sigma_e\sigma_{\varepsilon(e)}v_i,\tau_{\varepsilon(e)}^{-1}\sigma_e^{-1}}$};
    \draw[->] (B) -- (D) node[midway, right] {$\varphi_{\sigma_{\sigma e}\sigma_{\varepsilon(\sigma e)}v_i,\tau_{\varepsilon(\sigma e)}^{-1}\sigma_{\sigma e}^{-1}}$};
    \draw[->] (C) -- (D) node[midway, below] {$\varphi_{\tau_{\varepsilon( e)}^{-1}\sigma_{\varepsilon( e)}v_i,\eta_{e,\sigma}}$};
    \end{tikzpicture}
    \end{center}
    To verify its commutativity, we compute with \eqref{eq:semi-linearactiononrep}:
    \begin{align*}
      \varphi_{\tau_{\varepsilon( e)}^{-1}\sigma_{\varepsilon( e)}v_i,\eta_{e,\sigma}}
      \circ
      \varphi_{s(e),\tau_{\varepsilon(e)}^{-1}\sigma_e^{-1}}
      &=
      \varphi_{s(e),\eta_{e,\sigma}\tau_{\varepsilon(e)}^{-1}\sigma_e^{-1}}\\
      &=
      \varphi_{s(e),\tau_{\varepsilon(\sigma e)}^{-1}\sigma_{\sigma e}^{-1}\sigma\sigma_e\sigma_e^{-1}}\\
      &=
      \varphi_{s(e),\tau_{\varepsilon(\sigma e)}^{-1}\sigma_{\sigma e}^{-1}\sigma}\\
      &=
      \varphi_{\sigma_{\sigma e}\sigma_{\varepsilon(\sigma e)}v_i,\tau_{\varepsilon(\sigma e)}^{-1}\sigma_{\sigma e}^{-1}}
      \circ
      \varphi_{s(e),\sigma}.
    \end{align*}
    $\bigoplus\limits_\eta\varphi_{\bullet,\eta_{e,\sigma}}$ induces a canonical $\eta_{e,\sigma}$-linear isomorphism
    \[
    \bigoplus_{\eta}
    M(\eta\tau_{\varepsilon(e)}^{-1}\sigma_{\varepsilon(e)} v_i)
    \to
    \bigoplus_{\eta}
    M(\eta_{e,\sigma}\eta\eta_{e,\sigma}^{-1}\tau_{\varepsilon(\sigma e)}^{-1}\sigma_{\varepsilon(\sigma e)} v_i)
    =
    \bigoplus_{\eta}
    M(\eta\tau_{\varepsilon(\sigma e)}^{-1}\sigma_{\varepsilon(\sigma e)} v_i),
    \]
    which together with the canonical inclusions for the $\eta={\bf1}$-summands induces a commutative square. Adding a further layer of identifications, we obtain a commutative square
    \begin{center}
    \begin{tikzpicture}[>=Stealth]
    \node (A) at (0,0) {$M(s(e))$};
    \node (B) at (6,0) {$M(s(\sigma e))$};
    \node (C) at (0,-2) {$W_i\otimes_{L_i}{}_iM_j\otimes_{L_j} L$};
    \node (D) at (6,-2) {$W_i\otimes_{L_i}{}_iM_j\otimes_{L_j} L$};

    \draw[->] (A) -- (B) node[midway, above] {$\varphi_{s(e),\sigma}$};
    \draw[->] (A) -- (C) node[midway, left] {};
    \draw[->] (B) -- (D) node[midway, right] {};
    \draw[->] (C) -- (D) node[midway, below] {${\bf1}\otimes{\bf1}\otimes\eta_{e,\sigma}$};
    \end{tikzpicture}
    \end{center}
    To conclude the proof, consider the diagram
    \begin{center}
    \begin{tikzpicture}[>=Stealth]
    \node (A) at (0,0) {$W_i\otimes_{L_i}{}_iM_j\otimes_{L_j} L$};
    \node (B) at (6,0) {$W_i\otimes_{L_i}{}_iM_j\otimes_{L_j} L$};
    \node (C) at (0,-2) {$M(t(e))$};
    \node (D) at (6,-2) {$M(t(\sigma e))$};

    \draw[->] (A) -- (B) node[midway, above] {${\bf1}\otimes{\bf1}\otimes\eta_{e,\sigma}$};
    \draw[->] (A) -- (C) node[midway, left] {};
    \draw[->] (B) -- (D) node[midway, right] {};
    \draw[->] (C) -- (D) node[midway, below] {$\varphi_{t(e),\sigma}$};
    \end{tikzpicture}
    \end{center}
    which commutes due to the definition of $M(-)$ via \eqref{eq:MvforWi} and the definition of $\varphi_{t(e),\sigma}$ via \eqref{eq:varphi2}, taking the vertical identifications via $\varphi_{v_j,\sigma_e\tau_{\varepsilon(e)}}$ and $\varphi_{v_j,\sigma_{\sigma e}\tau_{\varepsilon(\sigma e)}}$ into account.
  \end{proof}
  This defines the functor $H$ on objects. For $v\in i$, $v=\sigma v_i$, we have the canonical identification
  \[
  M(v)=W_{i,K}\otimes_KL=(W_{i,K}\otimes_KL_i)\otimes_{L_i,\sigma}L= W_i\otimes_{L_i,\sigma}L,
  \]
  which does not depend on the choice of $\sigma$ modulo $G_{v_i}$, and under which $\varphi_{v,\tau}$ becomes $w\otimes a\mapsto w\otimes\tau a\colon W_i\otimes_{L_i,\sigma}L\to W_i\otimes_{L_i,\tau\sigma}L$. A morphism $(\psi_i)_i\colon(W_i,f_{ij})\to(W_i',f_{ij}')$ in $\Rep(S(\Gamma_K))$ induces the $L$-linear maps $\psi_i\otimes{\bf1}_L\colon W_i\otimes_{L_i,\sigma}L\to W_i'\otimes_{L_i,\sigma}L$, i.e., maps $M(v)\to M'(v)$ for all $v\in V$, which commute with the $\varphi_{v,\tau}$ and, by the construction of $\phi_e$ through $f_{ij}\otimes{\bf1}_L$, with the edge maps. This makes $H$ a functor. Different choices of the $K$-structures \eqref{eq:KrationalWi} lead to naturally isomorphic functors, as follows from Theorem \ref{thm:species_equivalence} below.

  \begin{theorem}\label{thm:species_equivalence}\rqleanmark{thm:species_equivalence}
  Let $\Gamma_K$ be a $K$-rational quiver. The category $\Rep(\Gamma_K)$ of finite-dimensional $K$-rational representations of $\Gamma_K$ is equivalent to the category $\Rep(S(\Gamma_K))$ of finite-dimensional representations of its associated $K$-species $S(\Gamma_K)$.
  \end{theorem}

  \begin{proof}
  The constructions are such that they are mutually inverse up to natural isomorphism. For a species representation $(W_i, f_{ij})$, applying $H$ gives a rational quiver representation $M_K$. Applying $F$ to $M_K$ gives back spaces $M(v_i)^{G_{v_i}} = (W_{i,K}\otimes_KL)^{G_{v_i}}=W_{i,K}\otimes_KL_i$, and the counit is
  \[
  FH(W)_i=W_{i,K}\otimes_KL_i\;\xrightarrow{\;\sim\;}\;W_i,\quad w\otimes a\mapsto wa.
  \]
  Since $\phi_e$ was defined through $f_{ij}\otimes{\bf1}_L$ and $f_{ij}$ of $F(M_K)$ is the descent of $\Phi_{ij}$, these isomorphisms carry the structure maps of $FH(W)$ to the $f_{ij}$.

  Conversely, starting with $M_K$, applying $F$ and then $H$ gives the representation with spaces $HF(M)(v)=W_{i,K}\otimes_KL$, where $W_{i,K}$ is the chosen $K$-structure of $W_i=M(v_i)^{G_{v_i}}$. The unit is given, for $v=\sigma v_i\in i$, by the inverse of
  \[
  HF(M)(v)=W_{i,K}\otimes_KL=W_i\otimes_{L_i,\sigma}L\;\xrightarrow{\;\sim\;}\;M(\sigma v_i)=M(v),\quad w\otimes a\mapsto\varphi_{v_i,\sigma}(w)\cdot a,
  \]
  which is an isomorphism by Proposition \ref{prop:generalextensionofscalars}, does not depend on the choice of $\sigma$ modulo $G_{v_i}$ (since $w\in M(v_i)^{G_{v_i}}$), and intertwines \eqref{eq:varphi2} with the $\varphi_{v,\tau}$ of $M_K$ by \eqref{eq:semi-linearactiononrep}. By the definitions of $\Phi_{ij}$ and of $\phi_e$ (and the independence statements above), it also intertwines the edge maps.

  A straightforward verification considering morphisms shows that $F \circ H \cong \mathrm{Id}_{\Rep(S(\Gamma_K))}$ and $H \circ F \cong \mathrm{Id}_{\Rep(\Gamma_K)}$.
  \end{proof}

  \begin{corollary}\label{cor:nilpotentrelations}\rqleanmark{cor:nilpotentrelations}
    The equivalence $F\colon\Rep(\Gamma_K)\to\Rep(S(\Gamma_K))$ restricts to an equivalence $\Rep^{\rm nil}(\Gamma_K)\simeq\Rep^{\rm nil}(S(\Gamma_K))$: a representation $M_K$ is nilpotent if and only if $F(M_K)$ is. For the $\QQ$-rational Gelfand quiver $\Gamma_\QQ$ of Example \ref{ex:gelfand}, $F$ furthermore restricts to an equivalence between the full subcategory of $\Rep^{\rm nil}(\Gamma_\QQ)$ of representations satisfying $a_-b_-=a_+b_+$ and $\Rep^{\rm nil}(S_{\rm Gelfand},R_{\rm Gelfand})$.
  \end{corollary}

  \begin{proof}
    Let $M_K\in\Rep(\Gamma_K)$ and $(W,f)=F(M_K)$. For $i,j\in I$, the identification \eqref{eq:WiiMjL} carries the $(\varepsilon,\eta)$-summand to $M(s(e'))$ for the edge $e':=\eta\tau_\varepsilon^{-1}e_\varepsilon$, which satisfies $s(e')=\eta\tau_\varepsilon^{-1}\sigma_\varepsilon v_i$ and $t(e')=v_j$, and as $\varepsilon$ and the coset of $\eta$ vary, $e'$ runs bijectively through the edges $e'\in E$ with $s(e')\in i$ and $t(e')=v_j$. By \eqref{eq:koecheractiononrep}, $\varphi_{t(e_\varepsilon),\eta\tau_\varepsilon^{-1}}\circ\phi_{e_\varepsilon}\circ\varphi_{v_i,\sigma_\varepsilon}=\phi_{e'}\circ\varphi_{v_i,\eta\tau_\varepsilon^{-1}\sigma_\varepsilon}$, so that by the definition of $\Phi_{ij}$
    \[
    f_{ij}\otimes{\bf1}_L\colon\bigoplus_{e'\colon s(e')\in i,\,t(e')=v_j}M(s(e'))\to M(v_j),\quad(m_{e'})_{e'}\mapsto\sum_{e'}\phi_{e'}(m_{e'}).
    \]
    By induction on $n$, using this description for the last step of a path and transporting the identification for the shorter path along the isomorphisms $\varphi_{v,\sigma}$ (with $\phi_{\sigma\gamma}=\varphi_{t(\gamma),\sigma}\circ\phi_\gamma\circ\varphi_{s(\gamma),\sigma}^{-1}$ for paths $\gamma$ in $\Gamma$, by \eqref{eq:koecheractiononrep}), one obtains for every path $p=(i_0,\dots,i_n)$ in $I$ an identification under which $f_p\otimes{\bf1}_L$ becomes $(m_\gamma)_\gamma\mapsto\sum_\gamma\phi_\gamma(m_\gamma)$, the sum running over the paths $\gamma$ in $\Gamma$ of length $n$ ending at $v_{i_n}$ whose $k$-th vertex lies in $i_k$ for all $k$. Since every path in $\Gamma$ is a $G$-translate of a path ending at one of the representatives $v_j$, it follows that $f_p=0$ for all paths $p$ of length $\geq n$ if and only if $\phi_\gamma=0$ for all paths $\gamma$ in $\Gamma$ of length $\geq n$. Hence $F$ preserves and reflects nilpotence, and as nilpotence is invariant under isomorphism, the equivalence $F$ restricts to an equivalence of the full subcategories of nilpotent objects. The second assertion follows in the same way from Example \ref{ex:gelfandspeciesrelation}, which shows that $F(M)$ satisfies $R_{\rm Gelfand}$ if and only if $M$ satisfies $a_-b_-=a_+b_+$.
  \end{proof}

  \begin{remark}\label{rmk:speciesstack}
    Given the base change formalism for $K$-rational quivers, $K$-species and their associated representations, it is easy to verify that the categorical equivalences for varying $K$ define an equivalence of stacks in the \'etale topology on $K$.
  \end{remark}

\section{Rational Harish-Chandra modules}\label{sec:HC}

\subsection{Rational Harish-Chandra modules}

  \label{def:rationalpairs}We understand rational $(\mathfrak{g}, K)$-modules as in \cite{Januszewski18}. Specifically, a \emph{pair} over a field $k$ of characteristic 0 consists of a $k$-Lie algebra $\mathfrak{a}_k$ and a linear algebraic group $B_k$ over $k$, whose identity component $B^\circ$ is reductive. Furthermore, $\mathfrak{b}_k = \Lie_k(B_k) \subseteq \mathfrak{a}_k$, and there is a $k$-rational action of $B_k$ on $\mathfrak{a}_k$ that extends the natural action on $\mathfrak{b}_k$. An $(\mathfrak{a}_k, B_k)$-module over $k$ is a $k$-vector space that is a colimit of rational $B_k$-representations and is endowed with a $k$-linear $\mathfrak{a}_k$-action compatible with the $B_k$-action.

  \label{def:sotworational}\rqleanmark{def:sotworational}We consider $\SL_2$ as a linear algebraic group over $\QQ$, to which we associate the $\QQ$-rational semisimple pair $(\mathfrak{sl}_{2,\QQ}, \SO(2)_\QQ)$, where $\mathfrak{sl}_{2,\QQ} \subseteq \QQ^{2\times 2}$ denotes the space of trace-0 matrices and
  \[
  \SO(2)_\QQ = \Spec \QQ[x, y]/\langle x^2 + y^2 - 1\rangle.
  \]
  Here, a specialization $(\xi,\eta)\in\CC^2$ of the pair $(x,y)$ corresponds to the matrix
  \[
  \begin{pmatrix} \xi & -\eta \\ \eta & \xi \end{pmatrix} \in \SO(2,\CC).
  \]
  This defines a $\QQ$-rational structure on $\SO(2, \CC) \subseteq \CC^{2\times 2}$.
  
  For brevity, we will write $(\mathfrak{g}_\QQ, K_\QQ) = (\mathfrak{sl}_{2,\QQ}, \SO(2)_\QQ)$ and $(\mathfrak{g}, K) = (\mathfrak{sl}_{2,\CC}, \SO(2, \CC))$ in the following.

  More generally, we consider the base change of this pair to an arbitrary field of characteristic $0$.

  \begin{definition}[Base change of the pair]\label{def:pairE}\rqleanmark{def:pairE}
    For a field $E$ of characteristic $0$, we put
    \[
    (\mathfrak{g}_E, K_E):=(E\otimes_\QQ\mathfrak{g}_\QQ,\, K_\QQ\times_\QQ E),
    \]
    Thus $\mathfrak{g}_E=\mathfrak{sl}_{2,E}$ and $K_E=\SO(2)_E=\Spec E[x, y]/\langle x^2 + y^2 - 1\rangle$.
  \end{definition}

  \begin{definition}[The field $L_E$]\label{def:sqrtminusone}\rqleanmark{def:sqrtminusone}
    Let $E$ be a field of characteristic $0$. We fix a square root $\sqrt{-1}$ of $-1$ in an algebraic closure of $E$ and put
    \[
    L_E:=E[\sqrt{-1}]\quad\text{and}\quad G_E:=\Gal(L_E/E).
    \]
    If $\sqrt{-1}\in E$, then $L_E=E$ and $G_E$ is trivial. If $\sqrt{-1}\not\in E$, then $L_E/E$ is quadratic, and we denote by $c\in G_E$ the non-trivial element, which we call {\em complex conjugation}. In either case, an element $\sigma\in G_E$ is determined by $\sigma(\sqrt{-1})\in\{\pm\sqrt{-1}\}$.
  \end{definition}

\subsection{Rational Modules}

  \label{def:sltwotriple}\rqleanmark{def:sltwotriple}We consider the classical triple
  \[
  H:=\begin{pmatrix} 0 & i \\ -i & 0 \end{pmatrix}, \quad
  X:=\frac{1}{2}\begin{pmatrix} 1 & -i \\ -i & -1 \end{pmatrix}, \quad
  Y:=\frac{1}{2}\begin{pmatrix} 1 & i \\ i & -1 \end{pmatrix}.
  \]
  This is a $\CC$-basis of $\mathfrak{g}$ and the usual relations are satisfied:
  \[
    [H, X] = 2X, \quad [H, Y] = -2Y, \quad [X, Y] = H.
  \]
  Let $E$ be a field of characteristic $0$. We read the entry $i$ as $\sqrt{-1}\in L_E$. Then $H,X,Y\in\mathfrak{g}_{L_E}$. If $\sqrt{-1}\not\in E$, we let the complex conjugation $c \in G_E = \Gal(L_E/E)$ act via its canonical action on the matrix entries. This gives:
  \[
  cH = -H, \quad cX = Y, \quad cY = X.
  \]
  The corresponding $E$-rational form of $\mathfrak{g}_{L_E}$, which consists of the matrices fixed by $c$, then agrees with $\mathfrak{g}_E=\mathfrak{sl}_{2,E}$, because $E$ is the fixed field of $c$ in $L_E$; for $E=\QQ$, this is $\mathfrak{sl}_{2,\QQ}$. If $\sqrt{-1}\in E$, then $H,X,Y\in\mathfrak{g}_E$.
  
  \label{prop:casimirdescends}\rqleanmark{prop:casimirdescends}For the Casimir element $C = H^2 - 2H + 4XY + 1 \in U(\mathfrak{g})$, we obtain
  \[
  cC = (cH)^2 - 2(cH) + 4(cX)(cY) + 1 = H^2 + 2H + 4YX + 1 = C,
  \]
  where the last equality uses $4XY=4YX+4H$, i.e., $[X,Y]=H$.
  Thus $C$ descends to $\QQ$. Explicitly, $(X+Y)^2+(i(X-Y))^2=2XY+2YX=4XY-2H$ gives
  \[
  C=-(iH)^2+(X+Y)^2+(i(X-Y))^2+1,
  \]
  where
  \[
  iH=\begin{pmatrix} 0 & -1 \\ 1 & 0 \end{pmatrix},\quad
  X+Y=\begin{pmatrix} 1 & 0 \\ 0 & -1 \end{pmatrix},\quad
  i(X-Y)=\begin{pmatrix} 0 & 1 \\ 1 & 0 \end{pmatrix}
  \]
  lie in $\mathfrak{g}_\QQ$. Hence $C\in U(\mathfrak{g}_\QQ)$, and $C\in U(\mathfrak{g}_E)$ for every field $E$ of characteristic $0$.
  
  \label{def:hcmodule}\rqleanmark{def:hcmodule}A Harish-Chandra module $M$ for $\SL_2(\RR)$ is a complex $(\mathfrak{g}, K)$-module that is finitely generated and admissible. Admissibility here means that for every character
  \[
  \chi_n\colon K=\SO(2, \CC) \ni \begin{pmatrix} \xi & -\eta \\ \eta & \xi \end{pmatrix} \mapsto (\xi+i\eta)^n \in \CC^\times,
  \]
  for $n \in \ZZ$ (their restrictions to $\SO(2,\RR)$ are precisely the continuous characters of $\SO(2,\RR)$), the isotypic component
  \[
  M_n := M[\chi_n]
  \]
  is finite-dimensional. We know that $M = \bigoplus_{n\in\ZZ} M_n$. This gives the abelian category $\HC(\mathfrak{g}, K)$ of Harish-Chandra modules over $\CC$.
  
  \begin{definition}[Generalized infinitesimal character]\label{def:geninfchar}\rqleanmark{def:geninfchar}
    For $\ell \in \ZZ_{\ge 0}$, let $\lambda := \ell^2$. We say that a Harish-Chandra module $M$ defined over a field $E/\QQ$ has \emph{generalized infinitesimal character} $\lambda$ if there exists an $m\ge 0$ such that $(C-\lambda)^m \cdot M = 0$.
  \end{definition}

  \begin{definition}[The Category $\HC_\lambda(\mathfrak{g}, K)_E$]\label{def:HClambdaE}\rqleanmark{def:HClambdaE}
    Let $E$ be a field of characteristic $0$. A Harish-Chandra module over $E$ is a $(\mathfrak{g}_E, K_E)$-module $M$ (Definition \ref{def:pairE}) which is finitely generated and admissible, the latter meaning that the isotypic components $(L_E\otimes_E M)[\chi_n]$, with $\chi_n(\xi,\eta)=(\xi+\sqrt{-1}\eta)^n$ on $K_{L_E}$, are finite-dimensional. We denote the full subcategory of Harish-Chandra modules over $E$ with generalized infinitesimal character $\lambda$, on which $-1_2 \in K_E(E)$ acts via $(-1)^{\ell-1}$, by $\HC_\lambda(\mathfrak{g}, K)_E$.
  \end{definition}

  \begin{remark}\label{rmk:HCabelian}\rqleanmark{rmk:HCabelian}
    The category $\HC_\lambda(\mathfrak{g}, K)_E$ is abelian. Its objects are $E$-vector spaces with additional structure and its morphisms are $E$-linear maps. The kernel and the cokernel of a morphism $f\colon M\to N$, taken in $E$-vector spaces, carry canonical structures of $(\mathfrak{g}_E,K_E)$-modules. They lie in $\HC_\lambda(\mathfrak{g}, K)_E$: admissibility, the generalized infinitesimal character and the action of $-1_2$ pass to subspaces and quotients, and the kernel is finitely generated since $U(\mathfrak{g}_E)$ is Noetherian. They satisfy the universal properties of the kernel and the cokernel. Moreover, as for $E$-vector spaces, $f$ is a monomorphism if and only if its kernel is $0$, and an epimorphism if and only if its cokernel is $0$.
  \end{remark}

  \begin{remark}\label{rmk:HClambdaC}\rqleanmark{rmk:HClambdaC}
    We remark that for $E=\CC$ we have $L_\CC=\CC$, and Definition \ref{def:HClambdaE} gives the full subcategory $\HC_\lambda(\mathfrak{g}, K)_\CC$ of the category $\HC(\mathfrak{g}, K)$ of Harish-Chandra modules over $\CC$ defined above. It consists of the modules with generalized infinitesimal character $\lambda$ on which $-1_2$ acts via $(-1)^{\ell-1}$. Here we use that the action of $K$ on a module in $\HC(\mathfrak{g}, K)$ is determined by its isotypic components $M_n$, on which $K$ acts via $\chi_n$.
  \end{remark}
  
  We remark that the torus $K$ splits over $\QQ[\sqrt{-1}]$, which implies that all irreducible Harish-Chandra modules in $\HC_\lambda(\mathfrak{g}, K)_\CC$ are already defined over $\QQ[\sqrt{-1}]$. Whether such a module descends to $\QQ$ is decided by comparing it with its complex conjugate (cf.\ Proposition 6.12 in \cite{Januszewski18}); these facts are not used in the sequel.
  
\subsection{Rational normalizations}

  We now consider Harish-Chandra modules over a base field $E$ of characteristic $0$. Let $M_E \in \HC_\lambda(\mathfrak{g}, K)_E$ be arbitrary. We denote its base change by
  \[
  M_{L_E} := L_E \otimes_E M_E.
  \]
  If $\sqrt{-1}\in E$, then $L_E=E$ and $M_{L_E}=M_E$. In the following, we identify an element $g \in U(\mathfrak{g}_{L_E})$ with the corresponding map it induces on the subspaces of $M_{L_E}$ and, for simplicity of notation, for $n\in\ZZ$ we write
  \[
  M_n := M_{L_E}[\chi_n].
  \]
  We begin our analysis with the following proposition.
  \begin{proposition}\label{prop:XY_C_relation}\rqleanmark{prop:XY_C_relation}
    For $m\ge 0$ and $k\in\ZZ$, the following relations hold:
    \begin{align}
      \left(4^m X^m Y^m \colon M_{k+1} \to M_{k+1}\right) &= \prod_{j=0}^{m-1} \left(C|_{M_{k+1}} - (k-2j)^2 \cdot \mathbf{1}_{M_{k+1}}\right), \label{eq:XYm_via_C} \\
      \left(4^m Y^m X^m \colon M_{-(k+1)} \to M_{-(k+1)}\right) &= \prod_{j=0}^{m-1} \left(C|_{M_{-(k+1)}} - (k-2j)^2 \cdot \mathbf{1}_{M_{-(k+1)}}\right). \label{eq:YXm_via_C}
    \end{align}
  \end{proposition}
  
  \begin{proof}
    For any $p \in \ZZ$, we have the two relations
    \begin{align}
      4XY &= C - H^2 + 2H - 1 = C-p^2 \quad\text{on } M_{p+1}, \label{eq:4XY_rel} \\
      4YX &= C - H^2 - 2H - 1 = C-p^2 \quad\text{on } M_{p-1}. \label{eq:4YX_rel}
    \end{align}
    For $m\ge 0$, consider the identity $4^{m+1}X^{m+1}Y^{m+1} = 4^m X^m(4XY)Y^m$. Restricting this to a map $M_{k+1} \to M_{k+1}$, the term $4XY$ acts as an endomorphism on the space $M_{k+1-2m}$. According to \eqref{eq:4XY_rel} with $p=k-2m$, this action is given by $C-(k-2m)^2$. Since $C$ commutes with $X$ and $Y$ in $U(\mathfrak{g})$, the first formula \eqref{eq:XYm_via_C} follows by induction.
    
    The second formula \eqref{eq:YXm_via_C} follows by a similar inductive argument using \eqref{eq:4YX_rel}, or by observing that complex conjugation (Galois symmetry) transforms \eqref{eq:XYm_via_C} into \eqref{eq:YXm_via_C}.
  \end{proof}
  
  \begin{proposition}\label{prop:xyscalarplusnilpotent}\rqleanmark{prop:xyscalarplusnilpotent}
    For $m\ge 0$ and $k\in\ZZ$, we have
    \begin{align}
      \left(4^m X^m Y^m \colon M_{k+1} \to M_{k+1}\right) &= \prod_{j=0}^{m-1} \left(\ell^2 - (k-2j)^2\right) \cdot \mathbf{1}_{M_{k+1}} + n, \label{eq:XYm_nilp} \\
      \left(4^m Y^m X^m \colon M_{-(k+1)} \to M_{-(k+1)}\right) &= \prod_{j=0}^{m-1} \left(\ell^2 - (k-2j)^2\right) \cdot \mathbf{1}_{M_{-(k+1)}} + n', \label{eq:YXm_nilp}
    \end{align}
    where $n$ and $n'$ are nilpotent operators.
  \end{proposition}
  
  \begin{proof}
    This follows from Proposition \ref{prop:XY_C_relation} and the fact that $C - \lambda \cdot \mathbf{1} = C - \ell^2 \cdot \mathbf{1}$ acts nilpotently on $M$.
  \end{proof}
  
  From now on until the end of this section we assume $\ell\geq 1$, so that $\gamma_\star=(\ell-1)!$ below is defined; the case $\ell=0$ is the subject of Theorem \ref{thm:cyclicEEequivalence}.

  \begin{corollary}\label{cor:X_Y_inverses}\rqleanmark{cor:X_Y_inverses}
    The two maps
    \[
    X^{\ell-1}\colon M_{-(\ell-1)} \to M_{\ell-1} \quad \text{and} \quad Y^{\ell-1}\colon M_{\ell-1} \to M_{-(\ell-1)}
    \]
    are invertible, and satisfy the following relations:
    \begin{equation} \label{eq:XY_inverses_rel}
      X^{\ell-1} \circ Y^{\ell-1} = \gamma_\star^2 \cdot \mathbf{1}_{M_{\ell-1}} + n, \quad Y^{\ell-1} \circ X^{\ell-1} = \gamma_\star^2 \cdot \mathbf{1}_{M_{-(\ell-1)}} + \bar{n},
    \end{equation}
    where
    \[
    \gamma_\star = (\ell-1)!,
    \]
    and $n$ and $\bar{n}$ are the nilpotent restrictions to $M_{\ell-1}$ and $M_{-(\ell-1)}$ of the same polynomial
    \begin{equation}\label{eq:nilpotentpartpolynomial}
      p(C):=4^{-(\ell-1)}\prod_{j=0}^{\ell-2}\left(C-(\ell-2-2j)^2\right)-\gamma_\star^2
    \end{equation}
    in $C$ with rational coefficients. If $\sqrt{-1}\not\in E$, then $\bar{n}$ is the complex conjugate of $n$, i.e., $c \circ n = \bar{n} \circ c$.
  \end{corollary}
  \begin{proof}
    It suffices to evaluate the constant prefactor on the right-hand side of \eqref{eq:XYm_nilp}. For the operator $4^{\ell-1}X^{\ell-1}Y^{\ell-1}$ acting on $M_{\ell-1}$, we set $m=\ell-1$ and $k=\ell-2$ in \eqref{eq:XYm_nilp}. The constant is
    \begin{align*}
      \prod_{j=0}^{\ell-2} \left(\ell^2 - (\ell-2-2j)^2\right) &= \prod_{j=0}^{\ell-2} \left(\ell - (\ell-2-2j)\right) \left(\ell + (\ell-2-2j)\right) \\
      &= \prod_{j=0}^{\ell-2} (2+2j) (2\ell-2-2j)\\
      &= \prod_{k=1}^{\ell-1} 4 k (\ell-k) \\
      &= 4^{\ell-1} ((\ell-1)!)^2.
    \end{align*}
    Since this is the constant for $4^{\ell-1}X^{\ell-1}Y^{\ell-1}$, the operator $X^{\ell-1}Y^{\ell-1}$ on $M_{\ell-1}$ has the scalar part $((\ell-1)!)^2 = \gamma_\star^2$. By \eqref{eq:YXm_nilp} with the same $m$ and $k$, the same holds for $Y^{\ell-1}X^{\ell-1}$ on $M_{-(\ell-1)}$, which proves \eqref{eq:XY_inverses_rel}. More precisely, \eqref{eq:XYm_via_C} and \eqref{eq:YXm_via_C} with $m=\ell-1$ and $k=\ell-2$ express $X^{\ell-1}Y^{\ell-1}$ on $M_{\ell-1}$ and $Y^{\ell-1}X^{\ell-1}$ on $M_{-(\ell-1)}$ by the same polynomial in $C$. Therefore $n$ and $\bar{n}$ are the restrictions of $p(C)$ in \eqref{eq:nilpotentpartpolynomial}. If $\sqrt{-1}\not\in E$, then $c$ commutes with $C$, since $cC=C$, fixes the rational coefficients of $p$, and maps $M_{\ell-1}$ onto $M_{-(\ell-1)}$. Hence $c\circ n=\bar{n}\circ c$.

    As $\gamma_\star^2\neq0$ and $n$, $\bar{n}$ are nilpotent, both composites in \eqref{eq:XY_inverses_rel} are invertible. Invertibility of $X^{\ell-1}\circ Y^{\ell-1}$ shows that $X^{\ell-1}\colon M_{-(\ell-1)}\to M_{\ell-1}$ is surjective and $Y^{\ell-1}\colon M_{\ell-1}\to M_{-(\ell-1)}$ is injective; invertibility of $Y^{\ell-1}\circ X^{\ell-1}$ shows that $X^{\ell-1}$ is injective and $Y^{\ell-1}$ is surjective. Hence both maps are invertible.
  \end{proof}
  
  \begin{corollary}\label{cor:T_operators}\rqleanmark{cor:T_operators}
    The two endomorphisms
    \begin{align}
      T_+ &:= \left(4^{\ell-1}\gamma_\star^{2}\right)^{-1} \cdot \prod_{j=0}^{\ell-2} \left(C|_{M_{\ell+1}} - (\ell - 2 - 2j)^2 \cdot \mathbf{1}_{M_{\ell+1}}\right), \label{eq:T_plus_def} \\
      T_- &:= \left(4^{\ell-1}\gamma_\star^{2}\right)^{-1} \cdot \prod_{j=0}^{\ell-2} \left(C|_{M_{-(\ell+1)}} - (\ell - 2 - 2j)^2 \cdot \mathbf{1}_{M_{-(\ell+1)}}\right), \label{eq:T_minus_def}
    \end{align}
    are automorphisms of $M_{\pm(\ell+1)}$ respectively, with the property
    \[
    T_{\pm} = \mathbf{1}_{M_{\pm(\ell+1)}} + n_{\pm},
    \]
    where $n_\pm$ is nilpotent.
  \end{corollary}

  \begin{proof}
    It suffices to note that the operator product defining $T_+$ is constructed from the right-hand side of identity \eqref{eq:XYm_via_C} for $m=\ell-1$ and $k=\ell-2$, by replacing the space on which the operators act, $M_{\ell-1}$, with $M_{\ell+1}$. The claim then follows from the same calculation as in the proof of Corollary \ref{cor:X_Y_inverses}, which evaluates the scalar part of that product as $4^{\ell-1}\gamma_\star^2$ --- precisely the factor normalized away in \eqref{eq:T_plus_def} and \eqref{eq:T_minus_def}.
  \end{proof}
  
  \begin{definition}\label{def:xstarystar}\rqleanmark{def:xstarystar}
    We put
    \begin{equation}
      X_\star:=\gamma_\star^{-1}\cdot X^{\ell-1}|_{M_{-(\ell-1)}}\quad\text{and}\quad Y_\star:=\gamma_\star^{-1}\cdot Y^{\ell-1}|_{M_{\ell-1}}.
    \end{equation}
  \end{definition}
  
  According to Corollary \ref{cor:X_Y_inverses} we have the relations
  \begin{equation}
    X_\star Y_\star = {\bf1}_{M_{\ell-1}}+n_\star,\quad
    Y_\star X_\star = {\bf1}_{M_{-(\ell-1)}}+\overline{n}_\star,
  \end{equation}
  where $n_\star=\gamma_\star^{-2}\cdot n$ and $\overline{n}_\star=\gamma_\star^{-2}\cdot\overline{n}$, the restrictions of $\gamma_\star^{-2}p(C)$ to $M_{\ell-1}$ and $M_{-(\ell-1)}$, are nilpotent and, if $\sqrt{-1}\not\in E$, conjugate to each other.

  \section{Unipotent stabilization}\label{sec:unipotentstabilization}

  The construction of a $\QQ$-rational structure on the quiver representations associated with Harish-Chandra modules requires us to construct semi-linear maps satisfying the cocycle condition \eqref{eq:semi-linearactiononrep}. However, with the given representation theoretic data, we can a priori guarantee the cocycle condition only up to a nilpotent error term. In order to resolve this obstacle while preserving all additional coherence conditions, we apply iterative procedures, akin to Newton's method, which are guaranteed to terminate after finitely many steps because the error terms are nilpotent. We present two closely related fundamental constructions for the situation at hand: unipotent stabilization and unipotent square roots.

  \subsection{Unipotent stabilization}

    \begin{proposition}\label{prop:unipotentstabilization}\rqleanmark{prop:unipotentstabilization}
      Let $L/K$ be a quadratic field extension with $K$ of characteristic $\neq 2$. Let $\tau\in\Gal(L/K)$ denote the non-trivial element. Let $W_K$ be a $K$-vector space. We equip $W_L = L \otimes_K W_K$ with the canonical semi-linear $\tau$-action. Let $W_\pm \subseteq W_L$ be two $L$-subspaces. Let $\varphi_\pm \colon W_\pm \to W_\mp^\tau$ be $L$-linear isomorphisms with the following property:
      \begin{equation}
        \varphi_-^\tau\circ\varphi_+ = {\bf1}_{W_+} + n
        \label{eq:nilpotency}
      \end{equation}
      with $n$ nilpotent of exponent $e \ge 1$, i.e., $n^{e-1} \neq 0$ and $n^e = 0$ (this presupposes $W_+\neq0$; for $W_+=0$ all assertions below are trivial).
      Then
      \begin{equation}
        \widetilde{\varphi}_\pm := \frac{1}{2}\left(\varphi_\pm + \left(\varphi_\mp^\tau\right)^{-1}\right)
        \label{eq:unipotentstabilizationiteration}
      \end{equation}
      is an $L$-linear isomorphism $\widetilde{\varphi}_+ \colon W_+ \to W_-^\tau$ with the properties
      \begin{equation}
        \widetilde{\varphi}_-^\tau \circ \widetilde{\varphi}_+ = {\bf1}_{W_+} + \widetilde{n},
        \label{eq:nilpotencytilde}
      \end{equation}
      with $\widetilde{n}$ nilpotent of exponent $\widetilde{e} \le \lceil e/2 \rceil$.

      \noindent\textbf{Addendum:} If $W_\pm = W_\mp^\tau$, then
      \begin{equation}
      \varphi_\pm\circ\varphi_\mp^\tau = \varphi_\mp^\tau\circ\varphi_{\pm}
        \label{eq:unipotentcommutators0}
      \end{equation}
      implies the relation
      \begin{equation}
        \widetilde{\varphi}_\pm \circ \widetilde{\varphi}_\mp^\tau = \widetilde{\varphi}_\mp^\tau \circ \widetilde{\varphi}_{\pm}.
        \label{eq:unipotentcommutators1}
      \end{equation}
      If, in addition to \eqref{eq:unipotentcommutators0},
      \begin{equation}
        \varphi_{\pm} = {\bf1}_{W_\pm} + m_\pm
        \label{eq:unipotentnilpotency0}
      \end{equation}
      with $m_\pm$ nilpotent, then
      \begin{equation}
        \widetilde{\varphi}_{\pm} = {\bf1}_{W_\pm} + \widetilde{m}_\pm
        \label{eq:unipotentnilpotency1}
      \end{equation}
      with $\widetilde{m}_\pm$ nilpotent.
    \end{proposition}

    \begin{proof}
      Since $\tau^2=1$, we have $\widetilde{\varphi}_-^\tau=\frac{1}{2}\left(\varphi_-^\tau+\varphi_+^{-1}\right)$. By definition, using \eqref{eq:nilpotency}, we get
      \begin{align*}
        \widetilde{\varphi}_-^\tau\circ\widetilde{\varphi}_+ &=
        \frac{1}{4}\left(\varphi_-^\tau+\left(\varphi_+\right)^{-1}\right)
        \circ \left(\varphi_++\left(\varphi_-^\tau\right)^{-1}\right)\\
        &=
        \frac{1}{4}\left(2\cdot{\bf1}_{W_+}+\varphi_-^\tau\circ\varphi_++\left(\varphi_+\right)^{-1}\circ\left(\varphi_-^\tau\right)^{-1}\right)\\
        &=
        \frac{1}{4}\left(2\cdot{\bf1}_{W_+}+({\bf1}_{W_+}+n)+\left({\bf1}_{W_+}+n\right)^{-1}\right)\\
        &=
        \frac{1}{4}\left(4\cdot{\bf1}_{W_+}+\sum_{k=2}^\infty(-n)^k\right)\\
        &=
        {\bf1}_{W_+}+\widetilde{n}
      \end{align*}
      with $\widetilde{n}=\frac{1}{4}n^2\left({\bf1}_{W_+}+n\right)^{-1}$, hence $\widetilde{n}^{\lceil e/2 \rceil}=0$. Moreover,
      \[
        \widetilde{\varphi}_+=\frac{1}{2}\,\varphi_+\circ\left({\bf1}_{W_+}+\left(\varphi_-^\tau\circ\varphi_+\right)^{-1}\right)
        =\frac{1}{2}\,\varphi_+\circ\Big(2\cdot{\bf1}_{W_+}+\sum_{k=1}^\infty(-n)^k\Big),
      \]
      where the second factor is $2\cdot{\bf1}_{W_+}$ plus a nilpotent operator, hence invertible since the characteristic is $\neq2$. Thus $\widetilde{\varphi}_+$ is an isomorphism, and by symmetry so is $\widetilde{\varphi}_-$.

      If we additionally assume $W_\pm = W_\mp^\tau$ and \eqref{eq:unipotentcommutators0}, we obtain
      \begin{align*}
        4\widetilde{\varphi}_\pm\circ\widetilde{\varphi}_\mp^\tau
        &=\left(\varphi_\pm+\left(\varphi_\mp^\tau\right)^{-1}\right)\circ \left(\varphi_\mp+\left(\varphi_\pm^\tau\right)^{-1}\right)^\tau\\
        &=2\cdot{\bf1}_{W_\pm}+\varphi_\pm\circ\varphi_\mp^\tau+\left(\varphi_\pm\circ\varphi_\mp^\tau\right)^{-1}\\
        &=2\cdot{\bf1}_{W_\pm}+\varphi_\mp^\tau\circ\varphi_\pm+\left(\varphi_\mp^\tau\circ\varphi_\pm\right)^{-1}\\
        &=4\widetilde{\varphi}_\mp^\tau\circ\widetilde{\varphi}_\pm.
      \end{align*}
      This shows \eqref{eq:unipotentcommutators1}.

      Assume additionally \eqref{eq:unipotentnilpotency0}. Then \eqref{eq:unipotentcommutators1} shows via
      \begin{align*}
        {\bf1}_{W_\pm}+m_{\pm}+m_{\mp}^\tau+m_{\pm}\circ m_{\mp}^\tau
        &=
        ({\bf1}_{W_\pm}+m_{\pm})\circ
        ({\bf1}_{W_\pm}+m_{\mp}^\tau)\\
        &=
        \varphi_\pm\circ\varphi_\mp^\tau\\
        &=
        \varphi_\mp^\tau\circ\varphi_\pm\\
        &=
        {\bf1}_{W_\pm}+m_{\mp}^\tau+m_{\pm}+m_{\mp}^\tau\circ m_{\pm},
      \end{align*}
      that $m_\pm$ and $m_{\mp}^\tau$ commute. Then \eqref{eq:unipotentstabilizationiteration} gives
      \[
        \widetilde{\varphi}_\pm=\frac{1}{2}\left({\bf1}_{W_\pm}+m_\pm+\left({\bf1}_{W_\pm}+m_\mp^\tau\right)^{-1}\right)
        ={\bf1}_{W_\pm}+\frac{1}{2}\Big(m_\pm+\sum_{k=1}^\infty\left(-m_\mp^\tau\right)^k\Big),
      \]
      and $\widetilde{m}_\pm:=\frac{1}{2}\big(m_\pm+\sum_{k\geq1}(-m_\mp^\tau)^k\big)$ is nilpotent as a sum of commuting nilpotent operators. This shows \eqref{eq:unipotentnilpotency1}.
    \end{proof}

    \begin{corollary}\label{cor:unipotentstabilization}\rqleanmark{cor:unipotentstabilization}
      Under the same assumptions as in Proposition \ref{prop:unipotentstabilization}, for the recursively defined sequences with $\varphi^{(0)}_\pm := \varphi_\pm$ and
      \begin{equation}
        \varphi^{(k+1)}_\pm := \frac{1}{2}\left(\varphi_\pm^{(k)}+\left(\varphi_\mp^{(k),\tau}\right)^{-1}\right),
      \end{equation}
      for $k \ge 0$, there exists a $k_0 \ge 0$ such that for all $k \ge k_0$:
      \[
      \varphi_\pm^{(k)}=\varphi_{\pm}^{(k_0)}\]
      and
      \begin{equation}
        \varphi_\mp^{(k),\tau}\circ\varphi_\pm^{(k)}={\bf1}_{W_\pm}.
        \label{eq:conjugateinverse}
      \end{equation}
      \noindent\textbf{Addendum:} If $W_\pm = W_\mp^\tau$ and \eqref{eq:unipotentcommutators0} and \eqref{eq:unipotentnilpotency0} hold, then for all $k \ge 0$:
      
      \begin{equation}
      \varphi_\pm^{(k)}\circ\left(\varphi_\mp^{(k)}\right)^\tau=\left(\varphi_\mp^{(k)}\right)^\tau\circ\varphi_{\pm}^{(k)},
        \label{eq:unipotentcommutatorsk}
      \end{equation}
      and
      \begin{equation}
        \varphi_{\pm}^{(k)}={\bf1}_{W_\pm}+m_\pm^{(k)}
        \label{eq:unipotentnilpotencyk}
      \end{equation}
      with $m_\pm^{(k)}$ nilpotent.
    \end{corollary}

    \begin{proof}
      Applying Proposition \ref{prop:unipotentstabilization} $k$ times, the error term of $\varphi_\mp^{(k),\tau}\circ\varphi_\pm^{(k)}$ has vanishing $\lceil e/2^k\rceil$-th power. Hence for every $k_0$ with $2^{k_0}\geq e$, the relations
      \[
      \varphi_\mp^{(k_0),\tau}\circ\varphi_\pm^{(k_0)}={\bf1}_{W_\pm}
      \]
      hold. If this is the case, it follows inductively from the recursion rule that $\varphi_\pm^{(k)} = \varphi_{\pm}^{(k_0)}$ for all $k \ge k_0$. The relations \eqref{eq:unipotentcommutatorsk} and \eqref{eq:unipotentnilpotencyk} follow inductively from \eqref{eq:unipotentcommutators1} and \eqref{eq:unipotentnilpotency1}.
    \end{proof}

    \begin{definition}[Unipotent Stabilization]\label{def:unipotentstabilization}\rqleanmark{def:unipotentstabilization}
      Let $L/K$ be a quadratic field extension with $K$ of characteristic $\neq 2$, and let $\tau \in \Gal(L/K)$ be the non-trivial automorphism. Let $W_K$ be a $K$-vector space and $W_L = L \otimes_K W_K$. Let $W_\pm \subseteq W_L$ be two $L$-subspaces. Let $\varphi_\pm \colon W_\pm \to W_\mp^\tau$ be $L$-linear isomorphisms satisfying \eqref{eq:nilpotency}.

      With the notation from Corollary \ref{cor:unipotentstabilization}, we define for $k \ge k_0$
      \[
        \varphi^{(\infty)}_\pm:=\varphi^{(k)}_\pm
      \]
      and call $\varphi^{(\infty)}_\pm$ the {\em unipotent stabilization} of the pair $(\varphi_+, \varphi_-)$.
    \end{definition}

    \begin{remark}\label{rmk:unipotentstabilization}\rqleanmark{rmk:unipotentstabilization}
      The unipotent stabilization for $(\varphi_+, \varphi_-)$ coincides with that for the pair $(\varphi_-, \varphi_+)$, up to transposition: the unipotent stabilization of $(\varphi_-,\varphi_+)$ is $(\varphi_-^{(\infty)},\varphi_+^{(\infty)})$, and the same holds for every iterate $(\varphi_\pm^{(k)})$.
    \end{remark}

    \begin{remark}\label{rmk:stabilizationaddenda}\rqleanmark{rmk:stabilizationaddenda}
      We record two further properties of the recursion of Corollary \ref{cor:unipotentstabilization}.
      \begin{itemize}
      \item[(i)] \emph{Naturality.} Let $(W_K',W_\pm',\varphi_\pm')$ be a second datum as in Proposition \ref{prop:unipotentstabilization}, and let $f\colon W_L\to W_L'$ be an $L$-linear map commuting with the $\tau$-actions (e.g., the base change of a $K$-linear map $W_K\to W_K'$), with $f(W_\pm)\subseteq W_\pm'$ and $f\circ\varphi_\pm=\varphi_\pm'\circ f$. Then $f\circ\varphi_\pm^{(k)}=\varphi_\pm^{\prime(k)}\circ f$ for all $k\geq0$, and hence $f\circ\varphi_\pm^{(\infty)}=\varphi_\pm^{\prime(\infty)}\circ f$. Indeed, $f$ commutes with $\varphi_\mp^\tau$, and therefore also with its inverse: $(\varphi_\mp^{\prime\tau})^{-1}\circ f=(\varphi_\mp^{\prime\tau})^{-1}\circ f\circ\varphi_\mp^\tau\circ(\varphi_\mp^\tau)^{-1}=f\circ(\varphi_\mp^\tau)^{-1}$; now induct on $k$. The map $f$ need not be invertible.
      \item[(ii)] \emph{The Addendum relative to an identification.} Let $e\colon W_-^\tau\to W_+$ be an $L$-linear isomorphism, and put
      \[
        P_k:=e\circ\varphi_+^{(k)},\qquad Q_k:=\varphi_-^{(k),\tau}\circ e^{-1}\qquad\in\End_L(W_+).
      \]
      Then $Q_k\circ P_k=\varphi_-^{(k),\tau}\circ\varphi_+^{(k)}$, and, since $\tau^2=1$, the recursion reads $P_{k+1}=\frac{1}{2}(P_k+Q_k^{-1})$ and $Q_{k+1}=\frac{1}{2}(Q_k+P_k^{-1})$. Consequently, if $P_0$ and $Q_0$ commute and are of the form ${\bf1}_{W_+}+(\text{nilpotent})$, then the same holds for $P_k$ and $Q_k$ for all $k\geq0$: $P_{k+1}$ and $Q_{k+1}$ lie in the commutative algebra generated by $P_k^{\pm1}$ and $Q_k^{\pm1}$, and $\frac{1}{2}\big(({\bf1}+m)+({\bf1}+m')^{-1}\big)-{\bf1}=\frac{1}{2}\big(m+\sum_{j\geq1}(-m')^j\big)$ is nilpotent for commuting nilpotent $m$, $m'$. If moreover $P_0=Q_0$, then $P_k=Q_k$ for all $k\geq0$. If $W_\pm=W_\mp^\tau$, then $e={\bf1}_{W_+}$ recovers the Addendum of Corollary \ref{cor:unipotentstabilization} (for the sign $+$).
      \end{itemize}
    \end{remark}

  \subsection{Unipotent square roots}
    
    \begin{proposition}\label{prop:unipotentsquareroots}\rqleanmark{prop:unipotentsquareroots}
      Let $L$ be a field of characteristic $\neq 2$ and let $W$ be an $L$-vector space. Let $\phi\colon W \to W$ be an automorphism with the property
      \begin{equation}
        \phi={\bf1}_{W}+n
        \label{eq:phinilpotency}
      \end{equation}
      with $n$ nilpotent.
      Then there exists an automorphism $\phi^{1/2}\colon W\to W$ with the following properties:
      \begin{itemize}
      \item[(a)] There exists a nilpotent operator $m\colon W \to W$ such that
        \begin{equation}
        \phi^{1/2}={\bf1}_{W}+m.
        \end{equation}
      \item[(b)] We have the following identity:
        \begin{equation}
        \left(\phi^{1/2}\right)^2=\phi.
        \end{equation}
      \item[(c)] For every $L$-linear automorphism $\psi\colon W\to W$,
        \begin{equation}
          \phi\circ\psi=\psi\circ\phi\quad\Rightarrow\quad \phi^{1/2}\circ\psi=\psi\circ\phi^{1/2}.
        \end{equation}
      \item[(d)] For every field automorphism $\sigma$ of $L$, every $L$-vector space $W'$ and every $\sigma$-semi-linear bijection $\psi\colon W \to W'$ (e.g., $W=W'=W_L = L \otimes_K W_K$ for a Galois extension $L/K$ and $\psi$ the action of $\sigma\in\Gal(L/K)$), the automorphism $\psi\circ\phi\circ\psi^{-1}$ of $W'$ again has property \eqref{eq:phinilpotency}, and
        \begin{equation}
          \left(\psi\circ\phi\circ\psi^{-1}\right)^{1/2}=
          \psi\circ\phi^{1/2}\circ\psi^{-1}.
          \label{eq:semilinearconjugation}
        \end{equation}
      \item[(e)] For every $L$-vector space $W'$, every automorphism $\phi'\colon W'\to W'$ with property \eqref{eq:phinilpotency} and every $L$-linear map $\psi\colon W\to W'$ --- not necessarily invertible ---,
        \begin{equation}
          \psi\circ\phi=\phi'\circ\psi\quad\Longrightarrow\quad \psi\circ\phi^{1/2}=\left(\phi'\right)^{1/2}\circ\psi.
          \label{eq:squarerootnaturality}
        \end{equation}
      \end{itemize}
      Furthermore, $\phi^{1/2}$ is uniquely determined by properties (a) and (b).
    \end{proposition}

    \begin{proof}
      The proof of the existence of $\phi^{1/2}$ satisfying (a) and (b) proceeds analogously to the proof of Proposition \ref{prop:unipotentstabilization}. Specifically, let $n$ from \eqref{eq:phinilpotency} be nilpotent with exponent $e_0 \ge 1$, i.e., $n^{e_0-1} \neq 0$ and $n^{e_0}=0$.

      We consider the recursively defined sequence with initial value $\phi_0:=\phi$ and the recurrence relation
      \begin{equation}
        \phi_{k+1}:=\frac{1}{2}\left(\phi_k+\phi_{k}^{-1}\circ\phi\right).
      \end{equation}
      These are automorphisms of $W$ with the properties
      \begin{equation}
        \phi_k={\bf1}_{W}+n_k,
      \end{equation}
      with $n_k$ nilpotent,
      \begin{equation}
        \phi_k\circ\phi=\phi\circ\phi_k,
      \end{equation}
      and
      \begin{equation}
        \phi_k^2-\phi = \widetilde{n}_k,
      \end{equation}
      with $\widetilde{n}_{k+1}$ nilpotent with exponent $e_{k+1} \le \lceil e_k/2 \rceil$.

      In particular, for $k \gg 0$, we have
      \[
      \phi_k^2=\phi,
      \]
      and thus, by the recurrence relation, $\phi_{k+1} = \phi_k$, i.e., the sequence converges in the discrete topology. Thus $\phi^{1/2} := \phi_k$ is an automorphism that satisfies (a) and (b).

      For this specific choice of $\phi^{1/2}$, statements (c) and (d) follow from the observation that the recurrence relation commutes with conjugation by the bijection $\psi$: we have $\psi\circ\phi_k^{-1}\circ\psi^{-1}=(\psi\circ\phi_k\circ\psi^{-1})^{-1}$, and $\psi\circ(\frac{1}{2}\alpha)\circ\psi^{-1}=\frac{1}{2}\,\psi\circ\alpha\circ\psi^{-1}$ for every endomorphism $\alpha$ of $W$, since $\sigma(\frac{1}{2})=\frac{1}{2}$ (in (c), $\sigma=\id$). Statement (e) follows by induction along the same recurrence: $\psi\circ\phi_0=\phi_0'\circ\psi$ by hypothesis, and $\psi\circ\phi_k=\phi_k'\circ\psi$ gives $\psi\circ\phi_k^{-1}=(\phi_k')^{-1}\circ\psi$ and hence $\psi\circ\phi_{k+1}=\phi_{k+1}'\circ\psi$; both sequences are eventually constant. The same statement holds for the normalization of Definition \ref{def:squareroots}, provided both square roots are taken with respect to the same $\gamma$.

      It remains to show that any other automorphism $\widetilde{\phi}\colon W \to W$ that satisfies (a) and (b) must coincide with $\phi^{1/2}$. To this end, we first observe that $\widetilde{\phi}$ commutes with $\phi$ because
      \[
      \widetilde{\phi}\circ\phi=\widetilde{\phi}^3=\phi\circ\widetilde{\phi}.
      \]
      Therefore, $\phi^{1/2}$ and $\widetilde{\phi}$ commute, according to (c).

      Let $\widetilde{\phi}={\bf1}_W+\widetilde{m}$ with $\widetilde{m}$ nilpotent. Since $\phi^{1/2}$ and $\widetilde{\phi}$ commute, so do $m$ and $\widetilde{m}$, and
      \[
      \left(\widetilde{\phi}-\phi^{1/2}\right)\circ\left(\widetilde{\phi}+\phi^{1/2}\right)=\widetilde{\phi}^2-\left(\phi^{1/2}\right)^2=\phi-\phi=0.
      \]
      Here $\widetilde{\phi}+\phi^{1/2}=2\cdot{\bf1}_W+(\widetilde{m}+m)$ is invertible, because $\widetilde{m}+m$ is nilpotent as a sum of commuting nilpotent operators and the characteristic of $L$ is $\neq2$. Hence $\widetilde{\phi}=\phi^{1/2}$, which was to be shown.
    \end{proof}
    
    \begin{remark}\label{rmk:sqrtinitialvalue}\rqleanmark{rmk:sqrtinitialvalue}
      The starting value $\phi_0 = {\bf1}_W$ leads after one iteration to the same automorphism $\phi_1$ as the starting value $\phi_0 = \phi$.
    \end{remark}

    \begin{corollary}\label{cor:sqrtinverses}\rqleanmark{cor:sqrtinverses}
      Let $\phi\colon W \to W$ be an automorphism of an $L$-vector space with property \eqref{eq:phinilpotency}. Then
      \begin{equation}
        \left(\phi^{1/2}\right)^{-1}=
        \left(\phi^{-1}\right)^{1/2}.
        \label{eq:squarerootinverses}
      \end{equation}
    \end{corollary}

    \begin{proof}
      The automorphism $\phi^{-1}={\bf1}_W+\sum_{k\geq1}(-n)^k$ again has property \eqref{eq:phinilpotency}. With $\phi^{1/2}={\bf1}_W+m$ as in (a), the automorphism $(\phi^{1/2})^{-1}={\bf1}_W+\sum_{k\geq1}(-m)^k$ is of the form ${\bf1}_W+(\text{nilpotent})$ and satisfies $((\phi^{1/2})^{-1})^2=(\phi^{1/2})^{-2}=\phi^{-1}$. By the uniqueness statement of Proposition \ref{prop:unipotentsquareroots}, applied to $\phi^{-1}$, it equals $(\phi^{-1})^{1/2}$.
    \end{proof}

    \begin{definition}[Unipotent square roots]\label{def:squareroots}\rqleanmark{def:squareroots}
      For a given automorphism $\phi\colon W \to W$ of an $L$-vector space $W$ with the property
      \begin{equation}
        \phi=\gamma^2{\bf1}_{W}+n
        \label{eq:phigammanilpotency}
      \end{equation}
      with $\gamma \in L^\times$ and $n$ nilpotent, we call
      \[
      \phi^{1/2}:=\gamma\cdot\left(\gamma^{-2}\phi\right)^{1/2}\colon W\to W
      \]
      the {\em (unipotent) square root} of $\phi$ (with respect to the square root $\gamma$ of $\gamma^2$).
    \end{definition}

    \begin{remark}\label{rmk:sqrtgaloisuniqueness}\rqleanmark{rmk:sqrtgaloisuniqueness}
      Once $\gamma$ is fixed, $\phi^{1/2}$ is the unique operator of the form $\gamma\cdot{\bf1}_W+(\text{nilpotent})$ whose square is $\phi$, by the uniqueness statement of Proposition \ref{prop:unipotentsquareroots} applied to $\gamma^{-2}\phi$; replacing $\gamma$ by $-\gamma$ replaces $\phi^{1/2}$ by $-\phi^{1/2}$. In this sense, $\phi^{1/2}$ is uniquely determined by \eqref{eq:phigammanilpotency} up to sign. In particular, for $W=W_L = L \otimes_K W_K$ for a quadratic extension $L/K$ with non-trivial Galois automorphism $c\colon L \to L$ and a $K$-vector space $W_K$, the relation
      \begin{equation}
        \left(\phi^{c}\right)^{1/2}=\left(\phi^{1/2}\right)^c
        \label{eq:squarerootgaloisequivariance}
      \end{equation}
      holds, provided that the square root of $\phi^c$ is understood with respect to $c\gamma$ (cf. also statement (d)). According to \eqref{eq:squarerootinverses}, the notation
      \[
      \phi^{-1/2}:=\left(\phi^{1/2}\right)^{-1}
      \]
      is well-defined and behaves as expected, provided that square roots of inverses are taken with respect to the corresponding inverses of the square roots of the scalars (in the unipotent case, the latter is not relevant).
    \end{remark}

\section{Rational quiver representations}

    Throughout this section, $E$ denotes a field of characteristic $0$, and $L_E$ and $G_E$ are as in Definition \ref{def:sqrtminusone}. We first define the target quivers over $E$.

    \begin{definition}[The $E$-rational Gelfand and cyclic quivers]\label{def:GammaE}\rqleanmark{def:GammaE}
      Let $\Gamma$ be the Gelfand quiver if $\ell\geq1$ and the cyclic quiver if $\ell=0$. We denote by $\Gamma_E$ the $E$-rational quiver given by $\Gamma$ and the following action of $G_E$. If $\sqrt{-1}\in E$, then $G_E$ is trivial and $\Gamma_E$ is $\Gamma$, split over $E$. If $\sqrt{-1}\not\in E$, then the non-trivial element $c\in G_E$ acts on $\Gamma$ as the non-trivial element of $\Gal(\QQ[\sqrt{-1}]/\QQ)$ does in Examples \ref{ex:gelfand} and \ref{ex:cyclic}. For $\ell\geq1$, the representations in $\Rep^{\mathrm{nil}}(\Gamma_E)$ are required to satisfy the relation $a_{-} b_{-} = a_{+} b_{+}$.
    \end{definition}

    \begin{lemma}\label{lem:GammaEbasechange}\rqleanmark{lem:GammaEbasechange}%
      We have $\Gamma_E\cong E\otimes_\QQ\Gamma_\QQ$, with the base change of Definition \ref{def:quiverbasechangegeneral}. If $E/\QQ$ is finite and contained in a Galois extension $L/\QQ$ with $\sqrt{-1}\in L$, with respect to which we regard $\Gamma_\QQ$ (cf.\ Remark \ref{rmk:finitequotient}), this is the base change of Definition \ref{def:quiverbasechange} (Lemma \ref{lem:quiverbasechangecompat}).
    \end{lemma}

    \begin{proof}
    We regard $\Gamma_\QQ$ with respect to $\QQ[\sqrt{-1}]/\QQ$, and we fix the $\QQ$-embedding $\iota\colon\QQ[\sqrt{-1}]\to L_E$ with $\iota(\sqrt{-1})=\sqrt{-1}$. Every $\sigma\in G_E$ maps $\sqrt{-1}$ to $\pm\sqrt{-1}$. Thus there is a unique $r(\sigma)\in\Gal(\QQ[\sqrt{-1}]/\QQ)$ with $\sigma\circ\iota=\iota\circ r(\sigma)$, and $r\colon G_E\to\Gal(\QQ[\sqrt{-1}]/\QQ)$ is a group homomorphism. By Definition \ref{def:GammaE}, $\Gamma_E$ is the quiver $\Gamma$ on which $\sigma\in G_E$ acts as $r(\sigma)$ acts on $\Gamma_\QQ$. If $\sqrt{-1}\in E$, both actions are trivial. If $\sqrt{-1}\not\in E$, then $r(c)$ is the non-trivial element, since $c(\sqrt{-1})=-\sqrt{-1}$.

    The vertex fields $L_i$ of $S(\Gamma_\QQ)$ and the fields $L(e_\varepsilon)$ of its bimodules are subfields of $\QQ[\sqrt{-1}]$. The fields of $S(\Gamma_\QQ)\otimes_\QQ E$ are quotients of $L_i\otimes_\QQ E$ and of $L(e_\varepsilon)\otimes_\QQ E$, so they embed into $L_E$ via $\iota$. In particular, $L_E$ splits $S(\Gamma_\QQ)\otimes_\QQ E$, and we compute $E\otimes_\QQ\Gamma_\QQ=\Gamma_E(S(\Gamma_\QQ)\otimes_\QQ E)$ with respect to $L_E/E$. Here $\Gamma_E(-)$ denotes the functor of Theorem \ref{thm:KrationalquiveretaleKspeciesantiequivalence} from \'etale $E$-species to $E$-rational quivers, while $\Gamma_E$ without argument is the quiver of Definition \ref{def:GammaE}; likewise for $\Gamma_\QQ(-)$ and $\Gamma_\QQ$. The two uses are consistent, since the lemma asserts $\Gamma_E\cong\Gamma_E(S(\Gamma_\QQ)\otimes_\QQ E)$. For $E=\QQ$ this is the canonical isomorphism $\Gamma_\QQ\cong\Gamma_\QQ(S(\Gamma_\QQ))$ of Step 1 in the proof of that theorem. As in the proof of Lemma \ref{lem:quiverbasechangecompat}, restriction along $A\to A\otimes_\QQ E$ is a $G_E$-equivariant bijection
    \[
    \Hom_E(A\otimes_\QQ E,L_E)\to\Hom_\QQ(A,L_E)
    \]
    for every finite \'etale $\QQ$-algebra $A$. This does not use that $E$ lies in a Galois extension of $\QQ$. Thus the vertices of $E\otimes_\QQ\Gamma_\QQ$ are identified with $\bigsqcup_{i\in I}\Hom_\QQ(L_i,L_E)$, and its edges likewise, compatibly with the source and target maps. Here $G_E$ acts by post-composition. Every $\QQ$-embedding $L_i\to L_E$ factors uniquely through $\iota$, since $L_i\subseteq\QQ[\sqrt{-1}]$. Therefore post-composition with $\iota$ is a bijection $\Hom_\QQ(L_i,\QQ[\sqrt{-1}])\to\Hom_\QQ(L_i,L_E)$, and the same holds on the edges. Thanks to $\sigma\circ\iota=\iota\circ r(\sigma)$, the element $\sigma\in G_E$ acts on the left hand side through $r(\sigma)$. In this way, we obtain an isomorphism of $E$-rational quivers from $E\otimes_\QQ\Gamma_\QQ$ to $\Gamma_\QQ(S(\Gamma_\QQ))$, on which $G_E$ acts through $r$. We compose it with the inverse of the canonical isomorphism $\Gamma_\QQ\to\Gamma_\QQ(S(\Gamma_\QQ))$ of Step 1 in the proof of Theorem \ref{thm:KrationalquiveretaleKspeciesantiequivalence}, which is $\Gal(\QQ[\sqrt{-1}]/\QQ)$-equivariant. By the first paragraph, this gives $E\otimes_\QQ\Gamma_\QQ\cong\Gamma_E$. The argument does not distinguish between the cases $\sqrt{-1}\in E$ and $\sqrt{-1}\not\in E$.

    Now let $E/\QQ$ be finite and contained in $L$ as in the second assertion. We apply the same argument to the inclusion $\QQ[\sqrt{-1}]\subseteq L$ in place of $\iota$ and to the restriction $\Gal(L/E)\to\Gal(\QQ[\sqrt{-1}]/\QQ)$ in place of $r$. It computes $\Gamma_E(S(\Gamma_\QQ)\otimes_\QQ E)$ with respect to $L/E$, and identifies it with $\Gamma$, on which $\Gal(L/E)$ acts through $\Gal(L/E)\subseteq\Gal(L/\QQ)\to\Gal(\QQ[\sqrt{-1}]/\QQ)$. This is the base change of $\Gamma_\QQ$ to $E$ in the sense of Definition \ref{def:quiverbasechange}, with the action of $\Gal(\QQ[\sqrt{-1}]/\QQ)$ on $\Gamma_\QQ$ inflated along $\Gal(L/\QQ)\to\Gal(\QQ[\sqrt{-1}]/\QQ)$. Thus the second assertion is Lemma \ref{lem:quiverbasechangecompat} for $\Gamma_\QQ$, and the isomorphism is the same comparison as in the general case.
    \end{proof}

    \subsection{$L_E$-rational modules and quiver representations}

    Up to and including Theorem \ref{thm:EEequivalence}, we assume $\ell\geq1$, as in Corollaries \ref{cor:X_Y_inverses} and \ref{cor:T_operators}. In this subsection $M_{L_E}$ denotes an arbitrary object of $\HC_\lambda(\mathfrak{g},K)_{L_E}$, not necessarily a base change, and $M_n:=M_{L_E}[\chi_n]$. The results of Section \ref{sec:HC} that we use apply to it: the proof of \eqref{eq:4XY_rel} and \eqref{eq:4YX_rel} uses only $C=H^2-2H+4XY+1=H^2+2H+4YX+1$ and that $H$ acts on $M_n$ by $n$, and the proofs of Proposition \ref{prop:XY_C_relation} (by the inductive argument), of the proposition following it, of the invertibility statements in Corollary \ref{cor:X_Y_inverses} and of Corollary \ref{cor:T_operators} use in addition only that $C$ commutes with $X$ and $Y$, that $C-\ell^2$ is nilpotent, and that the rational numbers $4$ and $\gamma_\star=(\ell-1)!$ are non-zero. Only the assertion $c\circ n=\bar n\circ c$ of Corollary \ref{cor:X_Y_inverses} involves complex conjugation, and it is used only from the subsection on the $E$-rational structure onwards. With a view to defining the $E$-rational quiver representation, we consider the commutative diagram
    \begin{equation}
    \begin{tikzcd}[column sep=huge,row sep=large]
      M_{-(\ell+1)} \arrow[r, arr, bend left=10, "X"] \arrow[d, arr, "\gamma_\star^2T_-"]  &
      M_{-(\ell-1)} \arrow[l, arr, bend left=10, "Y"] \arrow[r, arr, bend right=10, "XX^{\ell-1}"'] \arrow[d, arr, "\gamma_\star X_\star"] &
      M_{\ell+1} \arrow[l, arr, bend right=10, "X^{-(\ell-1)}Y"'] \arrow[d, arr, "{\bf1}_{M_{\ell+1}}"] \\
      M_{-(\ell+1)} \arrow[r, arr, bend left=10, "Y^{-(\ell-1)}X"] & M_{\ell-1} \arrow[l, arr, bend left=10, "YY^{\ell-1}"] \arrow[r, arr, bend right=10, "X"'] &
      M_{\ell+1} \arrow[l, arr, bend right=10, "Y"']
    \end{tikzcd}
    \end{equation}
    Commutativity refers to the condition that all $4$ squares, the two \lq{}in the front\rq{} with right facing horizontal arrows, as well as the two \lq{}in the back\rq{} with the left facing arrows, commute. Note that the relation of the Gelfand quiver is satisfied (Lemma 1.7 in \cite{ABR}); for the normalized form of the diagram we verify this directly below.

    We recall the normalization $X_\star=\gamma_\star^{-1}X^{\ell-1}|_{M_{-(\ell-1)}}$. In a first step, we renormalize this diagram to
    \begin{equation}
    \begin{tikzcd}[column sep=huge, row sep=large]
      M_{-(\ell+1)} \arrow[r, arr, bend left=10, "X"] \arrow[d, arr, "\gamma_\star T_-"]  &
      M_{-(\ell-1)} \arrow[l, arr, bend left=10, "Y"] \arrow[r, arr, bend right=10, "XX_\star"'] \arrow[d, arr, "X_\star"] &
      M_{\ell+1} \arrow[l, arr, bend right=10, "X_\star^{-1}Y"'] \arrow[d, arr, "{\bf1}_{M_{\ell+1}}"] \\
      M_{-(\ell+1)} \arrow[r, arr, bend left=10, "Y^{-(\ell-1)}X"] &
      M_{\ell-1} \arrow[l, arr, bend left=10, "YY^{\ell-1}"] \arrow[r, arr, bend right=10, "X"'] &
      M_{\ell+1} \arrow[l, arr, bend right=10, "Y"']
    \end{tikzcd}
    \end{equation}
    The normalization $Y_\star=\gamma_\star^{-1}Y^{\ell-1}|_{M_{\ell-1}}$ provides us with the commutative renormalized diagram
    \begin{equation}\label{eq:normalizedcomparison}
    \begin{tikzcd}[column sep=huge, row sep=large]
      M_{-(\ell+1)} \arrow[r, arr, bend left=10, "X"] \arrow[d, arr, "T_-"] &
      M_{-(\ell-1)} \arrow[l, arr, bend left=10, "Y"] \arrow[r, arr, bend right=10, "XX_\star"'] \arrow[d, arr, "X_\star"] &
      M_{\ell+1} \arrow[l, arr, bend right=10, "X_\star^{-1}Y"'] \arrow[d, arr, "{\bf1}_{M_{\ell+1}}"] \\
    	M_{-(\ell+1)} \arrow[r, arr, bend left=10, "Y_\star^{-1}X"] &
        M_{\ell-1} \arrow[l, arr, bend left=10, "YY_\star"] \arrow[r, arr, bend right=10, "X"'] & M_{\ell+1} \arrow[l, arr, bend right=10, "Y"']
    \end{tikzcd}
    \end{equation}
    Its commutativity is checked directly. The two right-hand squares commute trivially: $X\circ X_\star={\bf1}_{M_{\ell+1}}\circ XX_\star$ and $X_\star\circ X_\star^{-1}Y=Y\circ{\bf1}_{M_{\ell+1}}$. For the two left-hand squares, \eqref{eq:YXm_via_C} with $m=\ell-1$ and $k=\ell-2$ gives
    \[
      Y_\star X_\star=\left(4^{\ell-1}\gamma_\star^2\right)^{-1}\prod_{j=0}^{\ell-2}\left(C-(\ell-2-2j)^2\right)\quad\text{on }M_{-(\ell-1)},
    \]
    which is the same polynomial in $C$ as $T_-$ on $M_{-(\ell+1)}$ by \eqref{eq:T_minus_def}. Since $C$ commutes with $X$ and $Y$, we get $Y_\star X_\star X=XT_-$ on $M_{-(\ell+1)}$, i.e., $X_\star\circ X=Y_\star^{-1}X\circ T_-$, and $YY_\star X_\star=T_-Y$ on $M_{-(\ell-1)}$, i.e., $YY_\star\circ X_\star=T_-\circ Y$.
    In this way, we obtain a normalized nilpotent quiver representation, which we record in the following definition.

    \begin{definition}[The functor $\EE_{L_E}$]\label{def:EEsplit}\rqleanmark{def:EEsplit}
    For $M_{L_E}\in\HC_\lambda(\mathfrak{g},K)_{L_E}$ we let $\EE_{L_E}(M_{L_E})$ denote the quiver representation
    \begin{equation*}
    \begin{tikzcd}[column sep=large]
    	M(-) \arrow[rr, arr, bend left=10, "\phi_{a_-}"] & & M(\star) \arrow[ll, arr, bend left=10, "\phi_{b_-}"] \arrow[rr, arr, bend right=10, "\phi_{b_+}"'] & & M(+) \arrow[ll, arr, bend right=10, "\phi_{a_+}"']
    \end{tikzcd}
    \end{equation*}
    over $L_E$, where
    \[
    M(-):= M_{L_E}[\chi_{-(\ell+1)}],\quad
    M(\star):= M_{L_E}[\chi_{-(\ell-1)}],\quad
    M(+):= M_{L_E}[\chi_{\ell+1}],
    \]
    and
    \[
    \phi_{a_+}:=X_\star^{-1}Y,\quad
    \phi_{a_-}:=X,\quad
    \phi_{b_+}:=XX_\star,\quad
    \phi_{b_-}:=Y.
    \]
    \end{definition}

    The identity
    \[
    \phi_{a_{-}}\circ\phi_{b_{-}} = \phi_{a_{+}}\circ\phi_{b_{+}}
    \]
    shows that we have indeed defined an $L_E$-rational representation of the Gelfand quiver. It holds because $[Y,X^\ell]=-\ell X^{\ell-1}(H+\ell-1)$ in $U(\mathfrak{g})$, by induction on $\ell$ from $[X,Y]=H$ and $[H,X]=2X$, and $H+\ell-1$ vanishes on $M(\star)=M_{-(\ell-1)}$. Hence $YXX_\star=\gamma_\star^{-1}YX^\ell=\gamma_\star^{-1}X^\ell Y=X_\star XY$ on $M(\star)$, i.e., $\phi_{a_+}\phi_{b_+}=X_\star^{-1}YXX_\star=XY=\phi_{a_-}\phi_{b_-}$. The representation is nilpotent: by \eqref{eq:4XY_rel} with $p=-\ell$ and the relation, both loops $\phi_{a_\pm}\phi_{b_\pm}$ on $M(\star)$ act as $\frac14(C-\ell^2)$, which is nilpotent on the finite-dimensional space $M(\star)$; every other vertex of a path is $\star$, so every path of length $\geq2m+1$ contains $m$ consecutive loops at $\star$, and therefore acts as zero for $m\gg0$. This is well-defined for every $L_E$-rational $(\mathfrak{g}, K)$-module
    \[
    M_{L_E}\in \HC_\lambda(\mathfrak{g},K)_{L_E}
    \]
    regardless of whether the latter module descends to a $(\mathfrak{g}, K)$-module over $E$ or not.

    \begin{proposition}\label{prop:splitequivalence}\rqleanmark{prop:splitequivalence}
    The assignment
    \[
    M_{L_E}\;\mapsto\;\EE_{L_E}(M_{L_E})
    \]
    is functorial and defines an equivalence between the category $\HC_\lambda(\mathfrak{g}, K)_{L_E}$ and the category of nilpotent finite-dimensional $L_E$-rational representations of the Gelfand quiver satisfying its relation.
    \end{proposition}

    \begin{proof}
    To a morphism $f$ we assign $\EE_{L_E}(f):=(f|_{M(-)},f|_{M(\star)},f|_{M(+)})$, which is a morphism of quiver representations because $f$ preserves the weight spaces and commutes with $X$ and $Y$, hence with $X_\star$ and $X_\star^{-1}$. The equivalence follows by a direct proof analogous to the classical case, which we now give. By \eqref{eq:4XY_rel} and \eqref{eq:4YX_rel}, for $j\neq\pm\ell$ the composites $4XY$ on $M_{j+1}$ and $4YX$ on $M_{j-1}$ equal $C-j^2=(\ell^2-j^2)\cdot{\bf1}+(\text{nilpotent})$ and are invertible, so that
    \begin{equation}
      X\colon M_{j-1}\to M_{j+1}\ \text{is bijective for } j\equiv\ell\pmod2,\ j\neq\pm\ell.
      \label{eq:Xbijective}
    \end{equation}
    Thus the non-zero weight spaces form three chains: $M_{-(\ell+1)-2k}$, $k\geq0$, mapped bijectively onto $M(-)=M_{-(\ell+1)}$ by $X^k$; $M_{-(\ell-1)+2k}$, $0\leq k\leq\ell-1$, onto which $M(\star)=M_{-(\ell-1)}$ is mapped bijectively by $X^k$; and $M_{\ell+1+2k}=X^k(M(+))$, $k\geq0$, again with $X^k$ bijective. Hence a morphism $f\colon A\to B$ in $\HC_\lambda(\mathfrak{g}, K)_{L_E}$ is determined by its restrictions to $A(-)$, $A(\star)$ and $A(+)$ (faithfulness). Conversely, let $(\psi_-,\psi_\star,\psi_+)$ be a morphism $\EE_{L_E}(A)\to\EE_{L_E}(B)$. Define $\psi\colon A\to B$ on each chain by transport along the bijections \eqref{eq:Xbijective} ($X_A$, $X_B$ denoting $X$ on $A$, $B$): $\psi:=X_B^{-k}\psi_-X_A^k$ on $A_{-(\ell+1)-2k}$, $\psi:=X_B^k\psi_\star X_A^{-k}$ on $A_{-(\ell-1)+2k}$ and $\psi:=X_B^k\psi_+X_A^{-k}$ on $A_{\ell+1+2k}$, and $\psi:=0$ on the weight spaces of the wrong parity. Then $\psi$ commutes with $H$ and $K$, and with $X$: inside each chain by construction; on $A_{-(\ell+1)}$, where $X=\phi_{a_-}$, by the edge equation for $a_-$; and on $A_{\ell-1}$, where $X=\phi_{b_+}X_\star^{-1}$ and $\psi=X_\star\psi_\star X_\star^{-1}$ (as $X_\star$ is a non-zero multiple of $X^{\ell-1}$), by the edge equation for $b_+$. It commutes with $C$: by \eqref{eq:4XY_rel} and \eqref{eq:4YX_rel}, $C=\ell^2+4\phi_{b_-}\phi_{a_-}$ on $A(-)$, $C=\ell^2+4\phi_{a_-}\phi_{b_-}$ on $A(\star)$ and $C=\ell^2+4\phi_{b_+}\phi_{a_+}$ on $A(+)$, which commute with $\psi_\bullet$ by the edge equations, and as $C$ commutes with $X$ this propagates along the chains. Finally $\psi$ commutes with $Y$: on $A_{-(\ell-1)}$, where $Y=\phi_{b_-}$, by the edge equation for $b_-$; on $A_{\ell+1}$, where $Y=X_\star\phi_{a_+}$, by the edge equation for $a_+$; and on every other $A_n$ we have $Y|_{A_n}=\frac14\left(X|_{A_{n-2}}\right)^{-1}\left(C-(n-1)^2\right)$ by \eqref{eq:4XY_rel} and \eqref{eq:Xbijective}, where every factor commutes with $\psi$. Hence $\psi$ is a morphism of $(\mathfrak{g},K)$-modules with $\EE_{L_E}(\psi)=(\psi_-,\psi_\star,\psi_+)$ (fullness). Essential surjectivity follows from the construction in the proof of Theorem \ref{thm:EEequivalence} below, with all rational structures omitted: for $V$ given, the module $M$ constructed there satisfies $\EE_{L_E}(M)\cong V$, via the isomorphism $({\bf1}_{V(-)},{\bf1}_{V(\star)},p(n_+)^{-1})$, where $X_\star$, read on $V(\star)$, is a polynomial $p(n_\star)$ in $n_\star$ with $p(0)=1$ (see below; note $\phi_{b_+}\circ p(n_\star)=p(n_+)\circ\phi_{b_+}$, $p(n_\star)\circ\phi_{a_+}=\phi_{a_+}\circ p(n_+)$).

    In more detail, we use the notation of that construction. The square root $\phi_\star^{1/2}$ is $\ell$ times the square root of $\ell^{-2}\phi_\star={\bf1}_{V(\star)}+4\ell^{-2}n_\star$, which is obtained by the recursion in the proof of Proposition \ref{prop:unipotentsquareroots}. Each step of this recursion stays in the subalgebra $L_E[n_\star]$ of $\End_{L_E}(V(\star))$, since the inverse of ${\bf1}+m$ with $m\in L_E[n_\star]$ nilpotent is $\sum_{k\geq0}(-m)^k\in L_E[n_\star]$. Therefore $\phi_\star^{1/2}$, every $\xi_{\star,+}^{(j)}$, and $X_\star=\widetilde{\gamma}^{-1}\prod_j\xi_{\star,+}^{(j)}$, the product over $j=-\ell+2,-\ell+4,\dots,\ell-2$, lie in $L_E[n_\star]$, i.e., $X_\star=p(n_\star)$ for a polynomial $p$ with coefficients in $L_E$. Since $X_\star-{\bf1}_{V(\star)}=\widetilde{n}_\star$ is nilpotent (cf.\ the paragraph on the natural isomorphism below) and $p(n_\star)-p(0)\cdot{\bf1}_{V(\star)}$ are nilpotent, we have $p(0)=1$ if $V(\star)\neq0$; if $V(\star)=0$, we choose $p=1$.

    Now let $p$ be any polynomial with $p(n_\star)=X_\star$ and $p(0)=1$. By the relation of the Gelfand quiver, $n_\star=\phi_{a_+}\phi_{b_+}$, while $n_+=\phi_{b_+}\phi_{a_+}$. We have $\phi_{b_+}n_\star=n_+\phi_{b_+}$ and $n_\star\phi_{a_+}=\phi_{a_+}n_+$, which gives the two identities noted above. As $n_+$ is nilpotent and $p(0)=1$, the endomorphism $p(n_+)$ of $V(+)$ is invertible. By construction, $X|_{M_{-(\ell+1)}}=\phi_{a_-}$, $Y|_{M_{-(\ell-1)}}=\phi_{b_-}$, $X|_{M_{\ell-1}}=\phi_{b_+}$ and $Y|_{M_{\ell+1}}=\phi_{a_+}$. So $\EE_{L_E}(M)$ has the edge maps $\phi_{a_-}$, $\phi_{b_-}$ and
    \[
    XX_\star=\phi_{b_+}\circ p(n_\star)=p(n_+)\circ\phi_{b_+},\quad
    X_\star^{-1}Y=p(n_\star)^{-1}\circ\phi_{a_+}=\phi_{a_+}\circ p(n_+)^{-1},
    \]
    and $({\bf1}_{V(-)},{\bf1}_{V(\star)},p(n_+)^{-1})$ is an isomorphism $\EE_{L_E}(M)\to V$.
    \end{proof}

    \begin{remark}\label{rmk:asymmetricsplitES}
      Essential surjectivity in Proposition \ref{prop:splitequivalence} can also be shown without square roots, by the device we use for $\ell=0$ in the proof of Theorem \ref{thm:cyclicEEequivalence}. On the module $M$ of the proof of Theorem \ref{thm:EEequivalence} one lets $X$ act by ${\bf1}$ along the chains of $V(+)$ and $V(\star)$ and $Y$ by ${\bf1}$ along the chain of $V(-)$. On the remaining weight spaces of these chains one puts $Y|_{M_n}:=\frac14(\phi_\bullet-(n-1)^2)$, respectively $X|_{M_n}:=\frac14(\phi_--(n+1)^2)$, with $\phi_\bullet=\ell^2{\bf1}+4n_\bullet$ as in that proof, and one lets $X|_{M_{\ell-1}}:=\gamma_\star\phi_{b_+}$, $Y|_{M_{\ell+1}}:=\gamma_\star^{-1}\phi_{a_+}$, $X|_{M_{-(\ell+1)}}:=\phi_{a_-}$, $Y|_{M_{-(\ell-1)}}:=\phi_{b_-}$. Then $X_\star=\gamma_\star^{-1}\cdot{\bf1}_{V(\star)}$ and $\EE_{L_E}(M)=V$ on the nose; we do not use this construction and do not prove these claims here. For $\ell\geq2$ it does not carry the rational structure: the chain of $V(\star)$ is mapped to itself by $n\mapsto-n$, and $X={\bf1}$ along it, while $Y|_{M_{j+1}}=\frac14(\phi_\star-j^2)$ there is in general not the identity. Thus the semi-linear map formed as in \eqref{eq:varphicj} need not intertwine $X$ with $Y$, which is why we use the symmetric construction with $\frac12\xi_{\bullet,\pm}^{(j)}$.
    \end{remark}

    \begin{remark}\label{rmk:splitequivalenceanyfield}\rqleanmark{rmk:splitequivalenceanyfield}
      Definition \ref{def:EEsplit} and the proof of Proposition \ref{prop:splitequivalence} use of $L_E$ only that it is a field of characteristic $0$. Let $F$ be any field of characteristic $0$, and consider the $F$-vector spaces $M=\bigoplus_{n\in\ZZ}M_n$ with operators $H$, $X$, $Y$ such that $H$ acts on $M_n$ by $n$, $[H,X]=2X$, $[H,Y]=-2Y$, $[X,Y]=H$, each $M_n$ is finite-dimensional, $M$ is finitely generated, $C-\ell^2$ is nilpotent and $M_n=0$ for $n\not\equiv\ell-1\pmod2$. Then the same formulas define an equivalence between the category of these modules, with the $F$-linear maps preserving the $M_n$ and commuting with $H$, $X$ and $Y$ as morphisms, and the category of nilpotent finite-dimensional representations of the Gelfand quiver over $F$ satisfying its relation. If $\sqrt{-1}\in F$, the former category is $\HC_\lambda(\mathfrak{g},K)_F$.
    \end{remark}

    \subsection{The $E$-rational structure}

    We now use the diagram \eqref{eq:normalizedcomparison} in the case $M_{L_E}=L_E\otimes_E M_E$ with $M_E\in\HC_\lambda(\mathfrak{g}, K)_E$ to define an $E$-rational structure $\varphi_\bullet$ on $\EE_{L_E}(M_{L_E})$ in the sense of Definition \ref{def:rationalrepresentations}. If $\sqrt{-1}\in E$, the group $G_E$ is trivial, $\varphi_\bullet$ is the trivial rational structure, and the following construction is obsolete; we therefore assume $\sqrt{-1}\not\in E$ in the remainder of this subsection.

    For this, we first observe that complex conjugation transforms the commutative diagram \eqref{eq:normalizedcomparison} into the diagram 
    \begin{equation}\label{eq:conjugatenormalizedcomparison}
    \begin{tikzcd}[column sep=huge, row sep=large]
      M_{-(\ell+1)} \arrow[r, arr, bend left=10, "X"] &
      M_{-(\ell-1)} \arrow[l, arr, bend left=10, "Y"] \arrow[r, arr, bend right=10, "XX_\star"'] &
      M_{\ell+1} \arrow[l, arr, bend right=10, "X_\star^{-1}Y"'] \\
      M_{-(\ell+1)} \arrow[u, arr, "{\bf1}_{M_{-(\ell+1)}}"] \arrow[r, arr, bend left=10, "Y_\star^{-1}X"] &
      M_{\ell-1} \arrow[l, arr, bend left=10, "YY_\star"] \arrow[r, arr, bend right=10, "X"'] \arrow[u, arr, "Y_\star"] &
      M_{\ell+1} \arrow[l, arr, bend right=10, "Y"'] \arrow[u, arr, "T_+"]
    \end{tikzcd}
    \end{equation}\rqleanmark{eq:conjugatenormalizedcomparison}
    whose commutativity can also be verified directly.

    We now apply unipotent stabilization from Proposition \ref{prop:unipotentstabilization} simultaneously to the three vertical arrows in \eqref{eq:normalizedcomparison}, i.e., we consider the iterative construction for the three situations simultaneously:
    \begin{itemize}
    \item[($+$)] $W_+=M_{\ell+1}$, $W_-=M_{-(\ell+1)}$, $\varphi_+={\bf1}_{M_{\ell+1}}$, $\varphi_-=T_-$,
    \item[($-$)] $W_+=M_{-(\ell+1)}$, $W_-=M_{\ell+1}$, $\varphi_+=T_-$, $\varphi_-={\bf1}_{M_{\ell+1}}$,
    \item[($\star$)] $W_+=W_-=M_{-(\ell-1)}$, $\varphi_+=\varphi_-=X_\star$; here $W_-^c=M_{\ell-1}$ and $\varphi_-^c=Y_\star$, so that the pair $(M_{\ell-1},Y_\star)$ is the $c$-conjugate description of this same instantiation, not a second one.
    \end{itemize}
    Here Proposition \ref{prop:unipotentstabilization} is applied to $L/K=L_E/E$, $\tau=c$ and $W_K=M_E$, so that $W^c=c(W)$ for a subspace $W$ and $\varphi^c=c\circ\varphi\circ c$ for a map $\varphi$. Complex conjugation $c={\bf1}\otimes c$ on $M_{L_E}$ maps $M_n$ to $M_{-n}$, satisfies $cX=Yc$, $cY=Xc$, hence $cC=Cc$ (as $cC=C$ in $U(\mathfrak{g})$), and fixes rational scalars. Since $T_\pm$ is the same polynomial in $C$ with rational coefficients on $M_{\pm(\ell+1)}$, and $X_\star$, $Y_\star$ are $\gamma_\star^{-1}X^{\ell-1}$, $\gamma_\star^{-1}Y^{\ell-1}$ with $\gamma_\star\in\QQ$, we get
    \[
      T_-^c=T_+\qquad\text{and}\qquad X_\star^c=Y_\star.
    \]
    In all three cases, the prerequisites of Proposition \ref{prop:unipotentstabilization} are satisfied, according to Corollary \ref{cor:T_operators} and Corollary \ref{cor:X_Y_inverses}: in case ($+$), $\varphi_-^c\circ\varphi_+=T_+$; in case ($-$), $\varphi_-^c\circ\varphi_+=T_-$; in case ($\star$), $\varphi_-^c\circ\varphi_+=Y_\star X_\star$. In the first iteration, because $T_-^c=T_+$ and $X_\star^c=Y_\star$,
    \begin{itemize}
      \item ${\bf1}_{M_{\ell+1}}$ is replaced by $\varphi_{+,+}^{(1)}:=\frac{1}{2}({\bf1}_{M_{\ell+1}}+T_+^{-1})$ and $T_-$ by $\varphi_{+,-}^{(1)}:=\frac{1}{2}(T_-+{\bf1}_{M_{-(\ell+1)}})$,
      \item $T_-$ is replaced by $\varphi_{-,+}^{(1)}:=\frac{1}{2}(T_-+{\bf1}_{M_{-(\ell+1)}})$ and ${\bf1}_{M_{\ell+1}}$ by $\varphi_{-,-}^{(1)}:=\frac{1}{2}({\bf1}_{M_{\ell+1}}+T_+^{-1})$,
      \item $X_\star$ is replaced by $\varphi_{\star,+}^{(1)}:=\frac{1}{2}(X_\star+Y_\star^{-1})$, and correspondingly $Y_\star$ by the $c$-conjugate $\varphi_{\star,-}^{(1)}:=\varphi_{\star,+}^{(1),c}=\frac{1}{2}(Y_\star+X_\star^{-1})$.
    \end{itemize}
    Since the case ($-$) is the case ($+$) with the roles of $\varphi_+$ and $\varphi_-$ interchanged, Remark \ref{rmk:unipotentstabilization} gives the transposition identities $\varphi_{-,\pm}^{(1)}=\varphi_{+,\mp}^{(1)}$, and by construction $\varphi_{\star,-}^{(1)}=\varphi_{\star,+}^{(1),c}$, so that in what follows we simply define
    \[
    \varphi_\pm^{(1)}:=\varphi_{\pm,+}^{(1)}\quad\text{and}\quad\varphi_\star^{(1)}:=\varphi_{\star,+}^{(1)}.
    \]

    The commutativity of diagrams \eqref{eq:normalizedcomparison} and \eqref{eq:conjugatenormalizedcomparison} shows that the diagram 
    \begin{equation}\label{eq:firstiteration}
    \begin{tikzcd}[column sep=huge, row sep=large]
      M_{-(\ell+1)} \arrow[r, arr, bend left=10, "X"] \arrow[d, arr, "\varphi_{-}^{(1)}"] &
      M_{-(\ell-1)} \arrow[l, arr, bend left=10, "Y"] \arrow[r, arr, bend right=10, "XX_\star"'] \arrow[d, arr, "\varphi_{\star}^{(1)}"] &
      M_{\ell+1} \arrow[l, arr, bend right=10, "X_\star^{-1}Y"'] \arrow[d, arr, "\varphi_{+}^{(1)}"] \\
      M_{-(\ell+1)} \arrow[r, arr, bend left=10, "Y_\star^{-1}X"] &
      M_{\ell-1} \arrow[l, arr, bend left=10, "YY_\star"'] \arrow[r, arr, bend right=10, "X"'] &
      M_{\ell+1} \arrow[l, arr, bend right=10, "Y"']
    \end{tikzcd}
    \end{equation}
    commutes. Indeed, each vertical arrow of \eqref{eq:firstiteration} is of the form $\frac{1}{2}(V+U^{-1})$, where $V$ is the corresponding vertical arrow of \eqref{eq:normalizedcomparison} and $U$ the corresponding (upward) vertical arrow of \eqref{eq:conjugatenormalizedcomparison}. Since the arrows $U$ are invertible and intertwine the lower row with the upper row, their inverses $U^{-1}$ intertwine the upper row with the lower row, as do the arrows $V$; hence so do the averages $\frac{1}{2}(V+U^{-1})$.

    The unipotent stabilizations $\varphi_{-,\pm}^{(\infty)}$, $\varphi_{+,\pm}^{(\infty)}$, $\varphi_{\star,\pm}^{(\infty)}$ in the sense of Corollary \ref{cor:unipotentstabilization} exist in the three cases $+$, $-$, and $\star$.

    As before, we have the relations $\varphi_{-,\pm}^{(\infty)}=\varphi_{+,\mp}^{(\infty)}$ and $\varphi_{\star,-}^{(\infty)}=\varphi_{\star,+}^{(\infty),c}$, which follow by induction. Therefore, we define as in the case of the first iteration:
    \[
    \varphi_\pm^{(\infty)}:=\varphi_{\pm,+}^{(\infty)}\quad\text{and}\quad\varphi_\star^{(\infty)}:=\varphi_{\star,+}^{(\infty)}.
    \]
    We remark that, thanks to \eqref{eq:conjugateinverse},
    \begin{equation}
      \varphi_{-}^{(\infty),c}\circ\varphi_+^{(\infty)}= {\bf1}_{M_{\ell+1}} \quad\text{and}\quad \varphi_{+}^{(\infty),c}\circ\varphi_-^{(\infty)}= {\bf1}_{M_{-(\ell+1)}},
      \label{eq:conjugateinversepm}
    \end{equation}
    and, in the situation ($\star$), where $W_+=W_-$ and $\varphi_+=\varphi_-=X_\star$, so that the recursion of Corollary \ref{cor:unipotentstabilization} produces the same sequence for both signs, \eqref{eq:conjugateinverse} gives the first and its conjugate by $c$ the second identity of
    \begin{equation}
      \varphi_{\star}^{(\infty),c}\circ\varphi_\star^{(\infty)}=
      {\bf1}_{M_{-(\ell-1)}} \quad\text{and}\quad \varphi_{\star}^{(\infty)}\circ\varphi_\star^{(\infty),c}=
      {\bf1}_{M_{\ell-1}}.
      \label{eq:conjugateinversestar}
    \end{equation}

    The compositions
    \begin{align}
      \varphi_+\colon M(+)\to M(-),&\quad\text{given by}\quad (c\colon M(+)\to M(-))\circ\varphi_+^{(\infty)}\label{eq:varphiplus}\\
      \varphi_-\colon M(-)\to M(+),&\quad\text{given by}\quad (c\colon M(-)\to M(+))\circ\varphi_-^{(\infty)}\label{eq:varphiminus}\\
      \varphi_\star\colon M(\star)\to M(\star),&\quad\text{given by}\quad (c\colon M(\star)\to M(\star))\circ\varphi_\star^{(\infty)}\label{eq:varphistar}
    \end{align}\rqleanmark{eq:varphiplus}
    are semi-linear isomorphisms, and we put $\varphi_{\bullet,c}:=\varphi_\bullet$ and $\varphi_{\bullet,1}:={\bf1}$. In the proof of Theorem \ref{thm:EEequivalence}, we show that these define an $E$-rational structure in the sense of Definition \ref{def:rationalrepresentations}.

    \subsection{The $E$-rational categorical equivalence}

    \begin{theorem}\label{thm:EEequivalence}\rqleanmark{thm:EEequivalence}
    Let $E$ be a field of characteristic $0$. For $\ell\geq 1$, the assignment
    \[
      \EE_E\colon
      M_{E}\mapsto \left(\EE_{L_E}(M_{L_E}),\varphi_\bullet\right)
    \]
    from $\HC_\lambda(\mathfrak{g}, K)_{E}$ to the category $\Rep^{\mathrm{nil}}(\Gamma_E)$ of $E$-rational finite-dimensional nilpotent representations of the $E$-rational Gelfand quiver is well-defined and an equivalence of categories.
    \end{theorem}

    \begin{proof}
    If $\sqrt{-1}\in E$, then $L_E=E$, the group $G_E$ is trivial, the rational structure $\varphi_\bullet$ is trivial and $\EE_E=\EE_{L_E}$. In this case the statement is Proposition \ref{prop:splitequivalence} over $L_E=E$. We therefore assume $\sqrt{-1}\not\in E$ in the following, and we denote by $c\in G_E$ the non-trivial element. We record what the construction of $\EE_{L_E}$ and $\varphi_\bullet$ and the proof below use.
    \begin{itemize}
    \item[(i)] Over $L_E$, the definition of $\EE_{L_E}$ and the proof that it is an equivalence use only the following. (See Proposition \ref{prop:splitequivalence}, and for essential surjectivity the construction of $M$ in the paragraph \lq{}Essential surjectivity\rq{} below and the computation of $X_\star$ on $V(\star)$ in the paragraph \lq{}The natural isomorphism\rq{}, without rational structures.) The module $M$ is the direct sum of its weight spaces $M_n$, on which $H$ acts by $n$ and $K$ via $\chi_n$. We have the relations $[H,X]=2X$, $[H,Y]=-2Y$, $[X,Y]=H$ and $C=H^2-2H+4XY+1$, with $C$ commuting with $X$ and $Y$ and $C-\ell^2$ nilpotent. Hence Propositions \ref{prop:XY_C_relation} and \ref{prop:xyscalarplusnilpotent} and the invertibility statements of Corollaries \ref{cor:X_Y_inverses} and \ref{cor:T_operators} hold, as explained at the beginning of this section. The non-zero rational numbers $2$, $4$, $\ell^2-j^2$ ($j\neq\pm\ell$) and $\gamma_\star=(\ell-1)!$ are invertible in $E$. Finally, we use Proposition \ref{prop:unipotentsquareroots} and Definition \ref{def:squareroots} for $L=L_E$ with $\gamma=\ell$.
    \item[(ii)] The passage to $E$ uses in addition only the following. Complex conjugation $c$ on $M_{L_E}=L_E\otimes_E M_E$ is a semi-linear involution with $c(M_n)=M_{-n}$, $cX=Yc$ and $cY=Xc$, which fixes rational scalars, so that $cC=Cc$, $T_-^c=T_+$ and $X_\star^c=Y_\star$. We use Proposition \ref{prop:unipotentstabilization}, Corollary \ref{cor:unipotentstabilization} and Remarks \ref{rmk:unipotentstabilization} and \ref{rmk:stabilizationaddenda} for $L/K=L_E/E$, $\tau=c$ and $W_K=M_E$. We also use statements (d) and (e) of Proposition \ref{prop:unipotentsquareroots}, \eqref{eq:squarerootinverses} and \eqref{eq:squarerootgaloisequivariance} for $c$, with $c(\ell)=\ell$ and $c(j)=j$. Furthermore, $\mathfrak{g}_E$ is the space of $c$-fixed elements of $\mathfrak{g}_{L_E}$, where $c(\alpha H+\beta X+\gamma Y)=-c(\alpha)H+c(\gamma)X+c(\beta)Y$, and $c(\chi_n)=\chi_{-n}$. Galois descent (Theorem \ref{thm:galoisdescent}) for $L_E/E$ is used for morphisms in the proof of fullness and for the descent of the constructed module in the proof of essential surjectivity. In the proof of fullness only, we use the element $\zeta:=\sqrt{-1}\in L_E$, through $c(\zeta)=-\zeta\neq0$ alone.
    \end{itemize}
    Apart from this $\zeta$, no element of $L_E\smallsetminus E$ occurs as a scalar.

    \smallskip
    \noindent{\bf Well-definedness.}
    Thanks to \eqref{eq:conjugateinversepm} and \eqref{eq:conjugateinversestar}, in the notation above, we have
    \[
    \varphi_+\circ\varphi_-=(c\circ\varphi_+^{(\infty)})\circ(c\circ\varphi_-^{(\infty)})=c^2\circ\varphi_+^{(\infty),c}\circ\varphi_-^{(\infty)}={\bf1}_{M(-)}
    \]
    and similarly $\varphi_-\circ\varphi_+={\bf1}_{M(+)}$ and $\varphi_\star\circ\varphi_\star={\bf1}_{M(\star)}$. Thus $\varphi_\bullet$ satisfy the cocycle condition \eqref{eq:semi-linearactiononrep}.

    As in the case of the first iteration \eqref{eq:firstiteration}, it follows by induction on $k$ that the diagram with vertical arrows $\varphi_-^{(k)}$, $\varphi_\star^{(k)}$, $\varphi_+^{(k)}$ commutes: its conjugate by $c$, with upward vertical arrows, commutes as well, and the averaging argument given after \eqref{eq:firstiteration}, with \eqref{eq:normalizedcomparison} and \eqref{eq:conjugatenormalizedcomparison} replaced by these two diagrams, shows that the diagram with vertical arrows $\varphi_\bullet^{(k+1)}$ commutes. For $k$ large this is the diagram
    \begin{equation}\label{eq:limitdiagram}
    \begin{tikzcd}[column sep=huge, row sep=large]
      M_{-(\ell+1)} \arrow[r, arr, bend left=10, "X"] \arrow[d, arr, "\varphi_{-}^{(\infty)}"] &
      M_{-(\ell-1)} \arrow[l, arr, bend left=10, "Y"] \arrow[r, arr, bend right=10, "XX_\star"'] \arrow[d, arr, "\varphi_{\star}^{(\infty)}"] &
      M_{\ell+1} \arrow[l, arr, bend right=10, "X_\star^{-1}Y"'] \arrow[d, arr, "\varphi_{+}^{(\infty)}"] \\
      M_{-(\ell+1)} \arrow[r, arr, bend left=10, "Y_\star^{-1}X"] &
      M_{\ell-1} \arrow[l, arr, bend left=10, "YY_\star"] \arrow[r, arr, bend right=10, "X"'] &
      M_{\ell+1} \arrow[l, arr, bend right=10, "Y"']
    \end{tikzcd}
    \end{equation}
    commutes.

    Due to the commutativity of \eqref{eq:limitdiagram}, $\varphi_\bullet$ also satisfy condition \eqref{eq:koecheractiononrep} and thus define a $E$-rational structure on $\EE_{L_E}(M_{L_E})$ in the sense of Definition \ref{def:rationalrepresentations}.

    \smallskip
    \noindent{\bf Functoriality} follows inductively from the recursive process of unipotent stabilization using the commutative diagrams \eqref{eq:normalizedcomparison}, \eqref{eq:conjugatenormalizedcomparison}, \eqref{eq:firstiteration} and in the limit case \eqref{eq:limitdiagram}. More precisely, let $f\colon M_E\to N_E$ be a morphism in $\HC_\lambda(\mathfrak{g},K)_E$. Its base change to $L_E$ preserves the weight spaces, commutes with $c$, and commutes with $X$, $Y$ and $C$, hence with $T_\pm$ and $X_\star$. By Remark \ref{rmk:stabilizationaddenda}(i), applied to each of the three situations ($+$), ($-$), ($\star$), it commutes with $\varphi_\bullet^{(\infty)}$, hence with the rational structures $\varphi_\bullet=c\circ\varphi_\bullet^{(\infty)}$; it commutes with the maps $\phi_{a_\pm}$, $\phi_{b_\pm}$, i.e., with $X$, $Y$, $XX_\star$ and $X_\star^{-1}Y$, since it commutes with $X$, $Y$ and $X_\star$, hence with $X_\star^{-1}$. Thus $f$ induces a morphism $\EE_E(f):=\EE_{L_E}({\bf1}\otimes f)$ of $E$-rational representations.

    \paragraph{Fully faithfulness.}\ 
    We have the following 2-commutative diagram:
    \begin{equation}
    \begin{tikzcd}
    	\HC_\lambda(\mathfrak{g},K)_{L_E} \arrow[r] & \Rep^{\mathrm{nil}}(\Gamma_{L_E}) \\
    	\HC_\lambda(\mathfrak{g},K)_E \arrow[u, arr, "{L_E\otimes_E -}"] \arrow[r] & \Rep^{\mathrm{nil}}(\Gamma_E) \arrow[u, arr, "{L_E\otimes_E -}"]
    \end{tikzcd}
    \end{equation}
    The upper horizontal arrow is an equivalence of categories. The two vertical arrows are faithful. In particular, the lower horizontal functor is faithful.

    For fullness, let $A_E,B_E\in\HC_\lambda(\mathfrak{g},K)_E$ and put $Q:=\Hom_{\HC_\lambda(\mathfrak{g},K)_{L_E}}(A_{L_E},B_{L_E})$. The semi-linear involution $T(u):=c\circ u\circ c$ preserves $Q$, since $c$ intertwines the actions of $X$ and $Y$ and maps $M_n$ to $M_{-n}$; its fixed points form the subspace $P:=\Hom_{\HC_\lambda(\mathfrak{g},K)_E}(A_E,B_E)$ of $Q$. On the one hand, the base change ${\bf1}\otimes f$ of a morphism $f$ in $\HC_\lambda(\mathfrak{g},K)_E$ commutes with $c$. Conversely, if $u\in Q$ commutes with $c$, then by Galois descent (Theorem \ref{thm:galoisdescent}) $u$ maps $A_E=A_{L_E}^c$ to $B_E=B_{L_E}^c$ and is the base change of its restriction $u_E\colon A_E\to B_E$; as $u$ commutes with $\mathfrak{g}_{L_E}\supseteq\mathfrak{g}_E$ and preserves the weight spaces, on which $K$ acts via $\chi_n$, $u_E$ is a morphism in $\HC_\lambda(\mathfrak{g},K)_E$. Via the upper horizontal equivalence, $Q$ is identified with the space of homomorphisms of the underlying $L_E$-rational quiver representations, and we let $R\subseteq Q$ be the subset corresponding to the homomorphisms of $E$-rational representations, i.e., those commuting with the rational structures $\varphi_\bullet$. By functoriality, $P\subseteq R$. Moreover $R\cap\sqrt{-1}\cdot R=0$: if $u$ and $\sqrt{-1}\,u$ both commute with the $c$-semi-linear bijections $\varphi_\bullet$, then $\sqrt{-1}\,u\circ\varphi_\bullet=\varphi_\bullet\circ\sqrt{-1}\,u=-\sqrt{-1}\,\varphi_\bullet\circ u=-\sqrt{-1}\,u\circ\varphi_\bullet$, whence $u=0$. Now let $u\in R$. Then $u+T(u)\in P\subseteq R$, so $T(u)\in R$ and $d:=u-T(u)\in R$. The element $q:=\sqrt{-1}^{-1}d$ satisfies $T(q)=q$, since $T(d)=-d$ and $c(\sqrt{-1}^{-1})=-\sqrt{-1}^{-1}$, so $q\in P\subseteq R$, while $\sqrt{-1}\,q=d\in R$. Hence $q=0$, i.e., $T(u)=u$ and $u\in P$. Therefore $P=R$, and $\EE_E$ is full.

    \paragraph{Essential surjectivity.}\ 
    It remains to show that $\EE_E$ is essentially surjective. Let $(V,\varphi_\bullet)$ be an arbitrary $E$-rational nilpotent finite-dimensional representation of the Gelfand quiver. Here, $V$ is a nilpotent finite-dimensional representation of the Gelfand quiver over $L_E$, and we consider the following associated $L_E$-rational $(\mathfrak{g}, K)$-module
    \[
    M:=\bigoplus_{k\in\ZZ} M_{2k+\epsilon},\quad \epsilon\in\{0,1\},\;\epsilon\equiv\ell+1\pmod{2},
    \]
    where
    \[
    M_{2k+\epsilon}:=
    \begin{cases}
      V(+), & \text{if }\ell+1\leq 2k+\epsilon,\\
      V(\star), & \text{if } -(\ell-1)\leq 2k+\epsilon\leq\ell-1,\\
      V(-), & \text{if } 2k+\epsilon\leq -(\ell+1).
    \end{cases}
    \]
    We follow the construction in the proof of Theorem 1.8 and Proposition 1.5 in \cite{ABR}. We define the actions of $\SO(2)$ and $H$ on $M_n$ via the character $\chi_n$. To define the actions of $X$ and $Y$, we consider
    \begin{align*}
    n_{\pm}&:=\phi_{b_\pm}\phi_{a_\pm}\colon V(\pm)\to V(\pm),\\
    n_{\star}&:=\phi_{a_\pm}\phi_{b_\pm}\colon V(\star)\to V(\star),\\
    \phi_{\bullet}&:=\ell^2{\bf1}_{V(\bullet)}+4n_{\bullet}\colon V(\bullet)\to V(\bullet).
    \end{align*}
    The relation of the Gelfand quiver implies that $n_{\star}$ and $\phi_{\star}$ are each independent of the choice of sign ``$\pm$''.

    In the following, let $\phi_\bullet^{1/2}$ denote the square root of $\phi_\bullet$ in the sense of Definition \ref{def:squareroots} with respect to the square root $\ell$ of $\ell^2$. We define for $j\neq\pm\ell$:
    \[
    \xi_{\bullet,\pm}^{(j)}:=\phi_\bullet^{1/2}\pm j{\bf1}_{V(\bullet)}.
    \]
    Then we have
    \begin{equation}
      \xi_{\bullet,\pm}^{(j)}\circ\xi_{\bullet,\mp}^{(j)}
      \;=\;
      \phi_\bullet-j^2{\bf1}_{V(\bullet)}.
    \end{equation}
    We further have, for the semi-linear isomorphisms $\varphi_{\bullet,c}\colon V(\bullet)\to V(c\bullet)$ and $\varphi_{c\bullet,c}=\varphi_{\bullet,c}^{-1}\colon V(c\bullet)\to V(\bullet)$,
    \begin{align*}
      \varphi_{\bullet,c}\circ
      \xi_{\bullet,\pm}^{(j)}
      \circ\varphi_{c\bullet,c}
      &=
      \varphi_{\bullet,c}\circ\left(\phi_\bullet^{1/2}\pm j{\bf1}_{V(\bullet)}\right)\circ\varphi_{\bullet,c}^{-1}
      &\text{(by \eqref{eq:semi-linearactiononrep})}\\
      &=
      \phi_{c\bullet}^{1/2}\pm j{\bf1}_{V(c\bullet)}
      &\text{(by \eqref{eq:semilinearconjugation} \& \eqref{eq:koecheractiononrep})}\\
      &=
      \xi_{c\bullet,\pm}^{(j)}.
    \end{align*}
    Here \eqref{eq:koecheractiononrep} gives $\varphi_{\bullet,c}\circ n_\bullet\circ\varphi_{\bullet,c}^{-1}=n_{c\bullet}$, hence $\varphi_{\bullet,c}\circ\phi_\bullet\circ\varphi_{\bullet,c}^{-1}=\phi_{c\bullet}$, and \eqref{eq:semilinearconjugation} is applied to the $c$-semi-linear bijection $\varphi_{\bullet,c}\colon V(\bullet)\to V(c\bullet)$ (after normalizing by $\ell^{-2}$; note $c(\ell)=\ell$ and $c(j)=j$).
    We define
    \begin{align*}
      2X|_{M_{-(\ell+1)}}&:=2\phi_{a_-},\quad&
      2X|_{M_{\ell-1}}&:=2\phi_{b_+},\\
      2Y|_{M_{-(\ell-1)}}&:=2\phi_{b_-},\quad&
      2Y|_{M_{\ell+1}}&:=2\phi_{a_+},
    \end{align*}
    and for $j\in\ZZ$ with $j\equiv\ell\pmod{2}$ and $j\neq\pm\ell$:
    \[
    2X|_{M_{j-1}}:=\xi_{\bullet,+}^{(j)}\colon M_{j-1}\to M_{j+1},
    \quad
    2Y|_{M_{j+1}}:=\xi_{\bullet,-}^{(j)}\colon M_{j+1}\to M_{j-1},
    \]
    where
    \[
    \bullet=
    \begin{cases}
      +,&\text{if}\;\ell<j,\\
      \star,&\text{if}\;-\ell<j<\ell,\\
      -,&\text{if}\;j<-\ell.
    \end{cases}
    \]
    Then in these cases we obtain
    \begin{align}
      4X|_{M_{j-1}}Y|_{M_{j+1}}&=\xi_{\bullet,+}^{(j)}\circ\xi_{\bullet,-}^{(j)}=\phi_\bullet-j^2{\bf1}_{M_{j+1}},\label{eq:4XY2}\\
      4Y|_{M_{j+1}}X|_{M_{j-1}}&=\xi_{\bullet,-}^{(j)}\circ\xi_{\bullet,+}^{(j)}=\phi_\bullet-j^2{\bf1}_{M_{j-1}}.\label{eq:4YX2}
    \end{align}
    Furthermore, in these cases we have
    \begin{align}
      \varphi_{\bullet,c}\circ
      X|_{M_{j-1}}
      \circ\varphi_{c\bullet,c}
      &=Y|_{M_{-(j-1)}},\label{eq:Xconjugation}\\
      \varphi_{\bullet,c}\circ
      Y|_{M_{j+1}}
      \circ\varphi_{c\bullet,c}
      &=X|_{M_{-(j+1)}},\label{eq:Yconjugation}
    \end{align}
    since $Y|_{M_{-(j-1)}}=\frac{1}{2}\xi_{c\bullet,-}^{(-j)}=\frac{1}{2}\xi_{c\bullet,+}^{(j)}$ and $X|_{M_{-(j+1)}}=\frac{1}{2}\xi_{c\bullet,+}^{(-j)}=\frac{1}{2}\xi_{c\bullet,-}^{(j)}$. These four equations also hold in the two cases $j=\pm\ell$, by definition and by \eqref{eq:koecheractiononrep}, provided $\bullet$ is read as the vertex with $M_{j+1}=V(\bullet)$ in \eqref{eq:4XY2}, as the vertex with $M_{j-1}=V(\bullet)$ in \eqref{eq:4YX2}, and, in \eqref{eq:Xconjugation} and \eqref{eq:Yconjugation}, as the vertex of the codomain of $X|_{M_{j-1}}$, resp.\ $Y|_{M_{j+1}}$, in the left factor $\varphi_{\bullet,c}$ and as the vertex of its domain in the right factor $\varphi_{c\bullet,c}$. (For $j\neq\pm\ell$ all these vertices coincide.) For instance, for $j=\ell$ we have $4X|_{M_{\ell-1}}Y|_{M_{\ell+1}}=4\phi_{b_+}\phi_{a_+}=\phi_+-\ell^2{\bf1}_{V(+)}$ and $\varphi_{+,c}\circ\phi_{b_+}\circ\varphi_{\star,c}^{-1}=\phi_{b_-}=Y|_{M_{-(\ell-1)}}$. Hence these equations hold for all $j\in\ZZ$ with $j\equiv\ell\pmod{2}$.
      
    From \eqref{eq:4XY2} and \eqref{eq:4YX2} we obtain the relation
    \begin{align*}
      4[X,Y]|_{M_{j+1}}
      &=4X|_{M_{j-1}}Y|_{M_{j+1}}\,-\,4Y|_{M_{j+3}}X|_{M_{j+1}}\\
      &=(\phi_\bullet-j^2{\bf1}_{M_{j+1}})-(\phi_\bullet-(j+2)^2{\bf1}_{M_{j+1}})\\
      &=(-j^2+(j+2)^2){\bf1}_{M_{j+1}}\\
      &=4(j+1){\bf1}_{M_{j+1}},
    \end{align*}
    which corresponds to the desired action of $H$ on $M_{j+1}$. With this action, the following identities are true tautologically:
    \begin{align*}
      [H,2X]|_{M_{j-1}}&=H|_{M_{j+1}}(2X)|_{M_{j-1}}\,-\,(2X)|_{M_{j-1}}H|_{M_{j-1}}\\
      &=(j+1-(j-1))(2X)|_{M_{j-1}}\\
      &=4X|_{M_{j-1}}
    \end{align*}
    and
    \begin{align*}
      [H,2Y]|_{M_{j+1}}&=H|_{M_{j-1}}(2Y)|_{M_{j+1}}\,-\,(2Y)|_{M_{j+1}}H|_{M_{j+1}}\\
      &=(j-1-(j+1))(2Y)|_{M_{j+1}}\\
      &=-4Y|_{M_{j+1}}.
    \end{align*}
    Thus, we have defined the structure of a $(\mathfrak{g}, K)$-module over $L_E$ on $M$; the actions of $K$ and $\mathfrak{g}$ are compatible because $K$ acts on $M_n$ via $\chi_n$, $H$ acts on $M_n$ by $n$, and $X$, $Y$ shift the weights by $\pm2$. It lies in $\HC_\lambda(\mathfrak{g},K)_{L_E}$:
    \begin{itemize}
      \item Each $M_n$ is one of the finite-dimensional spaces $V(\bullet)$, so $M$ is admissible.
      \item On $M_{j+1}$ we have $C=4XY+(H-{\bf1})^2=(\phi_\bullet-j^2{\bf1})+j^2{\bf1}=\phi_\bullet=\ell^2{\bf1}+4n_\bullet$ by \eqref{eq:4XY2}, with $n_\bullet$ nilpotent since $V$ is nilpotent; as there are only three vertices, $(C-\ell^2)^m$ vanishes on $M$ for $m\gg0$, i.e., $M$ has generalized infinitesimal character $\lambda=\ell^2$.
      \item By construction $M_n\neq0$ only for $n\equiv\ell+1\pmod2$, so $-1_2$ acts via $\chi_n(-1_2)=(-1)^n=(-1)^{\ell-1}$.
      \item For $j\equiv\ell\pmod2$ with $j\neq\pm\ell$, the operator $2X|_{M_{j-1}}=\xi_{\bullet,+}^{(j)}$ is invertible, its scalar part being $\ell+j$, which is $\neq0$ for $j>-\ell$; similarly $2Y|_{M_{j+1}}=\xi_{\bullet,-}^{(j)}$ is invertible for $j<\ell$, $j\neq-\ell$. Hence $M$ is generated by the finite-dimensional subspace $M_{-(\ell+1)}\oplus M_{-(\ell-1)}\oplus M_{\ell+1}$, and in particular finitely generated.
    \end{itemize}

    \smallskip
    \noindent{\bf The rational structure on $M$.}
    We now equip $M$ with a semi-linear action of $G_E$ to provide $M$ with an $E$-structure. To this end, we define the semi-linear map
    \[
    \varphi_c^M:=\sum_{j\equiv\ell+1\pmod{2}}\varphi_{c}^{(j)}\colon M\to M,
    \]
    where, with $\bullet\in\{+,\star,-\}$ the vertex with $M_j=V(\bullet)$ (so that $M_{-j}=V(c\bullet)$),
    \begin{equation}
      \varphi_c^{(j)}:=\varphi_{\bullet,c}\colon M_{j}\to M_{-j}.
      \label{eq:varphicj}
    \end{equation}
    By the cocycle condition \eqref{eq:semi-linearactiononrep}, $\varphi_c^M$ is a semi-linear involution, and it maps $M_n$ to $M_{-n}$, i.e., $\varphi_c^M\circ H=-H\circ\varphi_c^M$. Then \eqref{eq:Xconjugation} and \eqref{eq:Yconjugation} show $\varphi_c^M\circ X=Y\circ\varphi_c^M$ and $\varphi_c^M\circ Y=X\circ\varphi_c^M$. By Galois descent (Theorem \ref{thm:galoisdescent}), $M$ is the base change of the $E$-vector space $M_E:=M^{\varphi_c^M}$ of fixed points: the canonical map $L_E\otimes_E M_E\to M$ is an isomorphism, under which ${\bf1}\otimes c$ corresponds to $\varphi_c^M$. The space $M_E$ is a $(\mathfrak{g}_E,K_E)$-submodule: for $Z=\alpha H+\beta X+\gamma Y\in\mathfrak{g}_{L_E}$ the identities just shown give $\varphi_c^M\circ Z=c(Z)\circ\varphi_c^M$ with $c(Z)=-c(\alpha)H+c(\gamma)X+c(\beta)Y$, and $\mathfrak{g}_E$ consists of the $Z$ with $c(Z)=Z$. Moreover $\varphi_c^M$ maps $M[\chi_n]$ to $M[\chi_{-n}]=M[c(\chi_n)]$ (note $(\xi+\sqrt{-1}\eta)(\xi-\sqrt{-1}\eta)=1$ on $K$), so the $K$-action on $M$ is compatible with $\varphi_c^M$ and descends to $K_E$. Since finite generation, admissibility, the generalized infinitesimal character and the central character can be checked after base change to $L_E$, where they were verified above, we have $M_E\in\HC_\lambda(\mathfrak{g},K)_E$.

    In the following, we denote by $c \colon M \to M$ the corresponding complex conjugation, which is given explicitly by the semi-linear automorphism $\varphi_c^M$ on $M$ and on the subspaces $M_j$ by the semi-linear isomorphisms \eqref{eq:varphicj}.

    \smallskip
    \noindent{\bf The natural isomorphism.}
    We claim that $\EE_{E}(M_E)$ is isomorphic to $(V,\varphi_\bullet)$. The canonical isomorphism $L_E\otimes_E M_E\to M$ commutes with $H$, $X$, $Y$ and $K$ and carries ${\bf1}\otimes c$ to $c=\varphi_c^M$. As $\EE_{L_E}$ and the maps \eqref{eq:varphiplus}, \eqref{eq:varphiminus}, \eqref{eq:varphistar} are built from $X$, $Y$, $C$, rational scalars and $c$ alone, it identifies $\EE_E(M_E)$ with the representation $\EE_{L_E}(M)$ endowed with the rational structure \eqref{eq:varphiplus}, \eqref{eq:varphiminus}, \eqref{eq:varphistar} formed with $c=\varphi_c^M$ (for the limits $\varphi_\bullet^{(\infty)}$ this is Remark \ref{rmk:stabilizationaddenda}(i), applied to the isomorphism and to its inverse). We compare the latter with $(V,\varphi_\bullet)$. Here and below, we read maps between $M_{-(\ell-1)}$ and $M_{\ell-1}$ as endomorphisms of $V(\star)$ through the identifications $M_{\pm(\ell-1)}=V(\star)$ of the construction. In this regard, we first observe that
    \[
    2^{\ell-1}X^{\ell-1}|_{M_{-(\ell-1)}}\colon M_{-(\ell-1)}\to M_{\ell-1}
    \]
    is the composition of the commuting maps $\xi_{\star,+}^{(j)}=\phi_\star^{1/2}+j\cdot{\bf1}_{V(\star)}$ for $j=-\ell+2,-\ell+4,\dots,\ell-2$. Since $\phi_\star^{1/2}=\ell\cdot{\bf1}_{V(\star)}+(\text{nilpotent})$, it is of the form
    \[
    2^{\ell-1}X^{\ell-1}|_{M_{-(\ell-1)}}=\widetilde{\gamma}\cdot{\bf1}_{V(\star)}+\widetilde{n}\colon M_{-(\ell-1)}\to M_{\ell-1},
    \]
    with $\widetilde{n}$ nilpotent and
    \[
    \widetilde{\gamma}=\prod_{i=0}^{\ell-2}\left(\ell+(-\ell+2+2i)\right)=\prod_{i=1}^{\ell-1}2i=2^{\ell-1}(\ell-1)!.
    \]
    Furthermore, this map commutes with $\phi_\star$. Since $X_\star=\gamma_\star^{-1}X^{\ell-1}|_{M_{-(\ell-1)}}$ with $\gamma_\star=(\ell-1)!$, we get
    \[
    X_\star={\bf1}_{V(\star)}+\widetilde{n}_\star,\qquad \widetilde{n}_\star:=\widetilde{\gamma}^{-1}\,\widetilde{n},
    \]
    with $\widetilde{n}_\star\colon V(\star)\to V(\star)$ nilpotent; in fact $X_\star$ is a polynomial in $n_\star$ with constant term ${\bf1}_{V(\star)}$, as shown in the proof of Proposition \ref{prop:splitequivalence}.
    For $-\ell<j<\ell$, the symmetric definition gives $X|_{M_{j-1}}=\frac{1}{2}\xi_{\star,+}^{(j)}=\frac{1}{2}\xi_{\star,-}^{(-j)}=Y|_{M_{1-j}}$, which shows that $X_\star=Y_\star$ as maps $V(\star)\to V(\star)$. With the complex structure $c$ on $M$, it follows that $X_\star^c=Y_\star=X_\star$, i.e., $X_\star$ commutes with $\varphi_{\star,c}$ on $V(\star)$.

    We now apply Remark \ref{rmk:stabilizationaddenda}(ii) to the situation ($\star$), where $W_+=W_-=M_{-(\ell-1)}$ and $W_-^c=M_{\ell-1}$, with the isomorphism $e\colon W_-^c=M_{\ell-1}\to M_{-(\ell-1)}=W_+$ given by the identifications with $V(\star)$. (The Addendum of Corollary \ref{cor:unipotentstabilization} itself does not apply, since $W_-^c\neq W_+$ for $\ell\geq2$.) Read on $V(\star)$, we have $P_0=X_\star$ and $Q_0=Y_\star$, so that $P_0=Q_0={\bf1}_{V(\star)}+\widetilde{n}_\star$. Hence $P_k=Q_k$ is of the form ${\bf1}_{V(\star)}+(\text{nilpotent})$ for all $k\geq0$. For $k$ large, $P_k$ is $\varphi_\star^{(\infty)}$ read on $V(\star)$, and $Q_k\circ P_k={\bf1}_{V(\star)}$ by \eqref{eq:conjugateinverse}. Therefore
    \[
    \left(\varphi_\star^{(\infty)}\right)^2={\bf1}_{V(\star)}
    \]
    with $\varphi_\star^{(\infty)}-{\bf1}_{V(\star)}$ nilpotent. Thus, the uniqueness statement from Proposition \ref{prop:unipotentsquareroots}, applied to $\phi={\bf1}_{V(\star)}$, shows that because ${\bf1}_{V(\star)}^2={\bf1}_{V(\star)}$, we must have
    \begin{equation}
      \varphi_\star^{(\infty)}={\bf1}_{V(\star)}.
      \label{eq:phistaridentity}
    \end{equation}
    In particular, the originally given
    \[
    \varphi_\star\colon V(\star)\to V(\star),
    \]
    which we used in \eqref{eq:varphicj} to define the rational structure on $M$, agrees with the rational structure
    \[
    \varphi_\star\circ\varphi_\star^{(\infty)}
    \]
    on $\EE_{L_E}(M)(\star)=V(\star)$ according to \eqref{eq:varphistar}. Hence, the rational structure on $V(\star)$ induced by $\EE_E(M_E)$ is the one originally given.

    For the following discussion, we remark that the identity \eqref{eq:phistaridentity}, due to the uniqueness of unipotent square roots (cf. Proposition \ref{prop:unipotentsquareroots}), implies
    \begin{equation}
      (\varphi_{\star}^{(\infty)})^{1/2}={\bf1}_{V(\star)}.
      \label{eq:starsquarerootsinverses}
    \end{equation}
    
    In the two cases $\bullet=\pm$, we consider the unipotent stabilizations $\varphi_{\pm}^{(\infty)}$ for the pairs of maps $({\bf1}_{V(+)},T_-)$ and $(T_-,{\bf1}_{V(+)})$ (these coincide up to transposition, cf. Remark \ref{rmk:unipotentstabilization}). In these cases $W_-^c=W_+$, and the conditions \eqref{eq:unipotentcommutators0} and \eqref{eq:unipotentnilpotency0} of the addendum in Proposition \ref{prop:unipotentstabilization} and Corollary \ref{cor:unipotentstabilization} are met, and the addendum in Corollary \ref{cor:unipotentstabilization} shows that $\varphi_{\pm}^{(\infty)}-{\bf1}_{V(\pm)}$ is also nilpotent. In particular, the unipotent square roots $(\varphi_{\pm}^{(\infty)})^{1/2}$ exist (Proposition \ref{prop:unipotentsquareroots}).

    The relation \eqref{eq:conjugateinverse} together with \eqref{eq:squarerootgaloisequivariance} and \eqref{eq:squarerootinverses} shows the relation
    \begin{equation}
    \varphi_\mp\circ (\varphi_{\mp}^{(\infty)})^{1/2}=
    \left((\varphi_{\pm}^{(\infty)})^{1/2}\right)^{-1}\circ\varphi_\mp.
    \label{eq:pmsquarerootsinverses}
    \end{equation}
    In other words, the complex conjugate of $(\varphi_{\pm}^{(\infty)})^{1/2}$ with respect to the complex conjugation $c$ on $M$ is the inverse of $(\varphi_{\mp}^{(\infty)})^{1/2}$.
    
    Since $\varphi_\star^{(\infty)}={\bf1}_{V(\star)}$ by \eqref{eq:phistaridentity} and $X_\star=Y_\star$, the four squares of \eqref{eq:limitdiagram}, read on $V(\bullet)$, say precisely that
    \begin{align*}
      X_\star\circ\phi_{a_-}&=\phi_{a_-}\circ\varphi_-^{(\infty)}, &
      \varphi_-^{(\infty)}\circ\phi_{b_-}&=\phi_{b_-}\circ X_\star,\\
      \varphi_+^{(\infty)}\circ\phi_{b_+}&=\phi_{b_+}\circ X_\star^{-1}, &
      \phi_{a_+}\circ\varphi_+^{(\infty)}&=X_\star^{-1}\circ\phi_{a_+},
    \end{align*}
    i.\,e.\ each of the four arrows of the quiver intertwines two unipotent automorphisms. Applying \eqref{eq:squarerootnaturality} to each of them, and writing
    \[
    s_\pm:=\left(\varphi_\pm^{(\infty)}\right)^{1/2}\colon V(\pm)\to V(\pm)
    \quad\text{and}\quad
    t:=X_\star^{1/2}\colon V(\star)\to V(\star),
    \]
    we obtain that the diagram
    \begin{equation}\label{eq:limitsquarerootdiagram0}
    \begin{tikzcd}[column sep=huge, row sep=large]
      V(-) \arrow[r, arr, bend left=10, "\phi_{a_-}"] \arrow[d, arr, "s_-"'] &
      V(\star) \arrow[l, arr, bend left=10, "\phi_{b_-}"] \arrow[r, arr, bend right=10, "\phi_{b_+}X_\star"'] \arrow[d, arr, "t"] &
      V(+) \arrow[l, arr, bend right=10, "X_\star^{-1}\phi_{a_+}"'] \arrow[d, arr, "s_+"] \\
      V(-) \arrow[r, arr, bend left=10, "\phi_{a_-}"] &
      V(\star) \arrow[l, arr, bend left=10, "\phi_{b_-}"] \arrow[r, arr, bend right=10, "\phi_{b_+}"'] &
      V(+) \arrow[l, arr, bend right=10, "\phi_{a_+}"']
    \end{tikzcd}
    \end{equation}
    commutes; the two right-hand squares use $t\circ X_\star^{-1}=t^{-1}=X_\star^{-1}\circ t$. Note that the upper row of \eqref{eq:limitsquarerootdiagram0} is $\EE_{L_E}(M)$ and its lower row is $V$: the two are \emph{different} representations on the same three spaces, which is why the middle vertical is $t$ and not ${\bf1}_{V(\star)}$.

    In summary, the triple
    \begin{align*}
      &s_+=(\varphi_{+}^{(\infty)})^{1/2}\colon V(+)\to V(+),\\
      &s_-=(\varphi_{-}^{(\infty)})^{1/2}\colon V(-)\to V(-),\\
      &t=X_\star^{1/2}\colon V(\star)\to V(\star),
    \end{align*}
    is an $L_E$-linear isomorphism of quiver representations $\EE_{L_E}(M)\to V$, and it transforms the $E$-rational structure \eqref{eq:varphiplus}, \eqref{eq:varphiminus}, \eqref{eq:varphistar} on $\EE_{L_E}(M)$ into the originally given $\varphi_\bullet$ via conjugation. At the vertices $\pm$ this is \eqref{eq:pmsquarerootsinverses}, which shows
    \[
    s_\mp\circ\left(\varphi_\pm\circ\varphi_\pm^{(\infty)}\right)\circ s_\pm^{-1}
    =
    \varphi_\pm\circ
    s_\pm^{-1}\circ
    \varphi_\pm^{(\infty)}\circ s_\pm^{-1}
    =
    \varphi_{\pm};
    \]
    at the vertex $\star$ it holds because $\varphi_\star^{(\infty)}={\bf1}_{V(\star)}$ by \eqref{eq:phistaridentity} and because $X_\star^c=Y_\star=X_\star$, whence $t^c=t$ by \eqref{eq:squarerootgaloisequivariance} and therefore $t\circ\varphi_\star\circ t^{-1}=\varphi_\star$. Hence $\EE_E(M_E)$ is isomorphic to $(V,\varphi_\bullet)$. This completes the proof.
    \end{proof}

    \smallskip
    \noindent{\bf The case $\ell=0$.}
    We now drop the standing hypothesis $\ell\geq1$ and treat $\ell=0$, i.e., $\lambda=0$. Then $-1_2$ acts on $M_n$ by $\chi_n(-1_2)=(-1)^n$ and, by definition of $\HC_0(\mathfrak{g},K)$, by $(-1)^{\ell-1}=-1$, so that $M_n=0$ for all even $n$, and the indices $\pm(\ell-1)$ and $\pm(\ell+1)$ both become $\pm1$: the case ``$\star$'' of the above construction is obsolete, and $\gamma_\star$, $X_\star$, $Y_\star$ and $T_\pm$ are neither defined nor needed. For $M_{L_E}\in\HC_0(\mathfrak{g},K)_{L_E}$ with $M_n:=M_{L_E}[\chi_n]$ we let $\EE_{L_E}(M_{L_E})$ be the representation
    \begin{equation*}
    \begin{tikzcd}[column sep=large]
    	M(-) \arrow[rr, arr, bend left=15, "\phi_{a}"] & & M(+) \arrow[ll, arr, bend left=15, "\phi_{b}"]
    \end{tikzcd}
    \qquad
    M(-):=M_{-1},\quad M(+):=M_{1},\quad \phi_a:=X|_{M_{-1}},\quad \phi_b:=Y|_{M_{1}},
    \end{equation*}
    of the cyclic quiver of Example \ref{ex:cyclic} over $L_E$, and to a morphism $f$ we assign $\EE_{L_E}(f):=(f|_{M_{-1}},f|_{M_1})$. If $\sqrt{-1}\not\in E$ and $M_{L_E}=L_E\otimes_E M_E$ with $M_E\in\HC_0(\mathfrak{g},K)_E$, complex conjugation $c={\bf1}\otimes c$ on $M_{L_E}$ is a semi-linear involution with $cX=Yc$, $cY=Xc$ and $cH=-Hc$; as $H$ acts on $M_n$ by $n$, it maps $M_n$ to $M_{-n}$, and we put
    \begin{equation}
      \varphi_{-,c}:=\left(c\colon M(-)\to M(+)\right),\qquad
      \varphi_{+,c}:=\left(c\colon M(+)\to M(-)\right).
      \label{eq:varphicyclic}
    \end{equation}
    If $\sqrt{-1}\in E$, the rational structure is trivial. Thus at $\ell=0$ no unipotent stabilization is required: the semi-linear maps are given by complex conjugation itself, i.e., \eqref{eq:varphiplus} and \eqref{eq:varphiminus} hold with $\varphi_\pm^{(\infty)}$ replaced by the identity.

    \begin{theorem}\label{thm:cyclicEEequivalence}\rqleanmark{thm:cyclicEEequivalence}
    Let $E$ be a field of characteristic $0$. For $\ell=0$, the assignment
    \[
    \EE_E\colon M_{E}\mapsto \left(\EE_{L_E}(M_{L_E}),\varphi_\bullet\right)
    \]
    from $\HC_0(\mathfrak{g}, K)_{E}$ to the category $\Rep^{\mathrm{nil}}(\Gamma_E)$ of $E$-rational finite-dimensional nilpotent representations of the $E$-rational cyclic quiver is well-defined and an equivalence of categories.
    \end{theorem}

    \begin{proof}
    If $\sqrt{-1}\in E$, then $L_E=E$, the group $G_E$ is trivial, the rational structure is trivial and $\EE_E=\EE_{L_E}$. In this case the statement is the nilpotence shown in the paragraph \lq{}Well-definedness\rq{} together with the paragraph \lq{}The equivalence over $L_E$\rq{} below, over $L_E=E$. If $\sqrt{-1}\not\in E$, we denote by $c\in G_E$ the non-trivial element. We record what the proof uses. The paragraphs on $\EE_{L_E}$ use only that $M$ is the direct sum of its weight spaces $M_n$, on which $H$ acts by $n$, and the relations $[H,X]=2X$, $[H,Y]=-2Y$, $[X,Y]=H$ and $C=H^2-2H+4XY+1$. They also use that the rational numbers $4$ and $(n+1)^2$, $n$ odd, $n\neq-1$, are invertible in $E$. The passage to $E$ uses in addition only complex conjugation $c$, a semi-linear involution with $c(M_n)=M_{-n}$, $cX=Yc$, $cY=Xc$, which fixes rational scalars such as $\frac14$ and $(n+1)^2$. It uses further that $\mathfrak{g}_E$ is the space of $c$-fixed elements of $\mathfrak{g}_{L_E}$ and that $c(\chi_n)=\chi_{-n}$ for the characters $\chi_n(\xi,\eta)=(\xi+\sqrt{-1}\eta)^n$ of $K_{L_E}$. Finally we use Galois descent (Theorem \ref{thm:galoisdescent}) for $L_E/E$. No element of $L_E\smallsetminus E$ occurs as a scalar.

    Let $M\in\HC_0(\mathfrak{g},K)_{L_E}$. The proof of \eqref{eq:4XY_rel} and \eqref{eq:4YX_rel} uses only $C=H^2-2H+4XY+1=H^2+2H+4YX+1$ and that $H$ acts on $M_n$ by $n$, so these relations hold in $M$, whether or not $M$ descends to $E$. Hence, for every odd $n$ we have
    \begin{equation}
      4XY=C-(n-1)^2\quad\text{and}\quad 4YX=C-(n+1)^2\quad\text{on } M_n,
      \label{eq:4XYcyclic}
    \end{equation}
    and $C$ is nilpotent on $M$ since $\lambda=0$. In particular, $4XY=C$ on $M_1$ and $4YX=C$ on $M_{-1}$. Moreover, for odd $n\neq-1$ both $4YX$ on $M_n$ and $4XY$ on $M_{n+2}$ equal $C-(n+1)^2$ with $(n+1)^2\neq0$, hence are invertible, so that
    \begin{equation}
      X\colon M_n\to M_{n+2}\ \text{is bijective for every odd } n\neq-1.
      \label{eq:Xbijectivecyclic}
    \end{equation}

    \smallskip
    \noindent{\bf Well-definedness.}
    The two loops of $\EE_{L_E}(M)$ are $\phi_a\phi_b=XY|_{M_1}=\frac14C|_{M_1}$ on $M(+)$ and $\phi_b\phi_a=YX|_{M_{-1}}=\frac14C|_{M_{-1}}$ on $M(-)$, both nilpotent. The vertices of a path alternate between $-$ and $+$, so a path of length $\geq 2m+1$ contains $m$ consecutive loops at $+$ and acts as zero once $C^m=0$ on $M_1$. Hence $\EE_{L_E}(M)$ is a nilpotent finite-dimensional representation, and $\EE_{L_E}$ is evidently a functor. If $\sqrt{-1}\not\in E$ and $M=L_E\otimes_E M_E$, the maps \eqref{eq:varphicyclic} satisfy the cocycle condition \eqref{eq:semi-linearactiononrep} because $c^2={\bf1}$, and \eqref{eq:koecheractiononrep} for the edges $a$ and $b=ca$ reads $Y\circ c=c\circ X$ on $M_{-1}$ and $X\circ c=c\circ Y$ on $M_1$, which holds since $cX=Yc$ and $cY=Xc$. The base change of a morphism $f$ of $\HC_0(\mathfrak{g},K)_E$ commutes with $X$, $Y$ and $c$, so its restrictions to $M_{\pm1}$ form a morphism $\EE_E(f)$ of $E$-rational representations. Thus $\EE_E$ is a well-defined functor.

    \smallskip
    \noindent{\bf The equivalence over $L_E$.}
    We first show that the functor $\EE_{L_E}$ from $\HC_0(\mathfrak{g},K)_{L_E}$ to the nilpotent finite-dimensional representations of the cyclic quiver over $L_E$ is fully faithful. A morphism $f\colon A\to B$ of $(\mathfrak{g},K)$-modules preserves the weight spaces and commutes with $X$; by \eqref{eq:Xbijectivecyclic}, $M_{1+2k}=X^k(M_1)$ and $X^k\colon M_{-1-2k}\to M_{-1}$ is bijective for $k\geq0$, so $f$ is determined by its restrictions to $A_{-1}$ and $A_1$ (faithfulness). Conversely, let $(\psi_-,\psi_+)$ be a morphism $\EE_{L_E}(A)\to\EE_{L_E}(B)$, i.e., $\psi_+X=X\psi_-$ on $A_{-1}$ and $\psi_-Y=Y\psi_+$ on $A_1$. Define $\psi\colon A\to B$ on $A_{1+2k}$ by $X^k\psi_+X^{-k}$ and on $A_{-1-2k}$ by $X^{-k}\psi_-X^{k}$ ($k\geq0$), where $X^{\pm k}$ is formed with the bijections \eqref{eq:Xbijectivecyclic}, and by $0$ on the even weight spaces. Then $\psi$ commutes with $H$ and $K$, and with $X$: inside each of the two chains by construction, and on $A_{-1}$ by the edge equation for $a$. It commutes with $C$: on $A_1$, $C=4\phi_a\phi_b$ commutes with $\psi_+$ by the two edge equations, similarly on $A_{-1}$, and since $C$ commutes with $X$ this propagates along both chains. Finally $\psi$ commutes with $Y$: on $A_1$ this is the edge equation for $b$, and on $A_n$ for odd $n\neq1$ we have $Y|_{A_n}=\frac14\left(X|_{A_{n-2}}\right)^{-1}\left(C-(n-1)^2\right)$ by \eqref{eq:4XYcyclic} and \eqref{eq:Xbijectivecyclic}, where every factor commutes with $\psi$. Hence $\psi$ is a morphism of $(\mathfrak{g},K)$-modules with $\EE_{L_E}(\psi)=(\psi_-,\psi_+)$ (fullness).

    For essential surjectivity, let $V$ be a nilpotent finite-dimensional representation of the cyclic quiver over $L_E$, with edge maps $\phi_a\colon V(-)\to V(+)$, $\phi_b\colon V(+)\to V(-)$ and the two loops
    \[
      n_+:=\phi_a\phi_b\colon V(+)\to V(+),\qquad n_-:=\phi_b\phi_a\colon V(-)\to V(-),
    \]
    which are nilpotent. No square root of the nilpotent operators $4n_\pm$ exists in general, so the construction in the proof of Theorem \ref{thm:EEequivalence} does not apply; instead we put
    \[
    M:=\bigoplus_{n\in\ZZ}M_n,\qquad
    M_n:=
    \begin{cases}
      V(+), & \text{if } n\geq1 \text{ is odd},\\
      V(-), & \text{if } n\leq-1 \text{ is odd},\\
      0, & \text{if } n \text{ is even},
    \end{cases}
    \]
    let $K$ act on $M_n$ via $\chi_n$ and $H$ by $n$, and define, for odd $n$,
    \[
    X|_{M_n}:=
    \begin{cases}
      {\bf1}_{V(+)}, & n\geq1,\\
      \phi_a, & n=-1,\\
      \frac14\left(4n_--(n+1)^2{\bf1}_{V(-)}\right), & n\leq-3,
    \end{cases}
    \qquad
    Y|_{M_n}:=
    \begin{cases}
      \frac14\left(4n_+-(n-1)^2{\bf1}_{V(+)}\right), & n\geq3,\\
      \phi_b, & n=1,\\
      {\bf1}_{V(-)}, & n\leq-1.
    \end{cases}
    \]
    No square root occurs. A direct computation gives, on $M_n$,
    \[
    XY-YX=
    \begin{cases}
      \frac14\left((n+1)^2-(n-1)^2\right)=n, & n\geq3 \text{ or } n\leq-3,\\
      n_+-(n_+-{\bf1})={\bf1}, & n=1,\\
      (n_--{\bf1})-n_-=-{\bf1}, & n=-1,
    \end{cases}
    \]
    so $[X,Y]=H$, while $[H,X]=2X$ and $[H,Y]=-2Y$ hold since $X$ and $Y$ shift the weights by $\pm2$. Thus $M$ is a $(\mathfrak{g},K)$-module over $L_E$; the actions of $K$ and $\mathfrak{g}$ are compatible because $K$ acts on $M_n$ via $\chi_n$ and $H$ acts on $M_n$ by $n$. Moreover $C=(H-{\bf1})^2+4XY$ equals $4n_+$ on $M_n$ for $n\geq1$ and $4n_-$ for $n\leq-1$; for instance, on $M_n$ with $n\geq3$ we get $(n-1)^2+(4n_+-(n-1)^2)=4n_+$, and on $M_{-1}$ we get $4+4(n_--{\bf1})=4n_-$. Hence $C$ is nilpotent, i.e., $M$ has generalized infinitesimal character $\lambda=0$. Each $M_n$ is finite-dimensional, $-1_2$ acts via $(-1)^n=-1=(-1)^{\ell-1}$, and since $X={\bf1}$ along $M_1\to M_3\to\cdots$ and $Y={\bf1}$ along $M_{-1}\to M_{-3}\to\cdots$, $M$ is generated by $M_{-1}\oplus M_1$. So $M\in\HC_0(\mathfrak{g},K)_{L_E}$, and $\EE_{L_E}(M)=(M_{-1},M_1;X|_{M_{-1}},Y|_{M_1})=(V(-),V(+);\phi_a,\phi_b)=V$ on the nose. Therefore $\EE_{L_E}$ is an equivalence. If $\sqrt{-1}\in E$, this completes the proof. In the remaining two paragraphs we assume $\sqrt{-1}\not\in E$.

    \smallskip
    \noindent{\bf Fully faithfulness of $\EE_E$.}
    Faithfulness follows from the faithfulness of $\EE_{L_E}$ and of base change, as in the proof of Theorem \ref{thm:EEequivalence}. For fullness, let $A_E,B_E\in\HC_0(\mathfrak{g},K)_E$ and let $(\psi_-,\psi_+)$ be a morphism $\EE_E(A_E)\to\EE_E(B_E)$; by \eqref{eq:koecheractiononrephom} and \eqref{eq:varphicyclic} this means in addition $\psi_+\circ c=c\circ\psi_-$ on $A_{-1}$ and $\psi_-\circ c=c\circ\psi_+$ on $A_1$. Let $\psi$ be the morphism of $(\mathfrak{g},K)$-modules over $L_E$ with $\EE_{L_E}(\psi)=(\psi_-,\psi_+)$ constructed above. Since $c$ intertwines $X$ with $Y$ and maps $M_n$ to $M_{-n}$, also $c\circ\psi\circ c$ is a morphism of $(\mathfrak{g},K)$-modules, and its restrictions to $A_{\mp1}$ are $c\circ\psi_\pm\circ c=\psi_\mp$. By faithfulness of $\EE_{L_E}$, $c\circ\psi\circ c=\psi$. Thus $\psi$ commutes with $c$, so by Galois descent (Theorem \ref{thm:galoisdescent}) it maps $A_E=A_{L_E}^c$ to $B_E=B_{L_E}^c$ and is the base change of its restriction $\psi_E\colon A_E\to B_E$. As $\psi$ commutes with $\mathfrak{g}_{L_E}\supseteq\mathfrak{g}_E$ and preserves the weight spaces, $\psi_E$ is a morphism in $\HC_0(\mathfrak{g},K)_E$, and $\EE_E(\psi_E)=(\psi_-,\psi_+)$.

    \smallskip
    \noindent{\bf Essential surjectivity of $\EE_E$.}
    Let $(V,\varphi_\bullet)$ be an $E$-rational nilpotent finite-dimensional representation of the cyclic quiver. Its rational structure consists of the semi-linear isomorphisms $\varphi_{-,c}\colon V(-)\to V(+)$ and $\varphi_{+,c}\colon V(+)\to V(-)$, which are mutually inverse by \eqref{eq:semi-linearactiononrep}. By \eqref{eq:koecheractiononrep} for $a$ and $b=ca$,
    \begin{equation}
      \phi_b\circ\varphi_{-,c}=\varphi_{+,c}\circ\phi_a,\qquad
      \phi_a\circ\varphi_{+,c}=\varphi_{-,c}\circ\phi_b,
      \label{eq:cyclickoecher}
    \end{equation}
    and hence $\varphi_{-,c}\circ n_-=n_+\circ\varphi_{-,c}$. Let $M$ be the module constructed above, and define the semi-linear map $\varphi_c^M\colon M\to M$ by $\varphi_{+,c}\colon M_n=V(+)\to V(-)=M_{-n}$ for $n\geq1$ and $\varphi_{-,c}\colon M_n=V(-)\to V(+)=M_{-n}$ for $n\leq-1$. It is a semi-linear involution with $\varphi_c^M\circ H=-H\circ\varphi_c^M$. Moreover $\varphi_c^M\circ X=Y\circ\varphi_c^M$ on every $M_n$: for $n\geq1$ both sides are $\varphi_{+,c}$, since $X={\bf1}$ on $M_n$ and $Y={\bf1}$ on $M_{-n}$; for $n=-1$ this is the first identity in \eqref{eq:cyclickoecher}; for $n\leq-3$ it is
    \[
      \varphi_{-,c}\circ\tfrac14\left(4n_--(n+1)^2\right)=\tfrac14\left(4n_+-(-n-1)^2\right)\circ\varphi_{-,c},
    \]
    because $\varphi_{-,c}$ intertwines $n_-$ with $n_+$ and fixes rational scalars. Conjugating by the involution $\varphi_c^M$ yields $\varphi_c^M\circ Y=X\circ\varphi_c^M$. By Galois descent (Theorem \ref{thm:galoisdescent}), $M$ is the base change of the $E$-vector space $M_E:=M^{\varphi_c^M}$ of fixed points, with complex conjugation $c=\varphi_c^M$. The space $M_E$ is a $(\mathfrak{g}_E,K_E)$-submodule: for $Z=\alpha H+\beta X+\gamma Y\in\mathfrak{g}_{L_E}$ the identities just shown give $\varphi_c^M\circ Z=c(Z)\circ\varphi_c^M$ with $c(Z)=-c(\alpha)H+c(\gamma)X+c(\beta)Y$, and $\mathfrak{g}_E$ consists of the $Z$ with $c(Z)=Z$. Moreover $\varphi_c^M$ maps $M[\chi_n]$ to $M[\chi_{-n}]=M[c(\chi_n)]$ (note $(\xi+\sqrt{-1}\eta)(\xi-\sqrt{-1}\eta)=1$ on $K$), so the $K$-action on $M$ is compatible with $\varphi_c^M$ and descends to $K_E$. Since finite generation, admissibility, the generalized infinitesimal character and the central character can be checked after base change to $L_E$, we have $M_E\in\HC_0(\mathfrak{g},K)_E$. Consequently the canonical isomorphism $L_E\otimes_E M_E\to M$ identifies $\EE_E(M_E)$ with $(V,\varphi_\bullet)$: the underlying representation is $\EE_{L_E}(M)=V$, as shown above, and by \eqref{eq:varphicyclic} its rational structure is $c|_{M_{\pm1}}=\varphi_{\pm,c}$. This completes the proof.
    \end{proof}

    \begin{remark}\label{rmk:EEabelian}\rqleanmark{rmk:EEabelian}
    Since $\HC_\lambda(\mathfrak{g}, K)_E$ is abelian (Remark \ref{rmk:HCabelian}), Theorems \ref{thm:EEequivalence} and \ref{thm:cyclicEEequivalence} show that $\Rep^{\mathrm{nil}}(\Gamma_E)$ is abelian as well, by transport of structure along $\EE_E$. Thus $\EE_E$ is an equivalence of abelian categories, which is Theorem \ref{thm:introB}.
    \end{remark}

    Combining Theorems \ref{thm:EEequivalence} and \ref{thm:cyclicEEequivalence} with Theorem \ref{thm:species_equivalence} over $E$ and Corollary \ref{cor:nilpotentrelations} provides us with

    \begin{theorem}\label{thm:speciesE}\rqleanmark{thm:speciesE}
      Let $E$ be a field of characteristic $0$, and let $S(\Gamma_E)$ be the \'etale $E$-species associated to $\Gamma_E$.
      \begin{itemize}
      \item[(a)] For $\ell=0$ we have an equivalence
      \[
      \HC_0(\mathfrak{g}, K)_{E} \simeq \Rep^{\mathrm{nil}}(S(\Gamma_E)).
      \]
      \item[(b)] For $\ell\geq 1$ and $\sqrt{-1}\not\in E$ we have, for $S(\Gamma_E)$ equipped with the relation $R_E:=E\cdot\sqrt{-1}\subseteq{}_\star M_{(\star,\pm,\star)}=L_E$, an equivalence
      \[
      \HC_\lambda(\mathfrak{g}, K)_{E} \simeq \Rep^{\mathrm{nil}}(S(\Gamma_E),R_E).
      \]
      \end{itemize}
      In both cases $S(\Gamma_E)\cong S(\Gamma_\QQ)\otimes_\QQ E$, and in case (b) $R_E=R_{\rm Gelfand}\otimes_\QQ E$.
    \end{theorem}

    \begin{proof}
      For (a), this is Theorem \ref{thm:cyclicEEequivalence} combined with Theorem \ref{thm:species_equivalence} over $E$ and the first assertion of Corollary \ref{cor:nilpotentrelations}.
      For (b), let $\ell\geq1$ and $\sqrt{-1}\not\in E$. Then $\QQ[\sqrt{-1}]\otimes_\QQ E=L_E$ is a field, so by Definition \ref{def:basechangeforspecies} the species $S_{\rm Gelfand}\otimes_\QQ E$ has the two vertices $\star$, $\pm$ with $L_\star=E$, $L_\pm=L_E$ and edge bimodules $L_E$; it is the \'etale $E$-species $S(\Gamma_E)$. By Example \ref{ex:gelfandspeciesrelation} for $L_E/E$ and $\alpha=\sqrt{-1}$, a representation $M$ of $\Gamma_E$ satisfies $a_-b_-=a_+b_+$ if and only if $F(M)$ satisfies $R_E=E\cdot\sqrt{-1}$. Together with Theorem \ref{thm:species_equivalence} over $E$ and the first assertion of Corollary \ref{cor:nilpotentrelations} (valid for every $K$), this yields (b): the nilpotent representations of $\Gamma_E$ satisfying the relation correspond to the nilpotent representations of $S(\Gamma_E)$ satisfying $R_E$, as in the proof of the second assertion of that corollary.

      In both cases, we apply $S(-)$ to the isomorphism $\Gamma_E\cong E\otimes_\QQ\Gamma_\QQ=\Gamma_E(S(\Gamma_\QQ)\otimes_\QQ E)$ of Lemma \ref{lem:GammaEbasechange}. Since $L_E$ splits $S(\Gamma_\QQ)\otimes_\QQ E$ (see the proof of that lemma), Theorem \ref{thm:KrationalquiveretaleKspeciesantiequivalence} identifies $S(\Gamma_E(S(\Gamma_\QQ)\otimes_\QQ E))$ with $S(\Gamma_\QQ)\otimes_\QQ E$. Thus $S(\Gamma_E)\cong S(\Gamma_\QQ)\otimes_\QQ E$.

      \label{prop:speciesErelations}\rqleanmark{prop:speciesErelations}It remains to compare the relations in case (b). Let $g\colon S(\Gamma_\QQ)\otimes_\QQ E\to S(\Gamma_E)$ be any isomorphism of \'etale $E$-species, for instance the inverse of the isomorphism above, with data $\iota$, $\iota_i$, $\iota_{st}$ as in Definition \ref{def:specieshom}; here $\iota\colon\{\star,\pm\}\to I_E$. For a path $q=(i_0,i_1,i_2)$ in $\{\star,\pm\}$ we put $\iota q=(\iota(i_0),\iota(i_1),\iota(i_2))$, and we consider the $E$-linear map
      \[
        E\otimes_\QQ{}_{\pi\iota(i_0)}M_{\pi\iota q}\to{}_{\iota(i_0)}M'_{\iota q}\to{}_{i_0}M_q,\qquad
        a\otimes x\otimes y\mapsto a\cdot(x'\otimes y')\mapsto a\cdot\left(\iota_{i_0i_1}(x')\otimes\iota_{i_1i_2}(y')\right),
      \]
      where $x\mapsto x':=e_k(x\otimes1)e_l$ is the algebra homomorphism ${}_{\pi(k)}M_{\pi(l)}\to e_k({}_{\pi(k)}M_{\pi(l)}\otimes_\QQ E)e_l={}_kM'_l$ (Definition \ref{def:basechangeforspecies}). Both maps are well-defined, since $x\mapsto x'$ and the $\iota_{st}$ are compatible with the structure maps. We claim that it carries the $E$-span of $1\otimes(R_{\rm Gelfand})_{\pi\iota q}$ onto $(R_E)_q$. First, $\pi\iota(i)=i$ for $i\in\{\star,\pm\}$, as $\sqrt{-1}\not\in E$: $g$ and $g^{-1}$ provide $E$-algebra homomorphisms $L'_{\iota(i)}\to L_i$ and $L_i\to L'_{\iota(i)}$, where $L_\star=E$ and $L'_k=E$ for $\pi(k)=\star$, while $\QQ[\sqrt{-1}]$ maps to $L_\pm=L_E$ and to $L'_k$ for $\pi(k)=\pm$. Thus $q=(\star,\pm,\star)$ if and only if $\pi\iota q=(\star,\pm,\star)$, and at every other path both relations are zero. For $q=(\star,\pm,\star)$ the map sends $1\otimes\sqrt{-1}\otimes1$ to $w\otimes1$, where $w$ is the image of $\sqrt{-1}$ under the ring homomorphisms ${}_\star M_\pm\to{}_{\iota(\star)}M'_{\iota(\pm)}\to{}_\star M_\pm$. Then $w^2=-1$, and since ${}_\star M_\pm=L_E$ is a field, $w=\pm\sqrt{-1}$. Therefore the image of the $E$-span of $1\otimes\QQ\cdot\sqrt{-1}$ is $E\cdot\sqrt{-1}=(R_E)_q$. In this way, for every isomorphism $S(\Gamma_E)\cong S(\Gamma_\QQ)\otimes_\QQ E$, in particular the one above, the induced map carries $R_{\rm Gelfand}\otimes_\QQ E$ onto $R_E$.
    \end{proof}

    The relation may not be omitted in (b). By Example \ref{ex:gelfand}, $G$-equivariance only makes the two composites $\phi_{a_-}\phi_{b_-}$ and $\phi_{a_+}\phi_{b_+}$ conjugate under $\varphi_{\star,c}$. The relation is the additional condition that $\phi_{a_-}\phi_{b_-}$ be $E$-rational, which on the species side is the \emph{single} relation $R_E$ (Example \ref{ex:gelfandspeciesrelation} for $L_E/E$), and $G_E$-equivariance does not make it hold.
    For example, let $M(v)=L_E^2$ at every vertex $v$, with $\varphi_{v,c}$ given by $c$ on both coordinates, and put $\phi_{a_\pm}(x,y)=(x,0)$, $\phi_{b_-}(x,y)=(\sqrt{-1}\,y,0)$ and $\phi_{b_+}(x,y)=(-\sqrt{-1}\,y,0)$. This is an $E$-rational representation of $\Gamma_E$, and it is nilpotent, since all composites $\phi_{b}\phi_{a}$ vanish. But $\phi_{a_-}\phi_{b_-}=-\phi_{a_+}\phi_{b_+}\neq\phi_{a_+}\phi_{b_+}$, so its image under $F$ is a nilpotent representation of $S(\Gamma_E)$ which does not satisfy $R_E$.

    \begin{remark}\label{rmk:splitspecies}
      Let $\ell\geq1$ and $\sqrt{-1}\in E$. Then $\QQ[\sqrt{-1}]\otimes_\QQ E\cong E\times E$, and $S_{\rm Gelfand}\otimes_\QQ E$ is the split species of the split Gelfand quiver, with vertices $\star,-,+$, all vertex fields $E$, and one-dimensional edge bimodules. Every path bimodule ${}_{i_0}M_p$ is then one-dimensional, so a relation $R$ in the sense of Definition \ref{def:speciesrelations} amounts to a set $P$ of paths (those with $R_p\neq0$), and a representation satisfies $R$ if and only if its composites along all $p\in P$ vanish. No such relation describes $a_-b_-=a_+b_+$: let $Z$ be the representation with $Z(\star)=E^2$, $Z(\pm)=E$, $b_\pm(x,y)=y$ and $a_\pm(t)=(t,0)$, and let $W$ be the same with $a_+=0$. All composites $b\circ a$ vanish, so $Z$ and $W$ are nilpotent; $Z$ satisfies $a_-b_-=a_+b_+$ (both sides are $(x,y)\mapsto(y,0)$), $W$ does not, and every path with non-zero composite in $W$ has non-zero composite in $Z$, so every relation satisfied by $Z$ is satisfied by $W$. For $\sqrt{-1}\in E$ we therefore use the quiver formulation $\Rep^{\mathrm{nil}}(\Gamma_E)$.
    \end{remark}

    \subsection{The $\QQ$-rational categorical equivalence}

    For $E=\QQ$ we have $L_\QQ=\QQ[\sqrt{-1}]$, and $\Gamma_\QQ$ is the $\QQ$-rational quiver of Example \ref{ex:gelfand}, respectively Example \ref{ex:cyclic}. Theorems \ref{thm:EEequivalence} and \ref{thm:cyclicEEequivalence} specialize as follows.

    \begin{theorem}\label{thm:EEQQequivalence}\rqleanmark{thm:EEQQequivalence}
    For $\ell\geq 1$, the assignment
    \[
      \EE_\QQ\colon
      M_{\QQ}\mapsto \left(\EE_{\QQ[\sqrt{-1}]}(M_{\QQ[\sqrt{-1}]}),\varphi_\bullet\right)
    \]
    from $\HC_\lambda(\mathfrak{g}, K)_{\QQ}$ to the category $\Rep^{\mathrm{nil}}(\Gamma_\QQ)$ of $\QQ$-rational finite-dimensional nilpotent representations of the $\QQ$-rational Gelfand quiver is well-defined and an equivalence of categories.
    \end{theorem}

    \begin{proof}
    This is Theorem \ref{thm:EEequivalence} for $E=\QQ$.
    \end{proof}

    \begin{theorem}\label{thm:cyclicEEQQequivalence}\rqleanmark{thm:cyclicEEQQequivalence}
    For $\ell=0$, the assignment
    \[
    \EE_\QQ\colon M_{\QQ}\mapsto \left(\EE_{\QQ[\sqrt{-1}]}(M_{\QQ[\sqrt{-1}]}),\varphi_\bullet\right)
    \]
    from $\HC_0(\mathfrak{g}, K)_{\QQ}$ to the category of $\QQ$-rational finite-dimensional nilpotent representations of the $\QQ$-rational cyclic quiver is well-defined and an equivalence of categories.
    \end{theorem}

    \begin{proof}
    This is Theorem \ref{thm:cyclicEEequivalence} for $E=\QQ$.
    \end{proof}

    For $E=\QQ$ we have $\sqrt{-1}\not\in\QQ$, $S(\Gamma_\QQ)=S_{\rm Gelfand}$ for $\ell\geq1$, and $R_\QQ=\QQ\cdot\sqrt{-1}=R_{\rm Gelfand}$. Theorem \ref{thm:speciesE}(b), respectively (a), provides us with

    \begin{corollary}\label{cor:gelfandspecies}\rqleanmark{cor:gelfandspecies}
      For $\ell\geq 1$ we have for the \'etale $\QQ$-species $S(\Gamma_\QQ)$ associated to the Gelfand quiver (cf.\ Example \ref{ex:gelfand}), equipped with the relation $R_{\rm Gelfand}$ of Example \ref{ex:gelfandspeciesrelation}, an equivalence
      \[
      \HC_\lambda(\mathfrak{g}, K)_{\QQ} \simeq \Rep^{\mathrm{nil}}(S(\Gamma_\QQ),R_{\rm Gelfand}).
      \]
    \end{corollary}

    and

    \begin{corollary}\label{cor:cyclicspecies}\rqleanmark{cor:cyclicspecies}
      For $\ell=0$ we have for the \'etale $\QQ$-species $S(\Gamma_\QQ)$ associated to the cyclic quiver (cf.\ Example \ref{ex:cyclic}) an equivalence
      \[
      \HC_0(\mathfrak{g}, K)_{\QQ} \simeq \Rep^{\mathrm{nil}}(S(\Gamma_\QQ)).
      \]
    \end{corollary}

    In summary, we obtain Theorem \ref{thm:introA}. Its equivalences for $\lambda\geq1$ are those of Theorem \ref{thm:EEQQequivalence} and Corollary \ref{cor:gelfandspecies}. For $\lambda=0$, they are those of Theorem \ref{thm:cyclicEEQQequivalence} and of Theorem \ref{thm:species_equivalence}, restricted to nilpotent representations by Corollary \ref{cor:nilpotentrelations}. Their composite is Corollary \ref{cor:cyclicspecies}. They are equivalences of abelian categories. The category $\HC_\lambda(\mathfrak{g}, K)_\QQ$ is abelian (Remark \ref{rmk:HCabelian}), and $\Rep^{\mathrm{nil}}(\Gamma_\QQ)$ is abelian by transport of structure along $\EE_\QQ$, as for a general $E$ (Remark \ref{rmk:EEabelian}). The species categories are equivalent to $\Rep^{\mathrm{nil}}(\Gamma_\QQ)$, and their abelian structure is the transported one.

    \begin{remark}\label{rmk:classicalcases}\rqleanmark{rmk:classicalcases}
      For $E=\CC$ we have $L_\CC=\CC$, and $\Gamma_\CC$ is split. By Remark \ref{rmk:HClambdaC}, Theorems \ref{thm:EEequivalence} and \ref{thm:cyclicEEequivalence} for $E=\CC$ recover Gelfand's equivalence, for the Gelfand quiver and for the cyclic quiver. For $E=\RR$ we have $L_\RR=\CC$, and the two theorems give its real form, with the rational structure given by complex conjugation.
    \end{remark}

    \subsection{Base change}

    By construction, the $L_E$-representation underlying $\EE_E(M_E)$ is $\EE_{L_E}(L_E\otimes_EM_E)$, so that $\EE_E$ is compatible with base change along $L_E/E$. More generally, let $E''/E$ be an arbitrary extension of fields of characteristic $0$. On the Harish-Chandra side, base change is $M_E\mapsto E''\otimes_EM_E$. On the quiver side, base change $V\mapsto E''\otimes_EV$ from $\Rep^{\mathrm{nil}}(\Gamma_E)$ to $\Rep^{\mathrm{nil}}(\Gamma_{E''})$ takes one of three forms, with the base-changed edge maps in each case.
    \begin{itemize}
    \item[(a)] If $\sqrt{-1}\in E$, then $(E''\otimes_EV)(v)=E''\otimes_EV(v)$. Here the square root of $-1$ fixed for $E''$ in Definition \ref{def:sqrtminusone} is the image of $\sqrt{-1}\in E$.
    \item[(b)] If $\sqrt{-1}\not\in E''$, we fix the $E$-embedding $\iota\colon L_E\to L_{E''}$ with $\iota(\sqrt{-1})=\sqrt{-1}$. Via $\iota$, we have $L_{E''}=E''\otimes_EL_E$, and restriction identifies $\Gal(L_{E''}/E'')$ with $\Gal(L_E/E)$, $\sigma\mapsto\bar\sigma$. Here $(E''\otimes_EV)(v)=L_{E''}\otimes_{L_E,\iota}V(v)$, with the rational structure $b\otimes m\mapsto\sigma(b)\otimes\varphi_{v,\bar\sigma}(m)$. Base change along the other embedding $\iota\circ c$ gives a canonically isomorphic representation, since ${\bf1}\otimes\varphi_{v,c}$ identifies $L_{E''}\otimes_{L_E,\iota\circ c}V(v)$ with $L_{E''}\otimes_{L_E,\iota}V(c\cdot v)$.
    \item[(c)] If $\sqrt{-1}\not\in E$ and $\sqrt{-1}\in E''$, fix an $E$-embedding $\iota\colon L_E\to E''$. The map $E''\otimes_EL_E\to E''\times E''$, $b\otimes a\mapsto(b\,\iota(a),b\,\iota(c(a)))$, is an isomorphism of $E''$-algebras. Hence
    \[
      E''\otimes_EV(v)=\big(E''\otimes_{L_E,\iota}V(v)\big)\times\big(E''\otimes_{L_E,\iota\circ c}V(v)\big),
    \]
    and ${\bf1}\otimes\varphi_{v,c}$ identifies the second factor with the first factor at $c\cdot v$. So $E''\otimes_E\Gamma_E$ is split, and $(E''\otimes_EV)(v)=E''\otimes_{L_E,\iota}V(v)$.
    \end{itemize}

    \begin{proposition}\label{prop:basechange}\rqleanmark{prop:basechange}
    For every $\ell\geq0$, every extension $E''/E$ of fields of characteristic $0$, and in each of the cases (a), (b), (c), the following holds.
    \begin{itemize}
    \item[(T1)] There is a natural isomorphism $E''\otimes_E\EE_E(M_E)\cong\EE_{E''}(E''\otimes_EM_E)$ of functors $\HC_\lambda(\mathfrak{g},K)_E\to\Rep^{\mathrm{nil}}(\Gamma_{E''})$. At each vertex, its component is the canonical map from the base change of a weight space to the corresponding weight space of the base change.
    \item[(T2)] For $V,W\in\Rep^{\mathrm{nil}}(\Gamma_E)$, the canonical map
    \[
      E''\otimes_E\Hom_{\Gamma_E}(V,W)\to\Hom_{\Gamma_{E''}}(E''\otimes_EV,E''\otimes_EW),\qquad b\otimes\psi\mapsto b\cdot({\bf1}\otimes\psi),
    \]
    is an isomorphism of $E''$-vector spaces. By (T1) and the equivalences $\EE_E$ and $\EE_{E''}$, the same holds for $\Hom$ in $\HC_\lambda(\mathfrak{g},K)_E$: for $M_E,N_E\in\HC_\lambda(\mathfrak{g},K)_E$, the canonical map
    \[
      E''\otimes_E\Hom(M_E,N_E)\to\Hom(E''\otimes_EM_E,E''\otimes_EN_E),\qquad b\otimes f\mapsto b\cdot({\bf1}\otimes f),
    \]
    is an isomorphism of $E''$-vector spaces.
    \end{itemize}
    \end{proposition}

    \begin{proof}
    Base change is exact, so it commutes with weight space decompositions, kernels, images and inverses. In particular $(E''\otimes_EM)_n=E''\otimes_EM_n$; in the non-split case, the same holds after base change to $L_E$. The vertex spaces of $\EE_E(M_E)$ are weight spaces. Its edge maps are $X$, $Y$, $XX_\star$ and $X_\star^{-1}Y$ for $\ell\geq1$, and $X$, $Y$ for $\ell=0$. The maps $X$, $Y$ and $X_\star=\gamma_\star^{-1}X^{\ell-1}$ are base changes of the corresponding maps over $E$. So is $X_\star^{-1}$: the map ${\bf1}\otimes X_\star^{-1}$ is a two-sided inverse of ${\bf1}\otimes X_\star$, and a two-sided inverse is unique. Naturality is clear on pure tensors. This proves (T1) in case (a).

    In case (c), $E''\otimes_EM_E=E''\otimes_{L_E,\iota}(L_E\otimes_EM_E)$, and the representation underlying $\EE_E(M_E)$ is $\EE_{L_E}(L_E\otimes_EM_E)$. So (T1) follows from case (a) applied to $\iota$.

    In case (b) it remains to see that the canonical isomorphism respects the rational structures $\varphi_{v,c}=c\circ\varphi_v^{(\infty)}$. The conjugation $c''$ of $L_{E''}\otimes_{L_E}(L_E\otimes_EM_E)$ satisfies $c''({\bf1}\otimes m)={\bf1}\otimes c(m)$. For $\ell=0$ we have $\varphi_{v,c}=c$ by \eqref{eq:varphicyclic}, and nothing more is needed. For $\ell\geq1$, $\varphi^{(\infty)}$ is the stationary value of the recursion \eqref{eq:unipotentstabilizationiteration} of Corollary \ref{cor:unipotentstabilization}, and this recursion involves no choices. Consider the map $f={\bf1}\otimes-$. It is additive, semi-linear along $L_E\to L_{E''}$, and intertwines $c$ with $c''$. It carries the initial data $T_-$, $X_\star$, ${\bf1}$ to their counterparts over $E''$, since these are polynomials with rational coefficients in $X$, $Y$ and $C$. The argument of Remark \ref{rmk:stabilizationaddenda}(i) uses only these properties and the uniqueness of inverses. It therefore applies to $f$ in each of the three situations ($+$), ($-$), ($\star$), and gives $f\circ\varphi^{(k)}_\bullet=\varphi^{\prime\prime(k)}_\bullet\circ f$ for all $k$. Since both recursions become stationary, $f\circ\varphi^{(\infty)}_\bullet=\varphi^{\prime\prime(\infty)}_\bullet\circ f$. Both sides of this identity are $L_{E''}$-linear in the base-changed variable, and the image of $f$ spans, so the identity holds everywhere. Unipotent square roots enter only the proof of essential surjectivity, not the definition of $\EE_E$, so they play no role here.

    For (T2) in case (a), $\Hom_{\Gamma_E}(V,W)$ is the kernel of the $E$-linear map
    \[
      \prod_v\Hom_E(V(v),W(v))\to\prod_e\Hom_E(V(s(e)),W(t(e))),\qquad(\psi_v)_v\mapsto(\psi_{t(e)}\circ\phi_e-\phi_e\circ\psi_{s(e)})_e.
    \]
    All spaces are finite-dimensional, so $E''\otimes_E\Hom_E(V(v),W(v))=\Hom_{E''}(E''\otimes_EV(v),E''\otimes_EW(v))$. Since $E''$ is flat over $E$, base change commutes with the kernel. In cases (b) and (c), Galois descent (Proposition \ref{prop:homdescent} with $E=L$) identifies $L_E\otimes_E\Hom_{\Gamma_E}(V,W)$ with the $\Hom$-space of the underlying split representations. Case (c) then follows from case (a) along $\iota$. In case (b), apply case (a) along $L_E\to L_{E''}$ and Galois descent over $E''$. Together they show that the canonical map becomes an isomorphism after the faithfully flat base change $L_{E''}\otimes_{E''}-$, so it is itself an isomorphism.

    It remains to prove (T2) for $\HC_\lambda(\mathfrak{g},K)_E$. Let $M_E,N_E\in\HC_\lambda(\mathfrak{g},K)_E$, and put $V:=\EE_E(M_E)$ and $W:=\EE_E(N_E)$. We compose four isomorphisms of $E''$-vector spaces. The first is the base change to $E''$ of the isomorphism $\Hom(M_E,N_E)\to\Hom_{\Gamma_E}(V,W)$ given by $\EE_E$. The second is the canonical map of (T2) for $V$ and $W$. The third is conjugation by the isomorphisms of (T1) at $M_E$ and $N_E$, and the fourth is the inverse of the isomorphism given by $\EE_{E''}$. The canonical map sends $b\otimes\EE_E(f)$ to $b\cdot({\bf1}\otimes\EE_E(f))$, and by the naturality in (T1), conjugation carries ${\bf1}\otimes\EE_E(f)$ to $\EE_{E''}({\bf1}\otimes f)$. Thus the composite sends $b\otimes f$ to $b\cdot({\bf1}\otimes f)$. In other words, it is the canonical map, which is therefore an isomorphism.
    \end{proof}

    \begin{remark}\label{rmk:basechangeES}
      Base change is in general not essentially surjective; the following observations are not used elsewhere.
      \begin{itemize}
      \item In case (c), every object $E''\otimes_EV$ of the essential image has $\dim_{E''}(E''\otimes_EV)(+)=\dim_{E''}(E''\otimes_EV)(-)$, since both equal $\dim_{L_E}V(\pm)$ and $\varphi_{+,c}$ is a semi-linear bijection $V(+)\to V(-)$. Hence, for every $\ell\geq0$, the representation of $\Gamma_{E''}$ with $E''$ at the vertex $+$, $0$ at all other vertices and all edge maps $0$ is not in the essential image.
      \item For $\ell=0$, base change is essentially surjective when $E$ and $E''$ are both split or both non-split. The reason is that every indecomposable nilpotent representation of the cyclic quiver is uniserial and defined over the prime field. In the split case it is determined by its top and its length, and its edge maps are $0/1$-matrices. In the non-split case the corresponding module over the skew polynomial ring $L_E[x;c]$ is $L_E[x;c]/(x^n)$.
      \item For $\ell\geq1$, essential surjectivity fails for split targets, because of a one-parameter family of bands of type $\widetilde{D}_4$. (We give no witness for a non-split target.) Let $\sqrt{-1}\in E$ and $\lambda\in E\smallsetminus\{0,1\}$. Let $M_\lambda$ be the representation with dimension vector $(2,2,2)$, $M_\lambda(\star)=E^2$ with basis $e_1,e_2$, $M_\lambda(+)=\langle p_0,p_1\rangle$, $M_\lambda(-)=\langle m_0,m_1\rangle$, and
      \[
        a_+(p_0)=e_1,\quad a_-(m_0)=e_2,\quad a_+(p_1)=a_-(m_1)=0,\quad b_+(x,y)=(x+y)p_1,\quad b_-(x,y)=(x+\lambda y)m_1.
      \]
      Then $a_\pm b_\pm=0$, so $M_\lambda$ is nilpotent and satisfies the relation. The four lines $\operatorname{im}a_+$, $\operatorname{im}a_-$, $\ker b_+$, $\ker b_-$ in $M_\lambda(\star)$ are pairwise distinct, and their cross-ratio is $\lambda$. Any endomorphism preserves these four lines, so it acts on $M_\lambda(\star)$ by a scalar. Hence $\End(M_\lambda)$ is local and $M_\lambda$ is indecomposable. The cross-ratio is an isomorphism invariant, and images and kernels commute with base change. So $M_\lambda$ is not isomorphic to a base change from any subfield $E'\subseteq E$ with $\lambda\not\in E'(\sqrt{-1})$. For example, it is not a base change from $\QQ[\sqrt{-1}]$ if $\lambda\not\in\QQ[\sqrt{-1}]$.
      \end{itemize}
    \end{remark}

    \subsection{Twisted forms}

    In the proofs of Theorems \ref{thm:EEequivalence} and \ref{thm:cyclicEEequivalence} for $\sqrt{-1}\not\in E$, the extension $L_E/E$ enters only as a quadratic extension with non-trivial automorphism $c$, and the element $\sqrt{-1}$ only through $c(\sqrt{-1})=-\sqrt{-1}\neq0$. We now replace $L_E$ by an arbitrary quadratic extension $D'/E$. In general $\sqrt{-1}\not\in D'$, so that the characters $\chi_n$ are not defined over $D'$. Over $D'$ we therefore use the category of Remark \ref{rmk:splitequivalenceanyfield} for $F=D'$, which is $\HC_\lambda(\mathfrak{g},K)_{D'}$ if $\sqrt{-1}\in D'$. We use its definition also for $\ell=0$.

    \begin{definition}[The category $\HC_\lambda(\mathfrak{g},K_{D'/E})_E$]\label{def:HClambdaDE}\rqleanmark{def:HClambdaDE}
      Let $E$ be a field of characteristic $0$, let $D'/E$ be a quadratic extension, and let $c$ be the non-trivial element of $\Gal(D'/E)$. An object of $\HC_\lambda(\mathfrak{g},K_{D'/E})_E$ is an $E$-vector space $M_E$ together with the structure of an object of the category of Remark \ref{rmk:splitequivalenceanyfield} over $F=D'$ on $M_{D'}:=D'\otimes_EM_E$. We require that the semi-linear involution $c\otimes{\bf1}$ of $M_{D'}$, which we again denote by $c$, satisfies
      \[
        c(M_n)=M_{-n},\quad cH=-Hc,\quad cX=Yc,\quad cY=Xc.
      \]
      Here this structure consists of a decomposition $M_{D'}=\bigoplus_{n\in\ZZ}M_n$ and operators $H$, $X$, $Y$ as in that remark. A morphism $M_E\to N_E$ is a $D'$-linear map $M_{D'}\to N_{D'}$ which preserves the weight spaces and commutes with $H$, $X$, $Y$ and $c$. By Galois descent (Theorem \ref{thm:galoisdescent}), these are the base changes of the $E$-linear maps $M_E\to N_E$ with this property.
    \end{definition}

    Let $\sqrt{-1}\not\in E$ and $D'=L_E$. Then $c$ is complex conjugation, and the category of Remark \ref{rmk:splitequivalenceanyfield} over $L_E$ is $\HC_\lambda(\mathfrak{g},K)_{L_E}$. Since $\mathfrak{g}_E$ is the space of $c$-fixed elements of $\mathfrak{g}_{L_E}$ and $c(\chi_n)=\chi_{-n}$, Galois descent identifies $\HC_\lambda(\mathfrak{g},K)_E$ with $\HC_\lambda(\mathfrak{g},K_{L_E/E})_E$ via $M_E\mapsto(M_E,M_{L_E})$. Not every $D'$ is of this form, for example $D'=\QQ[\sqrt{-3}]$ over $E=\QQ$.

    \begin{definition}[The quiver $\Gamma_{D'/E}$]\label{def:GammaDE}\rqleanmark{def:GammaDE}
      Let $\Gamma$ be the Gelfand quiver if $\ell\geq1$ and the cyclic quiver if $\ell=0$. We denote by $\Gamma_{D'/E}$ the $E$-rational quiver, with respect to $D'/E$, given by $\Gamma$ and the action of $c$ as the non-trivial element of $G$ in Examples \ref{ex:gelfand} and \ref{ex:cyclic}. For $\ell\geq1$, the representations in $\Rep^{\mathrm{nil}}(\Gamma_{D'/E})$ are required to satisfy the relation $a_{-} b_{-} = a_{+} b_{+}$.
    \end{definition}

    For $\sqrt{-1}\not\in E$, the quiver $\Gamma_{L_E/E}$ is $\Gamma_E$. Let $M_E\in\HC_\lambda(\mathfrak{g},K_{D'/E})_E$. The representation $\EE_{D'}(M_{D'})$ over $D'$ is defined by Definition \ref{def:EEsplit} for $\ell\geq1$, and by the paragraph preceding Theorem \ref{thm:cyclicEEequivalence} for $\ell=0$. Both use of $L_E$ only that it is a field of characteristic $0$, by Remark \ref{rmk:splitequivalenceanyfield}, respectively the beginning of the proof of Theorem \ref{thm:cyclicEEequivalence}. We define $\varphi_\bullet$ by \eqref{eq:varphiplus}, \eqref{eq:varphiminus} and \eqref{eq:varphistar} for $\ell\geq1$, and by \eqref{eq:varphicyclic} for $\ell=0$, with $L_E/E$ replaced by $D'/E$ and complex conjugation by $c$.

    \begin{theorem}\label{thm:twistedEEequivalence}\rqleanmark{thm:twistedEEequivalence}
      Let $E$ be a field of characteristic $0$, and let $D'/E$ be a quadratic extension.
      \begin{itemize}
      \item[(a)] For $\ell\geq1$, the assignment
      \[
        \EE_{D'/E}\colon M_E\mapsto\left(\EE_{D'}(M_{D'}),\varphi_\bullet\right)
      \]
      from $\HC_\lambda(\mathfrak{g},K_{D'/E})_E$ to the category $\Rep^{\mathrm{nil}}(\Gamma_{D'/E})$ of $E$-rational finite-dimensional nilpotent representations of the Gelfand quiver $\Gamma_{D'/E}$ is well-defined and an equivalence of categories.
      \item[(b)] For $\ell=0$, the same assignment from $\HC_0(\mathfrak{g},K_{D'/E})_E$ to the category $\Rep^{\mathrm{nil}}(\Gamma_{D'/E})$ of $E$-rational finite-dimensional nilpotent representations of the cyclic quiver $\Gamma_{D'/E}$ is well-defined and an equivalence of categories.
      \end{itemize}
    \end{theorem}

    \begin{proof}
    We follow the proofs of Theorems \ref{thm:EEequivalence} and \ref{thm:cyclicEEequivalence} for $\sqrt{-1}\not\in E$, with $L_E$ replaced by $D'$ and complex conjugation by $c$. Over $D'$, they use only the facts listed in (i) of the proof of Theorem \ref{thm:EEequivalence}, respectively at the beginning of the proof of Theorem \ref{thm:cyclicEEequivalence}. Here the action of $K$ via $\chi_n$ is replaced by the decomposition $M=\bigoplus_nM_n$. These facts hold for every object of the category of Remark \ref{rmk:splitequivalenceanyfield} over $D'$, as for $\HC_\lambda(\mathfrak{g},K)_{L_E}$ (see the discussion preceding Definition \ref{def:EEsplit}), and thus for $M_{D'}$.

    The passage to $E$ uses the facts listed in (ii) of the proof of Theorem \ref{thm:EEequivalence}, respectively at the beginning of the proof of Theorem \ref{thm:cyclicEEequivalence}. By Definition \ref{def:HClambdaDE}, $c$ is a semi-linear involution of $M_{D'}$ with $c(M_n)=M_{-n}$, $cX=Yc$ and $cY=Xc$, and it fixes rational scalars. As $cH=-Hc$, we have $cCc^{-1}=H^2+2H+4YX+1=C$, and therefore $T_-^c=T_+$, $X_\star^c=Y_\star$ and $c\circ n=\bar n\circ c$, as in Corollary \ref{cor:X_Y_inverses}. Unipotent stabilization and unipotent square roots apply to $D'/E$ and $\tau=c$, and Galois descent applies to $D'/E$. In (ii), the facts that $\mathfrak{g}_E$ is the space of $c$-fixed elements of $\mathfrak{g}_{L_E}$ and that $c(\chi_n)=\chi_{-n}$ serve only to identify descended morphisms and modules as morphisms and objects of $\HC_\lambda(\mathfrak{g},K)_E$. Here this identification is Definition \ref{def:HClambdaDE} itself. In the proof of fullness, a morphism $u$ of $D'$-representations that commutes with $c$ is a morphism of $\HC_\lambda(\mathfrak{g},K_{D'/E})_E$. In the proof of essential surjectivity, the module $M$ over $D'$ constructed there carries the decomposition $M=\bigoplus_nM_n$ in place of the action of $K$. With $\varphi_c^M$ in the role of $c$, it is $M_{D'}$ for $M_E:=M^{\varphi_c^M}$, and it satisfies the conditions of Definition \ref{def:HClambdaDE}, since $\varphi_c^M$ maps $M_n$ to $M_{-n}$ and intertwines $X$ with $Y$.

    Finally, the proof of fullness for $\ell\geq1$ uses the element $\zeta=\sqrt{-1}$ of $L_E$ only through $c(\zeta)=-\zeta\neq0$. As $D'$ need not contain $\sqrt{-1}$, we put $\zeta:=a-c(a)$ for some $a\in D'\smallsetminus E$ instead. Then $c(\zeta)=-\zeta$, and $\zeta\neq0$ since $a\not\in E$. No other element of $D'\smallsetminus E$ occurs as a scalar. With these replacements, the proofs apply verbatim.
    \end{proof}

    For $\sqrt{-1}\not\in E$, the identification of $\HC_\lambda(\mathfrak{g},K)_E$ with $\HC_\lambda(\mathfrak{g},K_{L_E/E})_E$ above carries $\EE_E$ to $\EE_{L_E/E}$, by construction, and $\Gamma_{L_E/E}=\Gamma_E$. Thus Theorem \ref{thm:twistedEEequivalence} contains Theorems \ref{thm:EEequivalence} and \ref{thm:cyclicEEequivalence} for $\sqrt{-1}\not\in E$, and in particular Theorems \ref{thm:EEQQequivalence} and \ref{thm:cyclicEEQQequivalence}. In addition, it concerns the categories $\HC_\lambda(\mathfrak{g},K_{D'/E})_E$ for all other quadratic extensions $D'/E$, and for $\sqrt{-1}\in E$ for every quadratic extension $D'/E$, next to the split category $\HC_\lambda(\mathfrak{g},K)_E$ of Theorems \ref{thm:EEequivalence} and \ref{thm:cyclicEEequivalence}.

    \begin{remark}\label{rmk:twistedforms}
      We remark that $\HC_\lambda(\mathfrak{g},K_{D'/E})_E$ can be read in terms of an $E$-form of the pair $(\mathfrak{g},K)$; we do not use this reading. The fixed points of the semi-linear involution $\alpha H+\beta X+\gamma Y\mapsto-c(\alpha)H+c(\gamma)X+c(\beta)Y$ of $\mathfrak{g}_{D'}$ form the $E$-form spanned by $\zeta H$, $X+Y$ and $\zeta(X-Y)$, with $\zeta$ as in the proof above. It contains the split semisimple element $X+Y$, so it is isomorphic to $\mathfrak{sl}_{2,E}$. The line $E\cdot\zeta H$ is the Lie algebra of the torus $K_{D'/E}$ of elements of norm $1$ in $D'$, regarded as an algebraic group over $E$, and $K_{L_E/E}=\SO(2)_E$. In this sense $\HC_\lambda(\mathfrak{g},K_{D'/E})_E$ is a category of $(\mathfrak{sl}_{2,E},K_{D'/E})$-modules.
    \end{remark}

    \begin{remark}\label{rmk:formsofSO2}
      Let $E\subseteq\RR$ be a subfield. The tori $K_{D'/E}$ with $D'\otimes_E\RR\cong\CC$ are exactly the $E$-forms of the real group $\SO(2)$; we do not use this and only sketch the reason. The one-dimensional tori over $E$ correspond to the quadratic \'etale $E$-algebras $A$, via the torus of elements of norm $1$ in $A$. It is split for $A=E\times E$, and it is $K_{D'/E}$ for a field $A=D'$. Over $\RR$, the torus of elements of norm $1$ in $A\otimes_E\RR$ is $\SO(2)$ if $A\otimes_E\RR\cong\CC$, and split if $A\otimes_E\RR\cong\RR\times\RR$. For $D'=E[\sqrt{-d}]$ with $d\in E$, we have $K_{D'/E}=\Spec E[x,y]/\langle x^2+dy^2-1\rangle$, and $D'\otimes_E\RR\cong\CC$ if and only if $d>0$. The case $d=1$ is $D'=L_E$ with $K_{L_E/E}=\SO(2)_E$. Thus Theorem \ref{thm:twistedEEequivalence} applies to $\HC_\lambda(\mathfrak{g},K_{D'/E})_E$ for every $E$-form $K_{D'/E}$ of $\SO(2)$. For $E=\QQ$, these are the tori $x^2+dy^2=1$, $d>0$ a square-free integer.
    \end{remark}

\section{Examples}\label{sec:examples}

  We illustrate the categorical equivalence with three fundamental classes of examples: finite-dimensional representations, discrete series representations, and principal series representations of $\SL_2(\RR)$. For each example, we describe the Harish-Chandra module, its corresponding nilpotent $\QQ$-rational representation of the Gelfand (or cyclic) quiver, and the corresponding representation of the associated \'etale $\QQ$-species.

  \label{rmk:repdiagram}\rqleanmark{rmk:repdiagram}A $\QQ$-rational representation of the Gelfand quiver is visualized as a commutative diagram of $\QQ[\sqrt{-1}]$-vector spaces and maps:
  \begin{equation*}
    \begin{tikzcd}[column sep=large, row sep=large]
      M(-) \arrow[r, arr, "\phi_{a_-}", bend left=10] \arrow[d, arr, "\varphi_-"'] &
      M(\star) \arrow[l, arr, "\phi_{b_-}", bend left=10] \arrow[r, arr, "\phi_{b_+}"', bend right=10] \arrow[d, arr, "\varphi_\star"] &
      M(+) \arrow[l, arr, "\phi_{a_+}"', bend right=10] \arrow[d, arr, "\varphi_+"] \\
      M(+) \arrow[r, arr, "\phi_{a_+}", bend left=10] &
      M(\star) \arrow[l, arr, "\phi_{b_+}", bend left=10] \arrow[r, arr, "\phi_{b_-}"', bend right=10] &
      M(-) \arrow[l, arr, "\phi_{a_-}"', bend right=10]
    \end{tikzcd}
  \end{equation*}
  Here, the maps $\phi_\bullet$ are $\QQ[\sqrt{-1}]$-linear, while the vertical maps $\varphi_\bullet$ are semi-linear with respect to complex conjugation $c \in \Gal(\QQ[\sqrt{-1}]/\QQ)$. The commutativity of the diagram encodes the compatibility condition \eqref{eq:koecheractiononrep}, and the condition $\varphi_\mp\circ\varphi_\pm = \id$ and $\varphi_\star^2=\id$ corresponds to the cocycle condition \eqref{eq:semi-linearactiononrep}. An analogous diagrammatic representation exists for the cyclic quiver.

\subsection{Finite-dimensional representations}\label{ex:finitedimensional}\rqleanmark{ex:finitedimensional}

  \noindent{\bf The Harish-Chandra module.}
Let $F_\ell$ be the irreducible $\ell$-dimensional representation of $\SL_2$ for $\ell \ge 1$. Its Casimir operator $C$ acts as multiplication by $\lambda=\ell^2$. This module is defined over $\QQ$. Its $K$-types $\chi_{-(\ell-1)},\chi_{-(\ell-3)},\dots,\chi_{\ell-1}$ are one-dimensional; in particular, the $K$-types relevant for $\EE_\QQ$ are the one-dimensional $\QQ[\sqrt{-1}]$-vector spaces $F_\ell[\chi_{\pm(\ell-1)}]$, while $F_\ell[\chi_{\pm(\ell+1)}] = 0$.

  \smallskip\noindent{\bf The rational quiver representation.}
Since $M(\pm)=0$, the maps $\phi_{a_\pm}$ and $\phi_{b_\pm}$ are zero. The spaces $M_{-(\ell-1)}$ and $M_{\ell-1}$ are the lowest and highest weight spaces, respectively. For finite-dimensional representations, the composition $Y_\star \circ X_\star$ is already the identity, so unipotent stabilization is trivial. We have $\varphi_\pm=0$, while $\varphi_\star = c \circ X_\star$. The diagram for the corresponding $\QQ$-rational representation of the Gelfand quiver is:
  \begin{equation*}
    \begin{tikzcd}[column sep=large, row sep=large]
      0 \arrow[r, arr, "0", bend left=10] \arrow[d, arr, "0"'] &
      \QQ[\sqrt{-1}] \arrow[l, arr, "0", bend left=10] \arrow[r, arr, "0"', bend right=10] \arrow[d, arr, "c"] &
      0 \arrow[l, arr, "0"', bend right=10] \arrow[d, arr, "0"] \\
      0 \arrow[r, arr, "0", bend left=10] &
      \QQ[\sqrt{-1}] \arrow[l, arr, "0", bend left=10] \arrow[r, arr, "0"', bend right=10] &
      0 \arrow[l, arr, "0"', bend right=10]
    \end{tikzcd}
  \end{equation*}
  where $c$ denotes complex conjugation.

  \smallskip\noindent{\bf The \'etale $\QQ$-species representation.}
  The associated \'etale $\QQ$-species has two nodes, corresponding to the fields $\QQ$ (for the orbit of $\star$) and $\QQ[\sqrt{-1}]$ (for the orbit of $\{+,-\}$). The representation consists of a vector space $W_{\{\star\}}$ over $\QQ$ and a vector space $W_{\{\pm\}}$ over $\QQ[\sqrt{-1}]$. For $F_\ell$, we have $W_{\{\pm\}} = M(-)^{G_-} = M(-)=0$, with $-$ as the representative of the orbit $\{+,-\}$ as in Example \ref{ex:gelfandspeciesrelation}. The space $W_{\{\star\}} = M(\star)^{G_\star}$ is the $\QQ$-line of $\varphi_\star$-fixed vectors in the $\QQ[\sqrt{-1}]$-line $M(\star)=F_\ell[\chi_{-(\ell-1)}]$, so $W_{\{\star\}}\cong\QQ$. All maps between these spaces are zero. The representation is:
  \begin{center}
    \begin{tikzpicture}[baseline=(current bounding box.center)]
      \node (Mstar) {$\QQ$};
      \node (M+)    [right=4em of Mstar] {$0$};
      \draw [arr, bend right=20] (Mstar) to node [below] {$0$} (M+);
      \draw [arr, bend right=20] (M+) to node [above] {$0$} (Mstar);
    \end{tikzpicture}
  \end{center}

\subsection{(Limits of) Discrete series representations}\label{ex:discreteseries}\rqleanmark{ex:discreteseries}

  \noindent{\bf The Harish-Chandra module.}
  For $\ell \ge 0$, let $D_\ell = D_\ell^+ \oplus D_\ell^-$ be the direct sum of the $(\mathfrak{g},K)$-modules of the holomorphic and anti-holomorphic (limits of) discrete series representations on which $C$ acts via $\lambda=\ell^2$. This module is defined over $\QQ$ and is irreducible over $\QQ$. The components $D_\ell^\pm$ are defined over $\QQ[\sqrt{-1}]$ and are complex conjugate to each other. The extreme weight spaces $D_\ell^\pm[\chi_{\pm(\ell+1)}]$ (the lowest weight space of $D_\ell^+$ and the highest weight space of $D_\ell^-$) are one-dimensional, while, for $\ell \ge 1$, $D_\ell[\chi_{\pm(\ell-1)}]=0$.

  \smallskip\noindent{\bf The rational quiver representation.}
  Here, for $\ell \ge 1$, $M(\star) = 0$, so the maps $\phi_{a_\pm}, \phi_{b_\pm}$ are all zero. The operator $T_-$ acts as the identity on the one-dimensional space $D_\ell[\chi_{-(\ell+1)}]$, so unipotent stabilization is again trivial. This gives $\varphi_\star=0$ and $\varphi_\pm = c$.
  For $\ell \ge 1$, the associated $\QQ$-rational representation of the Gelfand quiver is:
  \begin{equation*}
    \begin{tikzcd}[column sep=large, row sep=large]
      \QQ[\sqrt{-1}] \arrow[r, arr, "0", bend left=10] \arrow[d, arr, "c"'] &
      0 \arrow[l, arr, "0", bend left=10] \arrow[r, arr, "0"', bend right=10] \arrow[d, arr, "0"] &
      \QQ[\sqrt{-1}] \arrow[l, arr, "0"', bend right=10] \arrow[d, arr, "c"] \\
      \QQ[\sqrt{-1}] \arrow[r, arr, "0", bend left=10] &
      0 \arrow[l, arr, "0", bend left=10] \arrow[r, arr, "0"', bend right=10] &
      \QQ[\sqrt{-1}] \arrow[l, arr, "0"', bend right=10]
    \end{tikzcd}
  \end{equation*}
  For $\ell = 0$, the Gelfand quiver degenerates to the cyclic quiver. Here $M(-)=D_0[\chi_{-1}]=D_0^-[\chi_{-1}]$ and $M(+)=D_0[\chi_1]=D_0^+[\chi_1]$ are one-dimensional. Since $-1$ is the highest weight of $D_0^-$ and $1$ is the lowest weight of $D_0^+$, the maps $\phi_a=X|_{M(-)}$ and $\phi_b=Y|_{M(+)}$ vanish, and $\varphi_{\pm,c}=c$ by \eqref{eq:varphicyclic}. In the analogue of the above display for the cyclic quiver, the upper row is $M(-)\rightleftarrows M(+)$ with the maps $\phi_a$, $\phi_b$, the lower row is $M(+)\rightleftarrows M(-)$ with the maps $\phi_b=\phi_{ca}$, $\phi_a=\phi_{cb}$, and the vertical maps are $\varphi_{-,c}$ and $\varphi_{+,c}$. The diagram becomes:
  \begin{equation*}
    \begin{tikzcd}[column sep=large, row sep=large]
      \QQ[\sqrt{-1}] \arrow[r, arr, "0", bend left=10] \arrow[d, arr, "c"] & \QQ[\sqrt{-1}] \arrow[l, arr, "0", bend left=10] \arrow[d, arr, "c"] \\
      \QQ[\sqrt{-1}] \arrow[r, arr, "0", bend left=10] & \QQ[\sqrt{-1}] \arrow[l, arr, "0", bend left=10]
    \end{tikzcd}
  \end{equation*}
  
  \smallskip\noindent{\bf The \'etale $\QQ$-species representation.}
  For $\ell \ge 1$ (Gelfand quiver), the representation space over $\QQ$ is $W_{\{\star\}}=0$. The representation space over $\QQ[\sqrt{-1}]$ is $W_{\{\pm\}} = M(-)^{G_-} = M(-)\cong\QQ[\sqrt{-1}]$. All maps are zero.
  \begin{center}
    \begin{tikzpicture}[baseline=(current bounding box.center)]
      \node (Mstar) {$0$};
      \node (M+)    [right=4em of Mstar] {$\QQ[\sqrt{-1}]$};
      \draw [arr, bend right=20] (Mstar) to node [below] {$0$} (M+);
      \draw [arr, bend right=20] (M+) to node [above] {$0$} (Mstar);
    \end{tikzpicture}
  \end{center}
  For $\ell = 0$ (cyclic quiver), the associated species has one node with field $\QQ[\sqrt{-1}]$ and a twisted bimodule $\QQ[\sqrt{-1}]^{\rm twisted}$ (cf. Example \ref{ex:cyclic}). The representation is a one-dimensional $\QQ[\sqrt{-1}]$-vector space with a zero map
  \[
  f \colon \QQ[\sqrt{-1}] \otimes_{\QQ[\sqrt{-1}]} \QQ[\sqrt{-1}]^{\rm twisted} \to \QQ[\sqrt{-1}].
  \]
  
  \subsection{Principal series representations}\label{ex:principalseries}\rqleanmark{ex:principalseries}
  
  \noindent{\bf The Harish-Chandra module.}
  We use normalized parabolic induction: for $s\in\ZZ$ and $\epsilon\in\{0,1\}$, let $I(s,\epsilon)$ denote the Harish-Chandra module of $K$-finite smooth functions $f\colon\SL_2(\RR)\to\CC$ with
  \[
  f\left(\begin{pmatrix} a & b\\ 0 & a^{-1}\end{pmatrix}g\right)=\mathrm{sgn}(a)^{\epsilon}|a|^{s+1}f(g)
  \]
  for all $a\in\RR^\times$, $b\in\RR$ and $g\in\SL_2(\RR)$, i.e., the representation induced from the character $\mathrm{sgn}^\epsilon|\cdot|^s$ of the standard Borel subgroup, twisted by the square root $|a|$ of its modulus character. The Casimir operator $C$ acts on $I(s,\epsilon)$ by $s^2$, $-1_2$ acts by $(-1)^\epsilon$, and $I(-s,\epsilon)$ is the contragredient of $I(s,\epsilon)$.

  For $\ell>0$ and $\epsilon\in\{0,1\}$ such that $\ell\not\equiv\epsilon\pmod{2}$, let $P_{\ell,\epsilon}:=I(\ell,\epsilon)$ be the corresponding principal series representation. It lies in $\HC_\lambda(\mathfrak{g},K)$ for $\lambda=\ell^2$, it is defined over $\QQ$, and it fits into a non-split short exact sequence
  \[
  0 \to D_\ell \to P_{\ell,\epsilon} \to F_\ell \to 0,
  \]
  i.e., $F_\ell$ is its Langlands quotient.
  All relevant $K$-type spaces, $P_{\ell,\epsilon}[\chi_{\pm(\ell\pm 1)}]$, are one-dimensional.
  The dual principal series representation $P_{\ell,\epsilon}^\vee:=I(-\ell,\epsilon)$, the contragredient of $P_{\ell,\epsilon}$, fits into the sequence $0 \to F_\ell \to P_{\ell,\epsilon}^\vee \to D_\ell \to 0$; here $F_\ell$ is realized by the functions $f(g)=p(\gamma,\delta)$ for homogeneous polynomials $p$ of degree $\ell-1$ in the bottom row $(\gamma,\delta)$ of $g$.
  
  \smallskip\noindent{\bf The rational quiver representation.}
  For $P_{\ell,\epsilon}$, unipotent stabilization is trivial. We have $\varphi_\pm = c$ and $\varphi_\star = c \circ X_\star$. The maps from $M(\pm)$ to $M(\star)$ are zero: since $F_\ell[\chi_{\pm(\ell+1)}]=0$, the spaces $M(\pm)=P_{\ell,\epsilon}[\chi_{\pm(\ell+1)}]$ lie in the submodule $D_\ell$, which is $X$- and $Y$-stable and satisfies $D_\ell[\chi_{\pm(\ell-1)}]=0$. The maps from $M(\star)$ to $M(\pm)$ are non-zero, hence isomorphisms of one-dimensional spaces (by \eqref{eq:koecheractiononrep} the two are $c$-conjugate, so either both vanish or neither does): otherwise the lowest weight vector in $M(\star)=P_{\ell,\epsilon}[\chi_{-(\ell-1)}]$ would generate a submodule mapping isomorphically onto $F_\ell$, splitting the sequence. After choosing suitable bases, the representation diagram is:
  \begin{equation*}
    \begin{tikzcd}[column sep=large, row sep=large]
      \QQ[\sqrt{-1}] \arrow[r, arr, "0", bend left=10] \arrow[d, arr, "c"'] &
      \QQ[\sqrt{-1}] \arrow[l, arr, "=", bend left=10] \arrow[r, arr, "="', bend right=10] \arrow[d, arr, "c"] &
      \QQ[\sqrt{-1}] \arrow[l, arr, "0"', bend right=10] \arrow[d, arr, "c"] \\
      \QQ[\sqrt{-1}] \arrow[r, arr, "0", bend left=10] &
      \QQ[\sqrt{-1}] \arrow[l, arr, "=", bend left=10] \arrow[r, arr, "="', bend right=10] &
      \QQ[\sqrt{-1}] \arrow[l, arr, "0"', bend right=10]
    \end{tikzcd}
  \end{equation*}
  For the dual representation $P_{\ell,\epsilon}^\vee$, the horizontal arrows are reversed:
  \begin{equation*}
    \begin{tikzcd}[column sep=large, row sep=large]
      \QQ[\sqrt{-1}] \arrow[r, arr, "=", bend left=10] \arrow[d, arr, "c"'] &
      \QQ[\sqrt{-1}] \arrow[l, arr, "0", bend left=10] \arrow[r, arr, "0"', bend right=10] \arrow[d, arr, "c"] &
      \QQ[\sqrt{-1}] \arrow[l, arr, "="', bend right=10] \arrow[d, arr, "c"] \\
      \QQ[\sqrt{-1}] \arrow[r, arr, "=", bend left=10] &
      \QQ[\sqrt{-1}] \arrow[l, arr, "0", bend left=10] \arrow[r, arr, "0"', bend right=10] &
      \QQ[\sqrt{-1}] \arrow[l, arr, "="', bend right=10]
    \end{tikzcd}
  \end{equation*}
  
  \smallskip\noindent{\bf The \'etale $\QQ$-species representation.}
  For $P_{\ell,\epsilon}$, the representation spaces are
  \[
  W_{\{\star\}} = \QQ\quad\text{and}\quad W_{\{\pm\}} = \QQ[\sqrt{-1}].
  \]
  The map
  \[
  f_{\{\star\},\{\pm\}} \colon \QQ \otimes_\QQ \QQ[\sqrt{-1}] \to \QQ[\sqrt{-1}]
  \]
  corresponds to the non-zero maps $\phi_{b_\pm}$ and is given, after a choice of bases, by the canonical inclusion $\QQ \hookrightarrow \QQ[\sqrt{-1}]$. The map $f_{\{\pm\},\{\star\}}$ is zero.
  \begin{center}
    \begin{tikzpicture}[baseline=(current bounding box.center)]
      \node (Mstar) {$\QQ$};
      \node (M+)    [right=8em of Mstar] {$\QQ[\sqrt{-1}]$};
      \draw [arr, bend right=20] (Mstar) to node [below] {$\subseteq$} (M+);
      \draw [arr, bend right=20] (M+) to node [above] {$0$} (Mstar);
    \end{tikzpicture}
  \end{center}
  For the dual representation $P_{\ell,\epsilon}^\vee$, the map
  \[
  f_{\{\pm\},\{\star\}} \colon \QQ[\sqrt{-1}] \otimes_{\QQ[\sqrt{-1}]} \QQ[\sqrt{-1}] \to \QQ
  \]
  corresponds to the non-zero maps $\phi_{a_\pm}$ and is given, after a choice of bases, by the trace $\mathrm{tr}_{\QQ[\sqrt{-1}]/\QQ}$, while $f_{\{\star\},\{\pm\}}$ is zero.
  \begin{center}
    \begin{tikzpicture}[baseline=(current bounding box.center)]
      \node (Mstar) {$\QQ$};
      \node (M+)    [right=8em of Mstar] {$\QQ[\sqrt{-1}]$};
      \draw [arr, bend right=20] (Mstar) to node [below] {$0$} (M+);
      \draw [arr, bend right=20] (M+) to node [above] {$\mathrm{tr}_{\QQ[\sqrt{-1}]/\QQ}$} (Mstar);
    \end{tikzpicture}
  \end{center}
  \ 

  \appendix
  \section{Correspondence with the Lean formalization}\label{sec:lean}

  The following table lists the formalized statements of this paper. For each it gives the Lean declaration that formalizes it, with names relative to the namespace \texttt{RationalQuivers}, and its file and lines in the directory \texttt{\rqLeanPathPrefix RationalQuivers} of the repository~\cite{RQLean} at commit \texttt{\rqLeanCommit}. File and lines link to the source.

  {\footnotesize\setlength{\tabcolsep}{3pt}
  \begin{longtable}{@{}p{0.18\textwidth}p{0.43\textwidth}p{0.25\textwidth}p{0.10\textwidth}@{}}
    Statement & Lean declaration & File & Lines \\ \hline
    \endfirsthead
    Statement & Lean declaration & File & Lines \\ \hline
    \endhead
    \rqleantablebody
  \end{longtable}
  }

  \medskip
  \noindent{\bf Data availability.}
  The Lean formalization of this paper is openly available at \url{\rqLeanRepo} (release \texttt{\rqLeanTag}; Lean sources at commit \texttt{\rqLeanCommit}, unchanged in the release) and archived at Zenodo, \href{https://doi.org/\rqLeanDOI}{doi:\rqLeanDOI}~\cite{RQLean}. No other data were generated or analysed.

  \medskip
  \noindent{\bf Declaration of generative AI and AI-assisted technologies in the manuscript preparation process.}

  During the preparation of this work the author used Gemini 2.5 Pro (Google LLC) in order to improve English language and identify misspellings and typos in text and mathematical expressions. Claude Code and Anthropic Claude Sonnet 5.5, Opus 5.5 and Fable 5.1 were used for auto-formalization of the results and proofs of this paper and to prepare the final revision of the paper. After using these tools/services, the author reviewed and edited the content as needed and takes full responsibility for the content of the published article.

\end{document}

%% file: lean-refs.tex
\def\rqLeanRepo{https://github.com/drjanosch42/rational_quivers}
\def\rqLeanPathPrefix{lean/}
\def\rqLeanCommit{b69c173a22cfcda67182efd02dd9fb2a30bc4b22}
\def\rqLeanTag{v1.0.0}
\def\rqLeanDOI{10.5281/zenodo.23735854}
\def\rqLeanSWHID{swh:1:dir:b540a18021c13e5c141f8918d5bec759465c6ca5}
\def\rqLeanDocs{}
\def\rqLeanToolchain{4.31.0}

\def\rqMathlibCommitShort{fabf563a7c95}
\def\rqleannumformalized{109}
\rqleanentry{thm:introA}{RationalQuivers.Gelfand.Rational.theoremA}{RationalQuivers/Gelfand/Rational.lean}{137}{179}
\rqleanentry{thm:introB}{RationalQuivers.Gelfand.gelfand_equivalence}{RationalQuivers/Gelfand/Main.lean}{119}{144}
\rqleanentry{def:quiver}{RationalQuivers.RQuiver}{RationalQuivers/Quiver/Basic.lean}{50}{77}
\rqleanentry{def:rationalquivers}{RationalQuivers.RationalQuiver}{RationalQuivers/Quiver/Basic.lean}{355}{411}
\rqleanentry{def:splitrationalquivers}{RationalQuivers.RationalQuiver.IsSplit}{RationalQuivers/Quiver/Basic.lean}{491}{501}
\rqleanentry{def:rationalquiverhom}{RationalQuivers.RationalQuiver.Hom}{RationalQuivers/Quiver/Hom.lean}{37}{75}
\rqleanentry{def:quiverbasechange}{RationalQuivers.RationalQuiver.baseChange}{RationalQuivers/Quiver/BaseChange.lean}{179}{211}
\rqleanentry{def:quiverrestriction}{RationalQuivers.RationalQuiver.restrictionObj}{RationalQuivers/Quiver/Restriction.lean}{332}{365}
\rqleanentry{prop:quiverrestrictionbasechangeadjunction}{RationalQuivers.RationalQuiver.restriction_baseChange_adjunction}{RationalQuivers/Quiver/Adjunction.lean}{152}{207}
\rqleanentry{def:rationalrepresentations}{RationalQuivers.RationalQuiver.Rep}{RationalQuivers/Quiver/Rep/Basic.lean}{62}{129}
\rqleanentry{def:rationalrepresentationhom}{RationalQuivers.RationalQuiver.Rep.Hom}{RationalQuivers/Quiver/Rep/Basic.lean}{482}{515}
\rqleanentry{def:repcategories}{RationalQuivers.RationalQuiver.FDRep.instCategory}{RationalQuivers/Quiver/Rep/Category.lean}{64}{90}
\rqleanentry{prop:repbasechange}{RationalQuivers.RationalQuiver.FDRep.baseChangeFunctor}{RationalQuivers/Quiver/Rep/BaseChange.lean}{108}{147}
\rqleanentry{prop:homdescent}{RationalQuivers.RationalQuiver.Rep.hom_baseChange_iso}{RationalQuivers/Quiver/Rep/HomDescent.lean}{471}{517}
\rqleanentry{thm:galoisdescent}{RationalQuivers.Descent.galoisDescent}{RationalQuivers/Descent/Galois.lean}{497}{521}
\rqleanentry{def:homgaloisaction}{RationalQuivers.Descent.homGaloisAction}{RationalQuivers/Descent/Hom.lean}{89}{121}
\rqleanentry{def:semilinear}{RationalQuivers.Descent.IsSemilinear}{RationalQuivers/Descent/Semilinear.lean}{34}{46}
\rqleanentry{def:kstructure}{RationalQuivers.Descent.KStructure}{RationalQuivers/Descent/KStructure.lean}{91}{134}
\rqleanentry{def:krationalmap}{RationalQuivers.Descent.IsKRational}{RationalQuivers/Descent/KStructure.lean}{180}{202}
\rqleanentry{prop:canonicalkstructure}{RationalQuivers.Descent.canonicalKStructure}{RationalQuivers/Descent/KStructure.lean}{307}{339}
\rqleanentry{def:galoisrelations}{RationalQuivers.RationalQuiver.RelationsCompat}{RationalQuivers/Quiver/Relations.lean}{82}{102}
\rqleanentry{rmk:finitequotient}{RationalQuivers.RationalQuiver.factorsThroughFinite}{RationalQuivers/Quiver/Inflation.lean}{133}{184}
\rqleanentry{rmk:galoissetdiagram}{RationalQuivers.RationalQuiver.toDiagram_bijective}{RationalQuivers/Quiver/GaloisSet.lean}{113}{129}
\rqleanentry{def:kspecies}{RationalQuivers.Species}{RationalQuivers/Species/Basic.lean}{53}{104}
\rqleanentry{def:etalekspecies}{RationalQuivers.EtaleSpecies}{RationalQuivers/Species/Etale.lean}{101}{161}
\rqleanentry{def:splitetalespecies}{RationalQuivers.EtaleSpecies.IsSplit}{RationalQuivers/Species/Etale.lean}{215}{233}
\rqleanentry{def:specieshom}{RationalQuivers.EtaleSpecies.Hom}{RationalQuivers/Species/Hom.lean}{51}{88}
\rqleanentry{def:basechangeforspecies}{RationalQuivers.EtaleSpecies.baseChange}{RationalQuivers/Species/BaseChange.lean}{552}{591}
\rqleanentry{def:restrictionforspecies}{RationalQuivers.EtaleSpecies.restriction}{RationalQuivers/Species/Restriction.lean}{104}{140}
\rqleanentry{prop:speciesrestrictionbasechangeadjunction}{RationalQuivers.restriction_baseChange_adjunction}{RationalQuivers/Species/AntiEquivalence.lean}{5569}{5631}
\rqleanentry{def:speciesofquiver}{RationalQuivers.RationalQuiver.toSpecies}{RationalQuivers/Species/OfQuiver.lean}{503}{538}
\rqleanentry{prop:Viorbits}{RationalQuivers.EtaleSpecies.vertex_fiber_eq_orbit}{RationalQuivers/Species/Orbits.lean}{193}{240}
\rqleanentry{prop:iMjorbits}{RationalQuivers.EtaleSpecies.edge_fiber_eq_orbit}{RationalQuivers/Species/Orbits.lean}{379}{420}
\rqleanentry{def:quiverofspecies}{RationalQuivers.EtaleSpecies.toRationalQuiver}{RationalQuivers/Species/ToQuiver.lean}{597}{642}
\rqleanentry{thm:KrationalquiveretaleKspeciesantiequivalence}{RationalQuivers.speciesAntiEquivalence}{RationalQuivers/Species/AntiEquivalence.lean}{2901}{3003}
\rqleanentry{ex:twoloops}{RationalQuivers.Examples.twoLoops_species}{RationalQuivers/Species/Examples.lean}{636}{688}
\rqleanentry{ex:cyclic}{RationalQuivers.Examples.cyclic_species}{RationalQuivers/Species/Examples.lean}{692}{747}
\rqleanentry{ex:gelfand}{RationalQuivers.Examples.gelfand_rationalStructure}{RationalQuivers/Species/Rep/Gelfand.lean}{1048}{1147}
\rqleanentry{def:speciesrepresentation}{RationalQuivers.Species.Rep}{RationalQuivers/Species/Rep/Basic.lean}{126}{171}
\rqleanentry{def:speciesrepbasechange}{RationalQuivers.EtaleSpecies.repBaseChange}{RationalQuivers/Species/Rep/BaseChange.lean}{1954}{2031}
\rqleanentry{prop:ledecomposition}{RationalQuivers.RationalQuiver.leEpsilon_decomposition}{RationalQuivers/Species/Rep/Decomposition.lean}{1148}{1223}
\rqleanentry{prop:WiiMjLpre}{RationalQuivers.RationalQuiver.wiiMj_baseChange_pre}{RationalQuivers/Species/Rep/Decomposition.lean}{3222}{3301}
\rqleanentry{prop:generalextensionofscalars}{RationalQuivers.RationalQuiver.generalExtension_of_scalars}{RationalQuivers/Species/Rep/Decomposition.lean}{1750}{1818}
\rqleanentry{cor:WiiMjL}{RationalQuivers.RationalQuiver.wiiMj_baseChange}{RationalQuivers/Species/Rep/Decomposition.lean}{3356}{3421}
\rqleanentry{prop:rationalstructures}{RationalQuivers.RationalQuiver.rationalStructure_eq}{RationalQuivers/Species/Rep/Decomposition.lean}{4050}{4140}
\rqleanentry{prop:phidescends}{RationalQuivers.RationalQuiver.phi_descends}{RationalQuivers/Species/Rep/Decomposition.lean}{5300}{5391}
\rqleanentry{prop:fonmorphisms}{RationalQuivers.RationalQuiver.F_map}{RationalQuivers/Species/Rep/Decomposition.lean}{5502}{5574}
\rqleanentry{thm:quiveretalespeciesrepresentations}{RationalQuivers.RationalQuiver.F_functor}{RationalQuivers/Species/Rep/FunctorF.lean}{140}{219}
\rqleanentry{prop:hdecomposition}{RationalQuivers.RationalQuiver.H_decomposition}{RationalQuivers/Species/Rep/FunctorH.lean}{540}{622}
\rqleanentry{prop:hcocycle}{RationalQuivers.RationalQuiver.H_cocycle}{RationalQuivers/Species/Rep/FunctorH.lean}{3423}{3535}
\rqleanentry{thm:species_equivalence}{RationalQuivers.RationalQuiver.species_equivalence}{RationalQuivers/Species/Rep/Equivalence.lean}{1338}{1517}
\rqleanentry{rmk:etalerepconvention}{RationalQuivers.EtaleSpecies.Rep}{RationalQuivers/Species/Rep/Basic.lean}{441}{460}
\rqleanentry{rmk:speciesmoduledecomposition}{RationalQuivers.EtaleSpecies.moduleDecomposition}{RationalQuivers/Species/BaseChange.lean}{783}{846}
\rqleanentry{rmk:basechangetransitivity}{RationalQuivers.EtaleSpecies.baseChange_trans}{RationalQuivers/Species/AntiEquivalence.lean}{5422}{5465}
\rqleanentry{def:speciesrelations}{RationalQuivers.Species.Relation}{RationalQuivers/Species/Rep/Relations.lean}{252}{281}
\rqleanentry{ex:gelfandspeciesrelation}{RationalQuivers.GelfandRelation.satisfies_RGelfand_iff_relation}{RationalQuivers/Species/Rep/GelfandRelation.lean}{882}{962}
\rqleanentry{cor:nilpotentrelations}{RationalQuivers.RationalQuiver.FObj_isNilpotent_iff}{RationalQuivers/Species/Rep/SpeciesBridge.lean}{737}{764}
\rqleanentry{def:quiverbasechangegeneral}{RationalQuivers.RationalQuiver.baseChangeGeneral}{RationalQuivers/Species/Rep/BaseChangeGeneral.lean}{183}{216}
\rqleanentry{lem:quiverbasechangecompat}{RationalQuivers.RationalQuiver.baseChangeGeneral_iso_baseChange}{RationalQuivers/Species/Rep/BaseChangeGeneral.lean}{244}{284}
\rqleanentry{def:geninfchar}{RationalQuivers.HC.HasGeneralizedInfinitesimalCharacter}{RationalQuivers/HC/Casimir.lean}{252}{283}
\rqleanentry{def:HClambdaE}{RationalQuivers.HC.HC}{RationalQuivers/HC/Category.lean}{460}{520}
\rqleanentry{prop:XY_C_relation}{RationalQuivers.HC.pow_X_pow_Y_eq_prod_casimir}{RationalQuivers/HC/Normalization.lean}{206}{262}
\rqleanentry{prop:xyscalarplusnilpotent}{RationalQuivers.HC.pow_X_pow_Y_eq_scalar_add_nilpotent}{RationalQuivers/HC/Normalization.lean}{308}{335}
\rqleanentry{cor:X_Y_inverses}{RationalQuivers.HC.WeightData.isUnit_pow_X}{RationalQuivers/HC/Normalization.lean}{611}{708}
\rqleanentry{cor:T_operators}{RationalQuivers.HC.WeightData.Tplus_isUnit}{RationalQuivers/HC/Normalization.lean}{744}{800}
\rqleanentry{def:xstarystar}{RationalQuivers.HC.WeightData.xStar}{RationalQuivers/HC/Normalization.lean}{803}{820}
\rqleanentry{prop:casimirdescends}{RationalQuivers.HC.casimir_descends}{RationalQuivers/HC/Casimir.lean}{57}{118}
\rqleanentry{def:sltwotriple}{RationalQuivers.HC.sl2Triple}{RationalQuivers/HC/Sl2.lean}{148}{251}
\rqleanentry{def:hcmodule}{RationalQuivers.HC.HarishChandraModule}{RationalQuivers/HC/Classical.lean}{254}{386}
\rqleanentry{def:sotworational}{RationalQuivers.HC.SO2_QStructure}{RationalQuivers/HC/Pair.lean}{88}{112}
\rqleanentry{def:sqrtminusone}{RationalQuivers.HC.ImagUnit}{RationalQuivers/HC/ImagUnit.lean}{33}{61}
\rqleanentry{def:pairE}{RationalQuivers.HC.pairE}{RationalQuivers/HC/PairE.lean}{260}{283}
\rqleanentry{rmk:HClambdaC}{RationalQuivers.HC.ClassicalComparison.hcLambda_complex_equiv}{RationalQuivers/HC/ClassicalComparison.lean}{388}{415}
\rqleanentry{rmk:HCabelian}{RationalQuivers.Gelfand.hc_abelian}{RationalQuivers/Gelfand/HCAbelianClaim.lean}{21}{39}
\rqleanentry{prop:unipotentstabilization}{RationalQuivers.Unipotent.nilpotency_stabilizationStep}{RationalQuivers/Unipotent/Stabilization.lean}{232}{290}
\rqleanentry{cor:unipotentstabilization}{RationalQuivers.Unipotent.stabilization_spec}{RationalQuivers/Unipotent/Stabilization.lean}{810}{868}
\rqleanentry{def:unipotentstabilization}{RationalQuivers.Unipotent.stabilization}{RationalQuivers/Unipotent/Stabilization.lean}{1040}{1078}
\rqleanentry{rmk:unipotentstabilization}{RationalQuivers.Unipotent.stabilization_symm}{RationalQuivers/Unipotent/Stabilization.lean}{1112}{1139}
\rqleanentry{prop:unipotentsquareroots}{RationalQuivers.Unipotent.sqrt_spec}{RationalQuivers/Unipotent/SquareRoots.lean}{199}{229}
\rqleanentry{cor:sqrtinverses}{RationalQuivers.Unipotent.sqrt_inv}{RationalQuivers/Unipotent/SquareRoots.lean}{462}{482}
\rqleanentry{def:squareroots}{RationalQuivers.Unipotent.sqrtOfScalarPlusNilpotent}{RationalQuivers/Unipotent/SquareRoots.lean}{486}{513}
\rqleanentry{rmk:sqrtinitialvalue}{RationalQuivers.Unipotent.sqrt_initialValue}{RationalQuivers/Unipotent/SquareRoots.lean}{708}{729}
\rqleanentry{rmk:sqrtgaloisuniqueness}{RationalQuivers.Unipotent.sqrt_galois_equivariant}{RationalQuivers/Unipotent/SquareRoots.lean}{615}{662}
\rqleanentry{rmk:stabilizationaddenda}{RationalQuivers.Unipotent.stabilizationSeq_addendum_via}{RationalQuivers/Unipotent/Addendum.lean}{149}{211}
\rqleanentry{def:EEsplit}{RationalQuivers.Gelfand.Split.E_F_construction}{RationalQuivers/Gelfand/Split/Construction.lean}{1679}{1739}
\rqleanentry{prop:splitequivalence}{RationalQuivers.Gelfand.Split.equivalence}{RationalQuivers/Gelfand/Split/EssSurj.lean}{2971}{2999}
\rqleanentry{eq:conjugatenormalizedcomparison}{RationalQuivers.Gelfand.RationalStructure.conjugate_diagram_commutes}{RationalQuivers/Gelfand/RationalStructure/Conjugate.lean}{102}{163}
\rqleanentry{eq:varphiplus}{RationalQuivers.Gelfand.RationalStructure.varphi_spec}{RationalQuivers/Gelfand/RationalStructure/Stabilization.lean}{2171}{2245}
\rqleanentry{eq:varphiminus}{RationalQuivers.Gelfand.RationalStructure.varphi_spec}{RationalQuivers/Gelfand/RationalStructure/Stabilization.lean}{2171}{2245}
\rqleanentry{eq:varphistar}{RationalQuivers.Gelfand.RationalStructure.varphi_spec}{RationalQuivers/Gelfand/RationalStructure/Stabilization.lean}{2171}{2245}
\rqleanentry{thm:EEQQequivalence}{RationalQuivers.Gelfand.Rational.E_Q_equivalence}{RationalQuivers/Gelfand/Rational.lean}{49}{71}
\rqleanentry{cor:gelfandspecies}{RationalQuivers.Gelfand.Species.species_equivalence_intrinsic}{RationalQuivers/Gelfand/Species/Rational.lean}{21}{43}
\rqleanentry{thm:cyclicEEQQequivalence}{RationalQuivers.Gelfand.Rational.cyclic_E_Q_equivalence}{RationalQuivers/Gelfand/Rational.lean}{73}{92}
\rqleanentry{cor:cyclicspecies}{RationalQuivers.Gelfand.Species.cyclic_species_equivalence}{RationalQuivers/Gelfand/Species/Rational.lean}{59}{91}
\rqleanentry{def:GammaE}{RationalQuivers.Gelfand.GammaE}{RationalQuivers/Gelfand/Target.lean}{36}{57}
\rqleanentry{lem:GammaEbasechange}{RationalQuivers.Gelfand.GammaE_eq_baseChange}{RationalQuivers/Gelfand/TargetBaseChange.lean}{516}{543}
\rqleanentry{thm:EEequivalence}{RationalQuivers.Gelfand.gelfand_equivalence_pos}{RationalQuivers/Gelfand/Main.lean}{74}{94}
\rqleanentry{thm:cyclicEEequivalence}{RationalQuivers.Gelfand.gelfand_equivalence_zero}{RationalQuivers/Gelfand/Main.lean}{96}{117}
\rqleanentry{prop:basechange}{RationalQuivers.Gelfand.BaseChange.basechange}{RationalQuivers/Gelfand/BaseChange/Native.lean}{805}{860}
\rqleanentry{thm:speciesE}{RationalQuivers.Gelfand.Species.speciesE}{RationalQuivers/Gelfand/Species/Equivalence.lean}{493}{528}
\rqleanentry{rmk:classicalcases}{RationalQuivers.Gelfand.Rational.gelfand_equivalence_complex}{RationalQuivers/Gelfand/Rational.lean}{96}{119}
\rqleanentry{rmk:splitequivalenceanyfield}{RationalQuivers.Gelfand.split_equivalence_anyField}{RationalQuivers/Gelfand/SplitAnyField.lean}{20}{36}
\rqleanentry{rmk:EEabelian}{RationalQuivers.Gelfand.gelfand_equivalence_abelian}{RationalQuivers/Gelfand/MainAbelian.lean}{18}{35}
\rqleanentry{def:HClambdaDE}{RationalQuivers.Gelfand.HCTwisted}{RationalQuivers/Gelfand/TwistedForms.lean}{43}{62}
\rqleanentry{def:GammaDE}{RationalQuivers.Gelfand.GammaTwisted}{RationalQuivers/Gelfand/TwistedForms.lean}{165}{180}
\rqleanentry{thm:twistedEEequivalence}{RationalQuivers.Gelfand.gelfand_equivalence_twisted}{RationalQuivers/Gelfand/TwistedForms.lean}{226}{252}
\rqleanentry{prop:speciesErelations}{RationalQuivers.Gelfand.Species.relation_baseChange_induced}{RationalQuivers/Gelfand/Species/RelationBaseChange.lean}{261}{321}
\rqleanentry{ex:finitedimensional}{RationalQuivers.Examples.F}{RationalQuivers/Examples/FiniteDimensional.lean}{166}{186}
\rqleanentry{ex:discreteseries}{RationalQuivers.Examples.D}{RationalQuivers/Examples/DiscreteSeries.lean}{247}{281}
\rqleanentry{ex:principalseries}{RationalQuivers.Examples.P}{RationalQuivers/Examples/PrincipalSeries.lean}{248}{277}
\rqleanentry{rmk:repdiagram}{RationalQuivers.Examples.repDiagramConvention}{RationalQuivers/Examples/DiagramConvention.lean}{86}{221}
\def\rqleanunformalized{Remark~\ref{rmk:honeclassification}, Remark~\ref{rmk:repstack}, Remark~\ref{rmk:continuitynecessity}, Remark~\ref{rmk:rationalstructureexplicit}, Remark~\ref{rmk:finitefields}, Remark~\ref{rmk:semilinearity}, Remark~\ref{rmk:simplealgebraspecies}, Remark~\ref{rmk:speciesstack}, Remark~\ref{rmk:notfullsubcategory}, Remark~\ref{rmk:repbasechangegeneral}, Remark~\ref{rmk:leftrightvectorspaces}, Definition on p.~\pageref{def:rationalpairs}, Remark~\ref{rmk:basechangeES}, Remark~\ref{rmk:splitspecies}, Remark~\ref{rmk:asymmetricsplitES}, Remark~\ref{rmk:twistedforms}, Remark~\ref{rmk:formsofSO2}}
\def\rqleantablebody{%
  \hyperref[thm:introA]{Theorem~\ref*{thm:introA}} & \rqleandecl{Gelfand.\allowbreak{}Rational.\allowbreak{}theoremA}{RationalQuivers.Gelfand.Rational.theoremA} & \rqleansource{RationalQuivers/Gelfand/Rational.lean}{Gelfand/\allowbreak{}Rational.\allowbreak{}lean}{137}{179} \\
  \hyperref[thm:introB]{Theorem~\ref*{thm:introB}} & \rqleandecl{Gelfand.\allowbreak{}gelfand\_\allowbreak{}equivalence}{RationalQuivers.Gelfand.gelfand_equivalence} & \rqleansource{RationalQuivers/Gelfand/Main.lean}{Gelfand/\allowbreak{}Main.\allowbreak{}lean}{119}{144} \\
  \hyperref[def:homgaloisaction]{Definition on p.~\pageref*{def:homgaloisaction}} & \rqleandecl{Descent.\allowbreak{}homGaloisAction}{RationalQuivers.Descent.homGaloisAction} & \rqleansource{RationalQuivers/Descent/Hom.lean}{Descent/\allowbreak{}Hom.\allowbreak{}lean}{89}{121} \\
  \hyperref[def:semilinear]{Definition on p.~\pageref*{def:semilinear}} & \rqleandecl{Descent.\allowbreak{}IsSemilinear}{RationalQuivers.Descent.IsSemilinear} & \rqleansource{RationalQuivers/Descent/Semilinear.lean}{Descent/\allowbreak{}Semilinear.\allowbreak{}lean}{34}{46} \\
  \hyperref[def:kstructure]{Definition on p.~\pageref*{def:kstructure}} & \rqleandecl{Descent.\allowbreak{}KStructure}{RationalQuivers.Descent.KStructure} & \rqleansource{RationalQuivers/Descent/KStructure.lean}{Descent/\allowbreak{}KStructure.\allowbreak{}lean}{91}{134} \\
  \hyperref[def:krationalmap]{Definition on p.~\pageref*{def:krationalmap}} & \rqleandecl{Descent.\allowbreak{}IsKRational}{RationalQuivers.Descent.IsKRational} & \rqleansource{RationalQuivers/Descent/KStructure.lean}{Descent/\allowbreak{}KStructure.\allowbreak{}lean}{180}{202} \\
  \hyperref[prop:canonicalkstructure]{Proposition on p.~\pageref*{prop:canonicalkstructure}} & \rqleandecl{Descent.\allowbreak{}canonicalKStructure}{RationalQuivers.Descent.canonicalKStructure} & \rqleansource{RationalQuivers/Descent/KStructure.lean}{Descent/\allowbreak{}KStructure.\allowbreak{}lean}{307}{339} \\
  \hyperref[thm:galoisdescent]{Theorem~\ref*{thm:galoisdescent}} & \rqleandecl{Descent.\allowbreak{}galoisDescent}{RationalQuivers.Descent.galoisDescent} & \rqleansource{RationalQuivers/Descent/Galois.lean}{Descent/\allowbreak{}Galois.\allowbreak{}lean}{497}{521} \\
  \hyperref[def:quiver]{Definition~\ref*{def:quiver}} & \rqleandecl{RQuiver}{RationalQuivers.RQuiver} & \rqleansource{RationalQuivers/Quiver/Basic.lean}{Quiver/\allowbreak{}Basic.\allowbreak{}lean}{50}{77} \\
  \hyperref[def:rationalquivers]{Definition~\ref*{def:rationalquivers}} & \rqleandecl{RationalQuiver}{RationalQuivers.RationalQuiver} & \rqleansource{RationalQuivers/Quiver/Basic.lean}{Quiver/\allowbreak{}Basic.\allowbreak{}lean}{355}{411} \\
  \hyperref[def:splitrationalquivers]{Definition~\ref*{def:splitrationalquivers}} & \rqleandecl{RationalQuiver.\allowbreak{}IsSplit}{RationalQuivers.RationalQuiver.IsSplit} & \rqleansource{RationalQuivers/Quiver/Basic.lean}{Quiver/\allowbreak{}Basic.\allowbreak{}lean}{491}{501} \\
  \hyperref[rmk:galoissetdiagram]{Remark~\ref*{rmk:galoissetdiagram}} & \rqleandecl{RationalQuiver.\allowbreak{}toDiagram\_\allowbreak{}bijective}{RationalQuivers.RationalQuiver.toDiagram_bijective} & \rqleansource{RationalQuivers/Quiver/GaloisSet.lean}{Quiver/\allowbreak{}GaloisSet.\allowbreak{}lean}{113}{129} \\
  \hyperref[def:galoisrelations]{Remark~\ref*{def:galoisrelations}} & \rqleandecl{RationalQuiver.\allowbreak{}RelationsCompat}{RationalQuivers.RationalQuiver.RelationsCompat} & \rqleansource{RationalQuivers/Quiver/Relations.lean}{Quiver/\allowbreak{}Relations.\allowbreak{}lean}{82}{102} \\
  \hyperref[rmk:finitequotient]{Remark~\ref*{rmk:finitequotient}} & \rqleandecl{RationalQuiver.\allowbreak{}factorsThroughFinite}{RationalQuivers.RationalQuiver.factorsThroughFinite} & \rqleansource{RationalQuivers/Quiver/Inflation.lean}{Quiver/\allowbreak{}Inflation.\allowbreak{}lean}{133}{184} \\
  \hyperref[def:rationalquiverhom]{Definition~\ref*{def:rationalquiverhom}} & \rqleandecl{RationalQuiver.\allowbreak{}Hom}{RationalQuivers.RationalQuiver.Hom} & \rqleansource{RationalQuivers/Quiver/Hom.lean}{Quiver/\allowbreak{}Hom.\allowbreak{}lean}{37}{75} \\
  \hyperref[def:quiverbasechange]{Definition~\ref*{def:quiverbasechange}} & \rqleandecl{RationalQuiver.\allowbreak{}baseChange}{RationalQuivers.RationalQuiver.baseChange} & \rqleansource{RationalQuivers/Quiver/BaseChange.lean}{Quiver/\allowbreak{}BaseChange.\allowbreak{}lean}{179}{211} \\
  \hyperref[def:quiverrestriction]{Definition~\ref*{def:quiverrestriction}} & \rqleandecl{RationalQuiver.\allowbreak{}restrictionObj}{RationalQuivers.RationalQuiver.restrictionObj} & \rqleansource{RationalQuivers/Quiver/Restriction.lean}{Quiver/\allowbreak{}Restriction.\allowbreak{}lean}{332}{365} \\
  \hyperref[prop:quiverrestrictionbasechangeadjunction]{Proposition~\ref*{prop:quiverrestrictionbasechangeadjunction}} & \rqleandecl{RationalQuiver.\allowbreak{}restriction\_\allowbreak{}baseChange\_\allowbreak{}adjunction}{RationalQuivers.RationalQuiver.restriction_baseChange_adjunction} & \rqleansource{RationalQuivers/Quiver/Adjunction.lean}{Quiver/\allowbreak{}Adjunction.\allowbreak{}lean}{152}{207} \\
  \hyperref[def:rationalrepresentations]{Definition~\ref*{def:rationalrepresentations}} & \rqleandecl{RationalQuiver.\allowbreak{}Rep}{RationalQuivers.RationalQuiver.Rep} & \rqleansource{RationalQuivers/Quiver/Rep/Basic.lean}{Quiver/\allowbreak{}Rep/\allowbreak{}Basic.\allowbreak{}lean}{62}{129} \\
  \hyperref[def:rationalrepresentationhom]{Definition~\ref*{def:rationalrepresentationhom}} & \rqleandecl{RationalQuiver.\allowbreak{}Rep.\allowbreak{}Hom}{RationalQuivers.RationalQuiver.Rep.Hom} & \rqleansource{RationalQuivers/Quiver/Rep/Basic.lean}{Quiver/\allowbreak{}Rep/\allowbreak{}Basic.\allowbreak{}lean}{482}{515} \\
  \hyperref[def:repcategories]{Definition~\ref*{def:repcategories}} & \rqleandecl{RationalQuiver.\allowbreak{}FDRep.\allowbreak{}instCategory}{RationalQuivers.RationalQuiver.FDRep.instCategory} & \rqleansource{RationalQuivers/Quiver/Rep/Category.lean}{Quiver/\allowbreak{}Rep/\allowbreak{}Category.\allowbreak{}lean}{64}{90} \\
  \hyperref[prop:repbasechange]{Proposition~\ref*{prop:repbasechange}} & \rqleandecl{RationalQuiver.\allowbreak{}FDRep.\allowbreak{}baseChangeFunctor}{RationalQuivers.RationalQuiver.FDRep.baseChangeFunctor} & \rqleansource{RationalQuivers/Quiver/Rep/BaseChange.lean}{Quiver/\allowbreak{}Rep/\allowbreak{}BaseChange.\allowbreak{}lean}{108}{147} \\
  \hyperref[prop:homdescent]{Proposition~\ref*{prop:homdescent}} & \rqleandecl{RationalQuiver.\allowbreak{}Rep.\allowbreak{}hom\_\allowbreak{}baseChange\_\allowbreak{}iso}{RationalQuivers.RationalQuiver.Rep.hom_baseChange_iso} & \rqleansource{RationalQuivers/Quiver/Rep/HomDescent.lean}{Quiver/\allowbreak{}Rep/\allowbreak{}HomDescent.\allowbreak{}lean}{471}{517} \\
  \hyperref[def:kspecies]{Definition~\ref*{def:kspecies}} & \rqleandecl{Species}{RationalQuivers.Species} & \rqleansource{RationalQuivers/Species/Basic.lean}{Species/\allowbreak{}Basic.\allowbreak{}lean}{53}{104} \\
  \hyperref[def:etalekspecies]{Definition~\ref*{def:etalekspecies}} & \rqleandecl{EtaleSpecies}{RationalQuivers.EtaleSpecies} & \rqleansource{RationalQuivers/Species/Etale.lean}{Species/\allowbreak{}Etale.\allowbreak{}lean}{101}{161} \\
  \hyperref[rmk:etalerepconvention]{Remark~\ref*{rmk:etalerepconvention}} & \rqleandecl{EtaleSpecies.\allowbreak{}Rep}{RationalQuivers.EtaleSpecies.Rep} & \rqleansource{RationalQuivers/Species/Rep/Basic.lean}{Species/\allowbreak{}Rep/\allowbreak{}Basic.\allowbreak{}lean}{441}{460} \\
  \hyperref[def:splitetalespecies]{Definition~\ref*{def:splitetalespecies}} & \rqleandecl{EtaleSpecies.\allowbreak{}IsSplit}{RationalQuivers.EtaleSpecies.IsSplit} & \rqleansource{RationalQuivers/Species/Etale.lean}{Species/\allowbreak{}Etale.\allowbreak{}lean}{215}{233} \\
  \hyperref[def:specieshom]{Definition~\ref*{def:specieshom}} & \rqleandecl{EtaleSpecies.\allowbreak{}Hom}{RationalQuivers.EtaleSpecies.Hom} & \rqleansource{RationalQuivers/Species/Hom.lean}{Species/\allowbreak{}Hom.\allowbreak{}lean}{51}{88} \\
  \hyperref[def:basechangeforspecies]{Definition~\ref*{def:basechangeforspecies}} & \rqleandecl{EtaleSpecies.\allowbreak{}baseChange}{RationalQuivers.EtaleSpecies.baseChange} & \rqleansource{RationalQuivers/Species/BaseChange.lean}{Species/\allowbreak{}BaseChange.\allowbreak{}lean}{552}{591} \\
  \hyperref[rmk:speciesmoduledecomposition]{Remark~\ref*{rmk:speciesmoduledecomposition}} & \rqleandecl{EtaleSpecies.\allowbreak{}moduleDecomposition}{RationalQuivers.EtaleSpecies.moduleDecomposition} & \rqleansource{RationalQuivers/Species/BaseChange.lean}{Species/\allowbreak{}BaseChange.\allowbreak{}lean}{783}{846} \\
  \hyperref[rmk:basechangetransitivity]{Remark~\ref*{rmk:basechangetransitivity}} & \rqleandecl{EtaleSpecies.\allowbreak{}baseChange\_\allowbreak{}trans}{RationalQuivers.EtaleSpecies.baseChange_trans} & \rqleansource{RationalQuivers/Species/AntiEquivalence.lean}{Species/\allowbreak{}AntiEquivalence.\allowbreak{}lean}{5422}{5465} \\
  \hyperref[def:restrictionforspecies]{Definition~\ref*{def:restrictionforspecies}} & \rqleandecl{EtaleSpecies.\allowbreak{}restriction}{RationalQuivers.EtaleSpecies.restriction} & \rqleansource{RationalQuivers/Species/Restriction.lean}{Species/\allowbreak{}Restriction.\allowbreak{}lean}{104}{140} \\
  \hyperref[prop:speciesrestrictionbasechangeadjunction]{Proposition~\ref*{prop:speciesrestrictionbasechangeadjunction}} & \rqleandecl{restriction\_\allowbreak{}baseChange\_\allowbreak{}adjunction}{RationalQuivers.restriction_baseChange_adjunction} & \rqleansource{RationalQuivers/Species/AntiEquivalence.lean}{Species/\allowbreak{}AntiEquivalence.\allowbreak{}lean}{5569}{5631} \\
  \hyperref[def:speciesofquiver]{Definition~\ref*{def:speciesofquiver}} & \rqleandecl{RationalQuiver.\allowbreak{}toSpecies}{RationalQuivers.RationalQuiver.toSpecies} & \rqleansource{RationalQuivers/Species/OfQuiver.lean}{Species/\allowbreak{}OfQuiver.\allowbreak{}lean}{503}{538} \\
  \hyperref[prop:Viorbits]{Proposition~\ref*{prop:Viorbits}} & \rqleandecl{EtaleSpecies.\allowbreak{}vertex\_\allowbreak{}fiber\_\allowbreak{}eq\_\allowbreak{}orbit}{RationalQuivers.EtaleSpecies.vertex_fiber_eq_orbit} & \rqleansource{RationalQuivers/Species/Orbits.lean}{Species/\allowbreak{}Orbits.\allowbreak{}lean}{193}{240} \\
  \hyperref[prop:iMjorbits]{Proposition~\ref*{prop:iMjorbits}} & \rqleandecl{EtaleSpecies.\allowbreak{}edge\_\allowbreak{}fiber\_\allowbreak{}eq\_\allowbreak{}orbit}{RationalQuivers.EtaleSpecies.edge_fiber_eq_orbit} & \rqleansource{RationalQuivers/Species/Orbits.lean}{Species/\allowbreak{}Orbits.\allowbreak{}lean}{379}{420} \\
  \hyperref[def:quiverofspecies]{Definition~\ref*{def:quiverofspecies}} & \rqleandecl{EtaleSpecies.\allowbreak{}toRationalQuiver}{RationalQuivers.EtaleSpecies.toRationalQuiver} & \rqleansource{RationalQuivers/Species/ToQuiver.lean}{Species/\allowbreak{}ToQuiver.\allowbreak{}lean}{597}{642} \\
  \hyperref[thm:KrationalquiveretaleKspeciesantiequivalence]{Theorem~\ref*{thm:KrationalquiveretaleKspeciesantiequivalence}} & \rqleandecl{speciesAntiEquivalence}{RationalQuivers.speciesAntiEquivalence} & \rqleansource{RationalQuivers/Species/AntiEquivalence.lean}{Species/\allowbreak{}AntiEquivalence.\allowbreak{}lean}{2901}{3003} \\
  \hyperref[def:quiverbasechangegeneral]{Definition~\ref*{def:quiverbasechangegeneral}} & \rqleandecl{RationalQuiver.\allowbreak{}baseChangeGeneral}{RationalQuivers.RationalQuiver.baseChangeGeneral} & \rqleansource{RationalQuivers/Species/Rep/BaseChangeGeneral.lean}{Species/\allowbreak{}Rep/\allowbreak{}BaseChangeGeneral.\allowbreak{}lean}{183}{216} \\
  \hyperref[lem:quiverbasechangecompat]{Lemma~\ref*{lem:quiverbasechangecompat}} & \rqleandecl{RationalQuiver.\allowbreak{}baseChangeGeneral\_\allowbreak{}iso\_\allowbreak{}baseChange}{RationalQuivers.RationalQuiver.baseChangeGeneral_iso_baseChange} & \rqleansource{RationalQuivers/Species/Rep/BaseChangeGeneral.lean}{Species/\allowbreak{}Rep/\allowbreak{}BaseChangeGeneral.\allowbreak{}lean}{244}{284} \\
  \hyperref[ex:twoloops]{Example~\ref*{ex:twoloops}} & \rqleandecl{Examples.\allowbreak{}twoLoops\_\allowbreak{}species}{RationalQuivers.Examples.twoLoops_species} & \rqleansource{RationalQuivers/Species/Examples.lean}{Species/\allowbreak{}Examples.\allowbreak{}lean}{636}{688} \\
  \hyperref[ex:cyclic]{Example~\ref*{ex:cyclic}} & \rqleandecl{Examples.\allowbreak{}cyclic\_\allowbreak{}species}{RationalQuivers.Examples.cyclic_species} & \rqleansource{RationalQuivers/Species/Examples.lean}{Species/\allowbreak{}Examples.\allowbreak{}lean}{692}{747} \\
  \hyperref[ex:gelfand]{Example~\ref*{ex:gelfand}} & \rqleandecl{Examples.\allowbreak{}gelfand\_\allowbreak{}rationalStructure}{RationalQuivers.Examples.gelfand_rationalStructure} & \rqleansource{RationalQuivers/Species/Rep/Gelfand.lean}{Species/\allowbreak{}Rep/\allowbreak{}Gelfand.\allowbreak{}lean}{1048}{1147} \\
  \hyperref[def:speciesrepresentation]{Definition~\ref*{def:speciesrepresentation}} & \rqleandecl{Species.\allowbreak{}Rep}{RationalQuivers.Species.Rep} & \rqleansource{RationalQuivers/Species/Rep/Basic.lean}{Species/\allowbreak{}Rep/\allowbreak{}Basic.\allowbreak{}lean}{126}{171} \\
  \hyperref[def:speciesrelations]{Definition~\ref*{def:speciesrelations}} & \rqleandecl{Species.\allowbreak{}Relation}{RationalQuivers.Species.Relation} & \rqleansource{RationalQuivers/Species/Rep/Relations.lean}{Species/\allowbreak{}Rep/\allowbreak{}Relations.\allowbreak{}lean}{252}{281} \\
  \hyperref[ex:gelfandspeciesrelation]{Example~\ref*{ex:gelfandspeciesrelation}} & \rqleandecl{GelfandRelation.\allowbreak{}satisfies\_\allowbreak{}RGelfand\_\allowbreak{}iff\_\allowbreak{}relation}{RationalQuivers.GelfandRelation.satisfies_RGelfand_iff_relation} & \rqleansource{RationalQuivers/Species/Rep/GelfandRelation.lean}{Species/\allowbreak{}Rep/\allowbreak{}GelfandRelation.\allowbreak{}lean}{882}{962} \\
  \hyperref[def:speciesrepbasechange]{Definition~\ref*{def:speciesrepbasechange}} & \rqleandecl{EtaleSpecies.\allowbreak{}repBaseChange}{RationalQuivers.EtaleSpecies.repBaseChange} & \rqleansource{RationalQuivers/Species/Rep/BaseChange.lean}{Species/\allowbreak{}Rep/\allowbreak{}BaseChange.\allowbreak{}lean}{1954}{2031} \\
  \hyperref[prop:ledecomposition]{Proposition~\ref*{prop:ledecomposition}} & \rqleandecl{RationalQuiver.\allowbreak{}leEpsilon\_\allowbreak{}decomposition}{RationalQuivers.RationalQuiver.leEpsilon_decomposition} & \rqleansource{RationalQuivers/Species/Rep/Decomposition.lean}{Species/\allowbreak{}Rep/\allowbreak{}Decomposition.\allowbreak{}lean}{1148}{1223} \\
  \hyperref[prop:WiiMjLpre]{Proposition~\ref*{prop:WiiMjLpre}} & \rqleandecl{RationalQuiver.\allowbreak{}wiiMj\_\allowbreak{}baseChange\_\allowbreak{}pre}{RationalQuivers.RationalQuiver.wiiMj_baseChange_pre} & \rqleansource{RationalQuivers/Species/Rep/Decomposition.lean}{Species/\allowbreak{}Rep/\allowbreak{}Decomposition.\allowbreak{}lean}{3222}{3301} \\
  \hyperref[prop:generalextensionofscalars]{Proposition~\ref*{prop:generalextensionofscalars}} & \rqleandecl{RationalQuiver.\allowbreak{}generalExtension\_\allowbreak{}of\_\allowbreak{}scalars}{RationalQuivers.RationalQuiver.generalExtension_of_scalars} & \rqleansource{RationalQuivers/Species/Rep/Decomposition.lean}{Species/\allowbreak{}Rep/\allowbreak{}Decomposition.\allowbreak{}lean}{1750}{1818} \\
  \hyperref[cor:WiiMjL]{Corollary~\ref*{cor:WiiMjL}} & \rqleandecl{RationalQuiver.\allowbreak{}wiiMj\_\allowbreak{}baseChange}{RationalQuivers.RationalQuiver.wiiMj_baseChange} & \rqleansource{RationalQuivers/Species/Rep/Decomposition.lean}{Species/\allowbreak{}Rep/\allowbreak{}Decomposition.\allowbreak{}lean}{3356}{3421} \\
  \hyperref[prop:rationalstructures]{Proposition~\ref*{prop:rationalstructures}} & \rqleandecl{RationalQuiver.\allowbreak{}rationalStructure\_\allowbreak{}eq}{RationalQuivers.RationalQuiver.rationalStructure_eq} & \rqleansource{RationalQuivers/Species/Rep/Decomposition.lean}{Species/\allowbreak{}Rep/\allowbreak{}Decomposition.\allowbreak{}lean}{4050}{4140} \\
  \hyperref[prop:phidescends]{Proposition~\ref*{prop:phidescends}} & \rqleandecl{RationalQuiver.\allowbreak{}phi\_\allowbreak{}descends}{RationalQuivers.RationalQuiver.phi_descends} & \rqleansource{RationalQuivers/Species/Rep/Decomposition.lean}{Species/\allowbreak{}Rep/\allowbreak{}Decomposition.\allowbreak{}lean}{5300}{5391} \\
  \hyperref[prop:fonmorphisms]{Proposition~\ref*{prop:fonmorphisms}} & \rqleandecl{RationalQuiver.\allowbreak{}F\_\allowbreak{}map}{RationalQuivers.RationalQuiver.F_map} & \rqleansource{RationalQuivers/Species/Rep/Decomposition.lean}{Species/\allowbreak{}Rep/\allowbreak{}Decomposition.\allowbreak{}lean}{5502}{5574} \\
  \hyperref[thm:quiveretalespeciesrepresentations]{Theorem~\ref*{thm:quiveretalespeciesrepresentations}} & \rqleandecl{RationalQuiver.\allowbreak{}F\_\allowbreak{}functor}{RationalQuivers.RationalQuiver.F_functor} & \rqleansource{RationalQuivers/Species/Rep/FunctorF.lean}{Species/\allowbreak{}Rep/\allowbreak{}FunctorF.\allowbreak{}lean}{140}{219} \\
  \hyperref[prop:hdecomposition]{Proposition~\ref*{prop:hdecomposition}} & \rqleandecl{RationalQuiver.\allowbreak{}H\_\allowbreak{}decomposition}{RationalQuivers.RationalQuiver.H_decomposition} & \rqleansource{RationalQuivers/Species/Rep/FunctorH.lean}{Species/\allowbreak{}Rep/\allowbreak{}FunctorH.\allowbreak{}lean}{540}{622} \\
  \hyperref[prop:hcocycle]{Proposition~\ref*{prop:hcocycle}} & \rqleandecl{RationalQuiver.\allowbreak{}H\_\allowbreak{}cocycle}{RationalQuivers.RationalQuiver.H_cocycle} & \rqleansource{RationalQuivers/Species/Rep/FunctorH.lean}{Species/\allowbreak{}Rep/\allowbreak{}FunctorH.\allowbreak{}lean}{3423}{3535} \\
  \hyperref[thm:species_equivalence]{Theorem~\ref*{thm:species_equivalence}} & \rqleandecl{RationalQuiver.\allowbreak{}species\_\allowbreak{}equivalence}{RationalQuivers.RationalQuiver.species_equivalence} & \rqleansource{RationalQuivers/Species/Rep/Equivalence.lean}{Species/\allowbreak{}Rep/\allowbreak{}Equivalence.\allowbreak{}lean}{1338}{1517} \\
  \hyperref[cor:nilpotentrelations]{Corollary~\ref*{cor:nilpotentrelations}} & \rqleandecl{RationalQuiver.\allowbreak{}FObj\_\allowbreak{}isNilpotent\_\allowbreak{}iff}{RationalQuivers.RationalQuiver.FObj_isNilpotent_iff} & \rqleansource{RationalQuivers/Species/Rep/SpeciesBridge.lean}{Species/\allowbreak{}Rep/\allowbreak{}SpeciesBridge.\allowbreak{}lean}{737}{764} \\
  \hyperref[def:sotworational]{Definition on p.~\pageref*{def:sotworational}} & \rqleandecl{HC.\allowbreak{}SO2\_\allowbreak{}QStructure}{RationalQuivers.HC.SO2_QStructure} & \rqleansource{RationalQuivers/HC/Pair.lean}{HC/\allowbreak{}Pair.\allowbreak{}lean}{88}{112} \\
  \hyperref[def:pairE]{Definition~\ref*{def:pairE}} & \rqleandecl{HC.\allowbreak{}pairE}{RationalQuivers.HC.pairE} & \rqleansource{RationalQuivers/HC/PairE.lean}{HC/\allowbreak{}PairE.\allowbreak{}lean}{260}{283} \\
  \hyperref[def:sqrtminusone]{Definition~\ref*{def:sqrtminusone}} & \rqleandecl{HC.\allowbreak{}ImagUnit}{RationalQuivers.HC.ImagUnit} & \rqleansource{RationalQuivers/HC/ImagUnit.lean}{HC/\allowbreak{}ImagUnit.\allowbreak{}lean}{33}{61} \\
  \hyperref[def:sltwotriple]{\S\ref*{def:sltwotriple}} & \rqleandecl{HC.\allowbreak{}sl2Triple}{RationalQuivers.HC.sl2Triple} & \rqleansource{RationalQuivers/HC/Sl2.lean}{HC/\allowbreak{}Sl2.\allowbreak{}lean}{148}{251} \\
  \hyperref[prop:casimirdescends]{Proposition on p.~\pageref*{prop:casimirdescends}} & \rqleandecl{HC.\allowbreak{}casimir\_\allowbreak{}descends}{RationalQuivers.HC.casimir_descends} & \rqleansource{RationalQuivers/HC/Casimir.lean}{HC/\allowbreak{}Casimir.\allowbreak{}lean}{57}{118} \\
  \hyperref[def:hcmodule]{Definition on p.~\pageref*{def:hcmodule}} & \rqleandecl{HC.\allowbreak{}HarishChandraModule}{RationalQuivers.HC.HarishChandraModule} & \rqleansource{RationalQuivers/HC/Classical.lean}{HC/\allowbreak{}Classical.\allowbreak{}lean}{254}{386} \\
  \hyperref[def:geninfchar]{Definition~\ref*{def:geninfchar}} & \rqleandecl{HC.\allowbreak{}HasGeneralizedInfinitesimalCharacter}{RationalQuivers.HC.HasGeneralizedInfinitesimalCharacter} & \rqleansource{RationalQuivers/HC/Casimir.lean}{HC/\allowbreak{}Casimir.\allowbreak{}lean}{252}{283} \\
  \hyperref[def:HClambdaE]{Definition~\ref*{def:HClambdaE}} & \rqleandecl{HC.\allowbreak{}HC}{RationalQuivers.HC.HC} & \rqleansource{RationalQuivers/HC/Category.lean}{HC/\allowbreak{}Category.\allowbreak{}lean}{460}{520} \\
  \hyperref[rmk:HCabelian]{Remark~\ref*{rmk:HCabelian}} & \rqleandecl{Gelfand.\allowbreak{}hc\_\allowbreak{}abelian}{RationalQuivers.Gelfand.hc_abelian} & \rqleansource{RationalQuivers/Gelfand/HCAbelianClaim.lean}{Gelfand/\allowbreak{}HCAbelianClaim.\allowbreak{}lean}{21}{39} \\
  \hyperref[rmk:HClambdaC]{Remark~\ref*{rmk:HClambdaC}} & \rqleandecl{HC.\allowbreak{}ClassicalComparison.\allowbreak{}hcLambda\_\allowbreak{}complex\_\allowbreak{}equiv}{RationalQuivers.HC.ClassicalComparison.hcLambda_complex_equiv} & \rqleansource{RationalQuivers/HC/ClassicalComparison.lean}{HC/\allowbreak{}ClassicalComparison.\allowbreak{}lean}{388}{415} \\
  \hyperref[prop:XY_C_relation]{Proposition~\ref*{prop:XY_C_relation}} & \rqleandecl{HC.\allowbreak{}pow\_\allowbreak{}X\_\allowbreak{}pow\_\allowbreak{}Y\_\allowbreak{}eq\_\allowbreak{}prod\_\allowbreak{}casimir}{RationalQuivers.HC.pow_X_pow_Y_eq_prod_casimir} & \rqleansource{RationalQuivers/HC/Normalization.lean}{HC/\allowbreak{}Normalization.\allowbreak{}lean}{206}{262} \\
  \hyperref[prop:xyscalarplusnilpotent]{Proposition~\ref*{prop:xyscalarplusnilpotent}} & \rqleandecl{HC.\allowbreak{}pow\_\allowbreak{}X\_\allowbreak{}pow\_\allowbreak{}Y\_\allowbreak{}eq\_\allowbreak{}scalar\_\allowbreak{}add\_\allowbreak{}nilpotent}{RationalQuivers.HC.pow_X_pow_Y_eq_scalar_add_nilpotent} & \rqleansource{RationalQuivers/HC/Normalization.lean}{HC/\allowbreak{}Normalization.\allowbreak{}lean}{308}{335} \\
  \hyperref[cor:X_Y_inverses]{Corollary~\ref*{cor:X_Y_inverses}} & \rqleandecl{HC.\allowbreak{}WeightData.\allowbreak{}isUnit\_\allowbreak{}pow\_\allowbreak{}X}{RationalQuivers.HC.WeightData.isUnit_pow_X} & \rqleansource{RationalQuivers/HC/Normalization.lean}{HC/\allowbreak{}Normalization.\allowbreak{}lean}{611}{708} \\
  \hyperref[cor:T_operators]{Corollary~\ref*{cor:T_operators}} & \rqleandecl{HC.\allowbreak{}WeightData.\allowbreak{}Tplus\_\allowbreak{}isUnit}{RationalQuivers.HC.WeightData.Tplus_isUnit} & \rqleansource{RationalQuivers/HC/Normalization.lean}{HC/\allowbreak{}Normalization.\allowbreak{}lean}{744}{800} \\
  \hyperref[def:xstarystar]{Definition~\ref*{def:xstarystar}} & \rqleandecl{HC.\allowbreak{}WeightData.\allowbreak{}xStar}{RationalQuivers.HC.WeightData.xStar} & \rqleansource{RationalQuivers/HC/Normalization.lean}{HC/\allowbreak{}Normalization.\allowbreak{}lean}{803}{820} \\
  \hyperref[prop:unipotentstabilization]{Proposition~\ref*{prop:unipotentstabilization}} & \rqleandecl{Unipotent.\allowbreak{}nilpotency\_\allowbreak{}stabilizationStep}{RationalQuivers.Unipotent.nilpotency_stabilizationStep} & \rqleansource{RationalQuivers/Unipotent/Stabilization.lean}{Unipotent/\allowbreak{}Stabilization.\allowbreak{}lean}{232}{290} \\
  \hyperref[cor:unipotentstabilization]{Corollary~\ref*{cor:unipotentstabilization}} & \rqleandecl{Unipotent.\allowbreak{}stabilization\_\allowbreak{}spec}{RationalQuivers.Unipotent.stabilization_spec} & \rqleansource{RationalQuivers/Unipotent/Stabilization.lean}{Unipotent/\allowbreak{}Stabilization.\allowbreak{}lean}{810}{868} \\
  \hyperref[def:unipotentstabilization]{Definition~\ref*{def:unipotentstabilization}} & \rqleandecl{Unipotent.\allowbreak{}stabilization}{RationalQuivers.Unipotent.stabilization} & \rqleansource{RationalQuivers/Unipotent/Stabilization.lean}{Unipotent/\allowbreak{}Stabilization.\allowbreak{}lean}{1040}{1078} \\
  \hyperref[rmk:unipotentstabilization]{Remark~\ref*{rmk:unipotentstabilization}} & \rqleandecl{Unipotent.\allowbreak{}stabilization\_\allowbreak{}symm}{RationalQuivers.Unipotent.stabilization_symm} & \rqleansource{RationalQuivers/Unipotent/Stabilization.lean}{Unipotent/\allowbreak{}Stabilization.\allowbreak{}lean}{1112}{1139} \\
  \hyperref[rmk:stabilizationaddenda]{Remark~\ref*{rmk:stabilizationaddenda}} & \rqleandecl{Unipotent.\allowbreak{}stabilizationSeq\_\allowbreak{}addendum\_\allowbreak{}via}{RationalQuivers.Unipotent.stabilizationSeq_addendum_via} & \rqleansource{RationalQuivers/Unipotent/Addendum.lean}{Unipotent/\allowbreak{}Addendum.\allowbreak{}lean}{149}{211} \\
  \hyperref[prop:unipotentsquareroots]{Proposition~\ref*{prop:unipotentsquareroots}} & \rqleandecl{Unipotent.\allowbreak{}sqrt\_\allowbreak{}spec}{RationalQuivers.Unipotent.sqrt_spec} & \rqleansource{RationalQuivers/Unipotent/SquareRoots.lean}{Unipotent/\allowbreak{}SquareRoots.\allowbreak{}lean}{199}{229} \\
  \hyperref[rmk:sqrtinitialvalue]{Remark~\ref*{rmk:sqrtinitialvalue}} & \rqleandecl{Unipotent.\allowbreak{}sqrt\_\allowbreak{}initialValue}{RationalQuivers.Unipotent.sqrt_initialValue} & \rqleansource{RationalQuivers/Unipotent/SquareRoots.lean}{Unipotent/\allowbreak{}SquareRoots.\allowbreak{}lean}{708}{729} \\
  \hyperref[cor:sqrtinverses]{Corollary~\ref*{cor:sqrtinverses}} & \rqleandecl{Unipotent.\allowbreak{}sqrt\_\allowbreak{}inv}{RationalQuivers.Unipotent.sqrt_inv} & \rqleansource{RationalQuivers/Unipotent/SquareRoots.lean}{Unipotent/\allowbreak{}SquareRoots.\allowbreak{}lean}{462}{482} \\
  \hyperref[def:squareroots]{Definition~\ref*{def:squareroots}} & \rqleandecl{Unipotent.\allowbreak{}sqrtOfScalarPlusNilpotent}{RationalQuivers.Unipotent.sqrtOfScalarPlusNilpotent} & \rqleansource{RationalQuivers/Unipotent/SquareRoots.lean}{Unipotent/\allowbreak{}SquareRoots.\allowbreak{}lean}{486}{513} \\
  \hyperref[rmk:sqrtgaloisuniqueness]{Remark~\ref*{rmk:sqrtgaloisuniqueness}} & \rqleandecl{Unipotent.\allowbreak{}sqrt\_\allowbreak{}galois\_\allowbreak{}equivariant}{RationalQuivers.Unipotent.sqrt_galois_equivariant} & \rqleansource{RationalQuivers/Unipotent/SquareRoots.lean}{Unipotent/\allowbreak{}SquareRoots.\allowbreak{}lean}{615}{662} \\
  \hyperref[def:GammaE]{Definition~\ref*{def:GammaE}} & \rqleandecl{Gelfand.\allowbreak{}GammaE}{RationalQuivers.Gelfand.GammaE} & \rqleansource{RationalQuivers/Gelfand/Target.lean}{Gelfand/\allowbreak{}Target.\allowbreak{}lean}{36}{57} \\
  \hyperref[lem:GammaEbasechange]{Lemma~\ref*{lem:GammaEbasechange}} & \rqleandecl{Gelfand.\allowbreak{}GammaE\_\allowbreak{}eq\_\allowbreak{}baseChange}{RationalQuivers.Gelfand.GammaE_eq_baseChange} & \rqleansource{RationalQuivers/Gelfand/TargetBaseChange.lean}{Gelfand/\allowbreak{}TargetBaseChange.\allowbreak{}lean}{516}{543} \\
  \hyperref[def:EEsplit]{Definition~\ref*{def:EEsplit}} & \rqleandecl{Gelfand.\allowbreak{}Split.\allowbreak{}E\_\allowbreak{}F\_\allowbreak{}construction}{RationalQuivers.Gelfand.Split.E_F_construction} & \rqleansource{RationalQuivers/Gelfand/Split/Construction.lean}{Gelfand/\allowbreak{}Split/\allowbreak{}Construction.\allowbreak{}lean}{1679}{1739} \\
  \hyperref[prop:splitequivalence]{Proposition~\ref*{prop:splitequivalence}} & \rqleandecl{Gelfand.\allowbreak{}Split.\allowbreak{}equivalence}{RationalQuivers.Gelfand.Split.equivalence} & \rqleansource{RationalQuivers/Gelfand/Split/EssSurj.lean}{Gelfand/\allowbreak{}Split/\allowbreak{}EssSurj.\allowbreak{}lean}{2971}{2999} \\
  \hyperref[rmk:splitequivalenceanyfield]{Remark~\ref*{rmk:splitequivalenceanyfield}} & \rqleandecl{Gelfand.\allowbreak{}split\_\allowbreak{}equivalence\_\allowbreak{}anyField}{RationalQuivers.Gelfand.split_equivalence_anyField} & \rqleansource{RationalQuivers/Gelfand/SplitAnyField.lean}{Gelfand/\allowbreak{}SplitAnyField.\allowbreak{}lean}{20}{36} \\
  \eqref{eq:conjugatenormalizedcomparison} & \rqleandecl{Gelfand.\allowbreak{}RationalStructure.\allowbreak{}conjugate\_\allowbreak{}diagram\_\allowbreak{}commutes}{RationalQuivers.Gelfand.RationalStructure.conjugate_diagram_commutes} & \rqleansource{RationalQuivers/Gelfand/RationalStructure/Conjugate.lean}{Gelfand/\allowbreak{}RationalStructure/\allowbreak{}Conjugate.\allowbreak{}lean}{102}{163} \\
  \eqref{eq:varphiplus}, \eqref{eq:varphiminus}, \eqref{eq:varphistar} & \rqleandecl{Gelfand.\allowbreak{}RationalStructure.\allowbreak{}varphi\_\allowbreak{}spec}{RationalQuivers.Gelfand.RationalStructure.varphi_spec} & \rqleansource{RationalQuivers/Gelfand/RationalStructure/Stabilization.lean}{Gelfand/\allowbreak{}RationalStructure/\allowbreak{}Stabilization.\allowbreak{}lean}{2171}{2245} \\
  \hyperref[thm:EEequivalence]{Theorem~\ref*{thm:EEequivalence}} & \rqleandecl{Gelfand.\allowbreak{}gelfand\_\allowbreak{}equivalence\_\allowbreak{}pos}{RationalQuivers.Gelfand.gelfand_equivalence_pos} & \rqleansource{RationalQuivers/Gelfand/Main.lean}{Gelfand/\allowbreak{}Main.\allowbreak{}lean}{74}{94} \\
  \hyperref[thm:cyclicEEequivalence]{Theorem~\ref*{thm:cyclicEEequivalence}} & \rqleandecl{Gelfand.\allowbreak{}gelfand\_\allowbreak{}equivalence\_\allowbreak{}zero}{RationalQuivers.Gelfand.gelfand_equivalence_zero} & \rqleansource{RationalQuivers/Gelfand/Main.lean}{Gelfand/\allowbreak{}Main.\allowbreak{}lean}{96}{117} \\
  \hyperref[rmk:EEabelian]{Remark~\ref*{rmk:EEabelian}} & \rqleandecl{Gelfand.\allowbreak{}gelfand\_\allowbreak{}equivalence\_\allowbreak{}abelian}{RationalQuivers.Gelfand.gelfand_equivalence_abelian} & \rqleansource{RationalQuivers/Gelfand/MainAbelian.lean}{Gelfand/\allowbreak{}MainAbelian.\allowbreak{}lean}{18}{35} \\
  \hyperref[thm:speciesE]{Theorem~\ref*{thm:speciesE}} & \rqleandecl{Gelfand.\allowbreak{}Species.\allowbreak{}speciesE}{RationalQuivers.Gelfand.Species.speciesE} & \rqleansource{RationalQuivers/Gelfand/Species/Equivalence.lean}{Gelfand/\allowbreak{}Species/\allowbreak{}Equivalence.\allowbreak{}lean}{493}{528} \\
  \hyperref[prop:speciesErelations]{proof of Theorem~\ref*{thm:speciesE}} & \rqleandecl{Gelfand.\allowbreak{}Species.\allowbreak{}relation\_\allowbreak{}baseChange\_\allowbreak{}induced}{RationalQuivers.Gelfand.Species.relation_baseChange_induced} & \rqleansource{RationalQuivers/Gelfand/Species/RelationBaseChange.lean}{Gelfand/\allowbreak{}Species/\allowbreak{}RelationBaseChange.\allowbreak{}lean}{261}{321} \\
  \hyperref[thm:EEQQequivalence]{Theorem~\ref*{thm:EEQQequivalence}} & \rqleandecl{Gelfand.\allowbreak{}Rational.\allowbreak{}E\_\allowbreak{}Q\_\allowbreak{}equivalence}{RationalQuivers.Gelfand.Rational.E_Q_equivalence} & \rqleansource{RationalQuivers/Gelfand/Rational.lean}{Gelfand/\allowbreak{}Rational.\allowbreak{}lean}{49}{71} \\
  \hyperref[thm:cyclicEEQQequivalence]{Theorem~\ref*{thm:cyclicEEQQequivalence}} & \rqleandecl{Gelfand.\allowbreak{}Rational.\allowbreak{}cyclic\_\allowbreak{}E\_\allowbreak{}Q\_\allowbreak{}equivalence}{RationalQuivers.Gelfand.Rational.cyclic_E_Q_equivalence} & \rqleansource{RationalQuivers/Gelfand/Rational.lean}{Gelfand/\allowbreak{}Rational.\allowbreak{}lean}{73}{92} \\
  \hyperref[cor:gelfandspecies]{Corollary~\ref*{cor:gelfandspecies}} & \rqleandecl{Gelfand.\allowbreak{}Species.\allowbreak{}species\_\allowbreak{}equivalence\_\allowbreak{}intrinsic}{RationalQuivers.Gelfand.Species.species_equivalence_intrinsic} & \rqleansource{RationalQuivers/Gelfand/Species/Rational.lean}{Gelfand/\allowbreak{}Species/\allowbreak{}Rational.\allowbreak{}lean}{21}{43} \\
  \hyperref[cor:cyclicspecies]{Corollary~\ref*{cor:cyclicspecies}} & \rqleandecl{Gelfand.\allowbreak{}Species.\allowbreak{}cyclic\_\allowbreak{}species\_\allowbreak{}equivalence}{RationalQuivers.Gelfand.Species.cyclic_species_equivalence} & \rqleansource{RationalQuivers/Gelfand/Species/Rational.lean}{Gelfand/\allowbreak{}Species/\allowbreak{}Rational.\allowbreak{}lean}{59}{91} \\
  \hyperref[rmk:classicalcases]{Remark~\ref*{rmk:classicalcases}} & \rqleandecl{Gelfand.\allowbreak{}Rational.\allowbreak{}gelfand\_\allowbreak{}equivalence\_\allowbreak{}complex}{RationalQuivers.Gelfand.Rational.gelfand_equivalence_complex} & \rqleansource{RationalQuivers/Gelfand/Rational.lean}{Gelfand/\allowbreak{}Rational.\allowbreak{}lean}{96}{119} \\
  \hyperref[prop:basechange]{Proposition~\ref*{prop:basechange}} & \rqleandecl{Gelfand.\allowbreak{}BaseChange.\allowbreak{}basechange}{RationalQuivers.Gelfand.BaseChange.basechange} & \rqleansource{RationalQuivers/Gelfand/BaseChange/Native.lean}{Gelfand/\allowbreak{}BaseChange/\allowbreak{}Native.\allowbreak{}lean}{805}{860} \\
  \hyperref[def:HClambdaDE]{Definition~\ref*{def:HClambdaDE}} & \rqleandecl{Gelfand.\allowbreak{}HCTwisted}{RationalQuivers.Gelfand.HCTwisted} & \rqleansource{RationalQuivers/Gelfand/TwistedForms.lean}{Gelfand/\allowbreak{}TwistedForms.\allowbreak{}lean}{43}{62} \\
  \hyperref[def:GammaDE]{Definition~\ref*{def:GammaDE}} & \rqleandecl{Gelfand.\allowbreak{}GammaTwisted}{RationalQuivers.Gelfand.GammaTwisted} & \rqleansource{RationalQuivers/Gelfand/TwistedForms.lean}{Gelfand/\allowbreak{}TwistedForms.\allowbreak{}lean}{165}{180} \\
  \hyperref[thm:twistedEEequivalence]{Theorem~\ref*{thm:twistedEEequivalence}} & \rqleandecl{Gelfand.\allowbreak{}gelfand\_\allowbreak{}equivalence\_\allowbreak{}twisted}{RationalQuivers.Gelfand.gelfand_equivalence_twisted} & \rqleansource{RationalQuivers/Gelfand/TwistedForms.lean}{Gelfand/\allowbreak{}TwistedForms.\allowbreak{}lean}{226}{252} \\
  \hyperref[rmk:repdiagram]{Remark on p.~\pageref*{rmk:repdiagram}} & \rqleandecl{Examples.\allowbreak{}repDiagramConvention}{RationalQuivers.Examples.repDiagramConvention} & \rqleansource{RationalQuivers/Examples/DiagramConvention.lean}{Examples/\allowbreak{}DiagramConvention.\allowbreak{}lean}{86}{221} \\
  \hyperref[ex:finitedimensional]{\S\ref*{ex:finitedimensional}} & \rqleandecl{Examples.\allowbreak{}F}{RationalQuivers.Examples.F} & \rqleansource{RationalQuivers/Examples/FiniteDimensional.lean}{Examples/\allowbreak{}FiniteDimensional.\allowbreak{}lean}{166}{186} \\
  \hyperref[ex:discreteseries]{\S\ref*{ex:discreteseries}} & \rqleandecl{Examples.\allowbreak{}D}{RationalQuivers.Examples.D} & \rqleansource{RationalQuivers/Examples/DiscreteSeries.lean}{Examples/\allowbreak{}DiscreteSeries.\allowbreak{}lean}{247}{281} \\
  \hyperref[ex:principalseries]{\S\ref*{ex:principalseries}} & \rqleandecl{Examples.\allowbreak{}P}{RationalQuivers.Examples.P} & \rqleansource{RationalQuivers/Examples/PrincipalSeries.lean}{Examples/\allowbreak{}PrincipalSeries.\allowbreak{}lean}{248}{277} \\
}